\documentclass[a4paper,reqno,11pt]{amsart}

\usepackage{amsmath,amsfonts,amssymb,amsthm,mathtools,mathrsfs}
\usepackage[T1]{fontenc}
\usepackage{lmodern}
\usepackage{microtype}
\usepackage{needspace}
\usepackage{etoolbox}
\usepackage[a4paper,scale={0.72,0.74},marginratio={1:1},
  footskip=7mm,headsep=10mm]{geometry}
\usepackage[hidelinks]{hyperref}
\usepackage{xcolor}
\usepackage{cancel}
\usepackage[normalem]{ulem}
\usepackage{tikz}
\usepackage{graphicx}
\usepackage{scalerel,stackengine}
\usepackage{mathabx}

\makeatletter
\patchcmd{\@setauthors}{\footnotesize}{\small}{}{}
\patchcmd{\@setauthors}{\MakeUppercase{\authors}}{%
  \patchcmd{\authors}{Alexander Dunlap}{\\Alexander Dunlap}{}{}%
  \patchcmd{\authors}{Rongfeng Sun}{\\Rongfeng Sun}{}{}%
  \authors}{}{}
\patchcmd{\@setaddresses}{\nobreak\addvspace\bigskipamount}
  {\addvspace\bigskipamount}{}{}
\def\@secnumfont{\bfseries}
\def\section{\@startsection{section}{1}%
  \z@{.9\linespacing\@plus\linespacing}{.5\linespacing}%
  {\normalfont\large\bfseries\centering}}
\def\subsection{\@startsection{subsection}{2}%
  \z@{.5\linespacing\@plus.7\linespacing}{-.5em}%
  {\normalfont\bfseries}}
\def\l@subsection{\@tocline{2}{0pt}{2.5pc}{3pc}{}}
\renewcommand{\tocsection}[3]{\indentlabel{\@ifnotempty{#2}{#2.\quad}}#3}
\renewenvironment{proof}[1][\proofname]{\par
\pushQED{\qed}%
\normalfont \topsep4\p@\@plus4\p@\relax
\trivlist
\item[\hskip\labelsep\bfseries #1\@addpunct{.}]\ignorespaces
}{%
\popQED\endtrivlist\@endpefalse
}
\makeatother

\DeclareFontFamily{U}{mathx}{\hyphenchar\font45}
\DeclareFontShape{U}{mathx}{m}{n}{
      <5> <6> <7> <8> <9> <10>
      <10.95> <12> <14.4> <17.28> <20.74> <24.88> mathx10
      }{}
\DeclareSymbolFont{mathx}{U}{mathx}{m}{n}
\DeclareFontSubstitution{U}{mathx}{m}{n}
\DeclareMathAccent{\widecheck}{0}{mathx}{"71}

\numberwithin{equation}{section}
\newtheoremstyle{mytheorem}
  {6pt plus 2pt minus 1pt}{6pt plus 2pt minus 1pt}
  {\itshape}{}{\bfseries}{.}{0.5em}{}
\theoremstyle{mytheorem}
\newtheorem{theorem}{Theorem}[section]
\newtheorem{theoremA}{Theorem}

\newtheorem{proposition}[theorem]{Proposition}
\newtheorem{lemma}[theorem]{Lemma}
\newtheorem{corollary}[theorem]{Corollary}
\newtheorem{remark}[theorem]{Remark}
\newtheorem{claim}[theorem]{Claim}
\theoremstyle{definition}
\newtheorem{definition}[theorem]{Definition}

\newcommand{\bbP}{\mathbb{P}}
\newcommand{\bbE}{\mathbb{E}}
\newcommand{\rmP}{\mathrm{P}}
\newcommand{\rmE}{\mathrm{E}}
\newcommand{\R}{\mathbb R}
\newcommand{\Z}{\mathbb Z}
\newcommand{\one}{\mathbf 1}
\newcommand{\zero}{\mathbf 0}
\newcommand{\e}{\mathrm e}
\newcommand{\Var}{\operatorname{Var}}
\newcommand{\Cov}{\operatorname{Cov}}
\newcommand{\cF}{\mathcal F}

\newcommand{\cZ}{\mathcal Z}
\newcommand{\cH}{\mathcal H}

\newcommand{\cL}{\mathcal L}
\newcommand{\cG}{\mathcal G}
\newcommand{\dd}{\mathrm d}
\newcommand{\gep}{\varepsilon}

\newcommand{\N}{\mathbb{N}}

\newcommand{\scrH}{\mathscr H}
\newcommand{\scrZ}{\mathscr Z}

\newcommand{\cst}{\mathsf c_0}

\newcommand{\lact}{\mathbin{\triangleleft}}
\newcommand{\ract}{\mathbin{\triangleright}}

\newcommand\pig[1]{\scalerel*[5pt]{\big#1}{%
  \ensurestackMath{\addstackgap[1.5pt]{\big#1}}}}
\newcommand\pigl[1]{\mathopen{\pig{#1}}}
\newcommand\pigr[1]{\mathclose{\pig{#1}}}

\title[Critical 2D KPZ]{The Critical 2D KPZ Process}
\author[]{Quentin Berger}
\address[Quentin Berger]{Laboratoire d'Analyse, G\'eom\'etrie et Applications, CNRS UMR 7539, Universit\'e Sorbonne Paris Nord, 99 avenue J.-B. Cl\'ement, 93436 Villetaneuse, France, and Institut Universitaire de France}
\email{quentin.berger@math.univ-paris13.fr}

\author[]{Francesco Caravenna}
\address[Francesco Caravenna]{Dipartimento di Matematica e Applicazioni, Universit\`a degli Studi di Milano-Bicocca, via Cozzi 55, 20125 Milano, Italy}
\email{francesco.caravenna@unimib.it}

\author[]{Jeremy Clark}
\address[Jeremy Clark]{Department of Mathematics, University of Mississippi, University, MS 38677, USA}
\email{jeremy@olemiss.edu}

\author[]{Alexander Dunlap}
\address[Alexander Dunlap]{Department of Mathematics, Duke University, Durham, NC 27708, USA}
\email{alexander.dunlap@duke.edu}

\author[]{Yu Gu}
\address[Yu Gu]{Department of Mathematics, University of Maryland, College Park, MD 20742, USA}
\email{ygu7@umd.edu}

\author[]{Ziyang Liu}
\address[Ziyang Liu]{The Hong Kong University of Science and Technology, Hong Kong}
\email{mazyliu@ust.hk}

\author[]{Shuta Nakajima}
\address[Shuta Nakajima]{Department of Mathematics, Keio University, 3-14-1 Hiyoshi, Yokohama 2238522, Kanagawa, Japan}
\email{njima@keio.jp}

\author[]{Rongfeng Sun}
\address[Rongfeng Sun]{Department of Mathematics, National University of Singapore, 10 Lower Kent Ridge Road, Singapore 119076}
\email{matsr@nus.edu.sg}

\author[]{Li-Cheng Tsai}
\address[Li-Cheng Tsai]{Department of Mathematics, University of Utah, 155 S 1400 E, Salt Lake City, UT 84112, USA}
\email{licheng.tsai@utah.edu}

\author[]{Zuodi Xie}
\address[Zuodi Xie]{Department of Mathematics, University of Warwick, Coventry CV4 7AL, UK}
\email{zuodi-xie.xie@warwick.ac.uk}

\author[]{Nikos Zygouras}
\address[Nikos Zygouras]{Department of Mathematics, University of Warwick, Coventry CV4 7AL, UK}
\email{N.Zygouras@warwick.ac.uk}

\date{\today}
\hypersetup{pdftitle={The Critical 2D KPZ Process},pdfauthor={}}

\begin{document}
\begin{abstract}
This work constructs a non-trivial solution to the ill-defined two-dimensional Kardar-Parisi-Zhang (KPZ) equation at criticality:
a distribution-valued stochastic process on the plane, which we call the \emph{Critical 2D KPZ Process}.
It arises as the limit in law of centred log partition functions of two-dimensional Gaussian directed polymers at criticality, viewed as discretised KPZ solutions.
It is also a measurable function of the Critical 2D Stochastic Heat Flow (SHF) via a renormalised Cole--Hopf transformation, obtained as the limit of the centred logarithm of the spatially mollified SHF as the smoothing is removed.
The Critical 2D KPZ process is \emph{non-Gaussian} and its covariance displays a log-log blow-up near the diagonal.
Gaussian fluctuations emerge when averaging on shrinking spatial scales.
\end{abstract}

\maketitle

\subsection*{AI statement.}

The results presented here were obtained through extensive interactions with generative AI (ChatGPT~5.5, 5.6 and~6); see Appendix~\ref{app:AI-declaration} for more details.
The proof builds on key recent and older advances in our field, combined in novel ways with substantial AI input. We view this manuscript as a ``community project’’ that represents more than a decade’s collective effort.
We have worked on making the presentation as accessible as possible, highlighting the main ideas and their relations to the literature.

\setcounter{tocdepth}{1}
\tableofcontents

\section{Introduction and main results}\label{sec:model}

The Kardar--Parisi--Zhang (KPZ) equation~\cite{KPZ86} describes the
evolution of a randomly growing interface. Formally, it can be written as
\begin{equation}\label{eq:formal-KPZ}
 \partial_t h(t,x)=\frac12\Delta h(t,x)
       +\frac12|\nabla h(t,x)|^2+\beta\xi(t,x),
 \qquad t>0,\quad x\in\R^d.
\end{equation}
Here $h$ is the interface height, $\xi$ is a space-time white noise, and
$\beta>0$ is the noise strength. The equation is a central model in the
KPZ universality class, which includes random growth models, interacting
particle systems and directed polymers. Its one-dimensional theory has
been studied extensively, with precise results on fluctuation laws and
scaling limits; see the surveys~\cite{Corwin12,QS15}.

The Cole--Hopf transformation $h=\log u$ connects \eqref{eq:formal-KPZ}
to the multiplicative stochastic heat equation (SHE)
\begin{equation}\label{eq:formal-SHE}
 \partial_t u(t,x)=\frac12\Delta u(t,x)+\beta u(t,x)\xi(t,x).
\end{equation}
The roughness of white noise makes these equations singular: even in
dimension one, the height is not differentiable and $|\nabla h|^2$ has
no classical meaning, although the multiplicative SHE admits an
It\^o solution. Major advances in singular stochastic partial
differential equations (SPDEs) include Hairer's construction of the
one-dimensional KPZ equation~\cite{Hairer13} and theory of regularity
structures~\cite{Hairer14}, paracontrolled distributions~\cite{GIP15},
and the renormalisation-group approaches of Kupiainen~\cite{Kup16}
and Duch~\cite{Duch25}. Energy solutions provide another approach
for one-dimensional stochastic Burgers/KPZ equations in stationary
settings~\cite{GJ14,GP18}.
Dimension two is critical for both KPZ and the multiplicative SHE,
and lies outside the scope of these general subcritical constructions.

In recent years, there has been significant progress in making sense of the
solution of the SHE in the critical dimension \(d=2\). Bertini and Cancrini~\cite{BC98}
first computed the limiting covariance of the solution of mollified 2D SHE in what would
later be called the critical window. Motivated by the study of the directed polymer model
and its phase transition, Caravenna, Sun, and Zygouras~\cite{CSZ17b} showed that the
regularised 2D SHE exhibits a phase transition on an
\emph{intermediate-disorder scale}. Below the critical point, the one-point distribution of the
solution converges to a log-normal limit, while at and above the critical
point, it converges to 0. Below the same critical point, the solutions of the regularised
SHE and the KPZ equation, regarded as random fields, were shown to have Gaussian fluctuations
that belong to the Edwards-Wilkinson universality class~\cite{CSZ17b,CD20,CSZ20,Gu20}.
In a finer critical window around the critical point, \cite{CSZ19b} showed that third moments of
the regularised SHE converge to finite limits, and Gu, Quastel, and Tsai~\cite{GQT21}
subsequently extended it to all finite moments. However, these moments diverge too fast to
show that the solution of the regularised SHE in the critical window converges to a unique limit.

The convergence to a unique limit, named the {\em Critical 2D Stochastic Heat Flow} (SHF),
was finally achieved in \cite{CSZ23} for the field of directed polymer partition functions, which
is the solution of the 2D SHE regularised through space-time discretisation. Thus the SHF gives
a meaning to the critical 2D SHE. Subsequently, Tsai~\cite{Tsai24} gave an axiomatic characterisation of
the SHF and showed that it is also the limit of the noise-mollified 2D SHE. Recently, there have also been
interesting universality results for the SHF~\cite{DHP26}.

There has also been recent progress on other critical singular SPDEs,
including Gaussian scaling limits for the two-dimensional anisotropic KPZ
and stochastic Burgers equations~\cite{CET23,CGT24,CMT25,CKM26}.
Gaussian fluctuation results have also been obtained for the weakly coupled
four-dimensional Anderson model~\cite{GR26,DS26}.

Making sense of the solution of 2D KPZ at criticality,
however, remains a challenge.
One cannot simply take the logarithm of the SHF because it is
a singular random measure~\cite{CSZ25}.
A recent result of Gu and Tsai~\cite{GT26}
created a significant opening towards defining the Critical 2D KPZ Process.
They showed
that the logarithm of the spatial average of the SHF on scale $\varepsilon$,
after suitable deterministic centring, has asymptotically Gaussian
fluctuations of order $\sqrt{\log\log\varepsilon^{-1}}$.

We focus here on the directed polymer partition functions with Gaussian disorder,
which solve a space-time discretised stochastic heat equation, while their logarithms are corresponding discrete KPZ solutions.
Our main result is to construct the \emph{Critical 2D KPZ Process} as the unique limit of
the logarithms of these partition functions,
after centring and
without any further scaling.
We also show that the centred
logarithm of the mollified SHF converges to the same field.
These convergences hold jointly at finitely many fixed
positive times, hence they define a \textit{limit process},
which is a measurable function of the SHF.

Two inputs play a central role in the construction.
Liu and Xie~\cite{LX26} recently made a crucial
observation that a Gaussian convexity property~\cite{AC15,Chen26}
can be applied to the directed
polymer with Gaussian disorder and hence to the SHF.
This yields sharp control on negative moments
of the polymer partition functions as well as of the mass of the SHF
on any ball.
Another fundamental ingredient
is Gu and Tsai's approximation and expansion of the SHF
as a product of independent contributions
over a geometric sequence of disjoint time intervals~\cite{GT26},
which can be transferred to polymer partition functions.

A key novel ingredient in the proof
is to introduce a factorisation at two times of the partition functions.
This allows us to control sharply the effect of erasing the disorder in a small time interval, first at a local level
(by covariance bounds derived from Gaussian convexity coupled to sharp moment estimates),
then at a global level (by using the approximation from \cite{GT26} to control how
discrepancy in the initial condition is propagated in time through the SHF).
A sketch of the proof strategy is given in
Section~\ref{sec:decomposition}.

We now define the relevant objects and the critical window, and we state the main results.
An overview of earlier work
and connections with the present paper is given in Section~\ref{subsec:related-work}.

\subsection{Directed polymers and the Critical 2D Stochastic Heat Flow}
Let us introduce the directed polymer model,
focusing on its partition functions:
we refer to Comets~\cite{Com17} and the recent
survey~\cite{Zyg24} for general background.

Let $\omega=(\omega(j,z))_{j\in\mathbb N,z\in\Z^2}$ be independent
standard Gaussian random variables, with law denoted $\bbP$ and expectation $\bbE$.
These variables play
the role of a random environment, or \emph{disorder}.
Let also $S$ be the simple symmetric
random walk on $\Z^2$.
We write
$\rmP_{a,z}$ and $\rmE_{a,z}$ for its law and expectation from
$S_a=z$, and we abbreviate $\rmP_z\coloneqq \rmP_{0,z}$, $\rmE_z\coloneqq \rmE_{0,z}$
and $\rmP = \rmP_{0,0}$, $\rmE \coloneqq  \rmE_{0,0}$.

For integers $0\le a\le b$ and $w,z\in\Z^2$,
let us define the directed polymer
\emph{point-to-point partition functions} by
\begin{equation}\label{eq:kernel}
 Z_{a,b}^{\beta}(w,z)
 \coloneqq \rmE_{a,w}\!\bigg[
 \exp\bigg\{\sum_{j=a+1}^{b}
 \bigl(\beta\omega(j,S_j)-\tfrac12\beta^2\bigr)\bigg\}
 \one_{\{S_b=z\}}  \bigg] \,,
\end{equation}
with $Z_{a,a}^{\beta}(w,z)=\one_{\{w=z\}}$.
\emph{Point-to-plane partition functions} are
obtained by removing the final
constraint $\one_{\{S_b=z\}}$, that is
\begin{equation}\label{eq:Zplane0}
      Z_{a,b}^{\beta}(w)
      \coloneqq \rmE_{a,w}\!\bigg[
 \exp\bigg\{\sum_{j=a+1}^{b}
 \bigl(\beta\omega(j,S_j)-\tfrac12\beta^2\bigr)\bigg\}
 \bigg] =\sum_{z\in\Z^2} Z_{a,b}^{\beta}(w,z)  \,.
\end{equation}
Note that $Z_{a,b}^{\beta}(\cdot)$ are \emph{random} functions,
depending on the realisation~$\omega$ of the disorder.
These partition functions can be seen as a
discretised version of the stochastic heat equation~\eqref{eq:formal-SHE}, see~\eqref{eq:ZN-solution} below;
we refer to \cite[Section~1.2]{CSZ20} and \cite[Sections~1.1--1.3]{CSZ26}.

The expected replica overlap between two independent
random walks $S, S'$ is
\begin{equation}\label{eq:RN}
      R_N\coloneqq \sum_{n=1}^N \rmP(S_n = S'_n)
      =\sum_{n=1}^N \rmP(S_{2n}=0)
      =\frac1\pi\log N+O(1) \,.
\end{equation}
We set the disorder strength $\beta=\beta_N$ in the critical window
(and omit it from notation):
\begin{equation}\label{eq:critical-tuning}
 \sigma_N^2\coloneqq \e^{\beta_N^2}-1
 \coloneqq \frac1{R_N}\Bigl(1+\frac{\vartheta}{\log N}\Bigr),
\end{equation}
where $\vartheta\in\R$ is fixed (or can be replaced by a deterministic
$\vartheta_N\to\vartheta$).
Under diffusive rescaling of space-time,
partition functions in the critical window
have a universal scaling limit:
a process $\scrZ^\vartheta = (\scrZ^\vartheta_{s,t}(\dd x, \dd y))_{0 \le s < t}$ of random measures
on $\R^2 \times \R^2$ called the
\emph{Critical 2D Stochastic Heat Flow (SHF)},
see \cite{CSZ23}.
Let us recall this convergence.
We normalise the SHF mean by the standard heat kernel:
\begin{equation*}
	\bbE[\scrZ^\vartheta_{s,t}(\dd x, \dd y)] = g_{t-s}(y-x) \, \dd x \, \dd y \,,
	\qquad \text{with} \qquad
	g_u(z)\coloneqq \frac1{2\pi u}\exp\Bigl(-\frac{|z|^2}{2u}\Bigr) \,.
\end{equation*}

\smallskip

\newcommand{\cellC}{\mathtt{C}}
For $t \geq 0$,
denote by $[t] \coloneqq  2 \lceil t/2 \rceil$ an even integer approximation,
so that $t \le [t] < t+2$.
To account for the periodicity of the simple random walk,
define the even lattice
\begin{equation*}
	\Z^2_{\rm even}\coloneqq \{z = (z_1, z_2) \in\Z^2:
      z_1 + z_2 \in 2\Z \}
\end{equation*}
and, for $x\in\R^2$, let $[x]$ be a nearest point of
$\Z^2_{\rm even}$ (with a fixed choice to
break ties).
Given $N\in\mathbb{N}$, we partition $\R^2$ into
cells of diameter $O(N^{-1/2})$ defined by
\begin{equation}\label{eq:lattice-cells}
 \cellC_N(z)\coloneqq \{x\in\R^2:[\sqrt{N/2}\,x]=z\},
 \qquad z\in\Z^2_{\rm even} \,.
\end{equation}
Note that each cell $\cellC_N(z)$ is a twisted square
with area $2 \cdot (\sqrt{2/N})^2 = 4/N$.

For continuous space-time parameters
$0 \le s < t$ and $x,y \in \R^2$,
we define the
\emph{diffusively rescaled point-to-point partition functions} by
\begin{equation}\label{eq:ZNxy}
      \cZ^{\vartheta,(N)}_{s,t}(x,y) \coloneqq  \frac N4 Z_{[Ns],[Nt]}^{\beta_N}
 \bigl([\sqrt{N/2}\,x],[\sqrt{N/2}\,y]\bigr)  \,,
\end{equation}
where the scale $\sqrt{N/2}$ accounts for the
coordinate variance $1/2$ of the simple random walk.
Similarly, recalling \eqref{eq:Zplane0}, the
\emph{diffusively rescaled point-to-plane partition functions} are
\begin{equation} \label{eq:ZNx}
      \cZ_{s,t}^{\vartheta,(N)}(x)
      \coloneqq  Z_{[Ns],[Nt]}^{\beta_N}([\sqrt{N/2}\,x])
      = \int_{\R^2} \cZ^{\vartheta,(N)}_{s,t}(x,y) \, \dd y \,.
\end{equation}

Note that $\cZ^{\vartheta,(N)}_{s,t}(x,y)$ and
$\cZ_{s,t}^{\vartheta,(N)}(x)$ are
random functions on $\R^2 \times \R^2$ and $\R^2$,
piecewise constant on cells~$\cellC_N(z)$. As $N\to\infty$
they become very \emph{intermittent} and one cannot expect
to obtain a non-trivial limit in a space of functions.
However, they do have a limit as \emph{generalised functions}
or \emph{distributions},
after integration against suitable test functions.
Since partition functions are nonnegative, one can actually
work in the simpler space of
locally finite \emph{measures}.
This is the main result of \cite{CSZ23}, which we recall as follows\footnote{Note that in \cite{CSZ23} the convention for the partition functions is slightly different, since it does not include the disorder variable at the terminal time.
This however does not affect the result.}.

\begin{theoremA}[\cite{CSZ23}, Theorems 1.1 and 9.1]
\label{thm:SHF-convergence}
In the critical window \eqref{eq:critical-tuning}, jointly for any
finite collection of pairs $0\le s<t$, we have the convergence
in distribution
\begin{equation}\label{eq:SHF-input}
 \cZ^{\vartheta,(N)}_{s,t}(x,y) \,\dd x\,\dd y
 \ \xrightarrow[\ N\to\infty \ ]{d}\
 \scrZ_{s,t}^{\vartheta}(\dd x,\dd y) \,, \qquad
  \cZ^{\vartheta,(N)}_{s,t}(x) \,\dd x
 \ \xrightarrow[\ N\to\infty \ ]{d}\
 \scrZ_{s,t}^{\vartheta}(\dd x) \,,
\end{equation}
in the vague topology of measures 
on $\R^2\times\R^2$ and $\R^2$
(i.e., integrating against compactly supported test
functions).
\end{theoremA}

The diffusively rescaled partition functions $\cZ_{t}^{\vartheta,(N)}$ from~\eqref{eq:ZNxy}
provide a space-time discretisation
of the stochastic heat equation \eqref{eq:formal-SHE}
on a lattice of mesh $O(N^{-1})$ in time
and $O(N^{-1/2})$ in space.
Indeed, a discretised solution of \eqref{eq:formal-SHE}
with initial condition $u_0: \R^2 \to \R^+$
is given by the ``averaged-to-point'' partition function
\begin{equation} \label{eq:ZN-solution}
      \cZ_{t}^{\vartheta,(N)}(u_0, x) \coloneqq  \int_{\R^2}
      \dd y \, u_0(y) \,
      \cZ_{0,t}^{\vartheta,(N)}(y,x) \,.
\end{equation}
For this reason, its limit in law as $N\to\infty$,
\begin{equation} \label{eq:SHF-solution}
      \scrZ_{t}^{\vartheta}(u_0, \dd x) \coloneqq  \int_{\R^2} u_0(y) \,
      \scrZ_{0,t}^{\vartheta}(\dd y, \dd x) ,
\end{equation}
provides a non-trivial ``candidate solution'' to the
2D stochastic heat equation \eqref{eq:formal-SHE}.

\begin{remark}[Noise-mollified SHE]
      \label{rem:noise-mollified-she-0}
Instead of discretising the stochastic heat equation by polymer partition functions, 
one can also regularise~\eqref{eq:formal-SHE} by mollifying the noise in space,
i.e.\ replacing $\xi$ by $\xi_\varepsilon=\rho_\varepsilon*\xi$ with
$\rho_\varepsilon(x)=\varepsilon^{-2}\rho(x/\varepsilon)$, 
where~$\rho$ is a smooth compactly supported probability density.
This equation is well-posed in the It\^o sense
and, for $\vartheta\in\R$,
we can consider its fundamental solutions
$u_{s,t}(x,y)
= u_{s,t}^{\vartheta,(\varepsilon)}(x,y)$
with noise strength
in the critical
window from \cite{BC98}, more
precisely \(\beta^2 = \beta^2_{\gep}(\vartheta)
\coloneqq  \frac{2\pi}{|\log \varepsilon|}(1+\frac{\theta + c_\rho}{2 |\log \varepsilon|})\) for an explicit constant $c_\rho$ \cite[Equation~(1.6)]{Tsai24}
and $\theta=\vartheta-\gamma$ with $\gamma$ the
Euler--Mascheroni constant:
\begin{equation}\label{eq:noise-mollified-she}
 \partial_t u_{s,t}(x,y)
 =\tfrac12\Delta_y u_{s,t}(x,y)
 +\beta_\varepsilon(\vartheta)
 \, \xi_\varepsilon(t,y) \, u_{s,t}(x,y),
 \qquad u_{s,s}(x, \cdot)=\delta_x( \cdot) \,.
\end{equation}
 It was proved by Tsai in \cite[Proposition~1.5]{Tsai24} that the Critical 2D SHF \((\scrZ_{s,t}^{\vartheta}(\dd x, \dd y))_{0\le s\le t}\) is also the limit of the noise-mollified SHE solution \((u_{s,t}^{\vartheta,(\varepsilon)}(x, y) \, \dd x \, \dd y )_{0\le s \le t }\) as \(\gep\downarrow 0\).
\end{remark}

\subsection{Main result: the Critical 2D KPZ Process}

Recall that the Cole--Hopf transformation $h = \log u$
formally maps the KPZ equation~\eqref{eq:formal-KPZ}
to the stochastic heat equation~\eqref{eq:formal-SHE}.
Then, in view of~\eqref{eq:ZN-solution} and~\eqref{eq:SHF-solution} above, there are at least two possible ways
to construct a ``candidate solution'' of the 2D
KPZ equation \eqref{eq:formal-KPZ}.

A first approximation of the sought-after KPZ solution
is the logarithm of partition functions \eqref{eq:ZN-solution},
viewed as discretised solutions of \eqref{eq:formal-SHE},
\emph{with suitable centring}:
for an initial condition $h_0: \R^2 \to \R$, we define
\begin{equation}\label{eq:HN}
\begin{split}
      \cH_t^{\vartheta,(N)}(h_0, x)
 \,\coloneqq \, & \log \cZ_{t}^{\vartheta,(N)}(\e^{h_0}, x)
 \,-\, c_t^{\vartheta,(N)} \\
 &\ \text{with} \quad
 c_t^{\vartheta,(N)} \coloneqq 
 \bbE\log \cZ_{t}^{\vartheta,(N)}(\one , x) \,,
\end{split}
\end{equation}
where the centring \(c_t^{\vartheta,(N)}\) does not depend on \(x\), by translation invariance.
Subtracting~$c_t^{\vartheta,(N)}$ is a form of \emph{renormalisation} 
and amounts to centring $\log \cZ_{t}^{\vartheta,(N)}(\e^{h_0}, x)$
in the special case $h_0 = \zero $.
We stress that this is uniformly close
to centring for general $h_0 \in L^\infty$: one can easily obtain that 
\begin{equation}
      \label{eq:compare-centrings}
 \bigl| c_t^{\vartheta,(N)} - \bbE\bigl[\log \cZ_{t}^{\vartheta,(N)}(\e^{h_0} , x)\bigr] \bigr| \leq \|h_0\|_{\infty} \,,
\end{equation}
see Section~\ref{sec:centering} for details.
Let us also stress that we show in Lemma~\ref{lem:centering} (and the following discussion) that
\begin{equation}
      \label{eq:centering-constant}
      c_t^{\vartheta,(N)} =
      \bbE \log \cZ_{t}^{\vartheta,(N)}(\e^{h_0}, x) + O(1)
      = -\tfrac12\log\log N + O(1) \qquad \text{as } N\to\infty \,.
\end{equation}

A second natural way to make sense of the KPZ solution
is to apply the Cole--Hopf transformation $h = \log u$
to the continuum
``SHF solution'' $\scrZ_{t}^{\vartheta}(u_0, \dd x)$ in
\eqref{eq:SHF-solution}. This is
a measure, but we can regularise it by removing the noise
in a small time interval of size $\delta > 0$
(we call this a temporal mollification),
or equivalently convolving with the heat kernel
\begin{equation}
      \label{eq:heat}
      g_v^{(x)}(y)
      \coloneqq  g_v(y-x)
      \coloneqq  \frac{1}{2\pi v} \, \e^{-\frac{|y-x|^2}{2v}} \,.
\end{equation}
We thus define an approximate renormalised logarithm of the SHF by\footnote{One could replace the time $t-\delta$ by $t$ in \eqref{eq:regularised-H}, which amounts to (spatially)
regularising the SHF solution $\scrZ_t^\vartheta(u_0,\cdot)$
by convolving with the heat kernel $g_\delta$.
This does not affect our main result, see Corollary~\ref{prop:spatial-identification}.}
\begin{equation}\label{eq:regularised-H}
\begin{split}
 {\scrH}_{t}^{\vartheta,(\delta)}(h_0, x)
\,\coloneqq \,
& \log \scrZ_{t-\delta}^\vartheta \bigl( \e^{h_0}, g_\delta^{(x)} \bigr)
\,-\, c_t^{\vartheta,(\delta)} \\
& \ \text{with} \quad
c_t^{\vartheta,(\delta)}
\coloneqq  \bbE \big[ \log \scrZ_{t-\delta}^\vartheta
\bigl( \one, g_\delta^{(x)} \bigr)
 \big] \,.
\end{split}
\end{equation}
We will justify in Section~\ref{sec:FC} that this is
a locally integrable function of~$x$.
As in the discrete case, $c_t^{\vartheta,(\delta)}$ provides an
approximate centring for $h_0\in L^\infty$: replacing
$\one$ with~$\e^{h_0}$ changes the expected logarithm by at most
$\|h_0\|_\infty$, uniformly in $\delta$ and $x$, exactly as in~\eqref{eq:compare-centrings}.

\smallskip

Our main result proves that both approximations
\eqref{eq:HN} and \eqref{eq:regularised-H} converge to the
same limiting stochastic process, which gives a non-trivial
notion of solution to the ill-defined two-dimensional
KPZ equation \eqref{eq:formal-KPZ} at criticality.
Convergence does not take place in a space
of functions but of genuine distributions.
These can no
longer be identified with measures, so we use local versions
of negative Sobolev spaces $H^{-a}_{\rm loc}(\R^2)$,
defined in \eqref{prop:eq-sobolev-convention}.

\begin{theorem}[Critical 2D KPZ Process]\label{thm:main}
Fix $\vartheta\in\R$.
There exists a distribution-valued process on $\R^2$,
labeled by initial condition
$h_0 \in L^\infty(\R^2) \cap C(\R^2)$,
which we denote by
\begin{equation*}
      \scrH^\vartheta
      = (\scrH_t^\vartheta(h_0, \cdot))_{t\ge 0}
\end{equation*}
and call the \emph{Critical 2D KPZ Process}.
It is a measurable function of the SHF,
arising as the limit in probability of regularised
log-SHF, see \eqref{eq:regularised-H}:
jointly at finitely many $t > 0$ and~$h_0$,
\begin{equation}\label{eq:continuous-limit}
\forall \, a>0\,, \qquad
{\scrH}_{t}^{\vartheta,(\delta)}(h_0,\cdot\,)
 \ \xrightarrow[\ \delta\downarrow 0\ ]{\bbP}\
 \scrH_{t}^\vartheta(h_0,\cdot\,)
 \qquad\text{in } H^{-a}_{\rm loc}(\R^2) \,.
\end{equation}
Convergence holds in $L^p(\bbP)$, \(p\in [1,\infty)\), when integrating
both sides against a test function $\varphi\in C_c^\infty(\R^2)$.

The same process $\scrH^\vartheta$
arises as the limit in law of
renormalised
log-partition functions, see \eqref{eq:HN},
with Gaussian disorder and critical tuning
\eqref{eq:critical-tuning}:
jointly at finitely many~$t>0$ and $h_0$,
\begin{equation}\label{eq:main-convergence}
 \forall \, a>0\,, \qquad
 \cH_{t}^{\vartheta,(N)}(h_0,\cdot\,)
 \ \xrightarrow[\ N\to\infty \ ]{d} \
 \scrH_{t}^\vartheta(h_0,\cdot\,)
 \qquad \text{in }
 H^{-a}_{\rm loc}(\R^2) \, .
\end{equation}
\end{theorem}

Let us stress that Theorem~\ref{thm:main} builds \textit{jointly} the Critical 2D KPZ Process
with different (bounded) initial conditions~\(h_0\)
from the \textit{same} mollified Critical 2D SHF,
with a common centring \(c_{t}^{\vartheta,(\delta)}\) that does not depend on the initial condition, see~\eqref{eq:regularised-H}.
We also show below in Remark~\ref{rem:initial-condition}
that if \(h_0,h_0'\) are bounded Lipschitz initial conditions, then \emph{deterministically}, the approximating fields \(\scrH_t^{\vartheta,(\delta)}\) verify, for any~\(\delta>0\) and any \(\psi\in C_c^{\infty}\),
\[
\Bigl| \scrH_t^{\vartheta,(\delta)}(h_0, \psi) - \scrH_t^{\vartheta,(\delta)}(h_0', \psi) \Bigr| \leq \|\psi\|_1 \,\bigl\|h_0'-h_0\bigr\|_{\infty} \,.
\]      
This can be thought of as a reflection of the comparison principle for Hamilton--Jacobi equations in this singular setting.
This property passes to the limit thanks to Theorem~\ref{thm:main}, and in particular we have for any test function \(\psi\in C_c^{\infty}\) that
\[
|\scrH_t^{\vartheta}(h_0, \psi) - \scrH_t^{\vartheta}(h_0', \psi)| \leq \|\psi\|_1\, \|h_0'-h_0\|_{\infty}\,.
\] 
We thus deduce that \(\scrH_t^{\vartheta}(h_0, \cdot)\)
is continuous w.r.t.\ the (bounded Lipschitz) initial condition,
and additionally
that all the 2D KPZ processes \(\scrH_t^{\vartheta}(h_0, \cdot )\) with different bounded initial conditions~\(h_0\)
remain \textit{deterministically} within \(O(1)\) distance.

\begin{remark}[Critical 2D KPZ Process
      from noise-mollified SHE]\label{rem:noise-mollified-she}
Another approximation of the KPZ solution
can be obtained from the fundamental
solutions $u_{s,t}^{\vartheta,(\varepsilon)}(x,y)$
of the space-mollified stochastic heat equation
\eqref{eq:noise-mollified-she} in the critical window,
which converge to the SHF,
see Remark~\ref{rem:noise-mollified-she-0}.
Writing
\begin{equation*}
      u^{\vartheta,(\varepsilon)}_{s,t}(u_0, x) \coloneqq 
      \int_{\R^2} u_0(y) \,
      u_{s,t}^{\vartheta,(\varepsilon)}(y,x)\,\dd y \,,
\end{equation*}
taking logarithm and a suitable centring, we define,
for an initial condition $h_0 \in L^\infty(\R^2)$,
\begin{equation}\label{eq:heps}
\begin{split}
      h^{\vartheta,(\varepsilon)}_t(h_0, x)
      \coloneqq  & \log u^{\vartheta,(\varepsilon)}_{t}(\e^{h_0}, x)
      - c_t^{\vartheta,(\varepsilon)} \\
      & \ \text{with} \quad
      c_t^{\vartheta,(\varepsilon)} \coloneqq 
      \bbE \log u^{\vartheta,(\varepsilon)}_{t}(\one, x) \,.
\end{split}
\end{equation}
By Itô's formula,
$h(t,x) = h^{\vartheta,(\varepsilon)}_t(h_0, x)$
solves the following
\emph{renormalised 2D KPZ equation}:
\begin{equation}\label{eq:renormalised-KPZ}
\begin{split}
\partial_t h(t,x) \,=\,
      & \frac12\Delta h(t,x)
      +\frac12|\nabla h(t,x)|^2+\beta_\varepsilon(\vartheta)  \, \xi_{\varepsilon}(t,x)
      \,-\, C_{t}^{\vartheta,(\varepsilon)} \\
      &\ \text{with} \qquad
      C_{t}^{\vartheta,(\varepsilon)} \coloneqq 
      \frac{\beta_\varepsilon(\vartheta)^2}{2} \varepsilon^{-2} \,
      \|\rho\|_{L^2}^2 \,+\,
      \frac{\dd}{\dd t} c_t^{\vartheta,(\varepsilon)}
       \,.
\end{split}
 \end{equation}
With minor technical adaptations
to the proof, we expect that the Critical 2D KPZ Process
$\scrH^\vartheta
= (\scrH^\vartheta_t(h_0,\cdot))_{t\ge 0}$ is also
the limit in distribution
of $(h^{\vartheta,(\varepsilon)}_t(h_0, \cdot))_{t \ge 0}$
as $\varepsilon \downarrow 0$, jointly at finitely
many $t > 0$.
We omit these adaptations for brevity.
\end{remark}

\subsection{First properties: local structure and non-Gaussianity}

In this section, we focus for simplicity on the flat initial condition \(h_0 =\zero\) and we denote \(\scrH_t^\vartheta(\cdot) \coloneqq \scrH_t^\vartheta(\zero,\cdot)\).
We next examine the law of the limit and the spatial singularities
revealed by averaging it on small scales.
For a nonnegative $\varphi\in C_c^\infty(\R^2)$ let us put
\begin{equation*}
 \varphi_\varepsilon^{(x)}(y)
 \coloneqq \varepsilon^{-2}\varphi((y-x)/\varepsilon).
\end{equation*}

\begin{theorem}[Local structure]\label{thm:properties}
Fix $t>0$ and $\vartheta\in\R$. The random distribution $\scrH_t^\vartheta$ of
Theorem~\ref{thm:main}, with flat initial condition $h_0=\zero$, has the following properties.
\begin{enumerate}
\item \label{i:small-scale}
For every fixed $x\in\R^2$ and nonnegative $\varphi\in C_c^\infty(\R^2)$, we have
\begin{equation}\label{eq:main-local-clt}
 \frac{\scrH_t^\vartheta(\varphi_\varepsilon^{(x)})}
      {\sqrt{\log\log(1/\varepsilon)}}\ \Longrightarrow\
 \mathcal N\bigl(0,({\textstyle \int}\varphi)^2\bigr)
 \qquad \text{ as } \gep\downarrow 0\,.
\end{equation}

\item \label{iii:covariance}
There is a nonnegative radial covariance kernel
$\mathsf C_t^\vartheta\in L^2_{\rm loc}(\R^2)$, continuous off zero, such that
for real $f,g\in C_c^\infty(\R^2)$,
\begin{equation}\label{eq:main-covariance}
 \Cov\bigl(\scrH_t^\vartheta(f),\scrH_t^\vartheta(g)\bigr)
 =\int_{\R^2 \times \R^2} f(x)g(y) \mathsf C_t^\vartheta(x-y)\,\dd x\,\dd y \,.
\end{equation}
Moreover, the covariance kernel has a \(\log\log\) blow-up near the diagonal:
\begin{equation}\label{eq:main-diagonal}
 \lim_{|z|\downarrow0}
 \frac{\mathsf C_t^\vartheta(z)}{\log\log(1/|z|)}=1.
\end{equation}
\end{enumerate}
\end{theorem}

As a corollary of Theorem~\ref{thm:properties}-\eqref{i:small-scale}, we
show that the Critical 2D KPZ Process cannot
locally be a signed measure.

\begin{proposition}
\label{prop:no-local-representative}
Almost surely, $\scrH_t^\vartheta$ cannot be represented as a locally finite
signed Radon-measure.
In particular, it cannot be represented by a locally integrable function
and it cannot be spatially constant.
\end{proposition}

More importantly, the next result shows that the
Critical 2D KPZ Process cannot be Gaussian,
in the stronger sense that no one-dimensional marginal obtained by the action
of~$\scrH_t^\vartheta$ on a nonnegative $C_c^\infty$ test function can be Gaussian.

\begin{theorem}[Non-Gaussianity]\label{thm:nongaussian}
For every $t>0$, $\vartheta\in\R$ and nonzero nonnegative
$\psi\in C_c^\infty(\R^2)$, the random variable
$\scrH_t^\vartheta(\psi)$ is non-Gaussian, including exclusion of
degenerate Gaussian laws.
\end{theorem}

\begin{remark}[Comparison with renormalisation group predictions]
\label{rem:rg-tuning}
Theorems~\ref{thm:main} and~\ref{thm:nongaussian} yield a
non-Gaussian limit in the critical window~\eqref{eq:critical-tuning}
with fixed $\vartheta$.
This may seem in contrast with the renormalisation group
calculation of Frey and T\"auber~\cite{FT94}: the cubic term in
their function $\beta_g$ in~(3.48) would predict that a non-trivial scaling limit requires
a negative shift of $\vartheta$ of order $\log\log N$.
However, that formula contains a misprint:
in dimension \(d=2\) the cubic terms  cancel out\footnote{ The coefficient of $g^2 F_\nu(d)$ inside the parentheses of~(3.48) should rather be $+3(d-1)/(8d)$, cf.\ equations~(3.33) and
(3.46)--(3.47) therein.},
leading to $\beta_g=-g^2+O(g^4)$.
This formula, already stated in~\cite[Eq.~(3.64)]{FT94}
and recovered in~\cite[Eq.~(44)]{Wie97}, shows that
the two-loop renormalisation group calculation is consistent with the
emergence of a non-Gaussian limit in the critical window, with no additional $\log\log$ tuning,
as in our main results.
\end{remark}

Let us end this section by observing that the Critical 2D KPZ Process
inherits the following rotation-translation invariance
and scaling covariance properties from the SHF.

\begin{proposition}\label{prop:invariance}
Let us fix any finite collection \((t_i,\varphi_i)_{1\leq i\leq k}\) of times and test functions.
The Critical 2D KPZ Process \(\scrH^\vartheta\), with flat initial condition $h_0=\zero$, has the following properties:
\begin{enumerate}
\item It is rotation-translation invariant in law: for any \(z\in \R^2\), \(\alpha \in \R\),
\begin{equation*}
      \pigl(\scrH_{t_i}^\vartheta(
      T_{z}^{\alpha}\varphi_i)\pigr)_{i=1}^k
      \stackrel{\rm dist}{=} \pigl(\scrH_{t_i}^\vartheta(\varphi_i)\pigr)_{i=1}^k \,,
\end{equation*}
where \(T_{z}^{\alpha} \varphi(\cdot) = \varphi(z +\e^{i\alpha} \cdot)\) denotes the rotated-translated version of~\(\varphi\).

\item It satisfies the following covariance relation:
for all \(\mathtt{a}>0\)
\begin{equation*}
      \pigl(\scrH_{\mathtt{a} t_i}^\vartheta(S_{\mathtt{a}}\varphi_i)\pigr)_{i=1}^k
      \stackrel{\rm dist}{=} \pigl(\scrH_{t_i}^{\vartheta + \log \mathtt{a}}(\varphi_i)\pigr)_{i=1}^k \,,
\end{equation*}
where \(S_{\mathtt{a}} \varphi(\cdot) = \frac{1}{\mathtt{a}} \varphi(\frac{\cdot}{\sqrt{\mathtt{a}}})\) denotes the
scaled version of~\(\varphi\).
\end{enumerate}
\end{proposition}

\begin{proof}
The distribution of the SHF
is translation invariant,
see \cite[Theorem~1.2]{CSZ23},
and it is also rotation invariant, as follows from the
rotational invariance of the mollified SHE
(with a radial mollifier \(\rho\))
which is preserved by taking the limit in distribution \cite[Proposition~1.5]{Tsai24}.
Then, the Critical 2D KPZ Process
\(\scrH^\vartheta\) inherits these
invariance properties, thanks to the (joint) convergence in distribution from Theorem~\ref{thm:main}.

As far as the second covariance property
is concerned, one can use the scaling
covariance relation of the SHF from~\cite[Theorem~1.2]{CSZ23}.
Recalling \({\scrH}_{t}^{\vartheta,(\delta)}(  \varphi)\) from \eqref{eq:regularised-H} (with flat initial condition \(h_0 = \zero\)),
and observing that
\(g_{\mathtt{a}\delta}^{(x)}(\cdot) = S_{\mathtt{a}} g_{\delta}^{(x/\sqrt{\mathtt{a}})}(\cdot)\),
this yields that
for any \(\delta>0\) and \(\mathtt{a}>0\),
\[
\pigl({\scrH}_{ \mathtt{a}t_i}^{\vartheta,(\mathtt{a}\delta)}(S_\mathtt{a} \varphi_i)\pigr)_{i=1}^k
\stackrel{\rm dist}{=} \pigl({\scrH}_{t_i}^{\vartheta +\log \mathtt{a},(\delta)}(\varphi_i)\pigr)_{i=1}^k \,.
\]
The desired covariance property then follows, thanks to the
(joint) convergence in distribution from Theorem~\ref{thm:main}.
\end{proof}

\subsection{Relation to earlier work}\label{subsec:related-work}

 In this section, we review contributions from the literature that played an important role towards the
  main results in this paper.

Our convergence proof in Theorem~\ref{thm:main} first applies a localisation argument that was used
in \cite[Section~2.1]{CSZ20} to prove Edwards-Wilkinson limit for the 2D KPZ
equation in the subcritical regime (see also \cite{Gu20}). The main idea is to approximate
the field by a local field, whose covariance has fast spatial decay. In the critical window
we consider here, such an approximation by a local field is no longer accurate. Nevertheless,
we use the decomposition~\eqref{eq:critical-decomposition} to separate a local
term, which is analogous to the field considered in the sub-critical regime, and a global term
that is estimated separately by adapting techniques from \cite{GT26} and applying
a Gaussian convexity property \cite{Chen26, LX26}.

A key ingredient in our analysis is the geometric block factorisation that had been used by Gu and Tsai~\cite{GT26}
to prove a central limit theorem for the average of the SHF on shrinking spatial scales. 
Such
decomposition on exponential time scales first appeared in the study of 2D SHE and directed polymer
in the subcritical regime \cite{CSZ17b}, and the idea of approximating by products of independent
block partition functions on different time scales was developed by Cosco and
Donadini~\cite{CD25} (see also \cite{CSZ25} for a different decomposition in the so-called
quasi-critical window). Gu and Tsai~\cite{GT26} extended the approximation idea from \cite{CD25} to
the critical window and developed an expansion to control the error of approximation, which was applied
to the SHF. We will adapt this expansion to lattice polymers and use it to track how
the partition function depends on the polymer endpoint distribution
at an intermediate time. In particular, we compare two such
distributions evolved through the same later disorder, obtaining
Proposition~\ref{prop:PF}. Together with negative-moment and
logarithmic-moment bounds, this comparison controls the global term
in~\eqref{eq:critical-decomposition}.

Passing from partition functions to their logarithms 
requires control of small values, which we obtain via 
negative moments. Liu and Xie~\cite{LX26} obtained sharp 
bounds, up to multiplicative constants, for positive and 
negative moments of spatial averages of the SHF over
shrinking balls, in terms of explicit powers of their corresponding second moments.
In our proof we need sharper 
moment estimates of this sort: Proposition~\ref{prop:VD} shows that, 
for a suitably averaged partition function over a 
short time interval,
the ratio of the second moment to the negative-first moment 
is close to one.
We also use bounds for the 
centred fourth moments of partition functions 
from~\cite{CSZ25}, based on the operator estimates of 
Caravenna, Cottini and Rossi~\cite{CCR25},
to control the errors in the factorisation 
and in the comparison of positive and negative moments.

As in \cite{LX26},  Gaussian convexity 
also plays an important role here, controlling lower tails at criticality,
as well as obtaining spatial decorrelations.
In addition, Gaussian convexity is also used to prove non-Gaussianity of the 
Critical 2D KPZ Process. This convexity structure is 
preserved, in a distributional sense, under weak 
limits~\cite{LX26}. Together with the representation of 
the resulting laws as convex functions of a standard 
Gaussian~\cite{FMP}, this yields a rigidity principle: 
Gaussianity would force certain limiting polymer endpoint 
measures to be deterministic. Using the flow property 
established by Clark and Mian~\cite{CM}, we show that this 
is incompatible with the singular spatial structure of 
the SHF~\cite{CSZ25}.

\subsubsection*{Notational conventions}

We distinguish the random-walk law and expectation $\rmP,\rmE$
from the disorder law and expectation $\bbP,\bbE$.
For real-valued random variables, the $L^p$ norm
is denoted by $\|X\|_{L^p}=(\bbE|X|^p)^{1/p}$.
For functions on $\R^2$ or $\Z^2$, we denote by
$\|f\|_p$ the $L^p$ norm with respect to the
Lebesgue measure or counting measure, respectively.

In the rest of the article, the precise values of the constants \(c,C\) may change from line to line; we may write \(c_{\vartheta},C_{\vartheta}\) if we want to highlight that the constants may depend on some parameter \(\vartheta\).
We also write \(f=O(g)\) if there is a constant \(C<+\infty\) such that \(f(x)\leq C g(x)\) and we write \(f\asymp g\) if \(f=O(g)\) and \(g=O(f)\); we write \(f=O_{\vartheta}(g)\) or \(f\asymp_{\vartheta} g \) if we want to highlight that the constants may depend on some parameter \(\vartheta\).

\section{Strategy of the proof and organisation of the paper}
\label{sec:decomposition}

In this section, we outline the strategy of the proof
of our main result Theorem~\ref{thm:main}
and we describe the organisation of the paper.
We first define some notation.

\subsection{Notation}

The transition probabilities of the simple random walk $S$
on $\Z^2$ are
\begin{equation}\label{eq:rwkernel}
\begin{gathered}
 q_n(z)\coloneqq \rmE_0[\one_{\{S_n=z\}}],\qquad
 q_n^{(w)}(z) \coloneqq 
 q_n(w,z)\coloneqq q_n(z-w)=\rmE_w[\one_{\{S_n=z\}}] \,.
\end{gathered}
\end{equation}
Let $\cF_{m,n}$ be the sigma field generated by the environment
$\omega = (\omega(j,z))$ on time $(m,n]$:
\begin{equation}\label{eq:cF}
 \cF_{m,n}\coloneqq \sigma\{\omega(j,z):m<j\le n,\ z\in\Z^2\} \,.
\end{equation}

Recall the point-to-point partition functions $Z_{a,b}^{\beta}(w,z)$ from \eqref{eq:kernel}, for points $w,z \in \Z^2$.
For an initial measure $\nu$ and a terminal function $f$ on $\Z^2$,
set
\begin{equation}\label{eq:Zav}
\begin{gathered}
	 Z_{a,b}^{\beta}(\nu,f) \coloneqq  \sum_{w,z}\nu(w) \, Z_{a,b}^{\beta}(w,z) \, f(z) \,, \\
       Z_{a,b}^{\beta}(\nu,z)\coloneqq \sum_w\nu(w)Z_{a,b}^{\beta}(w,z) \,,
       \qquad
       Z_{a,b}^{\beta}(w,f)\coloneqq \sum_zZ_{a,b}^{\beta}(w,z)f(z) \,.
\end{gathered}
\end{equation}
Point-to-plane partition functions are obtained for $f = \one$
(often omitted from notation):
\begin{equation}\label{eq:Zplane}
 Z_{a,b}^{\beta}(w)\coloneqq Z_{a,b}^{\beta}(w,\one)
      =\sum_zZ_{a,b}^{\beta}(w,z) \,, \quad\
       Z_{a,b}^{\beta}(\nu)\coloneqq Z_{a,b}^{\beta}(\nu,\one)
       =\sum_{w,z} \nu(w)\, Z_{a,b}^{\beta}(w,z) \,.
\end{equation}
\emph{We henceforth set $\beta = \beta_N$ in the critical window
\eqref{eq:critical-tuning} 
and omit it from notation.}

Combining the notation introduced so far,
we define \emph{a smoothed version $\widehat Z_{a,b}(x)$
of the point-to-plane partition function
$Z_{a,b}(x)$}, where we average out disorder until
time~$a$:
\begin{equation}\label{eq:Wab}
 \widehat Z_{a,b}(x)
 \coloneqq  \bbE[Z_{0,b}(x)\mid\cF_{a,b}]
 = Z_{a,b}(q_a^{(x)},\one)
 =\sum_y q_a(y-x)Z_{a,b}(y,\one) .
\end{equation}
Thus the walk evolves without disorder up to time $a$,
while later disorder remains.

The polymer endpoint distributions,
also called endpoint profiles,
starting from a point $w\in\Z^2$
or a probability~$\nu$ on~$\Z^2$ are
\begin{equation}\label{eq:endpoint}
	 \mu_{a,b}^{(w)}(z)\coloneqq \frac{Z^{\beta}_{a,b}(w,z)}{Z^{\beta}_{a,b}(w)}  \,, \qquad
 \mu_{a,b}^{(\nu)}(z)
 \coloneqq  \frac{Z^{\beta}_{a,b}(\nu,z)}{Z^{\beta}_{a,b}(\nu)} \,.
\end{equation}
Starting from a single point corresponds to $\nu=\delta_w=q_0^{(w)}$.

We will use functional notation for the SHF,
analogous to \eqref{eq:Zav} and \eqref{eq:Zplane}:
\begin{equation} \label{eq:averaged-SHF}
\begin{split}
	\scrZ^\vartheta_{s,t}(\varphi,\psi)
	& \coloneqq  \int_{\R^2\times\R^2} \varphi(x) \, \scrZ^\vartheta_{s,t}(\dd x,
	\dd y) \, \psi(y) \,, \\
	\scrZ^\vartheta_{s,t}(\varphi)
	& \coloneqq  \scrZ^\vartheta_{s,t}(\varphi, \one)
	= \int_{\R^2\times\R^2} \varphi(x) \, \scrZ^\vartheta_{s,t}(\dd x, \dd y) \,.
\end{split}
\end{equation}

When the starting time is~$0$, we often omit it from notation and simply write
\begin{equation*}
	\cF_n \coloneqq  \cF_{0,n} \,, \qquad
	Z_{n}^{\beta} \coloneqq  Z_{0,n}^{\beta} \,, \qquad
	\mu_n^{(w)}\coloneqq \mu_{0,n}^{(w)} \,, \qquad\scrZ^\vartheta_{t}
      \coloneqq  \scrZ^\vartheta_{0,t} \,.
\end{equation*}

\subsection{Strategy of the proof of Theorem~\ref{thm:main}}
\label{subsec:strategy-approximation}

We describe the main ideas in the
proof of the convergence results \eqref{eq:continuous-limit}
and \eqref{eq:main-convergence},
with both sides
evaluated on a test function $\varphi \in C_c(\R^2)$.
We focus for simplicity on time $t=1$ and
initial condition \(h_0 = \zero\);
the extension to finitely many times and test functions and to an initial condition \(h_0 \in L^{\infty}(\R^2) \cap C(\R^2)\) follows the same strategy and is detailed in Section~\ref{sec:joint-convergence}.

Recall
${\scrH}_{t}^{\vartheta,(\delta)}$ from \eqref{eq:regularised-H}
and $\cH_t^{\vartheta,(N)}$ from \eqref{eq:HN},
and let us omit the dependence on \(\vartheta\) in the notation from now on.
We actually prefer to work with \emph{time-reversed versions}
(the distribution is unchanged\footnote{Strictly speaking, this is not true in the discrete setting \textit{i.e.}\ for the polymer partition functions defined in \eqref{eq:kernel}: time-reversal changes the convention of the absence or presence of a disorder variable at the starting or end point.
This however does not affect the results.}): 
let us call ${\scrH}^{(\delta)}(\varphi)$
the $\delta$-regularised log-SHF averaged over the
starting point, with flat terminal condition:
\begin{equation}\label{eq:regularised-Hbis}
\begin{split}
 {\scrH}^{(\delta)}(\varphi)
\,\coloneqq \,
\int_{\R^2} \varphi(x) \,
\pig\{ \log \scrZ_{\delta,1} \bigl( g_\delta^{(x)} \bigr)
\,-\, \bbE \big[ \log \scrZ_{\delta,1}
\bigl( g_\delta^{(0)} \bigr)
 \big] \pig\} \, \dd x \,.
\end{split}
\end{equation}
We similarly let $\cH^{(N)}(\varphi)$ be
the renormalised point-to-plane log-partition
function averaged over the starting point:
restricting for simplicity to even values of $N$, we define
\begin{equation}\label{eq:HNbis}
\begin{split}
      \cH^{(N)}(\varphi)
      &\,\coloneqq \,
      \int_{\R^2}
      \varphi(x) \, \pig\{
      \log \cZ_{1}^{(N)}(x)
       \,-\,
       \bbE\big[ \log \cZ_{1}^{(N)}(0)\big]
       \pig\} \, \dd x \\
       & \,=\,
       \sum_{z\in\Z^2_{\rm even}}
       \varphi^{(N)}(z) \,
       \pig\{
      \log Z_{N}(z)
       \,-\,
       \bbE\big[ \log Z_{N}(0) \big]\,
       \pig\} \,,
\end{split}
\end{equation}
where in the last equality,
recalling the cells $\cellC_N(z)$ from \eqref{eq:lattice-cells},
we identify $\cZ_{1}^{(N)}(x)
= Z_{N}(z)$ for $x \in \cellC_N(z)$
(see \eqref{eq:ZNx} and \eqref{eq:Zplane0})
and we define
\begin{equation}
\label{def:varphi-N}
\varphi^{(N)}(z)
\,\coloneqq\,  \int_{\cellC_N(z)} \varphi(x)\,\dd x
\,\approx\, \tfrac{4}{N} \varphi\pig(\tfrac{z}{\sqrt{N/2}}\pig) \,.
\end{equation}

We now state a preliminary version of Theorem~\ref{thm:main}
which retains the essential challenges.

\begin{theorem}[Critical 2D KPZ Process, preliminary version]
\label{thm:toy}
For any test function $\varphi \in C_c(\R^2)$
there exists a random variable $\scrH(\varphi) \in L^2$
such that,
in the critical scaling~\eqref{eq:critical-tuning}
\begin{equation}\label{eq:toy}
      \scrH^{(\delta)}(\varphi)
      \ \xrightarrow[\ \delta \downarrow 0 \ ]{L^2} \
      \scrH(\varphi) \,, \qquad
      \cH^{(N)}(\varphi)
      \ \xrightarrow[\ N\to \infty \ ]{d} \
      \scrH(\varphi) \,.
\end{equation}
\end{theorem}

\emph{We will prove the two convergences
in \eqref{eq:toy} simultaneously.}
The idea is to interpolate between
the $\delta$-regularised log-SHF
$\scrH^{(\delta)}(\varphi)$ and
the log-partition function $\cH^{(N)}(\varphi)$
by means of a $\delta$-regularised log-partition function:
we thus define, recalling $\widehat Z$ from \eqref{eq:Wab},
\begin{equation}\label{eq:discrete-regularised-H}
\begin{split}
      \cH^{(N;\delta)}(\varphi)
 &\coloneqq  \sum_{z\in\Z^2_{\rm even}}\!\!\!
 \varphi^{(N)}(z) \,
 \pig\{\log\widehat Z_{m,N}(z)-
 \bbE \big[ \log\widehat
 Z_{m,N}(0) \big] \pig\} \\
 & = \sum_{z\in\Z^2_{\rm even}} \!\!\!
 \varphi^{(N)}(z) \,
 \pig\{\log Z_{m,N}(q_{m}^{(z)})-
 \bbE \big[ \log
 Z_{m,N}(q_{m}^{(0)}) \big] \pig\} \quad
 \text{with } \, m \coloneqq  [\delta N] \,.
\end{split}
\end{equation}
Here disorder is removed until time
$m = [\delta N]$, while the distribution
$q_{m}^{(z)}$ converges to $g_\delta^{(x)}$
after spatial rescaling.
Intuitively, we will show that
\begin{equation}\label{eq:interpo}
      \scrH^{(\delta)}(\varphi)
       \ \underset{\rule{0em}{1.2em}\substack{\text{as }N \to \infty\\
       \rule{0em}{.65em}\text{for \emph{fixed}} \; \delta}}{\overset{\rule[-.3em]{0em}{.2em}d}{\approx}} \
       \cH^{(N;\delta)}(\varphi)
       \ \underset{\rule{0em}{1.2em}\substack{\text{as }\delta \downarrow 0\\
       \rule{0em}{.65em}\text{\emph{uniformly}
       for large} \; N}}{\overset{\rule[-.3em]{0em}{.2em}L^2}{\approx}} \
       \cH^{(N)}(\varphi) \,,
\end{equation}
as formulated precisely in the next two results.

\begin{proposition}[Log-polymer to log-SHF]
\label{prop:FC}
Jointly for finitely many $\delta \in (0,1)$, we have
\begin{equation}\label{target:fixed-regularisation}
 \cH^{(N;\delta)}(\varphi)
 \ \xrightarrow[\ N \to \infty\ ]{d}\ \scrH^{(\delta)}(\varphi) \,,
\end{equation}
and the limiting variable $\scrH^{(\delta)}(\varphi)$
belongs to $L^2$.
\end{proposition}

\begin{proposition}[Log-polymer smoothing]
\label{prop:smo}
We have
\begin{equation}\label{eq:strategy-goal}
 \varepsilon_\delta\coloneqq \limsup_{N\to\infty}
 \big\| \cH^{(N)}(\varphi)
       -\cH^{(N;\delta)}(\varphi) \big\|_{L^2}
       \ \xrightarrow[\ \delta\downarrow 0 \ ]{} \ 0 \,.
\end{equation}
More precisely there is $C_\varphi < \infty$ such that,
for all sufficiently small $\delta > 0$,
\begin{equation}
      \label{eq:strategy-goal-L2}
      \varepsilon_\delta
      \leq \frac{C_{\varphi}}{|\log \delta|^{1/40}}  \,.
\end{equation}
\end{proposition}

Proposition~\ref{prop:FC} is not surprising:
the convergence of polymer partition functions to the SHF,
see Theorem~\ref{thm:SHF-convergence}, suggests
that for fixed $\delta > 0$ we should have
convergence in distribution of $\cH^{(N;\delta)}(\varphi)$
to $\scrH^{(\delta)}(\varphi)$ as $N\to\infty$.
Turning this intuition into a proof
will require us to control the \emph{lower
tail} of partition functions,
in order to deal with the singularity
of the logarithm at zero.
The details are provided in Subsection~\ref{sec:FC}.

Proposition~\ref{prop:smo} is much more challenging:
the strategy for its proof, which is the technical core of this paper, is described in the next Subsection~\ref{sec:factorization}.

Combining both propositions, we can now
easily prove Theorem~\ref{thm:toy}.

\begin{proof}[Proof of Theorem~\ref{thm:toy}]
For fixed $\delta,\delta'>0$, combining
\eqref{target:fixed-regularisation} and Portmanteau
we get
\begin{equation}\label{eq:L2-Cauchy}
 \big\| \scrH^{(\delta)}(\varphi)
       -\scrH^{(\delta')}(\varphi) \big\|_{L^2}
       \,\le\, \limsup_{N\to\infty}
       \big\| \cH^{(N;\delta)}(\varphi)
       - \cH^{(N;\delta')}(\varphi) \big\|_{L^2}
      \,\le\, \varepsilon_\delta+\varepsilon_{\delta'}    \ ,
\end{equation}
where we used the triangle inequality and the definition
of $\varepsilon_\delta$, see \eqref{eq:strategy-goal}.
Since $\varepsilon_\delta \to 0$, this shows that the random variables
$(\scrH^{(\delta)}(\varphi))_{\delta \in (0,1)}$ are
Cauchy in $L^2$ as $\delta \downarrow 0$,
hence they converge in $L^2$ to a limiting random variable
$\scrH(\varphi)$, which must be a measurable
function of the SHF.
This proves the first convergence in \eqref{eq:toy}.

We next define $d(X,Y) \coloneqq  \sup | \bbE[\psi(X)] - \bbE[\psi(Y)]|$
where the supremum runs over $1$-bounded and $1$-Lipschitz
functions
$\psi: \R \to \R$, namely
$|\psi(x)-\psi(y)| \le |x-y|$ and
$|\psi(x)| \le 1$ for all $x,y\in\R$.
Note that $d(X,Y)$ is a distance between the \emph{laws}
of random variables $X, Y$ which metrises convergence in
distribution and we can write
\begin{equation*}
      d\big(\cH^{(N)}(\varphi), \scrH(\varphi)\big)
      \le d\big(\cH^{(N)}(\varphi), \cH^{(N;\delta)}(\varphi)\big)
      + d\big(\cH^{(N;\delta)}(\varphi),
      \scrH^{(\delta)}(\varphi)\big)
      + d\big(\scrH^{(\delta)}(\varphi), \scrH(\varphi)\big) \,.
\end{equation*}
We have $d\big(\cH^{(N;\delta)}(\varphi),
      \scrH^{(\delta)}(\varphi)\big) \to 0$
as $N\to\infty$ for any fixed $\delta > 0$,
by Proposition~\ref{prop:FC}. Since
$d(X,Y) \le \bbE[|X-Y|] \le \| X-Y \|_{L^2}$ for $X=\cH^{(N)}(\varphi)$ and $Y=\cH^{(N;\delta)}(\varphi)$,
it follows by the definition \eqref{eq:strategy-goal}
of $\varepsilon_\delta$ that
\begin{equation*}
      \limsup_{N\to\infty}
      d\big(\cH^{(N)}(\varphi), \scrH(\varphi)\big)
      \,\le \,
      \varepsilon_\delta
      \,+\, d\big(\scrH^{(\delta)}(\varphi), \scrH(\varphi)\big) \,.
\end{equation*}
Both terms in the right hand side vanish as $\delta \downarrow 0$,
by Proposition~\ref{prop:smo} and the first part of the proof.
This proves the second convergence in \eqref{eq:toy}.
\end{proof}

\subsection{Strategy of the proof of Proposition~\ref{prop:smo}}
\label{sec:factorization}

We need to show that $\cH^{(N)}(\varphi)$
and $\cH^{(N;\delta)}(\varphi)$ are close in $L^2$
for small $\delta > 0$, uniformly over large~$N\in\N$.

\smallskip

We first give some intuition.
By \eqref{eq:HNbis} and \eqref{eq:discrete-regularised-H},
we should compare the partition function $Z_N(z)$ and
its smoothed version $\widehat Z_{m,N}(z)$
for $m = [\delta N]$. By definition
\begin{equation}
      \widehat Z_{m,N}(z) = Z_{m,N}(q_m^{(z)}) \,,
\end{equation}
and we can express $Z_N(x)$ similarly:
by \eqref{eq:endpoint},
a Chapman-Kolmogorov decomposition gives
\begin{equation} \label{eq:CK}
\begin{split}
      Z_N(z) = \sum_{w} Z_{m}(z,w) \, Z_{m,N}(w)
      = Z_m(z) \, \sum_{w} \mu_{m}^{(z)}(w) \, Z_{m,N}(w)
      = Z_m(z) \, Z_{m,N}(\mu_m^{(z)}) \,.
\end{split}
\end{equation}
The problem is then to compare two partition functions
$Z_{m,N}(\mu_m^{(z)})$ and $Z_{m,N}(q_m^{(z)})$, which are
over the same time interval $(m,N]$ but started with
different initial conditions at time~$m$, namely, 
the polymer endpoint distribution $\mu_m^{(z)}$ and
the random walk endpoint distribution~$q_m^{(z)}$.

\smallskip
The key idea is to introduce an intermediate time $n$ with $m\ll n \ll N$. We will choose
\begin{equation}\label{eq:entrance-times}
 m=[\delta N] \,,\qquad n=[\sqrt\delta\,N] \,,
\end{equation}
so that $1 \ll m \ll n \ll N$ as $N\to\infty$ and then
$\delta \downarrow 0$. Let us decompose the two partition functions $Z_N(z) $
and $\widehat Z_{m,N}(z)$ at the same time $n$, as in  \eqref{eq:CK}:
\begin{eqnarray} \label{eq:CKn}
      Z_N(z) &=& Z_n(z) \, Z_{n,N}(\mu_n^{(z)}) \,, \\
      \widehat Z_{m,N}(z) &=& \widehat Z_{m,n}(z)
      \, Z_{n,N}(\widehat\mu_{m,n}^{(z)}) \,, \label{eq:CKnhat}
\end{eqnarray}
where $\mu_n^{(z)}$ and $\widehat\mu_{m,n}^{(z)}$ are the polymer endpoint distributions associated with
the partition functions $Z_n(z)$ and $\widehat Z_{m,n}(z)=Z_{m,n}(q_m^{(z)})$:
\begin{equation}\label{eq:hatmu}
  \mu_n^{(z)}(w)
  =\frac{Z_{0, n}(z, w)}
           {Z_n(z)}
 \qquad \text{ and } \qquad
 \widehat\mu_{m,n}^{(z)}(w)
 =\frac{Z_{m,n}(q_m^{(z)},w)}{\widehat Z_{m,n}(z)}
\end{equation}

To prove Proposition~\ref{prop:smo}
(recalling $\cH^{(N)}(\varphi)$
and $\cH^{(N;\delta)}(\varphi)$
from \eqref{eq:HNbis} and
\eqref{eq:discrete-regularised-H}), we need to control the fluctuations of the field of the difference of the logarithms
\begin{equation}\label{eq:critical-decomposition}
 \log Z_N(z) - \log\widehat Z_{m,N}(z)
                  \,=\, L_{m,n}(z)+G_{m,n,N}(z),
\end{equation}
where we used the decompositions \eqref{eq:CKn} and \eqref{eq:CKnhat} to separate the difference of the local terms
\begin{align}\label{eq:L}
 L_{m,n}(z)
 &\,\coloneqq \, 
 \log Z_{n}(z)-\log\widehat Z_{m,n}(z) \,,
 \end{align}
 and the difference of the global terms
\begin{align} \label{eq:G}
 G_{m,n,N}(z)
 &\,\coloneqq \,\log\frac{Z_{n,N}(\mu_n^{(z)})}
            {Z_{n,N}(\widehat\mu_{m,n}^{(z)})} \,.
\end{align}
Here are the key heuristics that explain why both the local difference and global difference are small after spatial averaging,
and what role our choice of $n$ plays.

\begin{itemize}
\item \emph{Both $\log Z_n(z)$ and $\log \widehat Z_{m,n}(z)$
are small in $L^2$
after centring and spatial averaging, without the need for cancellation between them.}
The reason is that they depend on disorder
up to time \(n \ll N\), so their spatial
correlations decay on the spatial scale \(\sqrt n\),
much smaller than the spatial averaging scale \(\sqrt N\).
Such a localisation argument has also been used in the study of
the subcritical 2D KPZ~\cite[Proposition~2.2]{CSZ20}. The key
new ingredient here is \eqref{eq:polymer-covariance-convexity},
which bounds the covariance of log-partition functions and is the consequence of
a powerful Gaussian convexity theorem~\cite[Theorem~1]{Chen26}. 

\item \emph{The distributions $\mu_n^{(z)}$ and
$\widehat\mu_{m,n}^{(z)}$ at time~$n$ are close.}
The reason is that, while they are far apart at time~$m$,
cf.\ $\mu_m^{(z)}$ and $\widehat\mu_{m,m}^{(z)} = q_m^{(z)}$,
they evolve through the same disorder on the time
interval $(m,n]$, and since \(n/m\) is large,
this common evolution
largely washes out their initial discrepancy.
This is proved through an adaptation of the
Gu--Tsai factorisation \cite[Proposition~2.1]{GT26}.

\item \emph{The partition functions $Z_{n,N}(\mu_n^{(z)})$
and $Z_{n,N}(\widehat\mu_{m,n}^{(z)})$ are close in $L^2$.}
The reason is that partition functions $Z_{n,N}(\mu)$
depend linearly on their initial profile~$\mu$,
so the part common to the two
profiles gives the same contribution, and their difference
only sees the small residual discrepancy left at time \(n\),
which we know to be small.
\end{itemize}

\smallskip

We now formulate the above heuristics into precise technical lemmas that
we will verify in order to conclude the proof of Proposition~\ref{prop:smo}.

Centring the field in \eqref{eq:critical-decomposition} and averaging in space gives
\begin{equation}\label{eq:averaged-decomposition}
 \cH^{(N)}(\varphi)
 \,=\, \cH^{(N;\delta)}(\varphi)
  \,+\, \cL^{(N;\delta)}(\varphi)
  \,+\, \cG^{(N;\delta)}(\varphi),
\end{equation}
where we set,
recalling also the definition~\eqref{def:varphi-N} of \(\varphi^{(N)}\),
\begin{align}
 \cL^{(N;\delta)}(\varphi)
 &\,\coloneqq \,\sum_{z\in\Z^2_{\rm even}}
 \varphi^{(N)}(z) \,
 \big\{
      L_{m,n}(z)-\bbE [L_{m,n}(0)] \,\big\}
      \, ,\label{eq:average-L}\\
 \cG^{(N;\delta)}(\varphi)
 &  \,\coloneqq \, \sum_{z\in\Z^2_{\rm even}}
 \varphi^{(N)}(z)\,
 \big\{G_{m,n,N}(z)-\bbE [G_{m,n,N}(0)]
 \big\}\,.\label{eq:average-G}
\end{align}
The proof of Proposition~\ref{prop:smo} follows
directly from the next two results, which are
the key technical ingredients of the paper.
Their proofs are detailed in
Subsections~\ref{subsec:local} and~\ref{subsec:PF}.

\begin{proposition}[Local approximation]\label{prop:local}
For all \(\varphi\in C_c(\R^2)\), there is a constant \(C_{\varphi}>0\) such that, for all sufficiently small \(\delta\),
\begin{equation}\label{eq:local-vanishes}
 \limsup_{N\to\infty}
 \bbE\pig[ \cL^{(N;\delta)}(\varphi)^2
 \pig] \leq \frac{C_{\varphi}}{|\log \delta|^{1/4}}.
\end{equation}
\end{proposition}

\begin{proposition}[Global approximation]\label{prop:global}
For all \(\varphi\in C_c(\R^2)\), there is a constant \(C_{\varphi}>0\) such that, for all sufficiently small \(\delta\),
\begin{equation}\label{eq:far-average}
 \limsup_{N\to\infty}
 \bbE\pig[
      \cG^{(N;\delta)}(\varphi)^2 \pig]
      \leq \frac{C_{\varphi}}{|\log \delta|^{1/20}}.
\end{equation}
\end{proposition}

\noindent
Taking square roots in \eqref{eq:local-vanishes}
and \eqref{eq:far-average}, and using
\eqref{eq:averaged-decomposition}, we obtain
\eqref{eq:strategy-goal-L2}.

\smallskip

Let us comment on the proof of
Proposition~\ref{prop:local}, presented in
Subsection~\ref{subsec:local}.
In the local term \(L_{m,n}(z)\), see \eqref{eq:L},
the two logarithms are not small at a fixed starting
point~\(z\), nor is their difference. However,
after centring and spatial averaging as in
\eqref{eq:average-L}, each logarithm separately
becomes small in \(L^2\).
To implement this argument at criticality
and control covariances of log-partition functions,
we exploit a novel covariance bound
\eqref{eq:polymer-covariance-convexity} that follows from
the same Gaussian convexity theorem~\cite[Theorem~1]{Chen26}
that was used in \cite{LX26} to derive moment comparison
bounds and Gaussian concentration of measure for the directed
polymer model. To conclude the covariance bound, a key ingredient
is Proposition~\ref{prop:VD}, whose proof adapts a factorisation technique
from \cite{GT26}.
\smallskip

We next comment on the global approximation,
Proposition~\ref{prop:global}.
Unlike the logarithms in \eqref{eq:L},
the global term \(G_{m,n,N}(z)\) in \eqref{eq:G}
is small in $L^2$ for any fixed starting point~$z$,
for a different reason.
The two endpoint distributions
\(\mu_n^{(z)}\) and \(\widehat\mu_{m,n}^{(z)}\), which start
from different profiles at time \(m\), evolve
through the same disorder on \((m,n]\). When \(n/m\) is large,
this common evolution washes out their initial discrepancy.
This is quantified in Proposition~\ref{prop:PF},
whose proof adapts the same factorisation technique from
\cite[Proposition~2.1]{GT26}, which identifies
the common part of the two endpoint distributions and isolates
the residual dependence on their initial profiles.
To transfer this control to the logarithms, we use
negative-moment bounds for the partition functions,
based again on Gaussian convexity \cite{Chen26,LX26},
together with the logarithmic-moment bounds in
Proposition~\ref{prop:log-tails}.
Together, these estimates allow us to pass from the comparison
of partition functions to their logarithms and
prove Proposition~\ref{prop:global}.

\subsection{Organisation of the rest of the article}
\label{sec:organisation}

Let us now give a brief outline on how the rest of the paper is organised.

\smallskip
\textbullet\
Section~\ref{sec:general} gives
Gaussian-convexity inequalities
and their polymer consequences, recalling some key results from \cite{Chen26} and \cite{LX26}.
Crucially, it derives a new covariance bound \eqref{eq:polymer-covariance-convexity} for log-partition functions, which is controlled through
the quotient of second and negative-first moments.

\smallskip
\textbullet\
Section~\ref{sec:convergence-estimates} proves the convergence in Proposition~\ref{prop:FC} and also
the local and global approximations of Propositions~\ref{prop:local}-\ref{prop:global}.
These last results rely on
quantitative inputs developed in Section~\ref{sec:factorisation},
notably:
(i) a sharp moment comparison (Proposition~\ref{prop:VD}) to control the local approximation;
(ii) a sharp endpoint estimate of the difference between \(\mu_n^{(x)}\) and \(\widehat{\mu}_{m,n}^{(x)}\) (Proposition~\ref{prop:PF}) to control the global approximation.

\smallskip
\textbullet\
Section~\ref{sec:joint-convergence} extends Theorem~\ref{thm:toy}
to joint convergence against test functions at finitely many times
and to non-flat initial condition.

\smallskip
\textbullet\
Section~\ref{sec:local-Sobolev}
extends the joint convergence of Section~\ref{sec:joint-convergence}
to convergence in local Sobolev spaces.
It concludes the proof of Theorem~\ref{thm:main}
and constructs the process \(\scrH^{\vartheta}\).

\smallskip
\textbullet\ Section~\ref{sec:factorisation} proves important estimates that are used Section~\ref{sec:convergence-estimates}.
It introduces
the geometric block factorisation of Gu--Tsai~\cite{GT26}
and adapts some of their results on the block
expansion to our discrete setting.
This is the main ingredient
(together with the moment inequalities of Section~\ref{sec:general}
and sharp second moment estimates) that
allows us to prove Proposition~\ref{prop:VD} and~\ref{prop:PF},
respectively in Subsections~\ref{sec:vanishing-defect} and \ref{sec:prop-PF}.

\smallskip
\textbullet\
Section~\ref{prop:sec-local-properties} proves Theorem~\ref{thm:properties}.
We first prove the local normality of the Critical 2D KPZ Process, which essentially relies on Gu--Tsai's (``local'') log-normality result for the SHF~\cite{GT26} together with some control when taking a spatial average.
Proposition~\ref{prop:no-local-representative} can then be derived from this local normality with diverging variance.
Finally, we conclude the proof of Theorem~\ref{thm:properties} by showing the \(\log\log\) divergence of
its covariances near the diagonal, through the study of the spectral measure associated with the covariance kernel.

\smallskip
\textbullet\
Finally, Section~\ref{ng:section} shows the non-Gaussianity of the Critical 2D KPZ Process.
It relies on the Ehrhard-type property of the process, together with Gaussian rigidity with respect to
the initial condition (see Remark~\ref{rem:initial-condition}, Lemmas~\ref{ng:positive-sums} and~\ref{ng:rigidity}).
A consequence of this is that if \(\scrH_t^\vartheta(\zero,\psi)\) were Gaussian,
then any weak limit of polymer endpoint distribution
would be deterministic and non-zero (see Lemmas~\ref{ng:determinism} and~\ref{ng:nonzero-limit}).
We show that this contradicts the singularity of the SHF with respect
to the Lebesgue measure, see~\cite{CSZ25}.

\smallskip
\textbullet\ Additionally, we collect in Appendix several technical results.
Appendix~\ref{app:technical-proposition}
gives a proof of a key technical estimate
for the block expansion from Section~\ref{sec:block-exp}, namely Proposition~\ref{prop:block-inputs-iii} (it is analogous to \cite[Corollary~3.5]{GT26} but cannot be directly adapted from it).
Appendix~\ref{app:renewal-comparison} provides several important moment estimates, in particular a sharp second moment estimate stated in Lemma~\ref{lem:second-moment-hat}.

\section{Gaussian convexity estimates}\label{sec:general}

The estimates in this section convert positive moments into
bounds on negative moments and covariances of logarithms.
We present general results
for convex functions of a Gaussian vector,
which we then specialise to
our directed polymers framework.

\subsection{Positive and negative moments}

Given a random variable $F$
such that $\bbE[\e^{\lambda F}]<\infty$ for all $\lambda > 0$,
we define for all $q\in\R$
\begin{equation}\label{eq:chen-Phi}
 \Phi(q) = \Phi_F(q)\coloneqq 
 \begin{cases}
 q^{-1}\log\bbE[\e^{qF}] &q\ne0 \,,\\
 \bbE[F] &q=0   \,.
 \end{cases}
\end{equation}
When $F$ is a convex function of a Gaussian vector,
we are going to exploit the following 
convexity property proved by Chen~\cite[Theorem~1]{Chen26} (convexity for $q>0$ was established earlier in \cite[Proposition~1]{AC15},
and an alternative proof based on Ehrhard's inequality was recently given
in \cite[Section~3]{FMP}).

\begin{theorem}[Gaussian convexity property]\label{thm:Chen}
Let $F=F(G)$ be a convex function of
a finite-dimensional Gaussian vector~$G$
such that $\bbE[\e^{\lambda F}]<\infty$ for all $\lambda > 0$.
Then the function $\Phi(q)$ for $q\in\R$ is convex.
\end{theorem}

Liu and Xie \cite{LX26} applied Theorem \ref{thm:Chen} to the directed polymer model and the Critical 2D Stochastic Heat Flow (SHF) and
derived many consequences, including bounds on variance, concentration of measure, moment comparisons, etc.
Here we focus on the following negative-moment comparison result: a positive second moment controls all
negative-moments.
It is convenient to normalise $F$ so that
$\bbE[\e^F] = 1$ (just replace $F$ by $F - \log \bbE[\e^F]$).

\begin{lemma}[Approximate log-normal bounds]\label{lem:one-point}
For $F$ as in the assumptions of Theorem~\ref{thm:Chen}
with $\bbE[\e^F] = 1$, we have
\begin{align}\label{eq:chen-one}
      \Var[F]
      & \leq-2 \,\bbE [F] \,,\\
      \label{eq:chen-two}
       -\bbE[F]
       & \leq \frac12\log\bbE[\e^{2F}] \,,
\end{align}
and also
\begin{align}
      \label{eq:momentp}
      \log\bbE[\e^{-pF}]
      &\leq -p \, \bbE[F] + \frac{p^2}{2}\Var[F] \,,
      \qquad p \ge 0 \,, \\
      \label{eq:one-point-negative}
      \bbE[\e^{-pF}]
      &\leq \bbE[\e^{2F}]^{p(p+1)/2},
      \qquad p \ge 0.
\end{align}
If $F$ is Gaussian, so that $\e^F$ is log-normal,
then all inequalities above are equalities.
\end{lemma}

\begin{proof}
Since
\(\bbE[\e^{qF}] =1+q \, \bbE[F] + (1+o(1)) \frac{q^2}{2} \,
\bbE[F^2]\)
as \(q\downarrow 0\),
we get that
\begin{equation*}
      \Phi'(0) = \frac12 \Var[F] \,.
\end{equation*}
Recall that $\Phi(0) = \bbE[F]$ and $\Phi(1) = 0$.
Convexity gives
\(\Phi(1) \leq \frac{1}{2}(\Phi(0)+\Phi(2))\),
that is \eqref{eq:chen-two}.

The supporting line at \(0\) gives
\(\Phi(q) \geq \Phi(0)+ \Phi'(0) \, q\).
For $q=1$ this yields \eqref{eq:chen-one},
while for $q = -p$ this yields \eqref{eq:momentp} (after
multiplication by~$-p$ and reversal of the inequality).

Finally, starting from \eqref{eq:momentp} and
applying \eqref{eq:chen-one} and \eqref{eq:chen-two},
we obtain \eqref{eq:one-point-negative}:
\begin{equation*}
      \log\bbE[\e^{-pF}]
      \leq (p + p^2)(- \bbE[F])
      \leq \frac{p(p+1)}{2} \, \log\bbE[\e^{2F}] \,.
      \qedhere
\end{equation*}
\end{proof}

\subsection{Covariance comparison}

It is useful
to define the following quantity:
\begin{equation}\label{eq:defect-general}
 \Delta(F)\coloneqq \Phi(1) -\Phi(0) - \Phi '(0) = -\bbE F-\frac12\Var[F] \,,
\end{equation}
which can be thought of as measuring the convexity of the normalised log moment-generating function $\Phi$ at $0$. Since $\Delta(F)=0$ when $F$ is Gaussian (recall that we use the normalisation $\bbE[\e^F]=1$), $\Delta (F)$ illustrates the ``Gaussianity defect'' through the convex distortion of \(F\),
or the ``log-Gaussianity defect'' of \(\e^F\).

\begin{lemma}[Gaussianity defect]\label{lem:defect}
For $F$ as in the assumptions of Theorem~\ref{thm:Chen}
with $\bbE[\e^F] = 1$, we have
\begin{equation}\label{eq:one-point-defect}
    0\leq\Delta(F)
 \leq\log\frac{\bbE[\e^{2F}]}{\bbE[\e^{-F}]} \,.
\end{equation}
\end{lemma}

\begin{proof}
The first inequality is simply \eqref{eq:chen-one}.
We next apply \eqref{eq:momentp} for $p=1$
and \eqref{eq:chen-two} to get
\begin{equation}
      \label{eq:chen-three}
       \log\bbE[\e^{-F}] \leq
       -\bbE[F] + \frac{1}{2} \Var[F]
       = -2 \, \bbE[F] - \Delta(F)
      \leq\log\bbE[\e^{2F}] - \Delta(F) \,,
\end{equation}
which is just a rewriting of \eqref{eq:one-point-defect}.
\end{proof}

We now present a simple  consequence of the
Gaussian convexity property in Theorem~\ref{thm:Chen},
which will play a key role in controlling the covariance
of logarithms of our random fields.

\begin{lemma}[Covariance bound]\label{lem:two-point}
Let $F_x = F_x(G)$, $F_y = F_y(G)$ be convex functions
of the same Gaussian vector,
as in the assumptions of Theorem~\ref{thm:Chen},
with
$\bbE[\e^{F_x}] = \bbE[\e^{F_y}] = 1$.  Then
\begin{equation}\label{eq:two-point-general}
 \Cov[F_x,F_y]
 \,\leq\, \log\bbE[\e^{F_x+F_y}]
 + \Delta(F_x) + \Delta(F_y) \,.
\end{equation}
If the Gaussian vector $G$ has independent
coordinates $G_i$, and $F_x,F_y$ are
coordinatewise nondecreasing functions,
then $\Cov[F_x,F_y] \geq 0$.
\end{lemma}

\begin{proof}
Define
\begin{equation*}
      H \,\coloneqq \, F_x+F_y-\log\bbE[\e^{F_x+F_y}] \,,
\end{equation*}
so that $H$ is a convex function of $G$
with $\bbE[\e^H]=1$. Then $\Var[H] \le -2 \, \bbE[H]$
by \eqref{eq:chen-one}, i.e.\
\begin{equation*}
      \Var[F_x] + \Var[F_y] + 2 \, \Cov[F_x,F_y]
      \le -2 \, \bbE[F_x] -2 \, \bbE[F_y]
      + 2 \, \log \bbE[\e^{F_x+F_y}] \,,
\end{equation*}
which gives \eqref{eq:two-point-general} after
rearranging the terms.
If $G$ has a product Gaussian law, the
lower bound $\Cov[F_x,F_y] \geq 0$ follows by the
FKG inequality for nondecreasing functions.
\end{proof}

\subsection{Ehrhard-type random variables}
Since the SHF arises as a weak limit of lattice partition functions, we need
a formulation of Gaussian convexity in terms of laws that are preserved
under weak convergence.
Following Liu and Xie~\cite[Definition~2.4]{LX26}, we introduce the Ehrhard-type class\footnote{Our definition of Ehrhard
type differs from \cite{LX26} by including constant random variables.}.

\begin{definition}[Ehrhard-type random variables]
\label{def:Ehrhard}
We call a real-valued random variable~$X$ and its distribution
\emph{Ehrhard type} if either $X$ is constant,
or if $\bbP(X<a)<1$ for every
$a\in\R$ and $a\mapsto\Phi^{-1}(\bbP(X<a))$ is concave, where $\Phi$
is the standard Gaussian distribution function.
\end{definition}

We recall the following consequence of \cite[Theorem~4.3]{FMP} and \cite[Corollary~2.2 and
Lemmas~2.5--2.6]{LX26}.

\begin{lemma}\label{ng:quantile-closure}
Every finite convex function of a finite-dimensional Gaussian vector is
of Ehrhard type. A real-valued random variable $X$ is of Ehrhard type
if and only if it admits a representation
\begin{equation}\label{eq:ehrhard-representation}
 X\stackrel{\mathrm d}=h(G),\qquad G\sim\mathcal N(0,1),\qquad
 h:\R\to\R\text{ nondecreasing and convex}.
\end{equation}
The class of Ehrhard-type laws is closed under weak convergence.
\end{lemma}

\begin{remark}\label{R:Ehrhard}
This representation implies that Lemmas~\ref{lem:one-point},
and \ref{lem:defect} 
apply if $F$ is replaced by any Ehrhard-type random variable
with proper centring.
Lemma~\ref{lem:two-point} also holds if~$F_x$ and~$F_y$ therein are replaced by any Ehrhard-type random 
variables that are convex functions of the same Gaussian vector or obtained as joint weak limits (with convergence of moments) of convex functions of the same Gaussian vectors.
\end{remark}

\subsection{Consequences for directed polymers and the SHF}
\label{sec:polymer-convexity}

The logarithms of the partition functions
$\widehat Z_{s,n}(x)$ in \eqref{eq:Wab} are convex and
nondecreasing in the (independent) Gaussian disorder variables.
Lemma~\ref{lem:defect} therefore gives
\begin{equation}\label{eq:polymer-one-point-convexity}
 \begin{aligned}
 0&\le -\bbE\log\widehat Z_{s,n}(x)
       -\tfrac12\Var(\log\widehat Z_{s,n}(x))\\
  &=\Delta(\log\widehat Z_{s,n}(x))
 \le\log\frac{\bbE[\widehat Z_{s,n}(x)^2]}
                 {\bbE[\widehat Z_{s,n}(x)^{-1}]}.
 \end{aligned}
\end{equation}
On the other hand, Lemma~\ref{lem:two-point} gives
\begin{equation}\label{eq:polymer-covariance-convexity}
 \begin{aligned}
 0&\le\Cov(\log\widehat Z_{s,n}(x),\log\widehat Z_{s,n}(y))\\
 &\le \log \bbE[\widehat Z_{s,n}(x)\widehat Z_{s,n}(y)]
   +\Delta(\log\widehat Z_{s,n}(x))
   +\Delta(\log\widehat Z_{s,n}(y)).
 \end{aligned}
\end{equation}
The bounds~\eqref{eq:polymer-one-point-convexity}--\eqref{eq:polymer-covariance-convexity},
together with the moment comparison in Proposition~\ref{prop:VD} below, will be used
to prove the local approximation in Proposition~\ref{prop:local}.
The general covariance bound in Lemma~\ref{lem:two-point} also gives the
estimates needed in Lemma~\ref{lem:cutoff-covariance}, which will yield the Sobolev
tightness in Lemma~\ref{lem:tightness}.

By taking the weak limit of polymer partition functions,
it was shown in the proof of Theorem 1.5 in
\cite[Section~4]{LX26} that the logarithm of SHF
$\log\scrZ_{s,t}^\vartheta(f, g)$
with \(f, g \ge 0\) is also Ehrhard
type. Nonnegative linear combinations of log partition functions are again convex
functions of the same Gaussian disorder, and hence their weak limits, such
as the fields in Theorem \ref{thm:main} integrated against nonnegative test functions,
are also Ehrhard type.
For example,
\begin{equation}\label{eq:F-delta}
 F_\delta(x)\coloneqq \log\scrZ_{\delta,1}^\vartheta(g_\delta^{(x)},\one),
 \qquad 0<\delta<1,
\end{equation}
is Ehrhard type.  Since $\bbE[\e^{F_\delta(x)}]=1$, the extension of the Gaussian convexity property
to Ehrhard-type laws in \cite[Theorem~2.9]{LX26} gives
\eqref{eq:chen-one}--\eqref{eq:chen-two}, and hence
\begin{equation}\label{eq:asymp-ElogZ-VarlogZ}
 0\le\Var(F_\delta(x))\le-2\bbE [F_\delta(x)]
 \le\log\bbE[\e^{2F_\delta(x)}]
 \le\log\log(1/\delta)+C_\vartheta.
\end{equation}
Here $\vartheta$ is fixed, and the bound holds uniformly in $x$ for
sufficiently small $\delta$. The last inequality follows from
$\bbE[\e^{2F_\delta(x)}]\le 1+C_\vartheta\log(1/\delta)$, obtained by
passing the lattice second-moment bound~\eqref{eq:FC-moments} to the
SHF limit at fixed $\delta$, using Fatou's lemma. This upper bound
will be needed in Section~\ref{prop:sec-local-properties}.

Let us conclude this section by observing that,
by Lemma~\ref{ng:quantile-closure} and Theorem~\ref{thm:main},
even though the Critical 2D KPZ Process is non-Gaussian, it is of Ehrhard type.

\begin{proposition}
      \label{prop:Ehrhard}
      For any \(t>0\), \(\vartheta \in \R\) and \(\varphi\in C_c^{\infty}(\R^2)\) with \(\varphi\ge 0\),
      the random variable \(\scrH_t^\vartheta(\varphi)\) is of Ehrhard type.
\end{proposition}

\section{Proofs of the convergence estimates}\label{sec:convergence-estimates}

In this section, we prove Propositions~\ref{prop:FC}, \ref{prop:local} and~\ref{prop:global} from Section~\ref{sec:decomposition},
which complete the proof of Theorem~\ref{thm:toy}. The proof of our main result Theorem~\ref{thm:main} will be completed in the next Sections~\ref{sec:joint-convergence} and~\ref{sec:local-Sobolev}.

\subsection{Preliminaries: some moment bounds}

Let us start with a few useful second moment estimates.
First, notice that by \cite[Proposition~A.2]{CSZ19a}, we have that for any fixed \(t>0\), \(\bbE[Z_{[Nt]}(0)^2] \sim \nu (t \e^{\vartheta - \gamma})\log N\) as \(N\to\infty\), with~\(\nu\) the Volterra function.
We also have the following uniform second moment bound on the point-to-plane partition function:
for any \(1\leq k \leq N\)
\begin{equation}\label{eq:uniform-two-replica}
 \sigma_N^2\bbE[Z_k(0)^2]\le \frac{C_\vartheta}{\log(1+N/k)} \,,
\end{equation}
see Claim~\ref{claim:VN}.

For deterministic probability measures \(\nu,\nu'\) on \(\Z^2\), we also have the following {\em first-collision} decomposition of covariances using their chaos representation, cf.\ \cite[Equations~(3.2)--(3.3)]{CSZ19b}:
\begin{equation*}
      \bbE[Z_{s,n}(\nu) Z_{s,n}(\nu')] -1 = \sum_{r=s+1}^n \sum_{x,y\in \Z^2} \nu(x)\nu'(y) q_{2(r-s)}(y-x) \sigma_N^2\bbE[Z_{n-r}(0)^2] \,.
\end{equation*}
In particular, recalling the definition~\eqref{eq:Wab} of \(\widehat Z_{s,n}(x)\), we get that
\begin{equation}\label{eq:local-replica-formula}
 \bbE[\widehat Z_{s,n}(x)\widehat Z_{s,n}(y)]-1
 =\sum_{r=s+1}^{n}q_{2r}(y-x)\,
       \sigma_N^2\bbE[Z_{n-r}(0)^2] \,.
\end{equation}
As a consequence, since \(\bbE[Z_{n-r}(0)^2]\leq \bbE[Z_{n}(0)^2]\),
we get from~\eqref{eq:uniform-two-replica} that
\begin{equation}\label{eq:local-replica-sum}
 \bbE[\widehat Z_{s,n}(x)\widehat Z_{s,n}(y)]-1
 \leq \frac{C}{\log(1+N/n)}\sum_{r=s+1}^{n} q_{2r}(y-x) \leq  C \frac{\mathsf G_n(y-x)}{\log(1+N/n)}  \,,
\end{equation}
where \(\mathsf G_n(z) = \sum_{r=1}^{n} q_{2r}(z)\)
is the Green function of the random walk up to time \(n\).
Note that $\mathsf G_n(0)=R_n$ with $R_N$ defined in \eqref{eq:RN}.
In particular, taking $x=y=0$ and using $q_{2r}(0)\le C/r$,
we obtain, for $1\le s\le n\le N$,
\begin{equation}\label{eq:rough-bound}
 \bbE[\widehat Z_{s,n}(0)^2]
 \le
 1+C\frac{R_n-R_s}{\log(1+N/n)}
 \le
 1+C\frac{\log(n/s)}{\log(1+N/n)}.
\end{equation}
The first inequality also holds for $s=0$, with $R_0=0$.

For $m=[\delta N]$, Lemma~\ref{lem:one-point} and
\eqref{eq:uniform-two-replica} also imply that
\begin{equation}\label{eq:FC-moments}
 \bbE[\widehat Z_{m,N}(x)^{-1}]
 \le \bbE[\widehat Z_{m,N}(x)^2]
 \le 1+C\sum_{r=m+1}^{N}q_{2r}(0)
 \le 1 +C \log \Big(\frac{N}{m}\Big) \leq 1+C \log(1/\delta) \,,
\end{equation}
where we have used the local limit theorem (see e.g. \cite[Theorem 2.3.11]{LL10}) to get that \(q_{2r}(0) \leq C/r\) for any \(r\geq 1\).

\subsection{Convergence with initial disorder removed: proof of Proposition~\ref{prop:FC}}\label{sec:FC}

From now on, for $0<\delta<1$ fixed, we let
\begin{equation*}
      m=[\delta N] \,.
\end{equation*}
For \(f\in L^1(\R^2)\), write
\(f^{(N)}(z)=\int_{\cellC_N(z)}f(x)\,\dd x\), as in \eqref{def:varphi-N}.
As a consequence of Theorem~\ref{thm:SHF-convergence}, jointly for finitely many
pairs $0\le s<t$ and test functions $\varphi\in C_c(\R^2)$, we have
\begin{equation}\label{eq:SHF-point-to-plane}
 Z_{[Ns],[Nt]}(\varphi^{(N)})
 =\int_{\R^2}\varphi(x)\cZ_{s,t}^{\vartheta,(N)}(x)\,\dd x
 \ \xrightarrow[\ N\to\infty \ ]{d} \
 \scrZ_{s,t}^\vartheta(\varphi) \,.
\end{equation}

\smallskip\noindent\emph{Step 1: convergence of partition functions.}
It follows by \eqref{eq:SHF-point-to-plane}
that for any \(x\in \R^2\), as \(N\to\infty\)
\begin{equation}
\label{eq:convergence-initial-0}
Z_{m,N}((g_{\delta}^{(x)})^{(N)})
 \ \xrightarrow[\ N\to\infty\ ]{d}
 \scrZ_{\delta,1}^\vartheta(g_\delta^{(x)}) \,.
\end{equation}
Strictly speaking, we cannot use~\eqref{eq:SHF-point-to-plane} since \(g_{\delta}^{(x)}\) does not have compact support, but we can apply the convergence~\eqref{eq:SHF-point-to-plane} to \(g_{\delta,A}^{(x)}\coloneqq \chi_A g_{\delta}^{(x)}\) with \(\chi_A\) some smooth function with compact support equal to \(1\) on the ball of radius \(A\).
We thus get \(Z_{m,N}((g_{\delta,A}^{(x)})^{(N)}) \xrightarrow{d} \scrZ_{\delta,1}^\vartheta(g_{\delta,A}^{(x)})\), and we may use the \(L^1\) bounds
\begin{equation}
\label{eq:L1-control}
\bbE\bigl[|Z_{m,N}(f^{(N)}) -Z_{m,N}(h^{(N)})|\bigr]\leq \| f-h \|_1
\quad \text{and} \quad
\bbE\bigl[|\scrZ_{\delta,1}^\vartheta(f)-\scrZ_{\delta,1}^\vartheta(h)|\bigr] \leq \| f-h \|_1
\end{equation}
applied to \(f=g_{\delta}^{(x)}\) and \(h=g_{\delta,A}^{(x)}\) to conclude the proof of \eqref{eq:convergence-initial-0}, by letting \(A\to\infty\).

Then, using again~\eqref{eq:L1-control}, letting \(x_N \coloneqq [\sqrt{N/2} x]\) we also have that
\begin{equation}
      \label{eq:L1-error}
\sup_{x} \bbE\bigl[ |Z_{m,N}(q_{m}^{(x_N)}) - Z_{m,N}((g_{\delta}^{(x)})^{(N)}) |\bigr] \leq \sup_x \sum_{y\in \Z^2_{\rm even}} \Bigl| q_{m}^{(x_N)}(y) -\int_{\cellC_N(y)}g_{\delta}^{(x)}(u)\,\dd u \Bigr| \,,
\end{equation}
which goes to \(0\) as \(N\to\infty\), by the local limit theorem (see \cite[Theorem 2.3.11]{LL10}).
We can thus deduce from~\eqref{eq:convergence-initial-0} that we have
\begin{equation}
\label{eq:convergence-initial}
\widehat Z_{m,N}(x_N) = Z_{m,N}(q_{m}^{(x_N)})
 \ \xrightarrow[\ N\to\infty\ ]{d}
 \scrZ_{\delta,1}^\vartheta(g_\delta^{(x)}) \,.
\end{equation}
We also stress that the convergence holds jointly for finitely many pairs $(\delta,x)$.

\smallskip\noindent\emph{Step 2: Convergence for the truncated logarithms.}
Define the logarithm truncated to $[-R,R]$, by setting
\begin{equation}\label{eq:truncated-logarithm}
 \log^{(R)}(u)=\max\{-R,\min\{\log u,R\}\}\quad(u>0),
 \qquad \log^{(R)}(0)=-R.
\end{equation}
We now show the convergence
\begin{equation}\label{eq:convergence-truncated}
\sum_{x\in\Z^2_{\rm even}}
 \varphi^{(N)}(x) \log^{(R)}\widehat Z_{m,N}(x) \ \xrightarrow[\ N\to\infty\ ]{d}
  \int_{\R^2} \varphi(x) \log^{(R)} \scrZ_{\delta,1}^\vartheta(g_\delta^{(x)}) \dd x \,.
\end{equation}

Let \(\varepsilon \in (0,1)\) be fixed, and partition the support of \(\varphi\) into a finite number \(k_{\varepsilon} \leq C_{\varphi} \varepsilon^{-2}\) of cells \((\cellC_i)_{1\leq i \leq k_{\varepsilon}}\), all cells being of diameter at most \(\varepsilon\).
For each cell \(\cellC_i\), let \(x_i\in \cellC_i\), and notice that a straightforward calculation shows that
\[
\sup_{x\in \cellC_i}\|g_\delta^{(x)} -g_\delta^{(x_i)} \|_1 \leq C \varepsilon /\delta \,,
\]
Hence, by~\eqref{eq:L1-control} and \eqref{eq:L1-error} and since \(\log^{(R)}\) is a Lipschitz function, we have (recalling that \(x_N = [\sqrt{N/2} x]\))
\[
\sup_{1\leq i\leq k_{\varepsilon}}\sup_{x\in \cellC_i} \bbE\bigl[\bigl|\log^{(R)}\widehat Z_{m,N}(x_N)  - \log^{(R)} Z_{m,N}((g_{\delta}^{(x_i)})^{(N)}) \bigr| \bigr]  \leq C_{\delta,R}\, \varepsilon +o_N(1)\,.
\]
This shows in particular that
\begin{equation}\label{eq:error-1}
\begin{split}
      &\bbE\biggl[ \Bigl| \sum_{x\in\Z^2_{\rm even}}
 \varphi^{(N)}(x) \log^{(R)}\widehat Z_{m,N}(x) - \sum_{i=1}^{k_{\varepsilon}} \Bigl(\int_{\cellC_i} \varphi(u) \dd u\Bigr) \log^{(R)} Z_{m,N}((g_{\delta}^{(x_i)})^{(N)})  \Bigr| \biggr]
 \\ &\qquad\leq C_{R,\varphi,\delta} \, \varepsilon + o_N(1) \,.
\end{split}
\end{equation}

Then, using the convergence~\eqref{eq:convergence-initial-0} and since \(\log^{(R)}\) is a bounded continuous function,
we get that \(\log^{(R)} Z_{m,N}((g_{\delta}^{(x_i)})^{(N)})\)
converges in distribution to \(\log^{(R)} \scrZ_{\delta,1}^\vartheta(g_\delta^{(x_i)})\)
jointly for all \(i\leq k_{\varepsilon}\).
This thus gives that
\begin{equation}
      \label{eq:convergence-sum-1}
      \sum_{i=1}^{k_{\varepsilon}} \Bigl(\int_{\cellC_i} \varphi(u) \dd u\Bigr) \log^{(R)} Z_{m,N}((g_{\delta}^{(x_i)})^{(N)}) \xrightarrow[\ N\to\infty\ ]{d} \sum_{i=1}^{k_{\varepsilon}} \Bigl(\int_{\cellC_i} \varphi(u) \dd u\Bigr) \log^{(R)} \scrZ_{\delta,1}^\vartheta(g_\delta^{(x_i)}) \,.
\end{equation}

Similarly as above, we also have that
\[
\sup_{1\leq i\leq k_{\varepsilon}}\sup_{x\in \cellC_i} \bbE\bigl[\bigl|\log^{(R)}\scrZ_{\delta,1}^\vartheta(g_\delta^{(x)})  - \log^{(R)} \scrZ_{\delta,1}^\vartheta(g_\delta^{(x_i)}) \bigr| \bigr]  \leq \cellC_{\delta,R}\, \varepsilon\,,
\]
so we obtain
\begin{equation}\label{eq:error-2}
\bbE\biggl[ \Bigl| \int_{\R^2} \varphi(x) \log^{(R)} \scrZ_{\delta,1}^\vartheta(g_\delta^{(x)}) \dd x
- \sum_{i=1}^{k_{\varepsilon}} \Bigl(\int_{\cellC_i} \varphi(u) \dd u\Bigr) \log^{(R)} \scrZ_{\delta,1}^\vartheta(g_\delta^{(x_i)}) \Bigr| \biggr]
\leq C_{R,\varphi,\delta} \, \varepsilon \,.
\end{equation}

All together, combining the convergence~\eqref{eq:convergence-sum-1} with \eqref{eq:error-1}--\eqref{eq:error-2}, letting \(N\to\infty\) then \(\varepsilon \downarrow 0\), this concludes the proof of~\eqref{eq:convergence-truncated}.

Since both terms in~\eqref{eq:convergence-truncated} are bounded (by \(R \|\varphi\|_1\)), this
also gives convergence of expectations
\begin{equation}
      \label{eq:convergence-truncated-expectations}
\sum_{x\in\Z^2_{\rm even}}
 \varphi^{(N)}(x) \bbE \bigl[\log^{(R)}\widehat Z_{m,N}(x)\bigr] \ \xrightarrow[\ N\to\infty\ ]{}
  \int_{\R^2} \varphi(x) \bbE \bigl[\log^{(R)} \scrZ_{\delta,1}^\vartheta(g_\delta^{(x)})\bigr] \dd x \,.
\end{equation}
Once more, let us stress that the convergences \eqref{eq:convergence-truncated}--\eqref{eq:convergence-truncated-expectations} hold jointly for any finite collection of fixed values of \(\delta\).

\smallskip\noindent\emph{Step 3: Removing the truncation.}
Since we have that
\[
 |\log u-\log^{(R)}(u)|\le\e^{-R}(u+u^{-1}),\qquad u>0,
\]
we get thanks to~\eqref{eq:FC-moments} that
\[
\bbE\bigl[ | \log\widehat Z_{m,N}(x) -\log^{(R)}\widehat Z_{m,N}(x) | \bigr] \leq\e^{-R} (2+C_{\delta}) \,.
\]
This thus gives that, for any \(R\geq 0\) and \(N\geq 1\),
\[
 \bbE\biggl[ \Bigl|\cH^{(N;\delta)}(\varphi)- \sum_{x\in\Z^2_{\rm even}}
 \varphi^{(N)}(x) \{\log^{(R)}\widehat Z_{m,N}(x) - \bbE\log^{(R)}\widehat Z_{m,N}(x) \} \Bigr|\biggr]
 \leq C_{\delta, \varphi}\e^{-R} \,.
\]

For the continuous counterpart, using~\eqref{eq:convergence-initial} and Portmanteau theorem, we also get from~\eqref{eq:FC-moments} that \(\bbE[ \scrZ_{\delta,1}^\vartheta(g_\delta^{(x)})^{-1}] \leq 1+C_{\delta}\).
In particular, we have that \(\scrZ_{\delta,1}^\vartheta(g_\delta^{(x)})\) is a.s.\ positive and also
\[
\bbE\bigl[ | \log \scrZ_{\delta,1}^\vartheta(g_\delta^{(x)}) -\log^{(R)} \scrZ_{\delta,1}^\vartheta(g_\delta^{(x)}) | \bigr] \leq\e^{-R} (2+C_{\delta})\,.
\]
This thus gives that for any \(R\geq 0\)
\[
 \bbE\biggl[ \Bigl|\scrH^{(\delta)}(\varphi) - \int_{\R^2} \varphi(x)\{\log^{(R)} \scrZ_{\delta,1}^\vartheta(g_\delta^{(x)}) - \bbE\log^{(R)} \scrZ_{\delta,1}^\vartheta(g_\delta^{(x)}) \}\,\dd x \Bigr|\biggr]
 \leq C_{\delta, \varphi}\e^{-R} \,.
\]
Moreover, \(|\log u|^2\le C(u+u^{-1})\), together with the first moment
\(\bbE[\scrZ_{\delta,1}^\vartheta(g_\delta^{(x)})]=1\) and the negative-first-moment
bound above, gives
\[
 \bbE\bigl[|\scrH^{(\delta)}(\varphi)|^2\bigr]
 \le C_\delta\|\varphi\|_1^2.
\]

Using \eqref{eq:convergence-truncated}--\eqref{eq:convergence-truncated-expectations} and then letting $R\to\infty$ this proves the convergence \eqref{target:fixed-regularisation}.
Once again, the convergence applies jointly for any finite collection of fixed
values of $\delta$.
\qed

\subsection{Local approximation: proof of Proposition~\ref{prop:local}}\label{subsec:local}

In this section, we consider the local approximation error from~\eqref{eq:L}, that
we rewrite as follows, recalling \eqref{eq:Wab}:
\begin{equation}\label{eq:local-as-difference}
 L_{m,n}(x)=\log\widehat Z_{0,n}(x)-\log\widehat Z_{m,n}(x).
\end{equation}
We will apply the covariance bound~\eqref{eq:polymer-covariance-convexity}, for which
the key technical input is the following one-point moment comparison,
proved in Section~\ref{sec:vanishing-defect} below. It is
a consequence of a Gaussian convexity property \cite[Theorem~1]{Chen26}
and the resulting moment comparison bound in Lemma~\ref{lem:defect} on
\[
\Delta(F)\coloneqq -\bbE F-\tfrac12\Var(F).
\]
It also relies on the block factorisation of Gu and
Tsai~\cite{GT26} that we detail in Section~\ref{sec:factorisation}.

\begin{proposition}[One-point moment comparison]
\label{prop:VD}
Recall
\(\widehat Z_{s,n}(x)=Z_{s,n}(q_s^{(x)},\one)\) from \eqref{eq:Wab},
with \(q_0^{(x)}=\delta_x\).
For every \(0<T_0<T_1<\infty\), there is a constant
\(C_{T_0,T_1}<\infty\) such that, for all sufficiently large \(N\),
uniformly in \(x\) and even integers \(0\le s\le n\) with
\(T_0N\le n\le T_1N\),
\begin{equation}\label{prop:eq-heat-defect}
 0\le\Delta(\log\widehat Z_{s,n}(x))
 \le\log\frac{\bbE[\widehat Z_{s,n}(x)^2]}
                 {\bbE[\widehat Z_{s,n}(x)^{-1}]}
 \le C_{T_0,T_1}.
\end{equation}
Moreover, there exist constants \(\delta_0\in(0,1)\) and \(C<\infty\)
such that, for every \(0<\delta\le\delta_0\) and \(n=[\sqrt\delta N]\),
uniformly in \(x\) and even integers \(0\le s\le n\),
\begin{equation}\label{target:VD}
 \log\frac{\bbE[\widehat Z_{s,n}(x)^2]}
              {\bbE[\widehat Z_{s,n}(x)^{-1}]}
 \le\frac{C}{|\log\delta|^{1/4}}+o_N(1).
\end{equation}
All constants may be chosen uniformly when \(\vartheta\) ranges over
a fixed compact interval, and \(o_N(1)\) tends to zero as
\(N\to\infty\) at fixed \(\delta\), uniformly in \(s,x\) and in
\(\vartheta\) over that interval.
\end{proposition}

The small-time bound \eqref{target:VD} quantifies how
\(\log\widehat Z_{s,n}(x)\) is close to a normal variable.
In this sense it complements the asymptotic log-normality
for locally averaged SHF proved in \cite[Theorem~1.1]{GT26}.
This quantitative comparison, with $N\to\infty$ before
$\delta\downarrow0$, is what makes the two $\Delta$ terms in
\eqref{eq:polymer-covariance-convexity} vanish.

\begin{proof}[Proof of Proposition~\ref{prop:local}]
Recalling \eqref{eq:entrance-times}, we take 
\begin{equation*}
      m=[\delta N] \,, \qquad
      n=[\sqrt\delta N] \,.
\end{equation*}
For $s\in\{0,m\}$, using spatial stationarity and the covariance bound~\eqref{eq:polymer-covariance-convexity} together with the bound \(\log u \leq u-1\), we get
\begin{equation}\label{eq:local-cov-bridge}
 \begin{aligned}
 0&\leq\Cov\pigl(\log\widehat Z_{s,n}(x),
                  \log\widehat Z_{s,n}(y)\pigr)\\
 &\leq\bbE[\widehat Z_{s,n}(x)\widehat Z_{s,n}(y)]-1
       +C|\log\delta|^{-1/4}+o_N(1).
 \end{aligned}
\end{equation}
where we have also used Proposition~\ref{prop:VD} to bound \(\Delta(\log\widehat Z_{s,n}(x))\) and \(\Delta(\log\widehat Z_{s,n}(y))\).

Now, recalling~\eqref{eq:local-replica-sum} and that \(n=[\sqrt{\delta} N]\), we obtain that for \(s\in \{0,m\}\)
\begin{align}
      \label{eq:bound-variance}
      & \Var\pigl(\sum_x
      \varphi^{(N)}(x)
      \log\widehat Z_{s,n}(x)\pigr)
       = \sum_{x,y}  \varphi^{(N)}(x)\varphi^{(N)}(y)
      \Cov\pigl(\log\widehat Z_{s,n}(x),
                  \log\widehat Z_{s,n}(y)\pigr)\\
      & \leq \frac{C}{|\log \delta|} \sum_{x,y} |\varphi^{(N)}(x)||\varphi^{(N)}(y)| \mathsf G_{n}(x-y)
            + \Bigl(\frac{C}{|\log\delta|^{1/4}}+o_N(1)\Bigr) \sum_{x,y} |\varphi^{(N)}(x)||\varphi^{(N)}(y)|  \,.\notag
\end{align}
All together, we get that
\[
\limsup_{N\to\infty} \Var\pigl(\sum_x
      \varphi^{(N)}(x)
      \log\widehat Z_{s,n}(x)\pigr) \leq \frac{C \mathcal{E}_{\sqrt{\delta}}(\varphi)}{|\log \delta|}   + \frac{C \|\varphi\|_1^2}{|\log\delta|^{1/4}} \,,
\]
where
\begin{equation}
      \label{eq:def-E-varphi}
\mathcal{E}_{\sqrt{\delta}}(\varphi) = \int_{\R^2\times \R^2} |\varphi(x)| |\varphi(y)| K_{\sqrt{\delta}}(x-y) \dd x \dd y \,,
\end{equation}
with
\[
 K_{\epsilon}(z)=4\int_0^{\epsilon}g_{2s}(z)\,\dd s
 =\int_0^{\epsilon}\frac{1}{\pi s}\e^{-|z|^2/(4s)}\,\dd s.
\]
This is the limit of \(\mathsf G_n([\sqrt{N/2}\,z])\) for \(z\ne0\) and
\(n/N\to\epsilon=\sqrt\delta\), with the spatial scaling from the definition \eqref{def:varphi-N} of $\varphi^{(N)}$, see \eqref{eq:lattice-cells}.
Since \(\int_{\R^2}K_\epsilon(z)\,\dd z=4\epsilon\), we have
\(\mathcal{E}_{\sqrt{\delta}}(\varphi)\le4\sqrt\delta\|\varphi\|_\infty\|\varphi\|_1\), which gives a negligible contribution
(we have kept this term in order to keep track
of the dependence of the constants on the
function \(\varphi\), see Remark~\ref{rem:better-bound-L} below).

We now apply these estimates to $\cL^{(N;\delta)}(\varphi)$, the object in Proposition~\ref{prop:local}:
recalling~\eqref{eq:average-L} and \eqref{eq:local-as-difference}, we have that
\[
\begin{split}
\cL^{(N;\delta)}(\varphi)
={}&\sum_x\varphi^{(N)}(x)
    \bigl\{\log\widehat Z_{0,n}(x)-\bbE\log\widehat Z_{0,n}(0)\bigr\}\\
 &-\sum_x\varphi^{(N)}(x)
    \bigl\{\log\widehat Z_{m,n}(x)-\bbE\log\widehat Z_{m,n}(0)\bigr\} \,.
\end{split}
\]
Since $\Var(X-Y)\le 2\Var(X)+2\Var(Y)$, applying the estimate~\eqref{eq:bound-variance}
for $s=0$ and $s=m$, and recalling that \(n= [\sqrt{\delta} N]\), we obtain that
\begin{equation}\label{target:local}
 \limsup_{N\to\infty} \bbE\bigl[|\cL^{(N;\delta)}(\varphi)|^2\bigr]
 \le C_\varphi |\log\delta|^{-1/4} \,.
\end{equation}
This proves Proposition~\ref{prop:local}.
\end{proof}

\begin{remark}
      \label{rem:better-bound-L}
      In fact, we may keep track of the dependence of the bound on the function~\(\varphi\):
      we have indeed obtained the following
      \begin{equation}
            \label{eq:limsup-local-L}
            \limsup_{N\to\infty} \bbE|\cL^{(N;\delta)}(\varphi)|^2 \leq \frac{C \mathcal{E}_{\sqrt{\delta}}(\varphi)}{|\log \delta|} +  \frac{C \|\varphi\|_1^2}{|\log\delta|^{1/4}} \,.
      \end{equation}
      In particular, this inequality holds even if the test function \(\varphi\) depends on \(\delta\).
      For instance we can apply~\eqref{eq:limsup-local-L} to \(\varphi_{\sqrt{\delta}}(\cdot) = \delta^{-1} \varphi(\cdot /\sqrt{\delta})\), which verifies \(\|\varphi_{\sqrt{\delta}}\|_1 = \|\varphi\|_{1}\).
      Scaling and splitting the time integral defining \(K_{1/\sqrt\delta}\) at time \(1\) give
      \[
       \mathcal E_{\sqrt\delta}(\varphi_{\sqrt\delta})
       =\mathcal E_{1/\sqrt\delta}(\varphi)
       \le C\|\varphi\|_\infty\|\varphi\|_1
          +C|\log\delta|\,\|\varphi\|_1^2.
      \]
      (The bound \(\mathcal E_{\sqrt\delta}(\varphi_{\sqrt\delta})
      \le4\sqrt\delta\|\varphi_{\sqrt\delta}\|_\infty
      \|\varphi_{\sqrt\delta}\|_1\) instead diverges as \(\delta^{-1/2}\).)
\end{remark}

\subsection{Intermezzo: about the deterministic centring}
\label{sec:centering}

As a corollary of Proposition~\ref{prop:VD} (and of the Gaussian convexity property), we are able 
to identify the order of the deterministic
centring constant in \eqref{eq:HN}.

\begin{lemma} \label{lem:centering}
Let \(t>0\) be fixed.
As \(N\to\infty\), we have that
\begin{equation*}
 \bbE\bigl[\log Z_{[Nt]}(0)\bigr]
 =-\tfrac12\log\log N+O_{\vartheta,t}(1).
\end{equation*}
\end{lemma}

\begin{proof}
By the Gaussian convexity property, we have
\[
 \tfrac12\log\bbE[Z_{[Nt]}(0)^{-1}]
 \le -\bbE[\log Z_{[Nt]}(0)]
 \le\tfrac12\log\bbE[Z_{[Nt]}(0)^2] \,.
\]
The first bound follows from the first inequality in~\eqref{eq:chen-three}
and the variance bound~\eqref{eq:chen-one}, and the second is~\eqref{eq:chen-two}.
Since \(\widehat Z_{0,[Nt]}(0)=Z_{[Nt]}(0)\), Proposition~\ref{prop:VD}-\eqref{prop:eq-heat-defect}
shows that 
\[
\Bigl| \bbE[\log Z_{[Nt]}(0)] +\tfrac12\log\bbE[Z_{[Nt]}(0)^2] \Bigr| \leq C \,.
\]

It thus only remains to estimate \(\log\bbE[Z_{[Nt]}(0)^2]\), but this is standard.
Indeed, from \cite[Proposition~A.2]{CSZ19a}, we have that \(\bbE[Z_{[Nt]}(0)^2] \sim C_{\vartheta,t} \log N\) as \(N\to\infty\).
We thus have that \(\log\bbE[Z_N(0)^2]=\log\log N+O_{\vartheta,t}(1)\), which proves the lemma.
\end{proof}

To recover the centring \(c_t^{\vartheta,(N)}\) in \eqref{eq:HN}, by time reversal and translation invariance, we have
\[
c_t^{\vartheta,(N)} \coloneqq 
 \bbE\bigl[\log \cZ_{t}^{\vartheta,(N)}(\one , x)\bigr]
 =\bbE[\log Z_{[Nt]}^{\beta_N}(0)]  = -\tfrac12\log\log N+O_{\vartheta,t}(1) \,,
\]
by Lemma~\ref{lem:centering}.
We can now compare the centring \(c_t^{\vartheta,(N)}\) with that for a general initial condition \(h_0\in L^\infty\).
By time reversal, we again have that 
\[
\bbE\bigl[\log \cZ_{t}^{\vartheta,(N)}(\e^{h_0} , x)\bigr] - c_t^{\vartheta,(N)} = \bbE\biggl[\log \frac{\cZ_{t}^{\vartheta,(N)}(x , \e^{h_0})}{\cZ_{t}^{\vartheta,(N)}(x , \one)}\biggr] = \bbE\bigl[\log \mu_{[Nt]}^{(x)}(\e^{h_0})\bigr]
\]
where \(\mu_{[Nt]}^{(x)}\) is the polymer measure started from \(x\). 
Since this is a probability measure, it readily yields
\(| c_t^{\vartheta,(N)} - \bbE[\log \cZ_{t}^{\vartheta,(N)}(\e^{h_0} , x)] | \leq \|h_0\|_{\infty}\),
that is~\eqref{eq:compare-centrings}.

\subsection{Global approximation: proof of Proposition~\ref{prop:global}}\label{subsec:PF}

Recall the definition~\eqref{eq:G} of the global approximation error
between times $n$ and $N$, namely
\[
 G_{m,n,N}(x)
 =\log Z_{n,N}(\mu_n^{(x)}) -
 \log Z_{n,N}(\widehat\mu_{m,n}^{(x)}).
\]
Note that both partition functions use the same disorder on $(n,N]$ and differ only in their
starting distributions at time $n$.
We rely on the following crucial
Proposition~\ref{prop:PF}, which controls the dependence on the starting distributions (see \eqref{eq:future-L1-PF} for a more precise statement),
by bounding the difference of the two partition functions and their second moments,
on an event of high probability.
Its proof again relies on the block decomposition of Gu and Tsai~\cite{GT26} and is postponed to Subsection~\ref{sec:prop-PF}.
Recall that $m=[\delta N]$ and $n=[\sqrt\delta N]$.

\begin{proposition}[Comparison of the partition functions]\label{prop:PF}
There exist constants $\delta_0\in(0,1)$ and $C<\infty$ such that
for every $0<\delta\le\delta_0$ and all sufficiently large $N$, there are events
$\mathcal G_{N,\delta}(x)\in\cF_n$ such that the following properties hold. First, uniformly in $x$, we have
\begin{equation}\label{eq:good-event-probability}
 \bbP\bigl(\mathcal G_{N,\delta}(x)^c\bigr)
 \le \frac{C}{|\log\delta|} \,.
\end{equation}
Second, we have the crucial comparison bound
\begin{equation}\label{target:PF}
 \bbE\bigl[\one_{\mathcal G_{N,\delta}(x)}
 |Z_{n,N}(\mu_n^{(x)})-Z_{n,N}(\widehat\mu_{m,n}^{(x)})|\bigr]
 \le \frac{C}{|\log\delta|} \,,
\end{equation}
Moreover, for $\nu=\mu_n^{(x)}$ or $\nu=\widehat\mu_{m,n}^{(x)}$,
\begin{equation}\label{eq:profile-good-energy}
 \bbE\bigl[\one_{\mathcal G_{N,\delta}(x)}Z_{n,N}(\nu)^2\bigr]
 \le C|\log\delta| \,.
\end{equation}
Finally, for every bounded deterministic terminal function \(f\), we also have
\begin{equation}\label{eq-weighted-PF-comparison}
 \bbE\bigl[\one_{\mathcal G_{N,\delta}(x)}
 |Z_{n,N}(\mu_n^{(x)},f)-Z_{n,N}(\widehat\mu_{m,n}^{(x)},f)|\bigr]
 \le C\|f\|_\infty\frac{\sqrt{\log|\log\delta|}}{|\log\delta|}.
\end{equation}
The constants may depend on $\vartheta$, but are uniform over $\vartheta$ in a compact set.
\end{proposition}
We also stress that the event \(\mathcal G_{N,\delta}(x)\) is in fact explicit (see~\eqref{eq:comparison-event}).
It is on this good event that we are able to control the difference between \(\mu_n^{(x)}\) and \(\widehat \mu_{m,n}^{(x)}\).
Aside from~\eqref{eq:good-event-probability} which already
appear in some form in \cite[Corollary~4.2]{GT26}
(we include the proof here to keep the paper self-contained),
the main novelty here is~\eqref{target:PF} (and~\eqref{eq:profile-good-energy}),
which controls how the partition function on \((n,N]\)
is affected by switching the initial condition from \(\mu_n^{(x)}\) to \(\widehat \mu_{m,n}^{(x)}\).
It strongly relies on Lemma~\ref{lem:profile-estimates}, which compares \(\mu_{n}\) and \(\widehat \mu_{m,n}\) with \(q_{n}\) (up to a leading correction) under some semi-norm that averages the disorder on \((n,N]\).

We also need the following proposition to control \(G_{m,n,N}\) on the event \(\mathcal G_{N,\delta}(x)^c\).
It controls the moments of \(\log Z_{n,N}(\nu)\). Its proof is postponed to Subsection~\ref{sec:NM}.

\begin{proposition}[Logarithmic moments]\label{prop:log-tails}
Fix $p\ge1$.
For every sufficiently small $\delta>0$,
the following bound holds for both $\nu=\mu_n^{(x)}$ and
$\nu=\widehat\mu_{m,n}^{(x)}$:
\begin{equation}\label{eq:future-log-moments}
 \limsup_{N\to\infty}\sup_x
 \|\log Z_{n,N}(\nu)\|_{L^p}
 \le C_p \log |\log\delta| \,.
\end{equation}
The constant $C_p$ may depend on $\vartheta$ but not on $\delta$.
For a positive deterministic terminal function \(f\) with
\(\|\log f\|_\infty<\infty\), the same bound holds with
\(Z_{n,N}(\nu)\) replaced by \(Z_{n,N}(\nu,f)\) and the right hand side replaced by $C_p \log |\log\delta|  +\| \log f \|_\infty$.
\end{proposition}

An important technical novelty in the proof is good control of the left tail of \(\log Z_{n,N}(\nu)\)
(the right tail is in fact easily controlled), see Lemma~\ref{lem:log-tail} below.
It greatly improves on the left-tail obtained in \cite[Equation~(5.2)]{GT26}.
The Gaussian convexity result from Section~\ref{sec:general}
plays a central role since it allows us to control negative moments of \(Z_{n,N}(\nu)\) in terms of its second moment, thanks to Lemma~\ref{lem:one-point}-\eqref{eq:one-point-negative}.

\begin{remark}
      \label{rem:Lp-boundedness}
      The moment bounds \eqref{eq:FC-moments} and the convergence
      \eqref{eq:convergence-initial} imply that, for every \(p\ge1\) and all
      sufficiently small \(\delta>0\),
      \begin{equation}
            \label{eq:Lp-bound-SHF}
            \|\log \scrZ_{\delta,1}^\vartheta(g_\delta^{(x)})\|_{L^p}
            \le C_p\log|\log\delta| \,.
      \end{equation}
      Indeed, for either \(X=\widehat Z_{m,N}(x)\) or
      \(X=\scrZ_{\delta,1}^\vartheta(g_\delta^{(x)})\), we have
      \(\bbE[X]=1\) and \(\bbE[X^{-1}]\le1+C\log(1/\delta)\),
      where the bound in the SHF case follows from weak convergence and Fatou. 
      Hence, for \(u\ge0\),
      \[
       \bbP(|\log X|>u)\le\min\{1,C|\log\delta|\e^{-u}\}.
      \]
      Integrating this tail bound proves \eqref{eq:Lp-bound-SHF}.
      By Minkowski's inequality and the centring in \eqref{eq:regularised-Hbis},
      we also obtain
      \begin{equation*}
           \|\scrH^{(\delta)}(\varphi)\|_{L^p}
           \le C_p\|\varphi\|_1\log|\log\delta| \,.
      \end{equation*}
\end{remark}

We now have the tools to conclude the proof of Proposition~\ref{prop:global}.

\begin{proof}[Proof of Proposition~\ref{prop:global}]
Fix $\delta>0$ sufficiently small and $x$, and write
$\mathcal G=\mathcal G_{N,\delta}(x)$.
We estimate $G_{m,n,N}(x)$ separately on $\mathcal G$ and
$\mathcal G^c$.

First of all, for $u,v>0$ and $0<\alpha<1$, a straightforward calculation gives that
\begin{equation}\label{eq:log-holder-NM}
 |\log u-\log v|
 \le C_\alpha|u-v|^\alpha
                 (u^{-\alpha}+v^{-\alpha}).
\end{equation}
Let us fix \(\alpha\) (we take \(\alpha=1/3\) below).
We get that
\begin{equation}\label{eq:good-log-NM-0}
     \begin{split}
      & \bbE[|G_{m,n,N}(x)|\one_{\mathcal G}] \\
      & \quad \leq C_{\alpha}  \bbE\bigl[ \one_{\mathcal G}\bigl|Z_{n,N}(\mu_n^{(x)})-  Z_{n,N}(\widehat\mu_{m,n}^{(x)})\bigr|^{\alpha} \bigl(Z_{n,N}(\mu_n^{(x)})^{-\alpha}+ Z_{n,N}(\widehat\mu_{m,n}^{(x)})^{-\alpha} \bigr)\bigr] \\
      & \quad \leq 2 C_{\alpha}  \bbE\bigl[ \one_{\mathcal G}\bigl|Z_{n,N}(\mu_n^{(x)})-  Z_{n,N}(\widehat\mu_{m,n}^{(x)})\bigr| \bigr]^{\alpha}  \max_{\nu=\mu_n^{(x)},\widehat\mu_{m,n}^{(x)}}\bbE\bigl[\one_{\mathcal G}Z_{n,N}(\nu)^{-\alpha/(1-\alpha)}\bigr]^{1-\alpha} \,,
     \end{split}
\end{equation}
using H\"older's inequality for the last line.

The first term can be controlled by Proposition~\ref{prop:PF}-\eqref{target:PF}.
For the second term, we use that with either starting distribution $\nu=\mu_n^{(x)}$ or $\nu=\widehat\mu_{m,n}^{(x)}$ (which are \(\cF_n\)-measurable),
the partition function
$Z_{n,N}(\nu)$ has conditional mean \(\bbE[Z_{n,N}(\nu)\mid \cF_n] =1\), and that \(\log Z_{n,N}(\nu)\) is convex in the Gaussian disorder on $(n,N]$.
We may thus apply Lemma~\ref{lem:one-point}-\eqref{eq:one-point-negative} with $p=\alpha/(1-\alpha)$, to obtain that
\[
 \bbE\bigl[Z_{n,N}(\nu)^{-\alpha/(1-\alpha)}\mid\cF_n\bigr]
 \le \bbE\bigl[Z_{n,N}(\nu)^2\mid\cF_n\bigr]^{\alpha/2(1-\alpha)^2}.
\]
Since $\mathcal G\in\cF_n$, we thus get that
\[
\bbE\bigl[\one_{\mathcal G}Z_{n,N}(\nu)^{-\alpha/(1-\alpha)}\bigr] \leq \bbE\bigl[\one_{\mathcal G}\bbE[Z_{n,N}(\nu)^2 \mid \cF_n]^{\alpha/2(1-\alpha)^2}\bigr] \leq
       \bbE\bigl[\one_{\mathcal G} Z_{n,N}(\nu)^2\bigr]^{\alpha/2(1-\alpha)^2} \,,
\]
having used H\"older's inequality for the last step
and bounded \(\bbP(\mathcal G)\leq 1\).
Using Proposition~\ref{prop:PF}-\eqref{eq:profile-good-energy}, this gives
\begin{equation*}
 \bbE\bigl[\one_{\mathcal G}Z_{n,N}(\nu)^{-\alpha/(1-\alpha)}\bigr]^{1-\alpha}
 \le C|\log\delta|^{ \alpha /2(1-\alpha)}.
\end{equation*}
Going back to~\eqref{eq:good-log-NM-0} and using also~\eqref{target:PF},
we thus end up with
\begin{equation}\label{eq:good-log-NM}
      \bbE\bigl[\one_{\mathcal G}|G_{m,n,N}(x)|\bigr]
      \le C |\log\delta|^{- \alpha}
      |\log\delta|^{\alpha /2(1-\alpha)}
      \leq C |\log\delta|^{-1/12} \,,
\end{equation}
where we have chosen \(\alpha =1/3\) for the last inequality
(the optimal value would be attained at \(\alpha =1- 2^{-1/2}\)).

\smallskip
On $\mathcal G^c$, Cauchy--Schwarz inequality gives
\[
 \bbE\bigl[\one_{\mathcal G^c}|G_{m,n,N}(x)|\bigr]
 \le \bbP(\mathcal G^c)^{1/2}
 \bigl\{\|\log Z_{n,N}(\mu_n^{(x)})\|_{L^2}
       +\|\log Z_{n,N}(\widehat\mu_{m,n}^{(x)})\|_{L^2}\bigr\}.
\]
Using Proposition~\ref{prop:PF}-\eqref{eq:good-event-probability} and
Proposition~\ref{prop:log-tails} with $p=2$, we conclude that
\begin{equation}\label{eq:bad-log-NM}
 \limsup_{N\to\infty}\sup_x
 \bbE\bigl[\one_{\mathcal G^c}|G_{m,n,N}(x)|\bigr]
 \le \frac{C \log |\log\delta|}{|\log\delta|^{1/2}}.
\end{equation}
This is bounded by $C|\log\delta|^{-1/12}$ for small $\delta$.
Adding \eqref{eq:good-log-NM} and \eqref{eq:bad-log-NM} therefore yields
\begin{equation}\label{eq:global-quantitative}
 \limsup_{N\to\infty}\sup_x\bbE|G_{m,n,N}(x)|
 \le \frac{C}{|\log\delta|^{1/12}}.
\end{equation}

Now, applying again Proposition~\ref{prop:log-tails}, we get that
\begin{equation}\label{eq:global-L2}
 \begin{split}
 \limsup_N\sup_x\|G_{m,n,N}(x)\|_{L^2}
 &\le\limsup_N\sup_x
       \|G_{m,n,N}(x)\|_{L^1}^{1/3}
       \|G_{m,n,N}(x)\|_{L^4}^{2/3}\\
 &\le \frac{C}{|\log \delta|^{1/36}} (\log |\log \delta|)^{2/3} \leq \frac{C'}{|\log \delta|^{1/40}} \,.
 \end{split}
\end{equation}

Finally, recalling the definition~\eqref{eq:average-G} of \(\cG^{(N;\delta)}(\varphi)\), by the triangle inequality we get
\begin{equation}
\label{eq:last-bound-global}
\|\cG^{(N;\delta)}(\varphi)\|_{L^2}
 \le 2\sum_{x\in\Z^2_{\rm even}} |\varphi^{(N)}(x)|
 \|G_{m,n,N}(x)\|_{L^2} \leq 2 \|\varphi\|_{L^1} \sup_x\|G_{m,n,N}(x)\|_{L^2} \,.
\end{equation}
Therefore,
\eqref{eq:global-L2} proves Proposition~\ref{prop:global}.
\end{proof}

\begin{remark}\label{rem:weighted-global}
We now extend the bound \eqref{eq:global-L2}
to non-constant terminal weights.
For a positive deterministic terminal function \(f\) with
\(\|\log f\|_\infty<\infty\), define
\begin{equation}\label{eq:weighted-G}
 G_{m,n,N}(x,f)
 \coloneqq \log\frac{Z_{n,N}(\mu_n^{(x)},f)}
                        {Z_{n,N}(\widehat\mu_{m,n}^{(x)},f)}.
\end{equation}
The comparison
\(\e^{-\|\log f\|_\infty}Z_{n,N}(\nu)\le Z_{n,N}(\nu,f)
 \le\e^{\|\log f\|_\infty}Z_{n,N}(\nu)\)
transfers the negative-moment estimate used in
\eqref{eq:good-log-NM-0} to \(f\).
Using \eqref{eq-weighted-PF-comparison} in place of \eqref{target:PF}
then introduces an additional factor \((\log|\log\delta|)^{1/6}\)
in \eqref{eq:good-log-NM}.
To obtain the analogue of \eqref{eq:global-quantitative} for $G_{m,n,N}(x,f)$, 
we estimate \(\bbE[\one_{\mathcal G^c}|G_{m,n,N}(x,f)|]\) as in
\eqref{eq:bad-log-NM}, using Proposition~\ref{prop:log-tails} with \(p=2\).
The same proposition with \(p=4\) bounds
\(\|G_{m,n,N}(x,f)\|_{L^4}\).
Interpolation as in \eqref{eq:global-L2} therefore gives
\begin{equation}\label{eq-weighted-global-L2}
 \limsup_{N\to\infty}\sup_x
 \|G_{m,n,N}(x,f)\|_{L^2}
 \le C_{\|\log f\|_\infty}|\log\delta|^{-1/40}.
\end{equation}
The bound is uniform over terminal functions with a common bound on
\(\|\log f\|_\infty\).
\end{remark}

\begin{remark}
      \label{rem:better-bound-G}
      Let us record here the dependence of the bound on the function \(\varphi\).
      Combining~\eqref{eq:last-bound-global} with~\eqref{eq:global-L2} yields
      \begin{equation}
            \label{eq:limsup-global-G}
            \limsup_{N\to\infty}  \bbE\bigl|\cG^{(N;\delta)}(\varphi)\bigr|^2
             \leq \frac{C \|\varphi\|_1^2}{|\log \delta|^{1/20}}.
      \end{equation}
      In particular, we may apply~\eqref{eq:limsup-global-G} to \(\varphi_{\delta}(\cdot) = \delta^{-2} \varphi(\cdot/\delta)\), which verifies \(\|\varphi_{\delta}\|_1 = \|\varphi\|_1\).
\end{remark}

\section{Extension to joint convergence and non-flat initial condition}
\label{sec:joint-convergence}

In this section, we extend the convergence in Theorem~\ref{thm:toy} 
to joint convergence at finitely many times and to non-flat initial condition \(f_0=\e^{h_0} \not \equiv \one\).
Let us first make the time-reversal more explicit than in Section~\ref{subsec:strategy-approximation}.
For
\(0\le t<T\), \(0<\delta<T-t\), we write
\begin{equation}
\label{prop:eq-reversed-temporal-field}
 \scrH_{t,T}^{(\delta)}(x,h_0)
 \coloneqq  \log\scrZ_{t+\delta,T}^{\vartheta}
              (g_\delta^{(x)},\e^{h_0})
 -\bbE\bigl[\log\scrZ_{t+\delta,T}^{\vartheta}
              (g_\delta^{(0)},\one)\bigr].
\end{equation}
Here \(x\) is the starting point at time \(t\), and \(h_0\) specifies the
terminal weight \(f_0=\e^{h_0}\) at time~\(T\). Note that disorder is removed on \([t,t+\delta]\). 
Let \(k\in \N\) and let \(0<\delta_i<t_i\le T\).
Then, invariance of the law of SHF under time reversal gives
\begin{equation}
\label{prop:eq-joint-time-reversal}
 \bigl(\scrH_{t_i}^{(\delta_i)}(h_0,\cdot)\bigr)_{i=1}^k
 \overset{d}{=}
 \bigl(\scrH_{T-t_i,T}^{(\delta_i)}(\cdot,h_0)\bigr)_{i=1}^k.
\end{equation}

We may similarly consider the discrete counterparts.
Let \(f_N\) be a discrete approximation of \(f_0=\e^{h_0}\), defined by the cell averages
\begin{equation}
      \label{eq:cell-averages}
      f_N(z)=|\cellC_N(z)|^{-1}\int_{\cellC_N(z)}\e^{h_0(y)}\,\dd y \,,
\end{equation}
with \(\cellC_N(z)\) the cell defined in \eqref{eq:lattice-cells}.
Let also \(x_N \coloneqq [\sqrt{N/2}\, x]\), so that \(x\in \cellC_N(x_N)\).
For \(0\le t_i<T\) and \(0<\delta<T-t_i\), define
\begin{equation}
\label{eq:def-cH-general}
 \cH^{(N)}_{t_i,T}(x,h_0)
 \coloneqq \log Z_{[Nt_i],[NT]}(x_N,f_N)-\bbE\log Z_{[Nt_i],[NT]}(x_N,\one),
\end{equation}
and also
\begin{equation}
\label{eq:def-cH-delta-general}
\cH^{(N;\delta)}_{t_i,T}(x,h_0)
\coloneqq \log Z_{[N(t_i+\delta)],[NT]}(q_{[\delta N]}^{(x_N)},f_N)
       -\bbE\log Z_{[N(t_i+\delta)],[NT]}(q_{[\delta N]}^{(x_N)},\one) \,.
\end{equation}
For a test function \(\varphi\), write
\(\scrH_{t,T}^{(\delta)}(\varphi,h_0)
=\int_{\R^2}\varphi(x)\scrH_{t,T}^{(\delta)}(x,h_0)\,\dd x\),
and similarly for the lattice fields.

We now extend Theorem~\ref{thm:toy} to 
a joint convergence at different times and different test functions, with common terminal weight \(f_0=\e^{h_0}\).
To this end, 
we extend Proposition~\ref{prop:FC} to a joint convergence with non-flat initial condition (see Lemma~\ref{lem:extension-convergence-delta} below)
and we also extend Proposition~\ref{prop:smo} to the case of general initial condition (see Lemma~\ref{lem:extension-approximation}).
Together, this proves the following finite-dimensional convergence that extends Theorem~\ref{thm:toy} as follows. 

\begin{theorem}[Critical 2D KPZ Process, finite-dimensional convergence]
      \label{thm:toy-joint}
Fix \(\vartheta\in\R\), \(T>0\),
\(h_0\in L^\infty(\R^2)\cap C(\R^2)\), and \(k\in\mathbb N\).
For \(1\le i\le k\), let \(0\le t_i<T\) and
\(\varphi_i \in C_c^{\infty}(\R^2)\).
We then have the following joint convergence, where the limits
\(\scrH_{t_i,T}^\vartheta(\varphi_i,h_0)\) are defined on the same probability space as the SHF:
\begin{equation*}
\begin{aligned}
 \bigl(\scrH_{t_i,T}^{\vartheta,(\delta)}(\varphi_i, h_0)\bigr)_{i=1}^k
 &\xrightarrow[\; \delta\downarrow0 \;]{L^p}
 \bigl(\scrH_{t_i,T}^\vartheta(\varphi_i,h_0)\bigr)_{i=1}^k \,, \quad \text{for} \,\, p\in [1,\infty)\\
 \bigl(\cH_{t_i,T}^{\vartheta,(N)}(\varphi_i,h_0)\bigr)_{i=1}^k
 &\xrightarrow[N\to\infty]{d}
 \bigl(\scrH_{t_i,T}^\vartheta(\varphi_i,h_0)\bigr)_{i=1}^k\,.
\end{aligned}
\end{equation*}
By time reversal~\eqref{prop:eq-joint-time-reversal},
\begin{equation*}
 \bigl(\scrH_{t_i,T}^\vartheta(\varphi_i,h_0)\bigr)_{i=1}^k
 \overset{d}{=}
 \bigl(\scrH_{T-t_i}^\vartheta(h_0,\varphi_i)\bigr)_{i=1}^k.
\end{equation*}
\end{theorem}

Once Lemmas~\ref{lem:extension-convergence-delta} and~\ref{lem:extension-approximation} are at hand, the proof of Theorem~\ref{thm:toy-joint}
is analogous to the proof in  Section~\ref{subsec:strategy-approximation}. 
The extension to convergence in \(H_{\rm loc}^{-a}\) for any \(a>0\) is proven in Section~\ref{sec:local-Sobolev} below.

\subsection{Extension of Proposition~\ref{prop:FC}}

The following result extends Proposition~\ref{prop:FC} to joint convergence at different starting times,
with a common terminal time \(T\) and terminal weight \(f_0=\e^{h_0}\).

\begin{lemma}[Joint convergence at fixed temporal regularisation]
\label{lem:extension-convergence-delta}
Fix \(\vartheta\in\R\), \(T>0\),
\(h_0\in L^\infty(\R^2)\cap C(\R^2)\), and \(k\in\mathbb N\).
For \(1\le i\le k\), let \(0\le s_i<T\), \(\delta_i>0\), and
\(x_i\in\R^2\). 
Under the critical tuning \eqref{eq:critical-tuning},
the SHF integrals
\(\scrZ_{s_i,T}^\vartheta(g_{\delta_i}^{(x_i)},\e^{h_0})\)
are almost surely positive, and letting \(f_N\) be defined as in~\eqref{eq:cell-averages} we have that
\begin{equation}\label{prop:eq-fixed-mass-convergence}
 \Bigl(\log Z_{[Ns_i],[NT]}^{\beta_N}
 \bigl(q_{[N\delta_i]}^{([\sqrt{N/2}\,x_i])},f_N\bigr)\Bigr)_{i=1}^k
 \xrightarrow[N\to\infty]{d}
 \left(\log\scrZ_{s_i,T}^\vartheta
       (g_{\delta_i}^{(x_i)},\e^{h_0})\right)_{i=1}^k.
\end{equation}
All absolute moments of positive order are finite and converge as well.
Moreover, for \(\varphi_1,\ldots,\varphi_k\in C_c^\infty(\R^2)\),
letting \(\varphi_i^{(N)}\) be defined as in \eqref{def:varphi-N}, we have
\begin{equation}\label{prop:eq-fixed-field-convergence}
\begin{gathered}
 \Bigl(\sum_z\varphi_i^{(N)}(z)
       \log Z_{[Ns_i],[NT]}^{\beta_N}
       \bigl(q_{[N\delta_i]}^{(z)},f_N\bigr)\Bigr)_{i=1}^k
 \xrightarrow[N\to\infty]{d}
 \Bigl(\int_{\R^2}\varphi_i(x)
       \log\scrZ_{s_i,T}^\vartheta
       (g_{\delta_i}^{(x)},\e^{h_0})\,\dd x\Bigr)_{i=1}^k.
\end{gathered}
\end{equation}
\end{lemma}

As a consequence of this lemma, taking \(s_i=t_i+\delta_i<T\) with
\(t_i\ge0\), we have the joint convergence
\begin{equation}
      \label{eq:joint-convergence-delta}
     \Big( \cH^{(N;\delta_i)}_{t_i,T}(\varphi_i,h_0)\Big)_{i=1}^k
     \xrightarrow[N\to\infty]{d}   \Big( \scrH^{(\delta_i)}_{t_i,T}(\varphi_i,h_0)\Big)_{i=1}^k \,.
\end{equation}
Indeed, this is only a matter of subtracting the centring constants (with flat final weight),
whose convergence follows from the moment convergence provided by Lemma~\ref{lem:extension-convergence-delta}.

\begin{proof}[Proof of Lemma~\ref{lem:extension-convergence-delta}]
To simplify notation, let us introduce the abbreviation
\[
 A_N^{(i)}(x;f)
 \coloneqq Z_{[Ns_i],[NT]}^{\beta_N}
       \bigl(q_{[N\delta_i]}^{([\sqrt{N/2}\,x])},f\bigr) \quad \mbox{and} \quad
 A^{(i)}(x;f) \coloneqq \scrZ_{s_i,T}^\vartheta(g_{\delta_i}^{(x)},f).
\]
We first establish the joint convergence of the partition functions before taking logarithm:
\begin{equation}\label{611}
 \bigl(A_N^{(i)}(x_i;f_N)\bigr)_{i=1}^k
 \xrightarrow[N\to\infty]{d}
 \bigl(A^{(i)}(x_i;\e^{h_0})\bigr)_{i=1}^k.
\end{equation}
For compactly supported initial test functions and bounded continuous
terminal functions, the joint convergence of such averaged partition functions
follows from \cite[Theorem~9.1]{CSZ23}, recalled
in Theorem~\ref{thm:SHF-convergence}. To allow Gaussian initial test functions such as \(g_{\delta_i}^{(x)}\), 
we use Step~1 of
Section~\ref{sec:FC}, in particular \eqref{eq:L1-control}--\eqref{eq:L1-error}.
The same estimates hold with an additional factor
\(\e^{\|h_0\|_\infty}\).

We next prove the a.s.\ strict positivity of \(A^{(i)}(x_i;\e^{h_0})\)
and control their logarithms. For terminal condition $\one$,
the moment bounds in \eqref{eq:FC-moments} give
\begin{equation}\label{eq:negative-moment}
 \bbE[A_N^{(i)}(x;\one)]=1,\qquad
 \bbE[A_N^{(i)}(x;\one)^{-1}]
 \le\bbE[A_N^{(i)}(x;\one)^2]
 \le C_{\delta_i,T-s_i}.
\end{equation}
These bounds are uniform in \(x\) and sufficiently large \(N\). The
constant can be chosen uniformly when \(\delta_i\) and \(T-s_i\) range
over compact subsets of \((0,\infty)\), and \(\vartheta\) ranges over a
compact subset of \(\R\). Since
\[
 \e^{-\|h_0\|_\infty}A_N^{(i)}(x;\one)
 \le A_N^{(i)}(x;f_N)
 \le \e^{\|h_0\|_\infty}A_N^{(i)}(x;\one),
\]
the corresponding negative-first moments are uniformly bounded as well.
Using the bound \eqref{eq:negative-moment} and applying weak convergence
and Fatou's Lemma, we obtain
\[
 \bbE[A^{(i)}(x;\e^{h_0})^{-1}]
 \le \e^{\|h_0\|_\infty}C_{\delta_i,T-s_i}<\infty.
\]
This shows that $A^{(i)}(x;\e^{h_0})>0$ almost surely and its logarithm is well-defined. This
allows us to pass the weak convergence of the partition functions in \eqref{611}
to the weak convergence of their logarithms in \eqref{prop:eq-fixed-mass-convergence}.

To establish the convergence of the moments of the logarithms, Markov's inequality and
\eqref{eq:negative-moment} give
\begin{equation}\label{eq-fixed-log-tails}
 \bbP\bigl(|\log A_N^{(i)}(x;\one)|>u\bigr)
 \le (1+C_{\delta_i,T-s_i})\e^{-u},\qquad u\ge0.
\end{equation}
Replacing \(\one\) by \(f_N\) only enlarges the tail bound by a factor of
\(\e^{\|h_0\|_\infty}\), and the same bounds hold for the limiting
SHF integrals. They imply uniform integrability of every positive power
of the absolute logarithm, and hence convergence of all absolute
moments of positive order.

To prove \eqref{prop:eq-fixed-field-convergence}, apply Steps~2 and~3
of Section~\ref{sec:FC} jointly to \(\varphi_1,\ldots,\varphi_k\).
The approximation by finite spatial sums is unchanged, since
\eqref{eq:L1-control}--\eqref{eq:L1-error} hold with an additional factor
\(\e^{\|h_0\|_\infty}\). The logarithmic tail bound
\eqref{eq-fixed-log-tails} removes the truncation uniformly in \(N\).
\end{proof}

\subsection{Extension of Proposition~\ref{prop:smo}}

We now extend the mean square approximation in Proposition~\ref{prop:smo} to the case of a general (non-flat) terminal weight.
Recall the definitions~\eqref{eq:def-cH-general} and \eqref{eq:def-cH-delta-general} of the fields \(\cH_{t,T}^{(N)}\)
and \( \cH_{t,T}^{(N;\delta)}\).

\begin{lemma}
\label{lem:extension-approximation}
Fix \(h_0\in L^\infty(\R^2)\cap C(\R^2)\) and let \(0\le t<T\).
For \(\varphi\in C_c^\infty(\R^2)\) and all sufficiently small
\(\delta>0\),
\begin{equation*}
 \limsup_{N\to\infty}
 \bigl\|\cH_{t,T}^{(N)}(\varphi,h_0) - \cH_{t,T}^{(N;\delta)}(\varphi,h_0)\bigr\|_{L^2}
 \le C_{T-t,\vartheta,\varphi,\|h_0\|_\infty}|\log\delta|^{-1/40} \,.
\end{equation*}
The bound is locally uniform in \((T-t,\vartheta)\in(0,\infty)\times\R\).
\end{lemma}

\begin{proof}
By time translation and diffusive rescaling, we may reduce to the case
\(t=0,T=1\), replacing \(\delta\) by \(\delta/(T-t)\) and
\(\vartheta\) by \(\vartheta+\log(T-t)\).
We thus drop the dependence on \(t,T\) in the following.

Let \(m=[\delta N]\) and \(n=[\sqrt\delta N]\), and note that
\(\e^{-\|h_0\|_\infty}\le f_N\le\e^{\|h_0\|_\infty}\).
The decomposition at time \(n\) in \eqref{eq:CKn}--\eqref{eq:CKnhat}
gives
\begin{equation}
\label{eq-weighted-error-decomposition}
\begin{aligned}
(\cH^{(N)}-\cH^{(N;\delta)})(\varphi,h_0)
=\cL^{(N;\delta)}(\varphi)
  +\sum_z\varphi^{(N)}(z)
     \bigl\{G_{m,n,N}(z,f_N)-\bbE G_{m,n,N}(0,\one)\bigr\},
\end{aligned}
\end{equation}
where \(\cL^{(N;\delta)}\) is defined in \eqref{eq:average-L} and
\(G_{m,n,N}(z,f_N)\) in \eqref{eq:weighted-G}.
The local term \(L_{m,n}\) in \eqref{eq:L} and its centring do not
depend on \(f_N\), so the local error \(\cL^{(N;\delta)}\) is unchanged
and Proposition~\ref{prop:local} still applies.
For the global term in \eqref{eq-weighted-error-decomposition}, we can apply
Remark~\ref{rem:weighted-global}-\eqref{eq-weighted-global-L2}
with \(f=f_N\), noting that
\(\|\log f_N\|_\infty\le\|h_0\|_\infty\).
Triangle inequality and
\(\sum_z|\varphi^{(N)}(z)|\le\|\varphi\|_1\) give the required bound,
while the centring constant $\bbE G_{m,n,N}(0,\one)$ is bounded by \eqref{eq:global-L2}.
This concludes the proof of the lemma.
\end{proof}

\subsection{Proof of Theorem~\ref{thm:toy-joint}}
The same argument as in \eqref{eq:L2-Cauchy}, together with \eqref{eq:joint-convergence-delta} applied
jointly to two regularisation parameters \(\delta,\delta'\), and Lemma~\ref{lem:extension-approximation}, 
shows that \(\scrH_{t_i,T}^{(\delta)}(\varphi_i,h_0)\) is Cauchy in \(L^2\)
as \(\delta\downarrow0\), for each \(i\). Denote its limit by
\(\scrH_{t_i,T}(\varphi_i,h_0)\). Letting \(\delta'\downarrow0\)
in the same estimate gives
\begin{equation}\label{prop:eq-temporal-approximation}
 \|\scrH_{t_i,T}^{(\delta)}(\varphi_i,h_0)
          -\scrH_{t_i,T}(\varphi_i,h_0)\|_{L^2}
 \le C_{T-t_i,\vartheta,\varphi_i,\|h_0\|_\infty}
          \left|\log\frac{\delta}{T-t_i}\right|^{-1/40},
\end{equation}
where the rescaling in Lemma~\ref{lem:extension-approximation}
replaces \(\delta\) by \(\delta/(T-t_i)\).
The bound holds for sufficiently small \(\delta/(T-t_i)\), locally
uniformly in \((T-t_i,\vartheta)\in(0,\infty)\times\R\).

This shows the convergence in \(L^2\) and we now show that 
the convergence also holds in $L^p$ for any $p \in [1,\infty)$.
The case \(p\in [1,2]\) follows from the $L^2$ convergence so we focus on the case \(p>2\) and some fixed~\(i\).
Let us set for simplicity $X_\delta=\scrH_{t_i,T}^{(\delta)}(\varphi_i,h_0)$ and set $\delta_j=(T-t_i)\exp(-2^j)$ for \(j\geq 1\).
By \eqref{prop:eq-temporal-approximation} we obtain that for any \(j\geq 1\)
\[
 \sup_{\delta_{j+1}\le\delta\le\delta_j}
 \|X_\delta-X_{\delta_j}\|_{L^2}
 \le C2^{-j/40},
\]
On the other hand, by Remark~\ref{rem:Lp-boundedness}, after time translation
and rescaling we obtain
\[
\sup_{\delta_{j+1}\le\delta\le\delta_j}
 \|X_\delta-X_{\delta_j}\|_{L^{2p}}
 \le C_p(j+1) \,,
\]
where the bound~\eqref{eq:Lp-bound-SHF} extends to any initial condition $h_0\in L^{\infty}(\R^2)$ because changing the
terminal weight from $\one$ to $\e^{h_0}$ changes
$X_\delta$ by at most $\|\varphi_i\|_1\|h_0\|_\infty$.
Therefore, by H\"older's inequality we get that 
\begin{equation}
      \label{eq:Lp-bound-Cauchy}
       \begin{split}
      \sup_{\delta_{j+1}\le\delta\le\delta_j}
      \|X_\delta-X_{\delta_j}\|_{L^p}  
      & \le \sup_{\delta_{j+1}\le\delta\le\delta_j}
      \|X_\delta-X_{\delta_j}\|_{L^{2}}^{1/(p-1)} \|X_\delta-X_{\delta_j}\|_{L^{2p}}^{(p-2)/(p-1)} \\
      & \leq  C_p\,2^{-j/[40(p-1)]}
                  (j+1)^{(p-2)/(p-1)}.
 \end{split}
\end{equation}
This shows in particular that \(\|X_{\delta_{j+1}}-X_{\delta_j}\|_{L^p} \) is summable in \(j\), so that $X_{\delta_j}$ converges
in $L^p$ to its already identified $L^2$ limit.
Applying again~\eqref{eq:Lp-bound-Cauchy} shows the \(L^p\) convergence for arbitrary $\delta\downarrow0$.

For the lattice fields $\cH_{t_i,T}^{\vartheta,(N)}(\varphi_i,h_0)$, we can apply the same 
bounded Lipschitz distance argument as in the proof of Theorem~\ref{thm:toy}, but now on \(\R^k\).
Letting \(N\to\infty\) and then \(\delta\downarrow0\), we obtain
\begin{equation}\label{prop:eq-test-function-convergence}
 \bigl(\cH_{t_i,T}^{(N)}(\varphi_i,h_0)\bigr)_{i=1}^k
 \xrightarrow[N\to\infty]{d}
 \bigl(\scrH_{t_i,T}(\varphi_i,h_0)\bigr)_{i=1}^k.
\end{equation}
The statement on time reversal is a direct consequence of \eqref{prop:eq-joint-time-reversal}.
\qed

\section{Convergence in local Sobolev spaces}
\label{sec:local-Sobolev}

Theorem~\ref{thm:toy-joint} gives joint convergence of the fields \(\cH^{\vartheta,(N)}\)
and \(\scrH^{\vartheta,(\delta)}\) against test functions. We now establish the convergence assertions
of Theorem~\ref{thm:main} for one fixed initial condition
\(h_0\in L^\infty(\R^2)\cap C(\R^2)\), jointly at finitely many times
and in the local Sobolev space \(H^{-a}_{\mathrm{loc}}(\R^2)\) for every \(a>0\).

Let us first recall the norm on the negative Sobolev space $H^{-s}$:
\begin{equation}\label{prop:eq-sobolev-convention}
 \|f\|_{H^{-s}}^2=(2\pi)^{-2}\int_{\R^2}
 (1+|\xi|^2)^{-s}|\widehat f(\xi)|^2\,\dd\xi, \quad \mbox{with} \quad \widehat f(\xi)=\int_{\R^2}\e^{-i\xi\cdot x}f(x)\,\dd x.
\end{equation}
The local Sobolev space \(H^{-s}_{\mathrm{loc}}(\R^2)\) is defined by requiring \(\chi f\in H^{-s}\) for every smooth compactly supported function
\(\chi\in C_c^\infty(\R^2)\).
In what follows, \emph{localising a field} means multiplying it by such a smooth compactly supported function \(\chi\).
We also denote
\(\|\cH \|_{L^2(\bbP;H^{-s})}^2 \coloneqq \bbE[ \|\cH\|_{H^{-s}}^2]\).

Recall the forward fields \(\cH_t^{\vartheta,(N)}(h_0,\cdot)\) and
\(\scrH_t^{\vartheta,(\delta)}(h_0,\cdot)\) from
\eqref{eq:HN} and \eqref{eq:regularised-H}. 
The following proposition gives the convergence in Theorem~\ref{thm:main}
for this fixed initial condition.

\begin{proposition}[Convergence in finite-dimensional distribution]
\label{prop:field-identification}
Fix \(\vartheta\in\R\), \(h_0\in L^\infty(\R^2)\cap C(\R^2)\),
and the critical tuning \eqref{eq:critical-tuning}.
There exists an SHF-measurable process
\((\scrH_t^\vartheta(h_0,\cdot))_{t>0}\) with values in
\(\bigcap_{a>0}H^{-a}_{\mathrm{loc}}(\R^2)\) such that for every
\(0<t_1<\cdots<t_k\) and \(a>0\), the following convergences hold
in \((H^{-a}_{\mathrm{loc}}(\R^2))^k\):
\begin{equation}\label{prop:eq-joint-field-convergence}
\begin{aligned}
 (\cH_{t_i}^{\vartheta,(N)}(h_0,\cdot))_{i=1}^k
 &\xrightarrow[N\to\infty]{d}
 (\scrH_{t_i}^\vartheta(h_0,\cdot))_{i=1}^k,\\
 (\scrH_{t_i}^{\vartheta,(\delta)}(h_0,\cdot))_{i=1}^k
 &\xrightarrow[\delta\downarrow0]{\bbP}
 (\scrH_{t_i}^\vartheta(h_0,\cdot))_{i=1}^k.
\end{aligned}
\end{equation}
Moreover, for every \(\chi\in C_c^\infty(\R^2)\) and \(s>0\),
\begin{equation}\label{prop:eq-joint-field-L2}
 \lim_{\delta\downarrow0}\sum_{i=1}^k
 \bbE\bigl\|\chi\bigl(\scrH_{t_i}^{\vartheta,(\delta)}(h_0,\cdot)
                   -\scrH_{t_i}^\vartheta(h_0,\cdot)\bigr)\bigr\|_{H^{-s}}^2=0.
\end{equation}
The conclusion for the lattice field also holds for deterministic
\(\vartheta_N\to\vartheta\) in \eqref{eq:critical-tuning}.
\end{proposition}

Instead of removing disorder on the final interval \([t-\delta,t]\),
as in the definition~\eqref{eq:regularised-H} of \(\scrH_{t}^{\vartheta,(\delta)}(h_0,\cdot)\), 
we may also spatially smooth the SHF solution
at time \(t\). 
For the same initial condition \(h_0\), and a test function
\(\varphi\in C_c^\infty(\R^2)\), define
\begin{equation}\label{prop:eq-direct-cutoff}
\begin{aligned}
 \widetilde{\scrH}_t^{\vartheta,(\delta)}(h_0,\varphi)
 \coloneqq \int_{\R^2}\varphi(x)\Bigl\{
 \log\scrZ_{0,t}^\vartheta(\e^{h_0},g_{\delta}^{(x)}) -\bbE\log\scrZ_{0,t}^\vartheta(\one,g_{\delta}^{(0)})\Bigr\}\,\dd x.
\end{aligned}
\end{equation}
The following corollary shows that the limit is the same as in
Proposition~\ref{prop:field-identification}.

\begin{corollary}
\label{prop:spatial-identification}
Fix \(\vartheta\in\R\) and
\(h_0\in L^\infty(\R^2)\cap C(\R^2)\).
For every \(0<t_1<\cdots<t_k\) and \(a>0\),
\begin{equation}\label{eq:spatial-limit}
 (\widetilde{\scrH}_{t_i}^{\vartheta,(\delta)}(h_0,\cdot))_{i=1}^k
 \xrightarrow[\delta\downarrow0]{\bbP}
 (\scrH_{t_i}^\vartheta(h_0,\cdot))_{i=1}^k
 \quad\text{in }(H^{-a}_{\mathrm{loc}}(\R^2))^k.
\end{equation}
For every \(\chi\in C_c^\infty(\R^2)\) and \(s>0\),
\begin{equation}\label{prop:eq-spatial-L2}
 \lim_{\delta\downarrow0}\sum_{i=1}^k
 \bbE\bigl\|\chi\bigl(\widetilde{\scrH}_{t_i}^{\vartheta,(\delta)}(h_0,\cdot)
      -\scrH_{t_i}^\vartheta(h_0,\cdot)\bigr)\bigr\|_{H^{-s}}^2=0.
\end{equation}
\end{corollary}

Following Section~\ref{sec:joint-convergence}, we prove these results
after performing the time reversal~\eqref{prop:eq-joint-time-reversal},
so that the initial condition \(\e^{h_0}\) becomes a terminal weight
at a fixed terminal time \(T>t_k\).
For Corollary~\ref{prop:spatial-identification}, we use the spatial regularisation
\(\widetilde{\scrH}_{t,T}^{\vartheta,(\delta)}(x,h_0)\) obtained by
replacing \(t+\delta\) by \(t\) in
\eqref{prop:eq-reversed-temporal-field}:
{{\[
		\begin{aligned}
			\widetilde{\scrH}_{t,T}^{\vartheta,(\delta)}
			(x,h_0)
			:={}&
			\log\mathscr Z_{t,T}^{\vartheta}
			\bigl(g_\delta^{(x)},e^{h_0}\bigr)-\mathbb E\Bigl[
			\log\mathscr Z_{t,T}^{\vartheta}
			\bigl(g_\delta^{(0)},1\bigr)
			\Bigr].
		\end{aligned}
		\]}}
Integrating in \(x\) against a test function \(\varphi\), we write
\(\scrH_{t,T}^{\vartheta,(\delta)}(\varphi,h_0)\), and similarly for
\(\widetilde{\scrH}\).

The rest of the section is organised as follows:
In Section~\ref{subsec:covariance-tightness},
we establish tightness
in \(H^{-a}_{\mathrm{loc}}\) for every \(a>0\) (Lemma~\ref{lem:tightness}),
using the spatial covariance bounds for flat data
(Lemma~\ref{lem:cutoff-covariance}) and comparison with the flat field.
Lastly, Section~\ref{subsec:proof-field-identification}
identifies the subsequential limits through their integrals against test
functions, and proves Proposition~\ref{prop:field-identification}
as well as Corollary~\ref{prop:spatial-identification}.

\subsection{Covariance bounds and tightness}
\label{subsec:covariance-tightness}

To obtain tightness in \(H^{-a}_{\mathrm{loc}}\) for every \(a>0\),
we first prove spatial covariance bounds for the KPZ fields with \(h_0=0\). These give
uniform Sobolev moment bounds, which can then be extended to more general $h_0$
by comparison with the flat fields. Tightness then follows
by compact embedding.

Fix \(T>0\) and \(\vartheta\in\R\). We use the time-reversed fields
\(\scrH_{0,T}^{\vartheta,(\delta)}(\cdot,h_0)\)  defined in
\eqref{prop:eq-reversed-temporal-field},  and
\(\widetilde{\scrH}_{0,T}^{\vartheta,(\delta)}(\cdot,h_0)\) introduced after
Corollary~\ref{prop:spatial-identification}.
For the lattice field, we use \eqref{eq:def-cH-general}
with starting time zero: for \(x\in \cellC_N(z)\),
\begin{equation}\label{eq:abuse-notation}
 \cH^{(N)}(x,h_0)
 \coloneqq \log Z_{0,[NT]}(z,f_N)-\bbE\log Z_{0,[NT]}(0,\one),
\end{equation}
where \(f_N\) is defined in \eqref{eq:cell-averages}.
As before, we suppress the terminal argument when \(h_0=0\).

\begin{lemma}[Covariance bounds]
\label{lem:cutoff-covariance}
The covariance kernel of the flat field
\(\widetilde{\scrH}_{0,T}^{\vartheta,(\delta)}\) with $h_0=0$ satisfies the following bounds: for
\(0<\delta,\delta'<T/2\) and \(x\ne y\),
\begin{equation}\label{prop:eq-mixed-cutoff-covariance}
 0\le\Cov\bigl(\widetilde{\scrH}_{0,T}^{\vartheta,(\delta)}(x),
                \widetilde{\scrH}_{0,T}^{\vartheta,(\delta')}(y)\bigr)
 \le C\Bigl[1+\log\Bigl(1+\log^+\frac1{|x-y|}\Bigr)\Bigr].
\end{equation}
The same bounds hold for the covariance kernels of
\(\scrH_{0,T}^{\vartheta,(\delta)}\), and \(\cH^{(N)}\) for all
sufficiently large \(N\).
The constant \(C\) is uniform when \(T\) and \(\vartheta\) range over
compact subsets of \((0,\infty)\) and \(\R\), respectively.
\end{lemma}

\begin{proof}[Proof of Lemma~\ref{lem:cutoff-covariance}]
	We prove \eqref{prop:eq-mixed-cutoff-covariance} for the spatially
	and temporally regularised $\log$-SHF fields
	\(\widetilde{\scrH}_{0,T}^{\vartheta,(\delta)}\)
	and \(\scrH_{0,T}^{\vartheta,(\delta)}\)
	via lattice approximations.

	Let us fix even integers \(0\le m,\ell<n\), with \(n/N\) in a fixed compact
	subset of \((0,\infty)\), and let us recall that we denote \(Z_{n}(q_m^{(z)}) = Z_{0,n}(q_m^{(z)},\one)\).
	By time-stationarity, \(Z_{n}(q_m^{(z)})\)
	has the same law as \(Z_{m,m+n}(q_m^{(z)},\one)\), so
	Proposition~\ref{prop:VD}, specifically \eqref{prop:eq-heat-defect},
	gives that \(\Delta(\log Z_{n}(q_m^{(z)}))\le C\) uniformly in \(m\).
	Therefore, Lemma~\ref{lem:two-point} (in particular \eqref{eq:two-point-general})
	gives similarly to~\eqref{eq:polymer-covariance-convexity} that
	\begin{equation}
		\label{eq:bound-covariance-log}
		0\le\Cov(\log Z_{n}(q_m^{(z)}),\log Z_{n}(q_\ell^{(w)}))
		\le C+\log \bbE[Z_{n}(q_m^{(z)})Z_{n}(q_\ell^{(w)})].
	\end{equation}
	
	Applying the first collision decomposition \eqref{eq:local-replica-formula} with \(s=0\)
	and averaging
	the initial positions against \(q_m^{(z)}\) and \(q_\ell^{(w)}\), we obtain
	\begin{equation}\label{prop:eq-mixed-first-collision}
		\bbE[Z_{n}(q_m^{(z)})Z_{n}(q_\ell^{(w)})]
		=1 +\sum_{k=1}^n q_{m+\ell+2k}(w-z)
		\sigma_N^2 \bbE[Z_{k,n}(0)^2] \,.
	\end{equation}
	By \eqref{eq:uniform-two-replica},
	\(\sigma_N^2\bbE[Z_{k,n}(0)^2]\le\sigma_N^2\bbE[Z_n(0)^2]\le C\)
	uniformly in \(1\le k\le n\).
	For time horizons larger than \(N\), the same bound follows from the last
	assertion of Claim~\ref{claim:VN}, with the constant depending on an upper
	bound for \(n/N\).
	Together with the Gaussian heat-kernel upper bound (see
	\cite[Propositions~2.1.2 and~2.4.6]{LL10},
	applied to the
	aperiodic two-step walk on \(\Z^2_{\rm even}\))
	this yields
	\begin{equation}\label{prop:eq-mixed-mass-bound}
		\begin{aligned}
			\bbE[Z_{n}(q_m^{(z)})Z_{n}(q_\ell^{(w)})]
			&\le1+C\sum_{k=1}^n\frac1{m+\ell+2k}
			\exp\left\{-\frac{c|w-z|^2}{m+\ell+2k}\right\}\\
			&\le C\left\{1+\log^+\frac{n}{1+|z-w|^2}\right\}.
		\end{aligned}
	\end{equation}
	To see the last inequality, the terms with \(m+\ell+2k\le|z-w|^2\)
	have a bounded sum by comparison with the corresponding Gaussian
	integral, while the remaining terms form a harmonic sum with upper endpoint
	\(m+\ell+2n<4n\).

	We now pass from these lattice estimates to the spatially regularised
	SHF field. Fix \(0<\delta,\delta'<T/2\) and \(x\ne y\), and take
	\[
	n=[NT],\qquad m=[N\delta],\qquad \ell=[N\delta'],\qquad
	z=[\sqrt{N/2}\,x],\qquad w=[\sqrt{N/2}\,y].
	\]
	Lemma~\ref{lem:extension-convergence-delta} gives joint convergence of the
	logarithms, as well as convergence of their mean and, by uniform
	integrability, their mixed second moments. Hence
	\[
	\begin{aligned}
		\Cov\bigl(\log Z_n(q_m^{(z)}),\log Z_n(q_\ell^{(w)})\bigr)
		\longrightarrow
		\Cov\bigl(\log\scrZ_{0,T}^\vartheta(g_\delta^{(x)},\one),
		\log\scrZ_{0,T}^\vartheta(g_{\delta'}^{(y)},\one)\bigr).
	\end{aligned}
	\]
	Since \(n/(1+|z-w|^2)\to2T/|x-y|^2\),
	\eqref{eq:bound-covariance-log} and \eqref{prop:eq-mixed-mass-bound}
	give \eqref{prop:eq-mixed-cutoff-covariance} for
	\(\widetilde{\scrH}_{0,T}^{\vartheta,(\delta)}\).
	
	For the temporal regularisation
	\(\scrH_{0,T}^{\vartheta,(\delta)}\), use instead
	\(\widehat Z_{m,n}(z)=Z_{m,n}(q_m^{(z)},\one)\).
	Lemma~\ref{lem:two-point} and Proposition~\ref{prop:VD} give
	\[
	0\le\Cov\bigl(\log\widehat Z_{m,n}(z),\log\widehat Z_{\ell,n}(w)\bigr)
	\le C+\log\bbE[\widehat Z_{m,n}(z)\widehat Z_{\ell,n}(w)].
	\]
	Suppose \(m\le\ell<n/2\). Averaging over the disorder on \((m,\ell]\)
	replaces \(q_m^{(z)}\) by \(q_\ell^{(z)}\). Time-stationarity then gives
	\[
	\bbE[\widehat Z_{m,n}(z)\widehat Z_{\ell,n}(w)]
	=\bbE[Z_{n-\ell}(q_\ell^{(z)})Z_{n-\ell}(q_\ell^{(w)})],
	\]
	so \eqref{prop:eq-mixed-mass-bound}, with duration \(n-\ell\) and
	both smoothing times equal to \(\ell\), bounds this mixed moment.
	With the same scaling as above, Lemma~\ref{lem:extension-convergence-delta}
	also gives
	\[
	\begin{aligned}
		\Cov\bigl(\log\widehat Z_{m,n}(z),\log\widehat Z_{\ell,n}(w)\bigr)
		\longrightarrow
		\Cov\bigl(\log\scrZ_{\delta,T}^\vartheta(g_\delta^{(x)},\one),
		\log\scrZ_{\delta',T}^\vartheta(g_{\delta'}^{(y)},\one)\bigr).
	\end{aligned}
	\]
	This proves \eqref{prop:eq-mixed-cutoff-covariance} for
	\(\scrH_{0,T}^{\vartheta,(\delta)}\) as well.
	
	For the lattice field \(\cH^{(N)}\), apply
	\eqref{eq:bound-covariance-log} and \eqref{prop:eq-mixed-mass-bound}
	with \(m=\ell=0\) and \(n=[NT]\).
	If \(x\in \cellC_N(z)\) and \(y\in \cellC_N(w)\), the cell definition
	\eqref{eq:lattice-cells} gives
	\[
	|x-y|^2\le C\frac{1+|z-w|^2}{N},
	\qquad
	\frac{n}{1+|z-w|^2}\le\frac{Cn}{N|x-y|^2}
	\quad(x\ne y).
	\]
	Since \(n/N\to T\), the discrete covariance bound therefore implies
	\eqref{prop:eq-mixed-cutoff-covariance} for \(\cH^{(N)}\), including
	when \(x\) and \(y\) belong to the same cell.
\end{proof}

\begin{lemma}[Sobolev bounds and tightness]\label{lem:tightness}
Fix \(h_0\in L^\infty(\R^2)\cap C(\R^2)\).
For every \(a>0\) and \(\chi\in C_c^\infty(\R^2)\),
\begin{equation}\label{prop:eq-uniform-Sobolev-moments}
 \max\left\{
 \begin{aligned}
 &\sup_{0<\delta<T/2}\bbE\|\chi
       \widetilde{\scrH}_{0,T}^{\vartheta,(\delta)}(\cdot,h_0)\|_{H^{-a}}^2,\\
 &\sup_{0<\delta<T/2}\bbE\|\chi
       \scrH_{0,T}^{\vartheta,(\delta)}(\cdot,h_0)\|_{H^{-a}}^2,\\
 &\limsup_{N\to\infty}\bbE\|\chi\cH^{(N)}(\cdot,h_0)\|_{H^{-a}}^2
 \end{aligned}\right\}
 \le C_{\chi,a,\|h_0\|_\infty}.
\end{equation}
The constant is locally uniform in \((T,\vartheta)\in(0,\infty)\times\R\).
All three families are tight in \(H^{-a}_{\mathrm{loc}}(\R^2)\).
\end{lemma}

\begin{proof}[Proof of Lemma~\ref{lem:tightness}]
We first prove the moment bounds for \(h_0=0\). We claim that the second moment for the $H^{-a}$ norm of $\chi \scrH_{0,T}^{\vartheta,(\delta)}$ has the representation
\begin{equation}\label{eqS: H -a norm represtt}
	\begin{aligned}
 &\bbE\pigl[\|
 \chi \scrH_{0,T}^{\vartheta,(\delta)}
 \|_{H^{-a}}^2\pigr]\\
 =&
\frac1{\Gamma(a)}\int_0^\infty u^{a-1}\e^{-u}
\iint g_{2u}(x-y)\chi(x)\chi(y)
\Cov\bigl( \scrH_{0,T}^{\vartheta,(\delta)}(x),\scrH_{0,T}^{\vartheta,(\delta)}(y)\bigr)\,\dd x\,\dd y\,\dd u \,.
	\end{aligned}
\end{equation}
Indeed, for \(a>0\), the change of variables \(v=(1+|\xi|^2)u\) in
\(\Gamma(a)=\int_0^\infty v^{a-1}\e^{-v}\,\dd v\) gives
\begin{equation}
      \label{magic-Gamma}
 (1+|\xi|^2)^{-a}
 =\frac1{\Gamma(a)}\int_0^\infty
       u^{a-1}\e^{-u}\e^{-u|\xi|^2}\,\dd u \,.
\end{equation}
Substituting this integral representation into the Sobolev norm
for $\|\chi \scrH_{0,T}^{\vartheta,(\delta)}\|_{H^{-a}}$
(recall~\eqref{prop:eq-sobolev-convention}) gives
\begin{equation*}
 \begin{aligned}
 \bbE\pigl[\|\chi \scrH_{0,T}^{\vartheta,(\delta)}\|_{H^{-a}}^2\pigr]
 &=\frac1{(2\pi)^2\Gamma(a)}\int_0^\infty u^{a-1}\e^{-u}
       \int_{\R^2}\e^{-u|\xi|^2}
       \bbE\Bigl[\bigl|\widehat{\chi \scrH}_{0,T}^{\vartheta,(\delta)}(\xi)\bigr|^2\Bigr]\,\dd\xi\,\dd u\\
 &=\frac1{\Gamma(a)}\int_0^\infty u^{a-1}\e^{-u}
       \iint\chi(x)\chi(y)\bbE\bigl[\scrH_{0,T}^{\vartheta,(\delta)}(x) \scrH_{0,T}^{\vartheta,(\delta)}(y)\bigr]\\
 &\qquad\qquad\qquad\qquad\qquad \times
       \left(\frac1{(2\pi)^2}\int_{\R^2}
      \e^{-u|\xi|^2+i\xi\cdot(x-y)}\,\dd\xi\right)\dd x\,\dd y\,\dd u  \,,
 \end{aligned}
\end{equation*}
The first identity uses Tonelli's theorem. For the second, expand
the Fourier transform and use Fubini for each fixed \(u>0\) and
\(\delta>0\), then integrate in \(u\). The finite logarithmic second
moments from Lemma~\ref{lem:extension-convergence-delta} justify the
fixed-\(u\) interchange.
Since \(\e^{-u|\xi|^2}=\widehat g_{2u}(\xi)\) and the field $\scrH_{0,T}^{\vartheta , (\delta)} (\cdot)$ has
mean zero, Fourier inversion identifies the last expression with
\eqref{eqS: H -a norm represtt}.

From \eqref{eqS: H -a norm represtt} and the covariance
upper bound in \eqref{prop:eq-mixed-cutoff-covariance}, we obtain
\[
 \bbE\pigl[\|\chi \scrH_{0,T}^{\vartheta,(\delta)}\|_{H^{-a}}^2\pigr]
 \le C_{\chi,a}\int_0^\infty u^{a-1}\e^{-u}
       \left(\int_{\R^2}g_{2u}(z)\Bigl[1+\log\Bigl(1+\log^+\frac1{|z|}\Bigr)\Bigr]\,\dd z\right)\dd u.
\]
We claim that there is a constant \(C>0\) such that for any \(u>0\),
\[
\int_{\R^2}g_{2u}(z)\Bigl[1+\log\Bigl(1+\log^+\frac1{|z|}\Bigr)\Bigr]\,\dd z \leq C [1+\log(1+\log^+(1/u))] \,.
\]
Indeed, the integral is bounded by a universal constant for \(u\geq 1\).
For \(u\in (0,1)\), we use the change of variable \(z=\sqrt u\,v\) and use that
\(\log^+(1/(\sqrt u|v|))\le\tfrac12\log(1/u)+\log^+(1/|v|)\),
giving the desired bound.
 Combining these estimates gives
\begin{equation}\label{eqS: moment for flat intial conditoin field}
\bbE\pigl[\|\chi \scrH_{0,T}^{\vartheta,(\delta)}\|_{H^{-a}}^2\pigr]
\le C_{\chi,a}\left\{1+\int_0^1
u^{a-1}[1+\log(1+\log(1/u))]\,\dd u\right\}<\infty,
\end{equation}
uniformly in \(\delta\).
The same calculation applies to the lattice field \(\cH^{(N)}\) and the spatially mollified field \(\widetilde{\scrH}_{0,T}^{\vartheta,(\delta)}\),
since their covariance kernels satisfy the same
bound by Lemma \ref{lem:cutoff-covariance}.
This proves \eqref{prop:eq-uniform-Sobolev-moments} for \(h_0=0\).

For general \(h_0\), compare the terminal weights using
\(\e^{-\|h_0\|_\infty}\le\e^{h_0},f_N\le\e^{\|h_0\|_\infty}\) ({{see  definition \eqref{eq:cell-averages}}}).
Since the centring is defined using the flat terminal condition,
\[
 \bigl|\scrH_{0,T}^{\vartheta,(\delta)}(x,h_0)
       -\scrH_{0,T}^{\vartheta,(\delta)}(x)\bigr|
 \le\|h_0\|_\infty,
\]
and the identical bound holds for
\(\widetilde{\scrH}_{0,T}^{\vartheta,(\delta)}\) and \(\cH^{(N)}\).
Using \(\|\cdot\|_{H^{-a}}\le\|\cdot\|_2\), we obtain
\[
 \bigl\|\chi\bigl(\scrH_{0,T}^{\vartheta,(\delta)}(\cdot,h_0)
                   -\scrH_{0,T}^{\vartheta,(\delta)}\bigr)\bigr\|_{H^{-a}}
 \le\|\chi\|_2\|h_0\|_\infty,
\]
and similarly for the other two fields. Combining this with the flat
moment bounds \eqref{eqS: moment for flat intial conditoin field} proves \eqref{prop:eq-uniform-Sobolev-moments}.

For tightness, apply \eqref{prop:eq-uniform-Sobolev-moments} with
\(a/2\) in place of \(a\). Markov's inequality gives
\[
 \sup_{0<\delta<T/2}
 \bbP\left(\|\chi\scrH_{0,T}^{\vartheta,(\delta)}(\cdot,h_0)
               \|_{H^{-a/2}}>R\right)
 \le C_{\chi,a,\|h_0\|_\infty}R^{-2}.
\]
Distributions supported on \(\operatorname{supp}(\chi)\) with
\(H^{-a/2}\) norm at most \(R\) form a relatively compact subset
of \(H^{-a}\). Thus the localised fields are tight in \(H^{-a}\).
Taking smooth cutoffs equal to one on successively larger balls, and
choosing the radii so that the corresponding probability bounds are
summable, gives tightness in \(H^{-a}_{\mathrm{loc}}\).
The same argument applies to the spatially regularised and lattice fields. 
\end{proof}

\subsection{Proof of Proposition~\ref{prop:field-identification} and Corollary~\ref{prop:spatial-identification}}
\label{subsec:proof-field-identification}

We combine Theorem~\ref{thm:toy-joint} with the tightness in
Lemma~\ref{lem:tightness}. Throughout the proof, fix
\(h_0\in L^\infty(\R^2)\cap C(\R^2)\) and a common terminal time
\(T>0\), and suppress \(\vartheta\) from the notation.
All time-reversed SHF fields below are constructed from the same SHF
with terminal weight \(\e^{h_0}\) at time \(T\).

\begin{proof}[Proof of Proposition~\ref{prop:field-identification}]
Fix \(0\le u<T\). Theorem~\ref{thm:toy-joint}
gives the scalar limits \(\scrH_{u,T}(\varphi,h_0)\), the
\(L^2\) approximation \eqref{prop:eq-temporal-approximation}, and the
joint lattice convergence \eqref{prop:eq-test-function-convergence}.

We next use the Sobolev bounds of Lemma~\ref{lem:tightness} to show
that these scalar limits define a random distribution on the same SHF
space. Fix \(\chi\in C_c^\infty(\R^2)\) and \(a>0\), and set
\(Y_{\delta,\delta'}=\chi(\scrH_{u,T}^{(\delta)}(\cdot,h_0)
 -\scrH_{u,T}^{(\delta')}(\cdot,h_0))\).
For \(R>0\), splitting the Fourier integral and using
\eqref{prop:eq-uniform-Sobolev-moments} with \(a/2\) gives
\begin{equation}\label{prop:eq-Sobolev-frequency-split}
\begin{aligned}
 \bbE\|Y_{\delta,\delta'}\|_{H^{-a}}^2
 &\le \frac1{(2\pi)^2}\int_{|\xi|\le R}
       \bbE|\widehat Y_{\delta,\delta'}(\xi)|^2\,\dd\xi
       +C_{\chi,a,\|h_0\|_\infty}(1+R^2)^{-a/2}.
\end{aligned}
\end{equation}
For each \(\xi\), the Fourier transform tends to zero in \(L^2\) by applying
\eqref{prop:eq-temporal-approximation} to the real and imaginary
parts of \(\chi(x)\e^{-i\xi\cdot x}\). Since \(Y_{\delta,\delta'}\)
is supported on \(\operatorname{supp}\chi\), Sobolev duality and
\eqref{prop:eq-uniform-Sobolev-moments} bound the integrand uniformly
on \(\{|\xi|\le R\}\), so dominated convergence applies.
First letting \(\delta,\delta'\downarrow0\) and then \(R\to\infty\)
shows that \(\chi\scrH_{u,T}^{(\delta)}(\cdot,h_0)\) is Cauchy in
\(L^2(\bbP;H^{-a})\). The local limits therefore define an
SHF-measurable random distribution \(\scrH_{u,T}(\cdot,h_0)\), with the integrals
against test functions identified above. This gives convergence
in probability in \(H^{-a}_{\mathrm{loc}}\) and the mean-square
convergence in \eqref{prop:eq-joint-field-L2}, jointly at finitely
many starting times.

For the lattice fields  $\left( \mathcal H_{t_i}^{\theta, (N)} (\cdot, h_0) \right)_{i=1} ^k$, Lemma~\ref{lem:tightness} gives tightness in
\((H^{-a}_{\mathrm{loc}})^k\). By
\eqref{prop:eq-test-function-convergence}, every subsequential weak limit
has the same law as the random distributions
\((\scrH_{u_i,T}(\cdot,h_0))_{i=1}^k\), which implies weak convergence in
 \((H^{-a}_{\mathrm{loc}})^k\).

Finally, take \(T>t_k\) and \(u_i=T-t_i\). Time reversal
\eqref{prop:eq-joint-time-reversal}, applied jointly to two
regularisation parameters, gives the asserted convergence for the
forward SHF fields. The lattice convergence follows by the corresponding
time reversal.

If \(\vartheta_N\to\vartheta\),
Lemma~\ref{lem:extension-convergence-delta} has the same limiting SHF,
while the approximation and tightness bounds are uniform for
\(\vartheta_N\) in a compact interval.
The preceding argument therefore proves the last assertion of the
proposition as well.
\end{proof}

\begin{remark}
The argument above also gives joint convergence for finitely many initial
conditions \(h_0^{(j)}\). The joint SHF convergence in
Theorem~\ref{thm:SHF-convergence} allows the proof of
Lemma~\ref{lem:extension-convergence-delta} to use the terminal weights
\(\e^{h_0^{(j)}}\) simultaneously, while the approximation estimates
apply to each initial condition.
\end{remark}

\begin{proof}[Proof of Corollary~\ref{prop:spatial-identification}]
Keep the common terminal time \(T\) and terminal weight \(\e^{h_0}\),
and fix \(0<u<T\). For \(0<\delta<\min\{u,T-u\}\), we have the exact identity
\begin{equation}\label{prop:eq-shifted-cutoff-identity}
 \widetilde{\scrH}_{u,T}^{(\delta)}(\cdot,h_0)
 =\scrH_{u-\delta,T}^{(\delta)}(\cdot,h_0).
\end{equation}
For \(\varphi\in C_c^\infty(\R^2)\), write
\begin{equation}\label{6.22}
\begin{aligned}
 &\widetilde{\scrH}_{u,T}^{(\delta)}(\varphi,h_0)
                 -\scrH_{u,T}(\varphi,h_0)\\
 =\ & \bigl(\scrH_{u-\delta,T}^{(\delta)}(\varphi,h_0)
                 -\scrH_{u-\delta,T}(\varphi,h_0)\bigr)
       +\bigl(\scrH_{u-\delta,T}(\varphi,h_0)
                 -\scrH_{u,T}(\varphi,h_0)\bigr).
\end{aligned}
\end{equation}
The first term tends to zero in \(L^2(\bbP)\) by
\eqref{prop:eq-temporal-approximation}, uniformly when \(u\) varies
in a compact subset of \((0,T)\).
For the second, it suffices to show that
\(u\mapsto\scrH_{u,T}(\varphi,h_0)\) is continuous in $L^2(\mathbb P)$.
 
By \cite[Definition~1.2(a) and Proposition~1.5]{Tsai24},
the SHF is continuous in its time
indices in the vague topology.
Truncating the Gaussian starting weight and the bounded continuous
terminal weight \(\e^{h_0}\), with the heat-kernel first-moment
bound as in Step~1 of Section~\ref{sec:FC}, gives continuity in
probability of \(\scrZ_{u+\varepsilon,T}(g_\varepsilon^{(x)},\e^{h_0})\).
The positivity and locally uniform logarithmic moment bounds in 
Lemma~\ref{lem:extension-convergence-delta} (and its proof) then give continuity of its
logarithm in \(L^2(\bbP)\). Integrating against \(\varphi\)
shows that \(u\mapsto\scrH_{u,T}^{(\varepsilon)}(\varphi,h_0)\)
is continuous in \(L^2(\bbP)\) on \([0,T-\varepsilon)\). By  \eqref{prop:eq-temporal-approximation}, these continuous maps
converge locally uniformly in \(L^2(\bbP)\) as
\(\varepsilon\downarrow0\). Their limit
\(\scrH_{u,T}(\varphi,h_0)\) is therefore continuous.
Both terms in \eqref{6.22} consequently tend to zero in \(L^2(\bbP)\).

The Sobolev bounds in Lemma~\ref{lem:tightness} and the Fourier
argument in \eqref{prop:eq-Sobolev-frequency-split} now give
\(\chi\widetilde{\scrH}_{u,T}^{(\delta)}(\cdot,h_0)
\to\chi\scrH_{u,T}(\cdot,h_0)\) in \(L^2(\bbP;H^{-a})\)
for every \(\chi\in C_c^\infty(\R^2)\) and \(a>0\).
This proves both assertions in Corollary \ref{prop:spatial-identification}, jointly at finitely many starting times.
Finally, take \(T>t_k\) and perform the time reversal \(u_i=T-t_i\)
to conclude the proof.
\end{proof}

\section{Geometric block factorisation and moment estimates}
\label{sec:factorisation}

This section supplies the three inputs used in the preceding convergence
proofs: the one-point comparison in Proposition~\ref{prop:VD}, the
comparison of partition functions on $(n,N]$ in
Proposition~\ref{prop:PF}, and the logarithmic moments in
Proposition~\ref{prop:log-tails}.
We adapt the block decomposition and interval expansion of
Gu--Tsai~\cite[Section~2.3 and Proposition~2.1]{GT26}:
the idea is to replace the
endpoint distribution on each geometric time interval by a heat kernel, which
produces independent partition functions.
One key idea of \cite{GT26} is to provide an exact expansion for
the error of this replacement.

We first describe the factorisation and then state the main estimates that are needed in this section.

\subsubsection*{Notation}

For kernels, we use the contractions of
\cite[Section~2.2]{GT26}.
The left triangle weights and sums the
starting point,
whereas the right triangle weights and sums the
terminal point:
\[
 (\nu\lact Z_{a,b})(y)=\sum_x\nu(x)Z_{a,b}(x,y),\qquad
 (Z_{a,b}\ract f)(x)=\sum_y Z_{a,b}(x,y)f(y).
\]
In the discrete setting, we can interpret \(Z_{a,b}\) as a matrix.
Then \(\nu\lact Z_{a,b}\) and  \(Z_{a,b} \ract f\) are simply the 
multiplication of \(Z_{a,b}\) with the row vector \(\nu\) on the left
and the column vector  \(f\) on the right, respectively.
We note that
$\nu\lact Z_{a,b}$ is an endpoint measure and applying both
triangles gives the scalar
$\nu\lact Z_{a,b}\ract f=Z_{a,b}(\nu,f)$ from
\eqref{eq:Zav}.
In particular, total terminal mass is obtained by closing with
$\ract\one$; if this contraction is absent, the endpoint remains free.

\begin{remark}\label{re:bra-ket}
It is also worth stressing that this notation is the analogue of the {\it bra-ket} notation in quantum
mechanics.
In particular $Z\ract f \equiv Z_{a,b} \,|\, f \rangle$ and $\nu\lact Z_{a,b} \equiv \langle \nu \,|\, Z_{a,b}$
This analogy is particularly useful to keep in mind when deriving the ``block-expansions'' in Section
\ref{sec:block-exp}.
\end{remark}

The circle composes successive kernels by summing their common
intermediate point: in other words,
 it corresponds to the matrix multiplication
\[
 (Z_{a,b}\circ Z_{b,c})(x,z)
 =\sum_y Z_{a,b}(x,y)Z_{b,c}(y,z)=Z_{a,c}(x,z).
\]
The last equality is the Chapman-Kolmogorov equation,
and we stress that the two factors use
the disorder on the disjoint intervals $(a,b]$ and $(b,c]$.

\subsection{The geometric block decomposition and factorisation}

We divide $(0,N]$ into geometric time intervals.
For \(\rho>1\), let \(K=\lfloor\log_\rho N\rfloor+1\) and define 
\begin{equation}\label{eq:adapted-grid}
 N_0=0,\qquad N_i=[  N\rho^{i-K}]\quad(1\le i<K),\qquad N_K=N \,.
\end{equation}
The time intervals satisfy $N_1\le2\rho$, $N_i\asymp N\rho^{i-K}$, and the
successive ratios (except the first one) are comparable to $\rho$.
We refer to the interval \((N_{i-1},N_i]\) as the \textit{\(i\)-th block}.

\begin{remark}\label{rem:choose-J}
We choose the geometric grid so that
$m=[\delta N]$ and $n=[\sqrt\delta N]$ are grid points, say $m = N_a$ and $n = N_b$,
with the same number of blocks between $m$ and $n$
as between $n$ and $N$,
i.e.\ $b-a = K - b$.
More precisely,
given $\rho>1$ large enough (to be fixed later), we write $\delta = \rho^{-2J}$ for a suitable~$J$, namely
\begin{equation}\label{eq:delta-grid-choice}
 J\coloneqq \frac{|\log\delta|}{2\log\rho} \,.
\end{equation}
To lighten notation, we will assume
that $J$ is an integer.
For large $N$,
the indices
\[
 a=K-2J,\qquad b=K-J
\]
then satisfy
\[
 N_a=[N\rho^{-2J}]=m,\qquad
 N_b=[N\rho^{-J}]=n.
\]
Thus $(m,n]$ consists of exactly $J$ blocks, as does $(n,N]$.

(For $\delta > 0$ small, if $J$ is not an integer, we can write $\delta=\tilde\rho^{-2\tilde J}$ for $\tilde J=\lfloor J\rfloor \ge 1$ and $ \tilde\rho=\rho^{J/\tilde J} \in [\rho,\rho^2]$.
All the identities still above
hold with $(\rho,J)$ replaced by $(\tilde\rho,\tilde J)$,
where $\delta$, $m$ and $n$ remain unchanged.
The estimates below continue to hold too with these replacements,
since $\tilde J$ differs from $J$ by less than one and 
$\tilde\rho\in[\rho,\rho^2]$;
see Remark~\ref{rem:block-rho-uniformity}.)
\end{remark}

All block formulas are first written for a starting point at the
origin.  The version starting at $x$ is obtained by translating every
heat-kernel starting distribution by $x$.  Let
\begin{align}\label{def:Zi}
\mathsf Z_i(x,y)\coloneqq Z_{N_{i-1},N_i}(x,y),
 \qquad
 \mathsf Z_i(q_{N_{i-1}})
 \coloneqq \mathsf Z_i(q_{N_{i-1}}, \one)
 \coloneqq q_{N_{i-1}}\lact\mathsf Z_i\ract\one
\end{align}
be respectively the kernel in block $i$ and its total mass under the
reference heat-kernel starting distribution. These interval partition functions are independent,
because the blocks contain disjoint disorder.

Let us now introduce the heat-kernel replacement, or ``rewired'' kernel,
\begin{align}
 \mathsf Z_i^{\mathrm{rewired}}(x,y)
 \coloneqq \mathsf Z_i\ract \one (x) \,q_{N_i}(y)
 =\mathsf Z_i(x, \one)  \,q_{N_i}(y)
\end{align}
which retains the total mass of $\mathsf Z_i$ but replaces
its endpoint distribution by $q_{N_i}$.
The superscript
$\mathrm{rewired}$ refers to this replacement.
We also consider the
``discrepancy'' kernel
\begin{equation}\label{eq:block-objects-D}
 D_i(x,y)\coloneqq \mathsf Z_i(x,y)-\mathsf Z_i^{\mathrm{rewired}}(x,y).
\end{equation}
It is easy to see that for any $x\in\Z^2$, it holds that
\begin{align}\label{D-mass}
D_i(x,\one)=D_i\ract \one (x) =  \sum_{z\in\Z^2} D_i(x,z)=0.
\end{align}
We also define normalised versions of the kernels:
\begin{equation}\label{eq:block-objects-hat}
\widecheck{\mathsf Z}_i(x,y)
 \coloneqq \mathsf Z_i(q_{N_{i-1}})^{-1}\mathsf Z_i(x,y),\qquad
 \widecheck{D}_i(x,y)
 \coloneqq \mathsf Z_i(q_{N_{i-1}})^{-1}D_i(x,y).
\end{equation}
With this notation, for every test function $f$, we have
\(
 q_{N_{i-1}}\lact\widecheck{\mathsf Z}_i\ract f
 =q_{N_i}(f)+q_{N_{i-1}}\lact\widecheck{D}_i\ract f.
\)

For consecutive blocks $r,\ldots,s$, define
\begin{equation}\label{eq:block-ratio}
 \mathscr R_{r,s}
 \coloneqq q_{N_{r-1}}\lact
       \widecheck{\mathsf Z}_r\circ\cdots\circ
       \widecheck{\mathsf Z}_s\ract\one
 =\frac{q_{N_{r-1}}\lact Z_{N_{r-1},N_s}\ract\one}
        {\prod_{i=r}^s q_{N_{i-1}}\lact\mathsf Z_i\ract\one},
\end{equation}
with by convention \(\mathscr R_{r,r-1}\coloneqq 1\).
This quantity compares the partition function on the
complete interval \((N_{r-1},N_s]\) with the product
of independent partition functions
on the different intervals \((N_{i-1},N_i]\),
with starting distribution replaced by \(q_{N_{i-1}}\).

We then have the following result that estimates \(\mathscr R_{r,s}\) and ratios of partition function on some good event.
Let us introduce, similarly as in \cite[Eq.~(2.24)]{GT26}, the events
\begin{equation}\label{eq:block-events}
 \Omega_{r,s}\coloneqq \bigcap_{i=r}^s
 \{|\mathsf Z_i(q_{N_{i-1}})-1|\le1/2\} \,,
\end{equation}
with by convention $\Omega_{r,r-1}$ the whole probability space.
Let us also introduce the following probability law weighted by the product of inverse interval partition functions:
\begin{equation}
\label{eq:weighted-law}
      \frac{\dd\mathbb Q_{r,s}}{\dd\bbP}\coloneqq 
\frac{\prod_{i=r}^s
      \mathsf Z_i(q_{N_{i-1}})^{-1}}
{\bbE\bigl[\prod_{i=r}^s
      \mathsf Z_i(q_{N_{i-1}})^{-1}\bigr]} \,.
\end{equation}

\begin{proposition}[Ratios estimates]
\label{prop:block-inputs-ii}
There is a constant \(C=C(\rho)>0\) 
such that for $r\le s\le K-1$ we have
\begin{equation}\label{eq:app-good-P}
 \bbP(\Omega_{r,s}^c) \le\frac{C}{K-s}
 \quad \text{ and } \quad
 \bbE\bigl[\one_{\Omega_{r,s}}
       |\mathscr R_{r,s}-1|^2\bigr]
      \le\frac{C}{K-s} \, .
\end{equation}
One also has
\begin{equation}\label{eq:app-good-Q}
    \mathbb Q_{r,s}(\Omega_{r,s}^c)
 \le\frac{C}{\sqrt{K-s}}
  \quad \text{ and } \quad
 \bbE_{\mathbb Q_{r,s}}\bigl[
 \one_{\Omega_{r,s}}|\mathscr R_{r,s}-1|^2\bigr]
 \le\frac{C}{\sqrt{K-s}}.
\end{equation}
\end{proposition}

The first part of the proposition, \textit{i.e.}~\eqref{eq:app-good-P}, is already present in \cite[Proposition~4.1]{GT26} in a slightly different setting;
we provide a proof here in order to explain more precisely the idea, which is needed afterwards.
The main novelty here is the control of the event
\(\Omega_{r,s}\) and of the ratio \(\mathscr R_{r,s}\)
under the weighted measure \(\mathbb Q_{r,s}\).
This will crucially allow us to control
the negative moment \(\bbE[\widehat Z_{s,n}(x)^{-1}]\)
that appears in~\eqref{target:VD} thanks to its ``product version'' \(\prod_{i=a}^{b}\bbE[Z_{N_{i-1},N_i}(q_{N_{i-1}})^{-1}]\) (see \eqref{eq:negative-first-moment} for the explicit decoupling identity).

\subsection{The block expansion}\label{sec:block-exp}

One idea of Gu--Tsai~\cite{GT26} is to expand the product
\(\mathscr R_{r,s} =q_{N_{r-1}}\lact\widecheck{\mathsf Z}_r\circ\cdots\circ\widecheck{\mathsf Z}_s\ract\one\),
using that
\begin{equation}\label{ZD}
 \widecheck{\mathsf Z}_i
 = \widecheck{D}_i +
 \frac{\mathsf Z_i \ract \one \otimes  q_{N_i} \lact }{ q_{N_{i-1}}\lact \mathsf Z_i \ract \one }  \;.
\end{equation}
More explicitly, specifying the coordinates, we can write
\begin{equation}\label{ZD-new}
 \widecheck{\mathsf Z}_i(x,y)
 = \widecheck{D}_i(x,y) +
 \frac{\mathsf Z_i \ract \one(x)\, q_{N_i} (y)}{ q_{N_{i-1}}\lact \mathsf Z_i \ract \one } \;.
\end{equation}
Recalling Remark \ref{re:bra-ket}, the particular notation in \eqref{ZD} is particularly useful in 
obtaining an expansion for \(\mathscr R_{r,s}\).
In particular, expanding the product in the definition~\eqref{eq:block-ratio} of \(\mathscr R_{r,s}\)
and decomposing over intervals of successive correctors \(\widecheck{D}_i\), we get the following exact identity,
which is the block expansion of
\cite[Proposition~2.1]{GT26} written for lattice kernels: 
\begin{equation}
      \mathscr R_{r,s}
 =\sum_{\mathcal I\in\mathfrak I(r,s)}
       \prod_{I\in\mathcal I} \mathscr D_I^{\mathrm{closed}}
= 1+  \sum_{\mathcal I\in\mathfrak I(r,s), \mathcal I \neq\emptyset}
       \prod_{I\in\mathcal I} \mathscr D_I^{\mathrm{closed}}
\,,
 \label{eq:R-interval}
\end{equation}
where $\mathfrak I(r,s)$ is the set of families of pairwise disjoint
non-degenerate integer intervals $[u,v]\subseteq[r,s]$, with $u<v$
(including the empty family) and
where we defined
\begin{equation}
  \mathscr D^{\mathrm{closed}}_{[u,v]}
\coloneqq
\mathscr D^{\mathrm{closed}}_{u,v}
 \coloneqq 
 q_{N_{u-1}}\lact
 \widecheck{D}_u\circ\cdots\circ\widecheck{D}_{v-1}
 \circ\widecheck{\mathsf Z}_v\ract\one \,.
 \label{eq:closed-activity}
\end{equation}
Let us mention that we also need later on to
perform an expansion for the endpoint distribution \(\mu_{N_{r-1},N_s}^{(q_{N_{r-1}})}\), in the same spirit as the above,
see~\eqref{eq:profile-interval-expansion} below.

We need the following estimates on the closed correctors.
We let
\[
D^{\mathrm{closed}}_{u,v}\coloneqq  q_{N_{u-1}}\lact
  D_u\circ\cdots\circ D_{v-1}
 \circ \mathsf Z_v\ract\one
\]
which is the non-normalised closed correction.

\begin{proposition}[Compositions of discrepancy kernels]
      \label{prop:block-inputs-iii}
There are constants $C,C_0$, independent of $N,\rho$,
such that, for $u<v\le K-1$
\begin{equation}\label{eq:app-centred-string}
 \|D^{\mathrm{closed}}_{u,v}\|_{L^2}
 \le
 \frac{C(\varepsilon_\rho)^{v-u}}
      {\sqrt{(K-u)(K-v)}}
      \qquad  \text{ with } \varepsilon_\rho
 \coloneqq \frac{C_0}{\sqrt{\log\rho}}<1 \,
\end{equation}
when $\rho>e^{C_0^2}$.
Additionally, for every $u< v\le K-1$,
\begin{equation}\label{eq:app-open-string}
 \bigl\|q_{N_{u-1}}\lact
 D_u\circ\cdots\circ D_{v-1}\circ Z_{N_{v-1},N}\ract\one\bigr\|_{L^2}
 \le\frac{C(\varepsilon_\rho)^{v-u-1}}{\sqrt{K-u}} \,.
\end{equation}
\end{proposition}

Let us stress here that this result is analogous to what is obtained by Gu and Tsai, see in particular~\cite[Corollary~3.5]{GT26} (our \(\rho \) corresponds to their \(b^{-2}\)).
Since we were not able to simply adapt their proof to our discrete setting,
we have decided to write a complete proof for the sake
of completeness; we present it in Appendix~\ref{app:technical-proposition}.

\begin{remark}[Uniformity in the parameters]
\label{rem:block-rho-uniformity}
We stress that in Propositions~\ref{prop:block-inputs-ii} and~\ref{prop:block-inputs-iii},
the bounds obtained are uniform in \(N\)
and the constants can be chosen uniformly for \(\vartheta\) in a compact set of \(\R\) and more importantly uniformly for \(\rho >1 \) large enough
e.g. replacing 
\ \(\gep_{\rho}\) by \(C_0/\sqrt{\log \rho}\) in Proposition~\ref{prop:block-inputs-iii}.
\end{remark}

\subsection{Ratios of partition functions: proof of Proposition~\ref{prop:block-inputs-ii}}
\label{app:scalar-and-ratios}
We prove Proposition~\ref{prop:block-inputs-ii} in
two steps: we first prove~\eqref{eq:app-good-P},
introducing all the necessary tools, and then
we turn to the proof of~\eqref{eq:app-good-Q}.

\medskip
\noindent
{\bf Proof of \eqref{eq:app-good-P}. }
First of all, using the fourth moment estimate
from Lemma~\ref{lem:moments-blocks}, we get that
\begin{equation}
      \label{eq:probability-Bi}
      \bbP\bigl(|\mathsf Z_i(q_{N_{i-1}})-1|>1/2\bigr) \leq  16 \,\bbE\bigl[|\mathsf Z_i(q_{N_{i-1}})-1|^4\bigr]\leq \frac{C}{(K-i)^2} \,.
\end{equation}
Using a union bound for $\Omega_{r,s}^c\coloneqq \bigcup_{i=r}^s  \{|\mathsf Z_i(q_{N_{i-1}})-1|>1/2\}$,
we obtain that $\bbP(\Omega_{r,s}^c) \le\frac{C}{K-s}$.

\smallskip
We next show the second estimate, namely
\begin{align}\label{ratios:main}
 \bbE\bigl[\one_{\Omega_{r,s}}
       |\mathscr R_{r,s}-1|^2\bigr]
       \le\frac{C}{K-s} \,.
\end{align}
Recall that \(\mathscr R_{r,s}\) is defined in~\eqref{eq:block-ratio} and recall also the expansion~\eqref{eq:R-interval}.
To proceed, we will need some auxiliary estimates.
First, denote by
$\mathcal B_i\coloneqq \{|\mathsf Z_i(q_{N_{i-1}})-1|\le1/2\}$
and define the truncated inverse multiplier
\[
  M_i\coloneqq \frac{\one_{\mathcal B_i}}{\mathsf Z_i(q_{N_{i-1}})}
\]
and the truncated normalized closed corrector
\[
 \widetilde{\mathscr D}^{\mathrm{closed}}_{u,v}
    =D^{\mathrm{closed}}_{u,v}\prod_{i=u}^v M_i.
\]
In particular, for any \(r\leq u\leq v\leq s\), 
we have \(\one_{\Omega_{r,s}} \mathscr{D}^{\mathrm{closed}}_{u,v} =\widetilde{\mathscr D}^{\mathrm{closed}}_{u,v} 
\prod_{i\notin [u,v]} \one_{\mathcal{B}_i} \),
and, noticing that \(M_i\leq 2\), 
Proposition~\ref{prop:block-inputs-iii}-\eqref{eq:app-centred-string} yields
\begin{equation}
      \label{eq:consequence-proposition-block}
      \|\widetilde{\mathscr D}^{\mathrm{closed}}_{u,v}\|_{L^2} \leq  \frac{C (2\gep_{\rho})^{v-u}}{\sqrt{(K-u)(K-v)}} \leq  \frac{C (2\gep_{\rho})^{v-u}}{K-v} \,.
\end{equation}

It is not difficult to see that $|M_i-1|\le2|\mathsf Z_{i}(q_{N_{i-1}})-1|$ a.s.,
including on $\mathcal B_i^c$, so that combined with the second moment estimate from Lemma~\ref{lem:moments-blocks}, it leads to
\begin{equation}\label{eq:app-first-multiplier}
 \|M_i-1\|_{L^2}\le2\sqrt{\Var\bigl(\mathsf Z_i(q_{N_{i-1}}))}\le  \frac{C}{\sqrt{K-i}},
\end{equation}
for any $i<K$.
We will use \eqref{eq:app-first-multiplier}
to obtain an improved estimate on
$|\bbE\widetilde{\mathscr D}^{\mathrm{closed}}_{u,v}|$
compared to that of a simple Cauchy--Schwarz's inequality.

Let us first decompose
\begin{equation}
      \label{eq:tilde-D-closed}
\bbE\bigl[\widetilde{\mathscr D}^{\mathrm{closed}}_{u,v}\bigr]
= \bbE\big[ D^{\mathrm{closed}}_{u,v} \,(M_u-1) \,\prod_{i=u+1}^v M_i \big]
+	\bbE\big[ D^{\mathrm{closed}}_{u,v}  \,\prod_{i=u+1}^v M_i \big]
\end{equation}
Then, denoting by $\bbE_u$ the expectation with respect to disorder in block $u$, we notice that
\[
\bbE\big[ D^{\mathrm{closed}}_{u,v}  \,\prod_{i=u+1}^v M_i \big] =
\bbE\Big[ q_{N_{u-1}} \lact\bbE_u\big[  D_u \big]\,
\circ\cdots \circ D_{v-1}\circ Z_{v} \ract \one  \,\prod_{i=u+1}^v M_i \Big] = 0 \,.
\]
In the above, the last identity follows from the fact that 
\[\bbE_u [ D_u(x,y) ] = \bbE_u[ \mathsf Z_u(x,y) -\mathsf Z^{\mathrm{rewired}}_u(x,y)] 
= q_{N_u-N_{u-1}}(y-x) - q_{N_u}(y)\]
and so
\(q_{N_{u-1}}\lact \bbE_u [D_u](y) = 0\).

Going back to~\eqref{eq:tilde-D-closed}, using the bound $M_i\le2$ and then Cauchy--Schwarz's inequality, we get that
\begin{align}\label{D-abs-est}
\big| \bbE\bigl[\widetilde{\mathscr D}^{\mathrm{closed}}_{u,v}\bigr] \big|
\leq
2^{v-u} \bigl\|D^{\mathrm{closed}}_{u,v}\bigr\|_{L^2}
 \bigl\|M_u-1 \bigr\|_{L^2}
 \le\frac{C(2\varepsilon_\rho)^{v-u}}{(K-v)^{3/2}},
\end{align}
where we have used \eqref{eq:app-first-multiplier} and Proposition~\ref{prop:block-inputs-iii}-\eqref{eq:app-centred-string} for the last inequality (together with the fact that \(K-u\geq K-v\)).

We now have the ingredients to prove \eqref{ratios:main}.
To this end, we expand $(\mathscr R_{r,s}-1)^2$ using \eqref{eq:R-interval}:
we obtain
\begin{align}\label{D-square-exp}
\bbE\bigl[\one_{\Omega_{r,s}}
       |\mathscr R_{r,s}-1|^2\bigr]
&= \sum_{\substack{\mathcal I, \mathcal J\in\mathfrak I(r,s) \\ \mathcal I, \mathcal J \neq \emptyset}}
       \bbE\Big[ \one_{\mathcal{B}_{\mathcal I, \mathcal J}} \prod_{I\in\mathcal I} \widetilde {\mathscr D}_I^{\mathrm{closed}}
       \prod_{J\in\mathcal J} \widetilde{\mathscr D}_J^{\mathrm{closed}} 
       \Big]
\end{align}
where \(\mathcal{B}_{\mathcal I,\mathcal J}\) accounts for the part of the event \(\Omega_{r,s}\) that has not been used in the definition of \(\widetilde {\mathscr D}_I\), \(\widetilde {\mathscr D}_J\); 
in particular \(\mathcal{B}_{\mathcal I,\mathcal J}\) is independent of the family \(\{\widetilde {\mathscr D}_I,\widetilde {\mathscr D}_J; I\in \mathcal I, J\in \mathcal J\}\).
Recall that $\mathfrak I(r,s)$ is the set of families of pairwise disjoint
non-degenerate integer intervals $[u,v]\subseteq[r,s]$, with $u<v$.
Even though intervals in \(\mathcal I\) and intervals in \(\mathcal J\) are internally disjoint, intervals from the two families can overlap.
Products over disjoint connected components are independent, so their expectations factorise.

If a component contains one interval from one family and none from the other, its expectation is estimated via \eqref{D-abs-est}.
For any other connected component \([a,b]\), Cauchy--Schwarz and block independence within each family give
\begin{equation}\label{overlap-block}
\begin{aligned}
 \bbE\Bigl[\Bigl|
   \prod_{\substack{I\in\mathcal I\\I\subset[a,b]}}
       \widetilde{\mathscr D}_I^{\mathrm{closed}}
   \prod_{\substack{J\in\mathcal J\\J\subset[a,b]}}
       \widetilde{\mathscr D}_J^{\mathrm{closed}}
 \Bigr|\Bigr]
 &\le
   \prod_{\substack{I\in\mathcal I\\I\subset[a,b]}}
       \|\widetilde{\mathscr D}_I^{\mathrm{closed}}\|_{L^2}
   \prod_{\substack{J\in\mathcal J\\J\subset[a,b]}}
       \|\widetilde{\mathscr D}_J^{\mathrm{closed}}\|_{L^2}\\
 &\le \frac{(2C\varepsilon_\rho)^{b-a}}{(K-b)^2} \,.
\end{aligned}
\end{equation}
Here we used \eqref{eq:consequence-proposition-block}: the sum of the interval lengths over both families is at least \(b-a\). Each interval \([u,v]\subset[a,b]\) contributes a factor \((K-v)^{-1}\le(K-b)^{-1}\), and at least two such factors occur.
As above, we take \(\rho\) large enough that \(2C\varepsilon_\rho<1\).

Going back to \eqref{D-square-exp}, we express the union of the intervals in \(\mathcal I\) and \(\mathcal J\) as \(\bigcup_{i=1}^m[a_i,b_i]\), where the \([a_i,b_i]\) are its connected components and \(r\le a_1<b_1<\cdots<a_m<b_m\le s\).
Then, using block independence and employing either \eqref{D-abs-est} (for isolated intervals)
or \eqref{overlap-block} (for overlapping intervals)
we obtain
\begin{equation}\label{D-square-2}
\bbE\bigl[\one_{\Omega_{r,s}}
      |\mathscr R_{r,s}-1|^2\bigr]
\leq	\sum_{m\geq 1} \sum_{r\leq a_1<b_1<\cdots<a_m<b_m\leq s} \sum_{\substack{\alpha_i \in \{1/2,1\}\\ \alpha_1+\cdots +\alpha_m \geq 1}} 
\prod_{i=1}^m \frac{(8 C\varepsilon_\rho)^{b_i-a_i} }{(K-b_i)^{1 + \alpha_i}} \,,
\end{equation}
where \(\alpha_i \coloneqq 1/2\) for isolated intervals and \(\alpha_i \coloneqq 1\) for overlapping intervals; the restriction \(\alpha_1+\cdots +\alpha_m \geq 1\) accounts for the fact that there cannot be a single isolated interval (in other words one must have \(\alpha_1=1\) when \(m=1\)).
We have also used that on \([a_i,b_i]\) there are at most \(4^{b_i-a_i}\) pairs of internally disjoint interval families, since each family is determined by an even-cardinality subset of the \(b_i-a_i+1\) possible endpoints, giving \(2^{b_i-a_i}\) choices.

First, we sum the above over the $a_i$ indices, obtaining
\[
\begin{split}
     \bbE\bigl[\one_{\Omega_{r,s}}
      |\mathscr R_{r,s}-1|^2\bigr]
       &\leq  \sum_{m\geq 1} (C' \gep_\rho)^m \sum_{r\leq b_1<\cdots<b_m\leq s} \sum_{\substack{\alpha_i \in \{1/2,1\}\\ \alpha_1+\cdots +\alpha_m \geq 1}} 
\prod_{i=1}^m (K-b_i)^{-(1+\alpha_i)} \\
& \leq \sum_{m\geq 1} (C' \gep_\rho)^m \sum_{\substack{\alpha_i \in \{1/2,1\}\\ \alpha_1+\cdots +\alpha_m \geq 1}} 
\prod_{i=1}^m (K-s)^{-\alpha_i} \,.
\end{split}
\]
Now, for \(m=1\) we have that necessarily \(\alpha_1\geq 1\), while for \(m\geq 2\) we use that \(\alpha_i\geq 1/2\):
this gives
\[
\bbE\bigl[\one_{\Omega_{r,s}}
      |\mathscr R_{r,s}-1|^2\bigr] \leq \frac{C' \gep_\rho}{K-s} + \sum_{m\geq 2} \Bigl(\frac{2 C' \gep_\rho}{\sqrt{K-s}}\Bigr)^m \leq \frac{C' \gep_\rho}{K-s} + \frac{C'' (\gep_\rho)^2}{K-s}
\]
where the last inequality holds provided that \(\gep_{\rho}\) is small enough.
This yields the desired bound \eqref{ratios:main},
concluding the proof of~\eqref{eq:app-good-P}.
\qed

\medskip\noindent
{\bf Proof of \eqref{eq:app-good-Q}. }
Recall the definition~\eqref{eq:weighted-law} of the weighted law \(\mathbb Q_{r,s}\).
Let us start with the first bound in~\eqref{eq:app-good-Q}.
Recalling that $\mathcal B_i\coloneqq \{|\mathsf Z_i(q_{N_{i-1}})-1|\le1/2\}$,
we have that
\[
\mathbb Q_{r,s}\big( \mathcal B_i^c \big)
= \frac{\bbE \big[ \mathsf Z_i(q_{N_{i-1}})^{-1} \one_{\mathcal B_i^c} \big]}{\bbE \big[\mathsf Z_i(q_{N_{i-1}})^{-1} \big]}
\leq  \bbE \big[\mathsf Z_i(q_{N_{i-1}})^{-4} \big]^{1/4} \, \bbP\big(\mathcal B_i^c \big)^{3/4}  \,.
\]
The equality follows from \eqref{eq:weighted-law} and block independence, since \(\mathcal B_i^c\) depends only on the disorder in block \(i\).
The inequality follows from Hölder's inequality
and the fact that \(\bbE[\mathsf Z_i(q_{N_{i-1}})^{-1}]\geq\bbE[\mathsf Z_i(q_{N_{i-1}})]^{-1}=1\).
Then, we have that \(\bbE[\mathsf Z_i(q_{N_{i-1}})^{-4}]\leq C\)
thanks to the negative
moment estimate of \eqref{eq:one-point-negative} from Lemma~\ref{lem:one-point}, with $p=4$.
Together with~\eqref{eq:probability-Bi} above, this shows that
\(\mathbb Q_{r,s}( \mathcal B_i^c)\leq C (K-i)^{-3/2}\).
Then, a union bound shows that
\[
\mathbb Q_{r,s}(\Omega_{r,s}^c)
\leq \sum_{i=r}^s \mathbb Q_{r,s}\big( \mathcal B_i^c \big)
\leq \frac{C}{\sqrt{K-s}} \,,
\]
which proves the first part of~\eqref{eq:app-good-Q}.

\smallskip
We now proceed to check the second inequality in \eqref{eq:app-good-Q}, namely
\begin{equation}\label{eq:app-good-Q2}
 \bbE_{\mathbb Q_{r,s}}\bigl[
 \one_{\Omega_{r,s}}|\mathscr R_{r,s}-1|^2\bigr]
 \le\frac{C}{\sqrt{K-s}}.
\end{equation}
The steps are similar to those taken for
 \eqref{ratios:main}, so we will
only provide a sketch.

As in \eqref{D-square-exp}, 
the starting point is the expansion 
\begin{equation}
      \label{eq:starting-point-Q}
	\bbE_{\mathbb Q_{r,s}}\bigl[\one_{\Omega_{r,s}}
       |\mathscr R_{r,s}-1|^2\bigr]
     =
      \sum_{\substack{\mathcal I, \mathcal J\in\mathfrak I(r,s) \\ \mathcal I, \mathcal J \neq \emptyset}}
       \bbE_{\mathbb Q_{r,s}} \Big[ \one_{\mathcal{B}_{\mathcal I, \mathcal J} } \prod_{I\in\mathcal I} \widetilde {\mathscr D}_I^{\mathrm{closed}}
       \prod_{J\in\mathcal J} \widetilde{\mathscr D}_J^{\mathrm{closed}}
       \Big] \,.
\end{equation}
To estimate this quantity we need the analogues of \eqref{D-abs-est} and \eqref{overlap-block} for the \(\mathbb Q\)-expectation.

For the analogue of~\eqref{D-abs-est}, we have 
\begin{align*}
\bbE_{\mathbb Q_{r,s}} \bigl[\widetilde{\mathscr D}^{\mathrm{closed}}_{u,v}\bigr]
&=	\frac{\bbE\Big[ D_{u,v}^{\rm closed} \prod_{i=u}^v M_i^2\Big]}{\prod_{i=u}^v \bbE \big[ Z^{-1}_i(q_{N_{i-1}}) \big]}\,
\end{align*}
for $r\leq u<v\leq s$.
The denominator is (as we have seen earlier) bounded below by $1$.
The numerator can be written as
\begin{align*}
& \bbE \Big[ D^{\mathrm{closed}}_{u,v} \,(M^2_u-1) \,\prod_{i=u+1}^v M_i^2 \Big]
+	\bbE \big[ D^{\mathrm{closed}}_{u,v}  \,\prod_{i=u+1}^v M_i^2 \big]
\end{align*}
The second term on the right-hand side vanishes, using again the observation that 
$q_{N_{u-1}} \lact \bbE_u [ D_u  ]=0$.
The first term can be estimated as in \eqref{D-abs-est} after writing $M^2_i-1=(M_i-1)(M_i+1)$
and using that  $M_i+1\leq 3$.

Finally, notice that the analogue of~\eqref{eq:consequence-proposition-block} holds (\textit{i.e.}\ we have a bound on the \(L^2(\mathbb Q)\) norm of \(\widetilde{\mathscr D}^{\mathrm{closed}}_{u,v}\)) thanks to Proposition~\ref{prop:block-inputs-iii}-\eqref{eq:app-centred-string},
using again that $M_i\leq 2$ together with the fact that $\bbE\big[ Z_i^{-1}(q_{N_{i-1}})\big]\geq 1$.
Therefore the analogue \eqref{overlap-block} for the $\mathbb Q_{r,s}$ expectation follows,
with the same line of proof.

All together, with the analogue of~\eqref{D-abs-est} and~\eqref{overlap-block} for the \(\mathbb Q\)-expectation,
the expansion~\eqref{eq:starting-point-Q} allows us to prove
the analogue of~\eqref{D-square-2} for the \(\mathbb Q\)-expectation
so one can conclude similarly as above.
\qed

\subsection{One-point estimate: proof of Proposition~\ref{prop:VD}}
\label{sec:vanishing-defect}

We first prove the small-time comparison \eqref{target:VD}, and then
deduce the bounded comparison \eqref{prop:eq-heat-defect} by averaging
out the disorder after a fixed small time.
By \eqref{eq:polymer-one-point-convexity},
\begin{equation*}
 0\le\Delta(\log\widehat Z_{s,n}(x))
 \le\log\frac{\bbE[\widehat Z_{s,n}(x)^2]}
                 {\bbE[\widehat Z_{s,n}(x)^{-1}]}.
\end{equation*}
It therefore suffices to bound the ratio of the second moment to the
negative-first moment. For the small-time comparison, set
\(n=[\sqrt\delta N]\); we will show that
\begin{equation}\label{eq:Delta-bound}
 \log\frac{\bbE[\widehat Z_{s,n}(x)^2]}
              {\bbE[\widehat Z_{s,n}(x)^{-1}]}
 \le\frac{C}{|\log\delta|^{1/4}}+o_N(1).
\end{equation}
By translation invariance, we only need to treat the case \(x=0\).

\smallskip\noindent\emph{Reducing to second moment estimates.}
First suppose that $s=N_u$ for some $0\le u<b$, on the grid
chosen in Remark~\ref{rem:choose-J}, so that $n=N_b$ and $K-b=J$.
The case $s=n$ is immediate since $\widehat Z_{n,n}(0)=1$.
By the block decomposition~\eqref{eq:block-ratio}, we have
\[
 \widehat Z_{s,n}(0)
 =\pigl(\prod_{i=u+1}^b\mathsf Z_i(q_{N_{i-1}})\pigr)
   \mathscr R_{u+1,b}.
\]
We introduce the law weighted by the product of inverse block partition functions, recall~\eqref{eq:weighted-law}:
\[
 \frac{\dd\mathbb Q}{\dd\bbP}
 =\frac{\pigl(\prod_{i=u+1}^b \mathsf Z_i(q_{N_{i-1}})\pigr)^{-1}}
        {\bbE\pigl[\prod_{i=u+1}^b \mathsf Z_i(q_{N_{i-1}})^{-1}\pigr]}.
\]
In particular, we have that
\begin{equation}
      \label{eq:negative-first-moment}
\bbE\bigl[\widehat Z_{s,n}(0)^{-1}\bigr]
=  \bbE_{\mathbb Q} \bigl[\mathscr R_{u+1,b}^{-1}\bigr] \bbE\pigl[\prod_{i=u+1}^b \mathsf Z_i(q_{N_{i-1}})^{-1}\pigr]
=  \bbE_{\mathbb Q} \bigl[\mathscr R_{u+1,b}^{-1}\bigr] \prod_{i=u+1}^b \bbE\bigl[\mathsf Z_i(q_{N_{i-1}})^{-1}\bigr]  \,,
\end{equation}
where in the second equality we also used that \((\mathsf Z_i(q_{N_{i-1}}))_{i\geq 1}\) are independent.

We may now use the second part of
Proposition~\ref{prop:block-inputs-ii}, namely~\eqref{eq:app-good-Q}, to get that
\begin{equation}\label{eq:VD-ratio-input}
 \mathbb Q(\Omega_{u+1,b}^c)\leq C (K-b)^{-1/2},
 \qquad
 \bbE_{\mathbb Q}\bigl[|\mathscr R_{u+1,b}-1|^2\one_{\Omega_{u+1,b}} \bigr]
 \leq C(K-b)^{-1/2}\,,
\end{equation}
Using that $\mathscr R^{-1}\geq2-\mathscr R$ for $\mathscr R>0$, we get that
\begin{equation}
      \label{eq:VD-ratio-output}
\begin{split}
      \bbE_{\mathbb Q}[\mathscr R_{u+1,b}^{-1}]\geq \bbE_{\mathbb Q}[\mathscr R_{u+1,b}^{-1} \one_{\Omega_{u+1,b}}]
      &\geq\mathbb Q(\Omega_{u+1,b}) + \bbE_{\mathbb Q}[(1-\mathscr R_{u+1,b})\one_{\Omega_{u+1,b}}]\\
      &\geq 1-\mathbb Q(\Omega_{u+1,b}^c)
      -\bbE_{\mathbb Q}[
            |\mathscr R_{u+1,b}-1|^2 \one_{\Omega_{u+1,b}}]^{1/2}
      \geq1-CJ^{-1/4} \,,
\end{split}
\end{equation}
where we have used Cauchy--Schwarz's inequality for the second inequality and then \eqref{eq:VD-ratio-input}, recalling also that \(K-b=J\).

On the other hand, we also claim that
\begin{equation}\label{eq:block-reciprocal-VD}
 \bbE\bigl[\mathsf Z_i(q_{N_{i-1}})^{-1}\bigr]
 = \bbE\bigl[\mathsf Z_i(q_{N_{i-1}})^{2}\bigr] \bigl(1+O\bigl((K-i)^{-3/2}\bigr)\bigr) \,.
\end{equation}
Indeed, writing $X_i=\mathsf Z_i(q_{N_{i-1}})-1$ we can use the identity \((1+X_i)^{-1}=1-X_i+X_i^2-\frac{X_i^3}{1+X_i}\) to get that
\[
\bbE\bigl[\mathsf Z_i(q_{N_{i-1}})^{-1}\bigr] =  \bbE\bigl[\mathsf Z_i(q_{N_{i-1}})^{2}\bigr] - \bbE\biggl[\frac{ (\mathsf Z_i(q_{N_{i-1}})-1)^3}{\mathsf Z_i(q_{N_{i-1}})}\biggr] \,.
\]
Now, by Hölder's inequality, we have that
\[
 \bbE\biggl[\frac{ |\mathsf Z_i(q_{N_{i-1}})-1|^3}{\mathsf Z_i(q_{N_{i-1}})}\biggr]
 \leq \bbE\bigl[|\mathsf Z_i(q_{N_{i-1}})-1|^4\bigr]^{3/4}
       \bbE\bigl[ \mathsf Z_i(q_{N_{i-1}})^{-4} \bigr]^{1/4}
 \leq C(K-i)^{-3/2}  \,.
\]
For the last inequality, we have used Lemma~\ref{lem:moments-blocks} to control the fourth moment and Lemma~\ref{lem:one-point}-\eqref{eq:one-point-negative} with \(p=4\) to bound \(\bbE[ \mathsf Z_i(q_{N_{i-1}})^{-4}] \leq \bbE[ \mathsf Z_i(q_{N_{i-1}})^{2}]^{10}\),
with the second moment being bounded by \(\bbE[ \mathsf Z_i(q_{N_{i-1}})^{2}] \leq  1 + \frac{C}{K-i} \leq 1+C\), again by Lemma~\ref{lem:moments-blocks}.

Since
$\sum_i(K-i)^{-3/2}\leq CJ^{-1/2}$, going back to~\eqref{eq:negative-first-moment} and using~\eqref{eq:VD-ratio-output} and~\eqref{eq:block-reciprocal-VD}, we end up with
\begin{equation}
      \label{eq:Ynegative-VD}
      \bbE\bigl[\widehat Z_{s,n}(0)^{-1}\bigr]  \geq \bigl( 1+ O(J^{-1/4})\bigr) \prod_{i=u+1}^b \bbE\bigl[\mathsf Z_i(q_{N_{i-1}})^{2}\bigr]   \,.
\end{equation}

\smallskip\noindent\emph{Estimating the second moments.}
We first prove \eqref{eq:Delta-bound} at grid points, i.e.\ for \(s=N_{u}\) with \(0\le u\le b\).
In view of~\eqref{eq:Ynegative-VD}, it suffices to show that
\begin{equation}
      \label{eq:ratio-second-moment}
      \frac{\bbE[\widehat Z_{s,n}(0)^2] }{\prod_{i=u+1}^b \bbE\bigl[\mathsf Z_i(q_{N_{i-1}})^{2}\bigr]} \leq 1+ C J^{-1} \,.
\end{equation}

We rely on the following sharp second moment estimate.
Its proof is classical and postponed to Appendix~\ref{app:renewal-comparison}.
To state it in a more natural way, let us introduce, analogously to \cite[Eq.~(1.12)]{BCT25}, the interpolating parameter
\[
\vartheta_N(s) \coloneqq \frac{\pi}{\sigma_N^2} - \pi R_s \,,
\]
which controls the distance to the critical window in time-\(s\) partition functions (we use a different sign convention from  \cite{BCT25}, in order to keep the parameter positive in the sub and quasi-critical regime).
The factor \(\pi\) is so that we have that  \(\vartheta_N(N) \to -\vartheta\) as \(N\to\infty\).
The following estimates apply whenever $\vartheta_N(t)>0$, in particular for $t \le \epsilon_0 N$ for
$\epsilon_0 > 0$ small enough. To be definite,
we set 
$\epsilon_0:=e^{-1-\max\{\vartheta,0\}}$.

\begin{lemma}\label{lem:second-moment-hat}
There exists a constant $C>0$ such that, for all sufficiently
large $N$ and all integers $0\le s\le t\le\epsilon_0N$,
\[
0\le
\frac{\vartheta_N(s)}{\vartheta_N(t)}
-\bbE[\widehat Z_{s,t}(0)^2]
\le
\frac{C(1+R_t-R_s)}{\vartheta_N(t)^2} \,,
\]
and therefore
\begin{equation}\label{eq:consequence-lem-SM}
\bbE[\widehat Z_{s,t}(0)^2]
=
\frac{\vartheta_N(s)}{\vartheta_N(t)}
\left(
1+O\left(
\frac{1+R_t-R_s}{\vartheta_N(s)\vartheta_N(t)}
\right)
\right).
\end{equation}
If $1\le s\le t/2$, the factor $1+R_t-R_s$ in these
bounds can be replaced by $\log(t/s)$.
\end{lemma}

The asymptotics of $R_t$ give, uniformly for integers
$0\le s<t$,
\begin{equation}\label{eq:vartheta-difference}
\vartheta_N(s)-\vartheta_N(t)
=\pi(R_t-R_s)
=\log\frac{t}{s\vee1}+O(1) \,,
\end{equation}
and since $\vartheta_N(N) \to -\vartheta = O(1)$,
we obtain in particular
\[
\vartheta_N(t)=\log\frac{N}{t\vee1}+O(1),
\qquad 0\le t\le N.
\]
For fixed $\rho$ and $J=K-b$ sufficiently large,
this implies
\begin{equation}\label{eq:estimate-vartheta-Ni}
\vartheta_N(N_i)\asymp K-i,
\qquad 0\le i\le b.
\end{equation}
In particular, since $n=N_b=[N\rho^{-J}]$,
we have $\vartheta_N(n)
=\log(N/n)+O(1)=J\log\rho+O(1)$.
Thus, for fixed $\rho>1$ and $J$ sufficiently large,
$\vartheta_N(n)\asymp J$.

For every $0\le s\le n$, we have
\[
\vartheta_N(s)
=\vartheta_N(n)+\pi(R_n-R_s)
\ge 1+R_n-R_s,
\]
provided $J$ is large enough.
Thus~\eqref{eq:consequence-lem-SM} and
$\vartheta_N(n)\asymp J$ give
\begin{equation}\label{eq:second-moment-sharp}
\bbE[\widehat Z_{s,n}(0)^2]
=
\frac{\vartheta_N(s)}{\vartheta_N(n)}
\bigl(1+O(J^{-1})\bigr),
\qquad 0\le s\le n \,.
\end{equation}
For every $1\le i\le b$ we can bound
$ R_{N_i}-R_{N_{i-1}}\le C_\rho $ by definition
of $N_i$
(also for $i=1$, since $N_0=0$ and
$N_1\le2\rho$).
Applying Lemma~\ref{lem:second-moment-hat} and
\eqref{eq:estimate-vartheta-Ni}, we obtain
\[
\bbE[\mathsf Z_i(q_{N_{i-1}})^2]
=
\frac{\vartheta_N(N_{i-1})}{\vartheta_N(N_i)}
(1-\eta_i)
\qquad \text{with} \quad
0\le\eta_i\le\frac{C}{(K-i)^2}.
\]
Since $\sum_{i=u+1}^b\eta_i
\le C\sum_{j=J}^{\infty}\frac1{j^2}
\le \frac{C'}J$,
for $J$ sufficiently large we get
\[
\begin{aligned}
\prod_{i=u+1}^b\bbE[\mathsf Z_i(q_{N_{i-1}})^2]
&=
\frac{\vartheta_N(N_u)}{\vartheta_N(n)}
\prod_{i=u+1}^b(1-\eta_i) = \frac{\vartheta_N(N_u)}{\vartheta_N(n)}
\bigl(1+O(J^{-1})\bigr).
\end{aligned}
\]
Together with~\eqref{eq:second-moment-sharp}, this yields,
for $s=N_u$,
\[
\frac{\bbE[\widehat Z_{s,n}(0)^2]}
{\prod_{i=u+1}^b\bbE[\mathsf Z_i(q_{N_{i-1}})^2]}
=1+O(J^{-1}).
\]
This proves~\eqref{eq:ratio-second-moment} and hence
\eqref{eq:Delta-bound} at every grid point, uniformly
in $0\le u\le b$.

\smallskip
We can now extend the proof to any \(s\in (0,n)\) which may not be the starting point of a block.
Let $s$ be even, with
$s\in(N_{u-1},N_u)$ for some $1\le u\le b$.
Averaging the disorder on $(s,N_u]$ and applying
conditional Jensen gives
\[
\bbE[\widehat Z_{s,n}(0)^{-1}]
\ge
\bbE[\widehat Z_{N_u,n}(0)^{-1}].
\]
On the other hand, the first-collision formula \eqref{eq:local-replica-formula} and $\sigma_N^2\bbE[Z_{n-k}(0)^2] \le \pi/\vartheta_N(n-k) \le \pi/\vartheta_N(n)$, see
Claim~\ref{claim:VN}, imply
\[
\begin{aligned}
0&\le
\bbE[\widehat Z_{s,n}(0)^2] -\bbE[\widehat Z_{N_u,n}(0)^2] 
=\sum_{k=s+1}^{N_u}q_{2k}(0)\sigma_N^2\bbE[Z_{n-k}(0)^2] \le
\frac{\pi(R_{N_u}-R_s)}{\vartheta_N(n)}
\le \frac CJ.
\end{aligned}
\]
Since $\bbE[\widehat Z_{N_u,n}(0)^2]\ge
\bbE[\widehat Z_{N_u,n}(0)]^2 = 1$, the ratio
of these second moments is at most $1+C/J$.
Combining this with the bound already proved at the grid points $s= N_u$
gives~\eqref{eq:Delta-bound} for every even $0\le s\le n$.
This concludes the proof of~\eqref{target:VD}.

\smallskip\noindent
\textbf{Proof of~\eqref{prop:eq-heat-defect}.}
Fix \(0<T_0<T_1<\infty\), and now let \(T_0N\le n\le T_1N\).
The expansion~\eqref{eq:local-replica-formula}, the last statement in Claim~\ref{claim:VN}, and the bound \(q_{2k}(0)\le C/k\) give
\begin{equation}\label{prop:eq-entrance-second-moment-upper}
 \begin{aligned}
 \bbE[\widehat Z_{s,n}(0)^2]
 &=1+\sigma_N^2\sum_{k=s+1}^n q_{2k}(0)\bbE[Z_{n-k}(0)^2]\\
 &\le1+C_{\vartheta,T_1}\sum_{k=s+1}^n\frac1k
 \le C\left\{1+\log\frac{n}{s\vee1}\right\}.
 \end{aligned}
\end{equation}
Fix \(0<\delta = \delta_{T_0}\le\delta_0\) sufficiently small, independently of
\(s,n,N\), with \(\sqrt\delta<T_0\), and put \(n_0=[\sqrt\delta N]\).
If \(s\ge n_0\), then \(n/s\le T_1/\sqrt\delta\), so
\eqref{prop:eq-entrance-second-moment-upper} and
\(\bbE[\widehat Z_{s,n}(0)^{-1}]\ge1\) give the required comparison.

For \(s<n_0\), averaging over the disorder on \((n_0,n]\) and
applying a conditional Jensen's inequality gives
\begin{equation*}
      \bbE[\widehat Z_{s,n}(0)^{-1}]
 \ge\bbE[\widehat Z_{s,n_0}(0)^{-1}]
 \ge c\,\bbE[\widehat Z_{s,n_0}(0)^2]
\end{equation*}
where the second inequality follows from the small-time comparison~\eqref{target:VD}
proven above.
Using the bound~\eqref{eq:second-moment-sharp},
we obtain
\begin{equation}\label{prop:eq-entrance-Jensen}
 \bbE[\widehat Z_{s,n_0}(0)^2]
 \ge \frac{1}{2} \, \frac{\vartheta_N(s)}{\vartheta_N(n_0)}
 \ge c_{T_0,T_1}\left\{1+\log\frac{n}{s\vee1}\right\} \,,
\end{equation}
where the last inequality holds
thanks to~\eqref{eq:vartheta-difference},
which shows in particular that
\[
\vartheta_{N}(n_0) = \frac12 |\log \delta| +O(1) \leq C_{T_0} \,,
\quad
\vartheta_N(s) = \log \frac{N}{s\vee 1} + O(1) \geq  c_{T_0,T_1} \left\{1+\log\frac{n}{s\vee1}\right\}   \,.
\]
Combining \eqref{prop:eq-entrance-Jensen} with
\eqref{prop:eq-entrance-second-moment-upper} proves
\eqref{prop:eq-heat-defect}.
This concludes the proof of Proposition~\ref{prop:VD}.
\qed

\subsection{Comparing the endpoint distributions: proof of Proposition~\ref{prop:PF}}
\label{sec:prop-PF}

The main goal of this section is to prove Proposition~\ref{prop:PF}.
It relies on estimating quantitatively the difference between the starting distributions
\[
 \mu_n
 =\mu_{N_0,N_b}^{(q_{N_0})},
 \qquad
 \widehat\mu_{m,n}
 =\mu_{N_a,N_b}^{(q_{N_a})} \,.
\]
If \(\eta = \nu - \nu'\) is a signed measure of total mass zero which is \(\cF_n\)-measurable, we define
\begin{equation}
      \label{eq:definition-semi-norm}
 \bigl\| \eta \bigr\|_{\cH_{n,N}}^2
 \coloneqq \bbE\bigl[Z_{n,N}(\eta)^2 \big| \cF_n\bigr] \,.
\end{equation}
Then, by using the conditional Cauchy--Schwarz inequality, we get
\begin{equation}\label{eq:future-L1-PF}
 \bbE\bigl[\one_{\mathcal G}
 |Z_{n,N}(\mu_n)-Z_{n,N}(\widehat\mu_{m,n})|\bigr] \le
 \bbE\bigl[\one_{\mathcal G}
 \|\mu_n-\widehat\mu_{m,n}\|_{\cH_{n,N}}\bigr]\,,
\end{equation}
for any \(\cF_n\)-measurable event \(\mathcal G\).
Thus, \eqref{target:PF} reduces to the control of \(\|\mu_n-\widehat\mu_{m,n}\|_{\cH_{n,N}}\).

To this end, in analogy with Subsection~\ref{sec:block-exp}, we perform an expansion
in terms of the endpoint distributions, since
\begin{align}\label{poly-meas}
\mu_{N_{r-1}, N_s}^{(q_{N_{r-1}})} 
= \frac{q_{N_{r-1}} \lact Z_{N_{r-1},N_s}}{q_{N_{r-1}} \lact Z_{N_{r-1}, N_s} \ract \one} 
= \frac{q_{N_{r-1}} \lact \mathsf Z_r\circ\cdots\circ  \mathsf Z_s}
{{q_{N_{r-1}} \lact Z_{N_{r-1}, N_s} \ract \one}} ,
\end{align}
which is a (random) measure.
Recall the definition
\begin{align}\label{Rrs-remind}
 \mathscr R_{r,s}
 \coloneqq q_{N_{r-1}}\lact
       \widecheck{\mathsf Z}_r\circ\cdots\circ
       \widecheck{\mathsf Z}_s\ract\one
 =\frac{q_{N_{r-1}}\lact Z_{N_{r-1},N_s}\ract\one}
        {\prod_{i=r}^s q_{N_{i-1}}\lact\mathsf Z_i\ract\one}.
\end{align}
Dividing the numerator and denominator in \eqref{poly-meas} by 
$\prod_{i=r}^s q_{N_{i-1}}\lact\mathsf Z_i\ract\one$
we have
\begin{align*}
	\mu_{N_{r-1}, N_s}^{(q_{N_{r-1}})} 
=	\frac{1}{\mathscr R_{r,s}}  q_{N_{r-1}}\lact \widecheck{\mathsf{Z}}_{r}\circ \cdots \circ 
	\widecheck{\mathsf{Z}}_{s} .
\end{align*}
Using
\begin{align}\label{ZD-remind}
 \widecheck{\mathsf Z}_i
 = \widecheck{D}_i +
 \frac{\mathsf Z_i \ract \one \otimes  q_{N_i} \lact }{ q_{N_{i-1}}\lact \mathsf Z_i \ract \one }  \,
 = \widecheck{D}_i + \widecheck{\mathsf Z}_i \ract \one \otimes  q_{N_i} \lact
\end{align}
(recall \eqref{ZD} for notational clarity), we get that
\begin{equation}
       \mu_{N_{r-1},N_s}^{(q_{N_{r-1}})}
 =q_{N_s}
   +\frac{1}{\mathscr R_{r,s}} \sum_{u=r}^s
       \mathscr R_{r,u-1}\mathscr D^{\mathrm{open}}_{u,s} \, ,
 \label{eq:profile-interval-expansion}
\end{equation}
where we have introduced the following open corrector
(which is a signed measure):
\begin{equation}
 \mathscr D^{\mathrm{open}}_{u,s}
 \coloneqq 
 q_{N_{u-1}}\lact
 \widecheck{D}_u\circ\cdots\circ\widecheck{D}_s .
 \label{eq:open-activity}
\end{equation}
Indeed, we can successively replace operators $\widecheck{\mathsf{Z}}_i$ from $i=s,s-1,...,r$ by the decomposition 
\eqref{ZD-remind} as (recall also \eqref{Rrs-remind}):
\begin{align*}
	\mu_{N_{r-1}, N_s}^{(q_{N_{r-1}})} 
&=	\frac{1}{\mathscr R_{r,s}} 
	 q_{N_{r-1}}\lact \widecheck{\mathsf{Z}}_{r}\circ \cdots \circ 
	\widecheck{\mathsf{Z}}_{s-1}\circ
	\Bigg\{ \widecheck{D}_{s} +\widecheck{\mathsf{Z}}_s\ract \one \times q_{N_s} \lact \Bigg\} \\
&=	\frac{1}{\mathscr R_{r,s}}  
	q_{N_{r-1}}\lact \widecheck{\mathsf{Z}}_{r}\circ \cdots \circ
	\widecheck{\mathsf{Z}}_{s-1}\circ 
	\widecheck{\mathsf{Z}}_s\ract \one \times q_{N_s} 
 +     \frac{1}{\mathscr R_{r,s}} 
	q_{N_{r-1}}\lact \widecheck{\mathsf{Z}}_{r}\circ \cdots \circ
	\widecheck{\mathsf{Z}}_{s-1}\circ  
	\widecheck{D}_{s}   \\
&=    q_{N_s} 
+	 \frac{1}{\mathscr R_{r,s}} 
	q_{N_{r-1}}\lact \widecheck{\mathsf{Z}}_{r}\circ \cdots \circ
	\Bigg\{ \widecheck{\mathsf{D}}_{s-1} + \widecheck{\mathsf{Z}}_{s-1} \ract \one \times 
	q_{N_{s-1}} \lact \Bigg\} \circ  
	\widecheck{D}_{s}  \\
&=	q_{N_s} 
+	 \frac{\mathscr R_{r,s-1}}{\mathscr R_{r,s}} q_{N_{s-1}} \lact \widecheck{D}_{s} 
+	 \frac{1}{\mathscr R_{r,s}} 
	q_{N_{r-1}}\lact \widecheck{\mathsf{Z}}_{r}\circ \cdots \circ
	\widecheck{\mathsf{Z}}_{s-2} \circ
	 \widecheck{{D}}_{s-1}  
	 \circ  
	\widecheck{D}_{s}   
\end{align*}
and continuing these successive replacements we arrive at \eqref{eq:profile-interval-expansion}.

Let us also introduce the following corrector, which is the leading (linear) contribution of the corrector in~\eqref{eq:profile-interval-expansion},
\begin{equation}\label{eq:common-correction}
 \Delta\mu^{\mathrm{open}}_{r,s}
 \coloneqq \sum_{u=r}^s\mathscr D^{\mathrm{open}}_{u,s}.
\end{equation}
We therefore rely on the following lemma, whose proof is postponed to the end of this subsection.
We state it for arbitrary consecutive blocks, since the proof of
Proposition~\ref{prop:log-tails} will also use it in this generality.
Recalling the definition~\eqref{eq:block-events} of the event \(\Omega_{r,s}\), we
define
\begin{equation}\label{eq:profile-good-event}
 \mathcal G_{r,s}
 \coloneqq \Omega_{r,s}\cap\{|\mathscr R_{r,s}-1|\le1/2\}.
\end{equation}

\begin{lemma}[Profile approximation and second moments]
\label{lem:profile-estimates}
Let \(\ell_*\) be large enough and $1\le r\le s\le K-\ell_*$.
Then, recalling the definition~\eqref{eq:profile-good-event}, we have
\begin{equation}
      \label{eq:profile-good-prob}
      \bbP(\mathcal G_{r,s}^c)
      \le\frac{C}{K-s} \,,
\end{equation}
and also
\begin{equation}
      \label{eq:profile-second-order}
      \bbE\pigl[\one_{\mathcal G_{r,s}}
 \|\mu_{N_{r-1},N_s}^{(q_{N_{r-1}})}
       -q_{N_s}-\Delta\mu^{\mathrm{open}}_{r,s}\|
       _{\cH_{N_s,N}}\pigr]
  \le\frac{C}{K-s}  \,.
\end{equation}
\end{lemma}

\begin{proof}[Proof of Proposition~\ref{prop:PF}]
  We will prove the proposition with the choice
\begin{equation}\label{eq:comparison-event}
 \mathcal G = \mathcal G_{N,\delta}(0)
 \coloneqq \Omega_{1,b}
 \cap\{|\mathscr R_{1,b}-1|\le1/2\}
 \cap\{|\mathscr R_{a+1,b}-1|\le1/2\}.
\end{equation}
It is the intersection of the two good events in
\eqref{eq:profile-good-event}, so
$\bbP(\mathcal G^c) \le C/J$ by Lemma~\ref{lem:profile-estimates}, since \(K-b=J\).
This shows~\eqref{eq:good-event-probability}, recalling that \(J\asymp |\log \delta|\), see Remark~\ref{rem:choose-J}.

In view of~\eqref{eq:future-L1-PF}, we need to control \(\|\mu_n-\widehat\mu_{m,n}\|_{\cH_{n,N}}\).
In the difference of the two sums defined by \eqref{eq:common-correction}, the terms with \(u\ge a+1\) cancel, giving
\begin{equation}\label{eq:common-cancellation}
 \Delta\mu^{\mathrm{open}}_{1,b}
 -\Delta\mu^{\mathrm{open}}_{a+1,b}
 =\sum_{u=1}^a\mathscr D^{\mathrm{open}}_{u,b}.
\end{equation}
By \eqref{eq:block-objects-hat} and \eqref{eq:open-activity},
\[
 Z_{N_s,N}(\mathscr D^{\mathrm{open}}_{u,s})
 =\frac{q_{N_{u-1}}\lact D_u\circ\cdots\circ D_s\circ Z_{N_s,N}\ract\one}
        {\prod_{i=u}^s\mathsf Z_i(q_{N_{i-1}})}.
\]
On \(\Omega_{u,s}\), each inverse factor is at most \(2\), so \eqref{eq:definition-semi-norm} and Proposition~\ref{prop:block-inputs-iii}-\eqref{eq:app-open-string}, with \(v=s+1\), give for \(1\le u\le s\le K-2\)
\begin{equation}\label{eq:open-activity-bound}
 \bbE\bigl[\one_{\Omega_{u,s}}\|\mathscr D^{\mathrm{open}}_{u,s}\|_{\mathcal H_{N_s,N}}\bigr]
 \le \bbE\bigl[\one_{\Omega_{u,s}}\|\mathscr D^{\mathrm{open}}_{u,s}\|_{\mathcal H_{N_s,N}}^2\bigr]^{1/2}
 \le\frac{C(2\varepsilon_\rho)^{s-u}}{\sqrt{K-u}} \,.
\end{equation}

Since \(K-b=J\), \(b-a=J\), and \(2\varepsilon_\rho<1\) for large \(\rho\), using \eqref{eq:common-cancellation} and the triangle inequality gives
\begin{equation}\label{eq:norm-corrector}
\bbE\bigl[\one_{\Omega_{1,b}}
\|\Delta\mu^{\mathrm{open}}_{1,b}
-\Delta\mu^{\mathrm{open}}_{a+1,b}\|_{\cH_{N_b,N}}^2\bigr]^{1/2}
\le \sum_{i=J}^{\infty} \frac{C (2\varepsilon_\rho)^{i}}{\sqrt{J+i}} \leq \frac{C(2\varepsilon_\rho)^J}{\sqrt J}.
\end{equation}

Therefore, if we apply the triangle inequality
\[
\begin{split}
     \|\mu_n-\widehat\mu_{m,n}\|_{\cH_{n,N}}  \leq  \|\mu_n-q_{N_{b}}- \Delta\mu^{\mathrm{open}}_{1,b} \|_{\cH_{n,N}} & +\|\widehat\mu_{m,n}- q_{N_{b}}- \Delta\mu^{\mathrm{open}}_{a+1,b}\|_{\cH_{n,N}} \\
     & \qquad \qquad + \|\Delta\mu^{\mathrm{open}}_{1,b}
 -\Delta\mu^{\mathrm{open}}_{a+1,b}\|_{\cH_{n,N}}  \,,
\end{split}
\]
together with \eqref{eq:profile-second-order} (once with $(r,s)=(1,b)$ and once with $(r,s)=(a+1,b)$, recalling that \(K-b=J\)) and the above~\eqref{eq:norm-corrector},
we finally obtain
\begin{equation}\label{eq:profile-comparison-seminorm}
 \bbE\bigl[\one_{\mathcal G}
 \|\mu_n-\widehat\mu_{m,n}\|_{\cH_{n,N}}\bigr]
 \le\frac{C}{J}+\frac{C(2\varepsilon_\rho)^J}{\sqrt J}
 \le \frac{C}{J}.
\end{equation}
Recalling~\eqref{eq:future-L1-PF}, this proves~\eqref{target:PF}, using again that \(J\asymp |\log \delta|\).

We next bound the second moment of $Z_{n,N}(\nu)$ as claimed in \eqref{eq:profile-good-energy}
for \(\nu = \mu_n, \widehat \mu_{m,n}\).
Recall~\eqref{eq:profile-interval-expansion}, which provides an expansion of \(\mu_n =\mu_{N_{0},N_b}^{(q_{N_{0}})}\) and \(\widehat \mu_{m,n} = \mu_{N_{a},N_b}^{(q_{N_{a}})}\); let us denote \(o=1\) if \(\nu = \mu_n\) and \(o=a+1\) if \(\nu = \widehat \mu_{m,n}\).
Then, using the inequality \((x+y)^2\leq 2 x^2+2y^2\), we get that
\[
\bbE\bigl[ Z_{n,N}(\nu)^2 \mid \cF_n\bigr] \leq  2 \bbE\bigl[ Z_{n,N}(q_{N_b})^2 \mid \cF_n\bigr] + \frac{2}{(\mathscr R_{o,b})^2} \bbE\bigl[ Z_{n,N}(\Delta\mu^{\mathrm{weighted}}_{o,b})^2 \mid \cF_n\bigr],
\]
where we have defined the signed measure
\begin{equation*}
\Delta\mu^{\mathrm{weighted}}_{r,s}
 \coloneqq \sum_{u=r}^s
       \mathscr R_{r,u-1}\mathscr D^{\mathrm{open}}_{u,s}\,.
\end{equation*}

Since \(\mathscr R_{o,b} \geq 1/2\) on the event \(\mathcal G \in \cF_n\), we thus get that
\begin{equation*}
      \bbE\bigl[\one_{\mathcal G} Z_{n,N}(\nu)^2 \bigr]\leq 2 \bbE\bigl[ Z_{n,N}(q_{N_b})^2\bigr] + 8 \bbE\bigl[ \one_{\Omega_{o,b}}  \| \Delta\mu^{\mathrm{weighted}}_{o,b}\|_{\mathcal H_{n,N}}^2\bigr],
\end{equation*}
where we have used the definition~\eqref{eq:definition-semi-norm} of the semi-norm \(\|\cdot\|_{\mathcal H_{n,N}}\); note that \(\Delta\mu^{\mathrm{weighted}}_{o,b}\) is a signed measure of zero total mass and is \(\cF_n\)-measurable.

Now, we can apply \eqref{eq:rough-bound} to bound
\[
\bbE\bigl[ Z_{n,N}(q_{N_b})^2\bigr] =
\bbE\bigl[ \widehat{Z}_{n,N}(0)^2 \bigr]
\le
1 + C \,\frac{\log(N/n)}{\log 2} \leq C\, (K-b) \,.
\]
For the remaining term, we may use the bound 
\begin{equation}
      \label{eq:bound-Delta-weighted}
      \bbE\bigl[ \one_{\Omega_{o,b}} \| \Delta\mu^{\mathrm{weighted}}_{o,b}\|_{\mathcal H_{n,N}}^2\bigr] \leq \frac{C}{K-b} \,,
\end{equation}
whose proof is postponed to the proof of Lemma~\ref{lem:profile-estimates} below.
Recalling that \(K-b=J\), we thus have obtained that, for \(\nu = \mu_n$ and $\widehat \mu_{m,n}\),
\[
\bbE\bigl[\one_{\mathcal G} Z_{n,N}(\nu)^2 \bigr] \leq C J + C/J \,,
\]
which concludes the proof of~\eqref{eq:profile-good-energy}, recalling that \(J \asymp |\log \delta|\).

\smallskip

It remains to prove \eqref{eq-weighted-PF-comparison} for a bounded
deterministic terminal function \(f\). Writing
\(\Delta\mu=\mu_n-\widehat\mu_{m,n}\),
we will control the fluctuations around the conditional mean
using the bound \eqref{eq:profile-comparison-seminorm} on
\(\|\Delta\mu\|_{\cH_{n,N}}\), defined in
\eqref{eq:definition-semi-norm}.
This need not vanish for a general \(f\),
so we will need to compare the endpoint distributions
obtained by evolving \(\mu_n\) and \(\widehat\mu_{m,n}\)
from time \(n\) to time \(N\) without disorder.

We first compute the conditional mean and variance of
\(Z_{n,N}(\Delta\mu,f)\).
Averaging over the
disorder on \((n,N]\) gives the conditional mean,
while the first-collision
calculation in \eqref{eq:local-replica-formula} gives
the conditional variance:
\begin{equation}\label{eq-weighted-conditional-moments}
\begin{aligned}
 \bbE[Z_{n,N}(\Delta\mu,f)\mid\cF_n]
 &=\sum_y(\Delta\mu*q_{N-n})(y)f(y),\\
 \Var(Z_{n,N}(\Delta\mu,f)\mid\cF_n)
 &=\sigma_N^2\sum_{r=1}^{N-n}\sum_w
   (\Delta\mu*q_r)(w)^2\,\bbE[Z_{n+r,N}(w,f)^2].
\end{aligned}
\end{equation}
For \(f=\one\), the conditional mean vanishes because
\(\Delta\mu(\one)=0\), so the conditional variance equals
\(\|\Delta\mu\|_{\cH_{n,N}}^2\) by
\eqref{eq:definition-semi-norm}.
Since
\(|Z_{n+r,N}(w,f)|\le\|f\|_\infty Z_{n+r,N}(w,\one)\),
we have
\[
 \bbE[Z_{n+r,N}(w,f)^2]
 \le \|f\|_\infty^2\bbE[Z_{n+r,N}(w,\one)^2] \,,
\]
which plugged into the variance formula in
\eqref{eq-weighted-conditional-moments} gives
\[
 \Var(Z_{n,N}(\Delta\mu,f)\mid\cF_n)
 \le \|f\|_\infty^2
      \Var(Z_{n,N}(\Delta\mu,\one)\mid\cF_n)
 =\|f\|_\infty^2\|\Delta\mu\|_{\cH_{n,N}}^2.
\]
By the triangle inequality and conditional Cauchy--Schwarz,
\[
 \bbE[|Z_{n,N}(\Delta\mu,f)|\mid\cF_n]
 \le
 \big|\bbE[Z_{n,N}(\Delta\mu,f)\mid\cF_n]\big|
 +\Var(Z_{n,N}(\Delta\mu,f)\mid\cF_n)^{1/2} \,.
\]
Using the mean formula above, multiplying by
\(\one_{\mathcal G}\), and applying
\eqref{eq:profile-comparison-seminorm}, we obtain
\begin{equation}\label{eq:weighted-future-L1-PF}
 \bbE[\one_{\mathcal G}|Z_{n,N}(\Delta\mu,f)|]
 \le \|f\|_\infty\left\{
       \frac{C}{|\log\delta|}
       +\bbE[\one_{\mathcal G}
              \|\Delta\mu*q_{N-n}\|_{\ell^1}]
       \right\}.
\end{equation}

We now estimate the last term.
Fix \(R\ge1\) and let
\(B=\{y\in\Z^2:|y|\le R\sqrt N\}\).
Recall from \eqref{eq:hatmu}, with \(z=0\), that
\(\mu_n\) is the endpoint distribution at time \(n\),
while \(\widehat\mu_{m,n}\) is the endpoint distribution
obtained by switching off the disorder on \((0,m]\).
Both are probability measures, so their convolutions with
\(q_{N-n}\) are nonnegative and have total mass one.

Since \(\mu_n*q_{N-n}\) and
\(\widehat\mu_{m,n}*q_{N-n}\) both have total mass one,
their difference \(\Delta\mu*q_{N-n}\) has total mass zero.
Using the triangle inequality on \(B^c\), and then expressing
the mass of this difference on \(B^c\) as minus its mass on \(B\),
we obtain
\[
\begin{aligned}
 \sum_{y\in B^c}|(\Delta\mu*q_{N-n})(y)|
 &\le (\mu_n*q_{N-n})(B^c)
      +(\widehat\mu_{m,n}*q_{N-n})(B^c)\\
 &=2(\widehat\mu_{m,n}*q_{N-n})(B^c)
   -\sum_{y\in B}(\Delta\mu*q_{N-n})(y)\\
 &\le 2(\widehat\mu_{m,n}*q_{N-n})(B^c)
      +\sum_{y\in B}|(\Delta\mu*q_{N-n})(y)|.
\end{aligned}
\]
Adding the sum over \(B\) gives
\begin{equation}\label{eq-smoothed-profile-L1}
 \|\Delta\mu*q_{N-n}\|_{\ell^1}
 \le
 2\sum_{y\in B}|(\Delta\mu*q_{N-n})(y)|
 +2(\widehat\mu_{m,n}*q_{N-n})(B^c).
\end{equation}

Consider first the sum over \(B\).
Since \(|B|\le CR^2N\), Cauchy--Schwarz gives
\begin{equation}\label{612}
 \sum_{y\in B}|(\Delta\mu*q_{N-n})(y)|
 \le CR\sqrt N\|\Delta\mu*q_{N-n}\|_{\ell^2}.
\end{equation}
To bound the right hand side, we use the variance identity
\eqref{eq-weighted-conditional-moments} with \(f=\one\):
\[
 \|\Delta\mu\|_{\cH_{n,N}}^2
 =
 \sum_{r=1}^{N-n}
 \sigma_N^2\bbE[Z_{N-n-r}(0)^2]\,
 \|\Delta\mu*q_r\|_{\ell^2}^2.
\]
For a fixed sufficiently small \(\varepsilon_0\in(0,1/4)\),
Claim~\ref{claim:VN} gives
\(\sigma_N^2\bbE[Z_{[\varepsilon_0N]}(0)^2]\ge c>0\)
for all sufficiently large \(N\).
For small \(\delta\), we have \(n\le N/4\).
Thus, whenever \(N/4\le r\le N/2\), monotonicity of the second
moment gives
\(\sigma_N^2\bbE[Z_{N-n-r}(0)^2]\ge c\).
Also,
\(\|\Delta\mu*q_r\|_{\ell^2}
\ge\|\Delta\mu*q_{N-n}\|_{\ell^2}\),
because convolution by \(q_{N-n-r}\) is a contraction in
\(\ell^2\).
Restricting the sum to these \(r\), of which there are at least
\(N/4\), yields
\begin{equation}\label{eq-smoothed-profile-L2}
 \|\Delta\mu\|_{\cH_{n,N}}^2
 \ge cN\|\Delta\mu*q_{N-n}\|_{\ell^2}^2.
\end{equation}
Combining \eqref{612} with
\eqref{eq-smoothed-profile-L2} and then applying
\eqref{eq:profile-comparison-seminorm}, we obtain
\[
 \bbE\Bigl[
   \one_{\mathcal G}
   \sum_{y\in B}|(\Delta\mu*q_{N-n})(y)|
 \Bigr]
 \le CR\,\bbE[
   \one_{\mathcal G}\|\Delta\mu\|_{\cH_{n,N}}
 ]
 \le \frac{CR}{|\log\delta|}.
\]

We next bound the expectation of the tail
\((\widehat\mu_{m,n}*q_{N-n})(B^c)\)
in \eqref{eq-smoothed-profile-L1}.
Introduce the abbreviations
\[
 Z_1=\widehat Z_{m,n}(0),\qquad
 Z_2=Z_{m,n}(q_m,q_{N-n}*\one_{B^c}).
\]
By definition and by averaging over the disorder,
\[
 (\widehat\mu_{m,n}*q_{N-n})(B^c)=\frac{Z_2}{Z_1},
 \qquad
 0\le Z_2\le Z_1,
 \qquad
 \bbE [Z_2] =q_N(B^c).
\]
By \eqref{eq:one-point-negative} and
\eqref{eq:local-replica-sum} with \((s,n)=(m,n)\),
we have
\(\bbE [Z_1^{-1}]\le\bbE [Z_1^2] \le C\),
uniformly for small \(\delta\) and sufficiently large \(N\).
Applying Cauchy--Schwarz to
\(Z_2/Z_1=\sqrt{Z_2}\,(\sqrt{Z_2}/Z_1)\),
then using \(Z_2\le Z_1\) and the Gaussian tail bound for \(q_N\),
gives
\[
\begin{aligned}
 \bbE[(\widehat\mu_{m,n}*q_{N-n})(B^c)]
 &=\bbE\left[\frac{Z_2}{Z_1}\right]
 \le
 \left(
   \bbE [Z_2]\,\bbE\left[\frac{Z_2}{Z_1^2}\right]
 \right)^{1/2}
 \le
 (\bbE [Z_2]\,\bbE [Z_1^{-1}])^{1/2}
 \le C\e^{-cR^2}.
\end{aligned}
\]

Substituting the estimates on \(B\) and \(B^c\) into
\eqref{eq-smoothed-profile-L1}, after multiplying by
\(\one_{\mathcal G}\) and taking expectations, we obtain
\[
 \bbE[\one_{\mathcal G}\|\Delta\mu*q_{N-n}\|_{\ell^1}]
 \le C\left(
   \frac{R}{|\log\delta|}+\e^{-cR^2}
 \right).
\]
Taking \(R = A  \sqrt{\log|\log\delta|}\)
for large~$A > 0$ and substituting into
\eqref{eq:weighted-future-L1-PF} proves
\eqref{eq-weighted-PF-comparison}.
\end{proof}

\begin{proof}[Proof of Lemma~\ref{lem:profile-estimates}]
For the first estimate~\eqref{eq:profile-good-prob},
we rely on Proposition~\ref{prop:block-inputs-ii}.
In particular, \eqref{eq:app-good-P} shows that uniformly for \(1\le r\le s\le K-1\) we have
\(\bbP(\Omega_{r,s}^c) \leq C /(K-s)\) and also
$$\bbP( |\mathscr R_{r,s}-1|>1/2)\leq \bbP(\Omega_{r,s}^c)+ 4 \bbE[\one_{\Omega_{r,s}} |\mathscr R_{r,s}-1|^2 ] \leq C/(K-s),$$
by Markov's inequality.
Recalling the definition~\eqref{eq:profile-good-event} of \(\mathcal G_{r,s}\), this proves~\eqref{eq:profile-good-prob}.

Before we prove \eqref{eq:profile-second-order}, let us show
the following auxiliary bound (we also prove~\eqref{eq:bound-Delta-weighted} at the same time):
\begin{equation}\label{eq:profile-VS}
 \| \one_{\Omega_{r,s}} \Delta\mu^{\mathrm{open}}_{r,s}\|_{L^2(\bbP,\mathcal H_{N_s,N})} \coloneqq \bbE\bigl[ \one_{\Omega_{r,s}} \|\Delta\mu^{\mathrm{open}}_{r,s}\|_{\cH_{N_s,N}}^2 \bigr]^{1/2}
 \le \frac{C}{\sqrt{K-s}} \,.
\end{equation}
Recalling the definition~\eqref{eq:common-correction} of \(\Delta\mu^{\mathrm{open}}_{r,s}\), we get by Minkowski's inequality that
\[
 \|\one_{\Omega_{r,s}} \Delta\mu^{\mathrm{open}}_{r,s} \|_{L^2(\bbP,\mathcal H_{N_s,N})}
 \le C\sum_{u=r}^s \| \one_{\Omega_{u,s}} \mathscr D^{\mathrm{open}}_{u,s}\|_{L^2(\bbP,\mathcal H_{N_s,N})} \leq C\sum_{u=r}^s \frac{(2\varepsilon_\rho)^{s-u}}{\sqrt{K-u}} \,,
\]
where the last inequality follows from \eqref{eq:open-activity-bound}.
This proves~\eqref{eq:profile-VS}.
In order to prove~\eqref{eq:bound-Delta-weighted},
\textit{i.e.}\ the same result with \(\Delta\mu^{\mathrm{weighted}}_{r,s}\), note that
\[
\one_{\Omega_{r,s}} \Delta\mu^{\mathrm{weighted}}_{r,s} = \sum_{u=r}^s \one_{\Omega_{r,u-1}}\mathscr R_{r,u-1}  \one_{\Omega_{u,s}}\mathscr D^{\mathrm{open}}_{u,s} \,,
\]
with
$\one_{\Omega_{r,u-1}}\mathscr R_{r,u-1}$ and
$\one_{\Omega_{u,s}}\mathscr D^{\mathrm{open}}_{u,s}$ depending on
disjoint blocks.
By independence, we thus have that
\[
\bigl\|\one_{\Omega_{r,s}} \Delta\mu^{\mathrm{weighted}}_{r,s} \bigr\|_{L^2(\bbP,\mathcal H_{N_s,N})} \leq C\sum_{u=r}^s \|\one_{\Omega_{r,u-1}}\mathscr R_{r,u-1} \|_{L^2} \bigl\| \one_{\Omega_{u,s}} \mathscr D^{\mathrm{open}}_{u,s}\bigr\|_{L^2(\bbP,\mathcal H_{N_s,N})} \,.
\]
Now, we have that \(\|\one_{\Omega_{r,u-1}}\mathscr R_{r,u-1} \|_{L^2}\leq C\) by Proposition~\ref{prop:block-inputs-ii}, see~\eqref{eq:app-good-P}, so the same geometric sum as above shows that
\[
\| \one_{\Omega_{r,s}} \Delta\mu^{\mathrm{weighted}}_{r,s}\|_{L^2(\bbP,\mathcal H_{N_s,N})} = \bbE\bigl[ \one_{\Omega_{r,s}} \|\Delta\mu^{\mathrm{weighted}}_{r,s}\|_{\cH_{N_s,N}}^2 \bigr]^{1/2} \leq \frac{C}{\sqrt{K-s}} \,,
\]
which is exactly \eqref{eq:bound-Delta-weighted}.

Going back to the proof of~\eqref{eq:profile-second-order}, we use the expansion \eqref{eq:profile-interval-expansion} to write the
exact identity
\begin{equation}
 \mu_{N_{r-1},N_s}^{(q_{N_{r-1}})}
       -q_{N_s}-\Delta\mu^{\mathrm{open}}_{r,s}
 =
 \frac1{\mathscr R_{r,s}}
 \sum_{u=r}^s(\mathscr R_{r,u-1}-1)
                  \mathscr D^{\mathrm{open}}_{u,s}
 +\frac{1-\mathscr R_{r,s}}{\mathscr R_{r,s}}
                  \Delta\mu^{\mathrm{open}}_{r,s}.
 \label{eq:profile-error-identity}
\end{equation}
Now, on $\mathcal G_{r,s}$, the denominators are at least $1/2$.
We thus have that
\[
\begin{split}
&\bbE \bigl[\one_{\mathcal G_{r,s}}
\|\mu_{N_{r-1},N_s}^{(q_{N_{r-1}})}-q_{N_s}-\Delta\mu^{\mathrm{open}}_{r,s}\|_{\mathcal H_{N_s,N}}\bigr] \\
& \leq 2 \sum_{u=r}^{s} \bbE\bigl[ \one_{\Omega_{r,s}}|\mathscr R_{r,u-1}-1| \, \| \mathscr D^{\mathrm{open}}_{u,s}\|_{\mathcal H_{N_s,N}}\bigr]
  + 2 \bbE\bigl[ \one_{\Omega_{r,s}}|\mathscr R_{r,s}-1| \cdot \| \Delta\mu^{\mathrm{open}}_{r,s}\|_{\mathcal H_{N_s,N}}\bigr].
\end{split}
\]
For $u>r$, using the independence of the blocks, we get
\[
 \bbE\bigl[\one_{\Omega_{r,u-1}}
 |\mathscr R_{r,u-1}-1|
 \one_{\Omega_{u,s}} \|\mathscr D^{\mathrm{open}}_{u,s}\|_{\cH_{N_s,N}}\bigr]
 \le
 \frac{C}{\sqrt{K-u}} \, \frac{(2\varepsilon_\rho)^{s-u}}
 {\sqrt{K-u}}\,,
\]
where we have also used that
\(\bbE[ \one_{\Omega_{r,u-1}}|\mathscr R_{r,u-1}-1| ] \leq C /(K-u+1)^{1/2}\)
thanks to Proposition~\ref{prop:block-inputs-ii}, see~\eqref{eq:app-good-P} and Cauchy--Schwarz's inequality,
together with~\eqref{eq:open-activity-bound}.
Note that the term with $u=r$ vanishes because
$\mathscr R_{r,r-1}=1$.
On the other hand, using Cauchy--Schwarz, we also have that
\[
\bbE\bigl[ \one_{\Omega_{r,s}}|\mathscr R_{r,s}-1| \one_{\Omega_{r,s}} \| \Delta\mu^{\mathrm{open}}_{r,s}\|_{\mathcal H_{N_s,N}}\bigr]
\leq
\frac{C}{\sqrt{K-s}} \times \frac{C}{\sqrt{K-s}} \,,
\]
where we have used Proposition~\ref{prop:block-inputs-ii}, see \eqref{eq:app-good-P}, and \eqref{eq:profile-VS}.
We thus get that
\begin{equation*}
      \bbE \bigl[\one_{\mathcal G_{r,s}}
\|\mu_{N_{r-1},N_s}^{(q_{N_{r-1}})}-q_{N_s}-\Delta\mu^{\mathrm{open}}_{r,s}\|_{\mathcal H_{N_s,N}}\bigr]
\leq \sum_{u=r}^{s} \frac{C(2\varepsilon_\rho)^{s-u}}
 {K-u} + \frac{C}{K-s}\,,
\end{equation*}
which concludes the proof of \eqref{eq:profile-second-order}.
\end{proof}

\subsection{Logarithmic moment bounds: proof of Proposition~\ref{prop:log-tails}}\label{sec:NM}

Recall the two ratios
\[
 Z_{n,N}(\mu_n^{(x)})=\frac{Z_N(x)}{Z_n(x)},
 \qquad
 Z_{n,N}(\widehat\mu_{m,n}^{(x)})
 =\frac{\widehat Z_{m,N}(x)}{\widehat Z_{m,n}(x)} \,.
\]
Our goal is to control the \(L^p\) norm of their logarithm.
(Again, by stationarity, we may only treat the case \(x=0\) and we drop it from the notation.)
For this, we rely on the following log-tail estimate,
which we prove below.
Recall that \(n= N_b\) with \(K-b=J\).

\begin{lemma}
      \label{lem:log-tail}
      There is a constant \(C>0\) such that, for \(\nu = \mu_n\) or \(\nu = \widehat\mu_{m,n}\),
      \[
      \bbP\bigl( |\log Z_{n,N}(\nu)| > u\bigr) \leq  \min\{1, C \sqrt{J} \e^{-u/3}\}  \qquad \text{ for all } u\geq 0\,.
      \]
\end{lemma}

Let us first conclude the proof of Proposition~\ref{prop:log-tails} thanks to this lemma.

\begin{proof}[Proof of Proposition~\ref{prop:log-tails}]
We start with the integral formula for the \(p\)-th moment
\[
 \bbE\bigl[|\log Z_{n,N}(\nu)|^p\bigr]
 =p\int_0^\infty u^{p-1}
       \bbP\bigl(|\log Z_{n,N}(\nu)|>u\bigr)\,\dd u \,.
\]
We now use Lemma~\ref{lem:log-tail}, and we split the integral into a part with \(u\leq C \log J\) where we bound the probability by \(1\) and a part with $u \geq C\log J$ where we use the bound \(C \sqrt{J}\e^{-u/3}\).
This gives
\[
\bbE\bigl[|\log Z_{n,N}(\nu)|^p\bigr]
 \le C_p (\log J)^p + C_p \sqrt{J} (\log J)^{p-1}\e^{- C \log J /3} \leq C_p' (\log J)^p \,,
\]
where the second inequality is valid when \(C\) is large enough, 
and thus \(\|\log Z_{n,N}(\nu)\|_{L^p} \leq C_p \log J\).
Since $J\asymp|\log\delta|$, this proves
\eqref{eq:future-log-moments}.

For a positive terminal function \(f\), the bounds
\(\e^{-\|\log f\|_\infty}\le f\le\e^{\|\log f\|_\infty}\) give
\[
 |\log Z_{n,N}(\nu,f)-\log Z_{n,N}(\nu)|
 \le \|\log f\|_\infty.
\]
The claimed extension follows by the triangle inequality in \(L^p\).
\end{proof}

\begin{proof}[Proof of Lemma~\ref{lem:log-tail}]
First of all, notice that $\bbE[Z_{n,N}(\nu)\mid \cF_n]=1$. Hence, by taking the exponential and using the conditional mean, we get that for any $u\ge0$,
\[
 \bbP\bigl(\log Z_{n,N}(\nu)>u\bigr)\le\e^{-u} \,,
\]
and we thus need to focus on the left tail of \(\log Z_{n,N}(\nu)\).
We do this by estimating their negative moments, for which the Gaussian convexity inequalities are central.

\smallskip\noindent\emph{The case \(\nu = \widehat \mu_{m,n}\).}
Writing \(Z_{n,N}(\widehat\mu_{m,n})=\widehat Z_{m,N}(0)/\widehat Z_{m,n}(0)\) and applying the Cauchy--Schwarz's inequality, we get
\[
 \begin{split}
      \bbE\bigl[ Z_{n,N}(\widehat\mu_{m,n})^{-1/2}\bigr]
      & = \bbE\bigl[ \widehat Z_{m,N}(0)^{-1/2}\widehat Z_{m,n}(0)^{1/2}\bigr] \\
      & \le
      \bbE[\widehat Z_{m,n}(0)]^{1/2}
      \bbE[\widehat Z_{m,N}(0)^{-1}]^{1/2}
      \leq \bbE[\widehat Z_{m,N}(0)^{2}]^{1/2},
 \end{split}
\]
where we have used \eqref{eq:one-point-negative} with \(p=1\)
for the last inequality.

Now, recalling that \(m=N_a\) with \(a=K-2J\), 
a rough second moment computation, see e.g.\ \eqref{eq:rough-bound}, gives the following
\[
\bbE[\widehat Z_{m,N}(0)^{2}] \leq C \log(N/m) \leq C J \,.
\]
We therefore end up with \(\bbE[Z_{n,N}(\widehat\mu_{m,n})^{-1/2}] \le C\sqrt{J}\) so that Markov's inequality yields,
for $u\ge0$,
\begin{equation}\label{eq:regularised-log-tail}
 \bbP\bigl(Z_{n,N}(\widehat\mu_{m,n})< \e^{-u}\bigr)
 \le C\sqrt{J}\,\e^{-u/2} ,
\end{equation}
which is what is claimed in Lemma~\ref{lem:log-tail}.

\smallskip\noindent\emph{The case \(\nu = \mu_{n}\).}
Applying directly the same argument as above shows that
\[
\bbE\bigl[ Z_{n,N}(\mu_{n})^{-1/2}\bigr] \leq \bbE \bigl[Z_N(0)^2\bigr]^{1/2} \leq C \sqrt{\log N} \,,
\]
where the last inequality follows from
\eqref{eq:uniform-two-replica} and
$\sigma_N^2\asymp1/\log N$.
Using Markov's inequality, this gives the bound
\begin{equation*}
      \bbP\bigl(Z_{n,N}(\mu_{n})< \e^{-u}\bigr)
 \le C\sqrt{\log N} \,\e^{-u/2} \,.
\end{equation*}
This is not a satisfying bound, except when \(u\) is large.
In fact, we will use this bound when \(\lceil J\e^{u/3}\rceil \geq K \geq c \log N \),
in which case we obtain
\begin{equation}
      \label{eq:log-tail-rough}
      \bbP\bigl(Z_{n,N}(\mu_{n})< \e^{-u}\bigr)
 \le C \sqrt{J} \,\e^{-u/3} \,.
\end{equation}

Let us now assume that \(u\geq 0\) is such that \(\lceil J\e^{u/3}\rceil \leq K-1\) and let us introduce the intermediate time \(N_1\leq M = N_{s} \leq n =N_b\) with \(s \coloneqq K- \lceil J\e^{u/3}\rceil \in \{1,\ldots, b\}\).
The reason for this specific choice of \(s\) will be clear below, see e.g.\ \eqref{eq:reason-choice-s}.
Then, for any event \(\mathcal G_M \in \cF_M\), applying Cauchy--Schwarz's inequality and
\eqref{eq:one-point-negative} with $p=1$,
conditionally on $\cF_M$, 
we get analogously as above that
\begin{equation}\label{eq: -1/2 moment and 2nd moment}
 \begin{split}
      \bbE\bigl[\one_{\mathcal G_M} Z_{n,N}(\mu_{n})^{-1/2}\bigr]
      = \bbE\biggl[\one_{\mathcal G_M} \Bigl(\frac{Z_n(0)}{Z_{M}(0)}\Bigr)^{1/2} \Bigl(\frac{Z_N(0)}{Z_{M}(0)}\Bigr)^{-1/2} \biggr]  \leq \bbE\bigl[  \one_{\mathcal G_M} \,  Z_{M,N}(\mu_M)^2 \bigr]^{1/2} \,,
 \end{split}
\end{equation}
where we have used that \(\bbE[ Z_n(0)\mid \cF_M] = Z_M(0)\)  and noticed that
\[
\bbE\biggl[\biggl(\frac{Z_N(0)}{Z_{M}(0)}\biggr)^2
\,\bigg|\,\cF_{M}\biggr] = \bbE\bigl[Z_{M,N}(\mu_M)^2 \mid \cF_M\bigr].
\]

Take $\mathcal G=\mathcal G_{1,s}$ as defined in
\eqref{eq:profile-good-event}, so that $\mathcal G\in\cF_M$.
Then
\begin{equation*}
 \bbP(\mathcal G^c)\le \frac{C}{K-s},
 \qquad
 \bbE\bigl[\one_{\mathcal G}Z_{M,N}(\mu_M)^2\bigr]
 \le C(K-s).
\end{equation*}
The first bound follows from \eqref{eq:profile-good-prob}, while
the second can be checked 
repeating the argument which proves
\eqref{eq:profile-good-energy} at time $M=N_s$
(note that $K-s\ge J$ is sufficiently large
for small $\delta$).
We therefore get that, applying Markov's inequality with \eqref{eq: -1/2 moment and 2nd moment} 
\begin{equation}
\label{eq:reason-choice-s}
\begin{split}
      \bbP\bigl(Z_{n,N}(\mu_n)< \e^{-u}\bigr)
      & \leq \bbP(\mathcal G^c)
      +  \bbE\bigl[\one_{\mathcal G}Z_{M,N}(\mu_M)^2\bigr]^{1/2}\e^{-u/2} \\
      & \leq \frac{C}{K-s} + C \sqrt{K-s}\,\e^{-u/2} \,.
\end{split}
\end{equation}
The reason for the choice
\(K-s \coloneqq \lceil J\e^{u/3}\rceil\)
is made so that both terms on the right-hand side
of~\eqref{eq:reason-choice-s} bounded by \(O(\sqrt{J}\e^{-u/3})\).
We end up with the following bound when \(\lceil J\e^{u/3}\rceil < K\):
\begin{equation}\label{eq:fine-log-tail}
 \bbP\bigl(Z_{n,N}(\mu_n)<\e^{-u}\bigr)
 \le C\sqrt J\,\e^{-u/3}.
\end{equation}
Combining with \eqref{eq:log-tail-rough} for the case \(\lceil J\e^{u/3}\rceil \geq K\), this concludes the proof.
\end{proof}

\section{Small-scale fluctuations and covariance structure}
\label{prop:sec-local-properties}

In this section, we prove Theorem~\ref{thm:properties}.
We recall that we only work here with flat initial condition \(h_0=\zero\).
By Proposition~\ref{prop:invariance},
\[
 \scrH_t^\vartheta(\zero,\varphi_\varepsilon^{(x)})
 \stackrel{\rm dist}{=}
 \scrH_1^{\vartheta+\log t}(\zero,
       \varphi_{\varepsilon/\sqrt t}^{(x/\sqrt t)}),
\]
so it suffices to consider only \(t=1\). 
We use the time-reversed fields \(\scrH^{(\delta)}\) and \(\scrH\) from
\eqref{eq:regularised-Hbis} and \eqref{eq:toy}, omitting the superscript \(\vartheta\) as in Section~\ref{subsec:strategy-approximation}.
The law of \(\scrH\) is that of \(\scrH_1^\vartheta(\zero,\cdot)\).

\subsection{Small-scale Gaussianity}

We prove here item~\eqref{i:small-scale} in Theorem~\ref{thm:properties}.
The main idea is to derive the result from Gu and Tsai's log-normality for the SHF \cite{GT26}.
Recall that we defined $F_\delta(x)=\log\scrZ_{\delta,1}^\vartheta(g_\delta^{(x)},\one)$
in~\eqref{eq:F-delta}, so that 
\[
\scrH^{(\delta)}(\varphi) = \int_{\R^2} \varphi(x) \{F_{\delta}(x)-\bbE[F_{\delta}(x)]\} \dd x \,.
\]
We then have the following log-normality, which differs from \cite[Theorem~1.3]{GT26} only by the centring term.

\begin{theorem}
      \label{thm:log-normality}
      For any fixed \(\vartheta\in\R\), for any \(x\in \R^2\) we have that
      \begin{equation}\label{eqY: local CLT}
            \frac{F_{\gep^2}(x)-\bbE[F_{\gep^2}(x)]}{\sqrt{\log \log (1/\gep)}} \xrightarrow[\ \gep \to 0 \ ]{d} \mathcal{N}(0,1)\,.
      \end{equation}
\end{theorem}

\begin{proof}
Theorem~1.3 of \cite{GT26} gives a deterministic centring $b_\varepsilon\sim-\frac12\log\log(1/\varepsilon)$ such that
\begin{equation}\label{eqY: Y epsilon, CLT}
Y_\varepsilon \coloneqq  \frac{F_{\varepsilon^2}(x)-b_\varepsilon}
{\sqrt{\log\log(1/\varepsilon)}}
\xrightarrow[\;\varepsilon\downarrow0\;]{d}\mathcal N(0,1)\,.
\end{equation}

We now claim that the second moment of $Y_\varepsilon$ is uniformly bounded:
\begin{equation}\label{eqY: second moment, claim}
\sup_{\varepsilon \in (0,1/10)} \mathbb E
[ |Y_\varepsilon| ^2 ] <\infty.
\end{equation}
Once \eqref{eqY: second moment, claim} is proven, from the weak convergence \eqref{eqY: Y epsilon, CLT}, we get $\lim_{\varepsilon \downarrow 0} \mathbb E [Y_\varepsilon] =0$ and thus  
\[
Y_\varepsilon - \mathbb E [Y_\varepsilon ] \xrightarrow[\varepsilon\downarrow0]{d}\mathcal N(0,1)
\] 
which is exactly \eqref{eqY: local CLT}. 

In what follows, we will prove \eqref{eqY: second moment, claim}. By~\eqref{eq:asymp-ElogZ-VarlogZ}, we have 
\begin{equation}\label{eqY: var}
	\Var Y_\varepsilon = \frac{\Var F_{\varepsilon^2} (x) }{\log \log \frac{1}{\varepsilon} } \le \frac{\log \log \frac{1}{\varepsilon^2} +C_\vartheta }{\log \log \frac{1}{\varepsilon} }\le C <\infty,\quad \forall \varepsilon \in (0,1/10)
\end{equation}
for some finite $C>0$. 
To bound $\mathbb E [|Y_\varepsilon|^2] = \Var Y_\varepsilon +  | \mathbb E [Y_\varepsilon] |^2 $, it suffices to bound $|\mathbb E [Y_\varepsilon] |$ uniformly for small $\varepsilon $. 
From \eqref{eqY: var}, $\{ Y_\epsilon -\mathbb E [Y_\varepsilon ]: \varepsilon \in(0,1/10) \}$ is tight. Furthermore, from the weak convergence \eqref{eqY: Y epsilon, CLT}, $\{ Y_\epsilon : \varepsilon \in(0,1/10) \}$ is also tight. 
It follows that $\mathbb E [Y_\varepsilon ]$ is uniformly bounded for $ \varepsilon \in(0,1/10)$, completing the proof.
\end{proof}

In the rest of this section, we work with a nonnegative \(\varphi\in C_c^\infty(\R^2)\).
Assume for simplicity that \(\int \varphi =1\).
By stationarity, we may also only consider the case \(x=0\) and recall that we defined
\(\varphi_\varepsilon(z) \coloneqq \varphi_\varepsilon^{(0)}(z)
=\varepsilon^{-2}\varphi(z/\varepsilon)\).
Finally, by translation invariance of SHF, we may assume that
the barycentre of~\(\varphi\) is \(\int y \varphi(y) \dd y =0\), that is, 
\begin{equation}\label{eqY: barycentre}
\int_{\R^2} y_i \varphi(y) \dd y =0,\quad y=(y_1,y_2),\ i=1,2.
\end{equation}

\begin{proof}[Proof of item~\eqref{i:small-scale} in Theorem~\ref{thm:properties}]
Our main task is to prove that there exists a constant \(C_\varphi\) such that,
uniformly for \(\gep>0\) small enough,
\begin{equation}
      \label{prop:eq-point-matched-comparison}
 \left\|\scrH(\varphi_\varepsilon)
 -\{F_{\varepsilon^2}(0)-\bbE F_{\varepsilon^2}(0)\}\right\|_{L^1}
 \le C_\varphi.
\end{equation}
Together with Theorem~\ref{thm:log-normality}, this concludes the proof of item~\eqref{i:small-scale} in Theorem~\ref{thm:properties}.

To prove~\eqref{prop:eq-point-matched-comparison}, we use the triangle inequality to get that the left-hand side is bounded by
\begin{equation}
\label{eq:uniform-control-L1}
     \|\scrH(\varphi_\varepsilon) - \scrH^{(\gep^2)}(\varphi_{\gep})\|_{L^2}  + \Bigl\| \int_{\R^2} \varphi_{\gep}(x) F_{\gep^2}(x) \dd x -  F_{\gep^2}(0) \Bigr\|_{L^1} \,,
\end{equation}
using also that \(\int_{\R^2} \varphi_{\gep}(x) \bbE[F_{\gep^2}(x)] \dd x = \bbE[F_{\gep^2}(0)]\)
so that the centring terms cancel out.

\smallskip
\noindent
\textit{First term in~\eqref{eq:uniform-control-L1}.}
We use the decomposition~\eqref{eq:averaged-decomposition}, to get that
\[
 \limsup_{N\to\infty}\|
 \cH^{(N)}(\varphi_\varepsilon)
 -\cH^{(N;\delta)}
       (\varphi_\varepsilon)\|_{L^2}
\leq \limsup_{N\to\infty} \|\cL^{(N;\delta)}(\varphi_\varepsilon)\|_{L^2}  + \limsup_{N\to\infty} \|\cG^{(N;\delta)}(\varphi_\varepsilon)\|_{L^2} \,.
\]
Now, we can use Propositions~\ref{prop:local} and~\ref{prop:global}, or more precisely the bounds~\eqref{eq:limsup-local-L} and \eqref{eq:limsup-global-G} (see  Remarks~\ref{rem:better-bound-L} and~\ref{rem:better-bound-G}), to get that for any \(\delta\) small enough,
\[
\limsup_{N\to\infty}\|
 \cH^{(N)}(\varphi_\varepsilon)
 -\cH^{(N;\delta)}
       (\varphi_\varepsilon)\|_{L^2}^2  \leq \frac{C \mathcal E_{\sqrt\delta}(\varphi_{\gep})}{ |\log \delta|} + \frac{C' \|\varphi_{\gep}\|_{1}^2}{|\log \delta|^{1/20}} \,,
\]
where \(\mathcal E_{\sqrt\delta}(\varphi_{\gep})\) is defined in~\eqref{eq:def-E-varphi}.
In particular, by the triangle inequality (and using that \((a+b)^2 \leq 2 a^2 +2 b^2\)) we get that for any \(\delta\leq \gep^2\),
\[
\limsup_{N\to\infty}\|
 \cH^{(N;\gep^2)}(\varphi_\varepsilon)
 -\cH^{(N;\delta)}
       (\varphi_\varepsilon)\|_{L^2}^2  \leq \frac{C \mathcal E_{\gep}(\varphi_{\gep})}{|\log \gep|} + \frac{C' \|\varphi_{\gep}\|_{1}^2}{|\log \gep|^{1/20}} \,,
\]
using also the monotonicity of \(\delta \mapsto \mathcal{E}_{\delta}(\varphi)\).
Now, as in Remark~\ref{rem:better-bound-L}, scaling gives
\[
 \mathcal E_{\gep}(\varphi_{\gep})=\mathcal E_{1/\gep}(\varphi)
 \leq C\|\varphi\|_\infty\|\varphi\|_1
       +C|\log\gep|\,\|\varphi\|_1^2.
\]
By the joint convergence in distribution of Proposition~\ref{prop:FC} and Portmanteau, as in~\eqref{eq:L2-Cauchy}, we therefore get that, uniformly for \(\gep\in (0,\gep_0)\) with \(\gep_0\) small enough and \(\delta \leq \gep^2\)
\begin{equation*}
\|\scrH^{(\gep^2)}(\varphi_\varepsilon)
 -\scrH^{(\delta)}(\varphi_\varepsilon)\|_{L^2}^2
 \leq C \|\varphi\|_1 \big(\|\varphi\|_{\infty}+\|\varphi\|_1 \big) \,.
\end{equation*}
Letting \(\delta \downarrow 0\) and using~\eqref{eq:toy}, namely
\(\|\scrH^{(\delta)}(\varphi_\varepsilon) - \scrH(\varphi_\varepsilon)\|_{L^2} \to 0\),
this shows that the first term in~\eqref{eq:uniform-control-L1} satisfies
\begin{equation}
      \label{eq:uniform-control-L2-first}
    \|\scrH^{(\gep^2)}(\varphi_{\gep})- \scrH(\varphi_\varepsilon)\|_{L^2} \leq C \|\varphi\|_1^{1/2} \big(\|\varphi\|_{\infty}+\|\varphi\|_{1}\big)^{1/2}\,,
\end{equation}
uniformly in \(\gep\in (0,\gep_0)\).

\smallskip
\noindent
\textit{Second term in~\eqref{eq:uniform-control-L1}.}
First of all, note that 
\[
F_{\varepsilon^2 }(x) -F_{\varepsilon^2 }(0) = \log \mathscr Z_{\varepsilon^2 ,1}^\vartheta (g_{\varepsilon^2}^{(x)},\one) - \log \mathscr Z_{\varepsilon^2 ,1} ^\vartheta (g_{\varepsilon^2}^{(0)},\one) 
= \log \int_{\R^2} \frac{g_{\varepsilon^2} (x-y)}{g_{\varepsilon^2} (y) } Q_\varepsilon (\dd y),
\]
where \(Q_\varepsilon\) is a probability measure on \(\R^2\), defined by
\[
Q_\varepsilon (\dd y) \coloneqq 
\frac{g_{\varepsilon^2} (y)\mathscr Z_{\varepsilon^2 ,1} ^\vartheta (\dd y, \one) }
	 {\int_{\R^2} g_{\varepsilon^2} (z)\mathscr Z_{\varepsilon^2 ,1} ^\vartheta (\dd z, \one)} \,.
\]
Now, by definition~\eqref{eq:heat} of the heat kernel we can write
\(
\frac{g_{\varepsilon^2} (x-y)}{g_{\varepsilon^2} (y) } 
= \exp(-\frac{|x|^2}{2\varepsilon^2 } +\frac{x \cdot y}{\varepsilon^2}),
\)
so that we end up with 
\[
F_{\varepsilon^2 }(x) -F_{\varepsilon^2 }(0) 
=  -\frac{|x|^2}{2\varepsilon^2 } + \log \int_{\R^2} \e^{ \frac{x \cdot y}{\varepsilon^2} } Q_\varepsilon (\dd y)
\geq -\frac{|x|^2}{2\varepsilon^2 } + x \cdot \int_{\R^2} \frac{y}{\varepsilon^2} Q_\varepsilon (\dd y) \,,
\]
where the last inequality follows by 
Jensen's inequality, since \(Q_{\varepsilon}\) is a probability measure.
As a result, recalling that \(\varphi\) is non-negative,
\begin{equation*}
\begin{aligned}
	\int_{\R^2} \varphi_\varepsilon(x) [F_{\varepsilon^2 }(x) -F_{\varepsilon^2 }(0)] \dd x 
	\ge & -\int_{\R^2} \frac{\varphi_\varepsilon(x) }{ 2\varepsilon^2} |x|^2 \dd x + 
	\int_{\R^2} \varphi_\varepsilon(x) x  \dd x  \cdot \int_{\R^2} \frac{y}{\varepsilon^2} Q_\varepsilon (\dd y)\\
	=& -\frac{1}{2}\int_{\R^2} \varphi(x) |x|^2 \dd x\,,
\end{aligned}
\end{equation*}
where we have used the scaling properties of \(\varphi_\varepsilon\) for the first term and that $\int_{\R^2} \varphi_\varepsilon(x) x  \dd x =0$ for the second term (recall that \(\varphi\) satisfies \eqref{eqY: barycentre}). 
In particular, this shows that 
\begin{equation}
      \label{eqY: int phi F -F}
      \Bigl(\int_{\R^2} \varphi_\varepsilon(x) [F_{\varepsilon^2 }(x) -F_{\varepsilon^2 }(0)] \dd x \Bigr)^- \leq \frac{1}{2}\int_{\R^2} \varphi(x) |x|^2 \dd x \,,
\end{equation}
where \(X^- = \max\{-X,0\}\) is the negative part of \(X\).
Since $\int \varphi_\varepsilon(x) [F_{\varepsilon^2 }(x) -F_{\varepsilon^2 }(0)] \dd x $ has mean zero by the spatial stationarity of $F_{\varepsilon^2}(\cdot)$
and since $\mathbb E[|X|] =2 \mathbb E [X^-]$ for any mean-zero random variable $X$, we get from~\eqref{eqY: int phi F -F} that the second term in~\eqref{eq:uniform-control-L1} verifies
\begin{equation}\label{eqY: second term bound}
	\mathbb E \biggl[ \Bigl|
	\int_{\R^2} \varphi_\varepsilon(x) \left(F_{\varepsilon^2 }(x) -F_{\varepsilon^2 }(0)\right) \dd x
      \Bigr| \biggr] 
	\le 
	\int_{\R^2} \varphi (x) |x|^2 \dd x \,.
\end{equation}

\smallskip
\noindent
In the end, plugging \eqref{eqY: second term bound} and \eqref{eq:uniform-control-L2-first} into \eqref{eq:uniform-control-L1} finishes the proof. 
\end{proof}

We conclude this section by estimating the variance of \(\scrH_t^\vartheta(\varphi_\varepsilon^{(x)})\).

\begin{proposition}[Variance of shrinking averages]
\label{prop:shrinking-variance}
For every fixed nonnegative smooth compactly supported
\(\varphi\) of integral one and satisfying \eqref{eqY: barycentre}, for any fixed \(\vartheta\in\R\),
\begin{equation}\label{prop:eq-shrinking-variance}
 \Var\bigl(\scrH_1^\vartheta(\varphi_\varepsilon)\bigr)
 \sim\log\log(1/\varepsilon) \qquad \text{ as } \gep\downarrow 0 \,.
\end{equation}
\end{proposition}

\begin{proof}
Notice that since \(\varphi\geq 0\) with \(\int\varphi=1\),
 Jensen's inequality yields that
\begin{equation*}
      \bigl(\scrH^{(\gep^2)}(\varphi_{\gep})\bigr)^2  = \Bigl(\int_{\R^2} \varphi_{\gep}(x) \{F_{\gep^2}(x)-\bbE[F_{\gep^2}(x)]\} \dd x \Bigr)^2 \leq \int_{\R^2} \varphi_{\gep}(x) \bigl(F_{\gep^2}(x)-\bbE[F_{\gep^2}(x)]\bigr)^2 \dd x \,,
\end{equation*}
so that in particular \(\bbE[(\scrH^{(\gep^2)}(\varphi_{\gep}))^2] = \Var(\scrH^{(\gep^2)}(\varphi_{\gep})) \leq \Var(F_{\gep^2}(0))\), using also the spatial stationarity of $F_{\varepsilon^2}$ and \(\int\varphi =1\).
Now, using the observation~\eqref{eq:asymp-ElogZ-VarlogZ},
we have that \(\Var(F_{\gep^2}(0)) \leq \log \log (1/\gep) + C_{\vartheta}\). By \eqref{eq:uniform-control-L2-first} {{and the triangle inequality}}, we thus get the upper bound
\[
\Var\bigl(\scrH(\varphi_\varepsilon)\bigr) 
= \mathbb E \bigl[
\scrH(\varphi_\varepsilon) ^2 
\bigr]
\leq  \log \log (1/\gep) + O( (\log \log 1/\gep)^{1/2}) \,.
\]

For the reverse bound, the convergence~\eqref{eq:main-local-clt} and Portmanteau give
\[
\liminf_{\gep \downarrow 0} \frac{\bbE\bigl[(\scrH(\varphi_{\gep}))^2 \bigr]}{\log \log (1/\gep)} \geq \bbE[Z^2]=1 \,,
\]
where \(Z\sim\mathcal N(0,1)\).
This proves the lower bound \(\Var\bigl(\scrH(\varphi_\varepsilon)\bigr) \geq (1+o(1)) \log \log (1/\gep) \) for small $ \varepsilon$, concluding the proof of Proposition~\ref{prop:shrinking-variance}.
\end{proof}

\subsection{No local signed-measure representative}

In this section, we prove
Proposition~\ref{prop:no-local-representative}.
Again, we focus on the case \(t=1\).
Consider a nonnegative \(\varphi\in C_c^{\infty}(B(0,1))\) with \(\int\varphi=1\) and let \(\varphi_{\gep}^{(x)}(y) \coloneqq \gep^{-2} \varphi((y-x)/\gep)\).

First of all, using the
small scale Gaussianity from Theorem~\ref{thm:properties}-\ref{i:small-scale}, see~\eqref{eq:main-local-clt},  for all $x\in \mathbb R^2$ and \(k\in \mathbb N\) we get that
\(\lim_{\gep\downarrow 0} \bbP(|\scrH(\varphi_\varepsilon^{(x)})| \leq k) =0\).
We can thus choose a sequence \(\gep_k \downarrow 0\) such that
\[
\bbP\pigl(\bigl|\scrH(\varphi_{\varepsilon_k}^{(x)})\bigr| \leq k\pigr) \leq 2^{-k} \,, \qquad \forall x\in \mathbb R^2. 
\]
Thus, applying Fubini--Tonelli, we get that for \(B =B(0,1)\),
\[
\bbE\Bigl[\int_B\sum_{k\ge1}
 \one_{\{|\scrH
      (\varphi_{\varepsilon_k}^{(x)})|\le k\}}\,\dd x\Bigr]
      \leq |B| \sum_{k=1}^{\infty} 2^{-k} <+\infty \,.
\]
This shows that almost surely, for almost every \(x\in B\), we have
\(|\scrH(\varphi_{\varepsilon_k}^{(x)})|> k\) eventually.
Put differently, and using a countable union of balls, we get that
\begin{equation}
      \label{eq:infinite}
      \mathrm{Leb}\pigl( \pigl\{x\in\R^2 \,: \, \limsup_{k\to\infty} |\scrH(\varphi_{\varepsilon_k}^{(x)})| <+\infty \pigr\} \pigr) =0 \qquad \bbP\text{-a.s.}
\end{equation}

Consider now the event that there is some
locally finite signed Radon measure on a non-empty open set, namely
\[
A \coloneqq \pigl\{ \exists\, U = B(z,r)\,, \mu \in  \mathcal{M}_{\rm loc}(U) \,\text{ s.t. } \, \scrH(\varphi) = \int_{U} \varphi(y) \mu(\dd y) \quad \text{ for all } \varphi \in C_c^\infty(U) \pigr\} \,,
\]
where \(z\in \R^2\), \(r>0\).
Here, \(\mathcal{M}_{\rm loc}(U)\) denotes the set of locally finite signed Radon measures on the open ball \(U=B(z,r)\).
(Note that the choice of the ball \(U\) and of
the Radon measure~\(\mu\) may depend on the realisation.)

Consider a realisation in this event \(A\) and some \(U=B(z,r)\) given by this realization.
Now, notice that
\(\varphi_{\gep}^{(x)} \in C_c^{\infty}(B(x,\gep))\) so \( \varphi_{\gep}^{(x)} \in C_c^{\infty}(U)\)
for all \(x\in B(z,r/2)\) and \(\gep \in (0,r/2)\).
Then, for any such~\(x\) and~\(\gep\), by definition of \(A\) we have that
\begin{equation*}
      \bigl|\scrH(\varphi_{\gep}^{(x)}) \bigr|
      = \Bigl|\int_{U} \varphi_{\gep}^{(x)}(y) \mu(\dd y) \Bigr|
      \leq  \|\varphi\|_{\infty}\, \gep^{-2} \, |\mu|(B(x,\gep)) \,.
\end{equation*}
Since \(|\mu|\) is locally finite on \(U=B(z,r)\),
by the Lebesgue differentiation theorem,
we have that
\(\limsup_{\gep\downarrow 0} \gep^{-2}\, |\mu|(B(x,\gep))<\infty\)
for almost every \(x\in B(z,r/2)\).
In other words, we have proven that
\[
\text{ on the event }A \,, \quad \mathrm{Leb}\pigl( \pigl\{x\in\R^2 \,: \, \limsup_{\gep\downarrow 0} |\scrH(\varphi_{\varepsilon}^{(x)})| <+\infty \pigr\} \pigr) \geq \mathrm{Leb}(B(z,r/2)) >0 \,.
\]
(We recall that  \(z\in \R^2\) and \(r>0\) may depend on the realisation.)
In view of~\eqref{eq:infinite}, this proves that \(\bbP(A)=0\), which is the main statement of Proposition~\ref{prop:no-local-representative}.

Since a locally integrable function
(and in particular a spatially
constant distribution)
defines a locally finite Radon measure, this proves the
remaining assertion.
\qed

\subsection{The covariance kernel: existence and continuity}
\label{prop:sec-covariance}

We now prove~\eqref{eq:main-covariance} in item~\eqref{iii:covariance} of Theorem~\ref{thm:properties}.
Again, we only deal with \(t=1\) and we omit the dependence on \(\vartheta\) for simplicity.

\subsubsection{Step 1. Existence of the covariance kernel}
Let \(f,g\in C_c^\infty(\R^2)\) be real test functions.
For \(0<\delta<1\) we have
\[
\Cov\bigl(\scrH^{(\delta)}(f),\scrH^{(\delta)}(g)\bigr)  = \int_{\R^2\times \R^2} f(x) g(y) \mathsf C_{\delta}(x-y) \dd x \dd y \,,
\]
where
\(\mathsf C_{\delta}(z) \coloneqq \Cov(F_{\delta}(x),F_{\delta}(x+z)) \geq 0\) (recall the definition~\eqref{eq:F-delta}),
using also translation invariance to get that it only depends on~\(z\).
By the \(L^2\) convergence in Theorem~\ref{thm:toy}, we have
\begin{equation}
      \label{eq:limit-covariances}
      \Cov\bigl(\scrH(f),\scrH(g)\bigr) = \lim_{\delta \downarrow 0} \Cov\bigl(\scrH^{(\delta)}(f),\scrH^{(\delta)}(g)\bigr) \,.
\end{equation}

Now, thanks to Lemma~\ref{lem:cutoff-covariance}, we get
\begin{equation}
      \label{eq:uniform-boundedness}
      \sup_{\delta \in (0,\delta_0)} \mathsf C_{\delta}(z) \leq C \,
      \bigl[1+ \log \bigl( 1+ \log^+ (1/|z|) \bigr)\bigr] \,.
\end{equation}
This shows that, for any ball \(B=B(0,R)\) in \( \R^2 \),
the family of functions
\((\mathsf C_{\delta})_{\delta \in (0,\delta_0)}\) is bounded in \(L^2(B)\).
Thus, by the Banach--Alaoglu theorem, one can find
a subsequence along which \(\mathsf C_{\delta}\) converges weakly in \(L^2(B)\).
By a diagonal extraction in a sequence of growing balls,
we can thus find a subsequence \((\delta_k)_{k\geq 1}\)
and some \(\mathsf C\in L^2_{\rm loc}(\R^2)\)
such that \(\mathsf C_{\delta_k}\) converges weakly to \(\mathsf C\).
We thus have that
\[
\lim_{k\to\infty} \int_{\R^2\times \R^2} f(x) g(y) \mathsf C_{\delta_k}(x-y) \dd x \dd y = \int_{\R^2\times \R^2} f(x) g(y) \mathsf C(x-y) \dd x \dd y\,,
\]
using the weak convergence applied to \(z\mapsto\int f(y+z)g(y)\,\dd y\) which is compactly supported.

Because of~\eqref{eq:limit-covariances}, the limit on
the right-hand side is equal to \(\Cov(\scrH(f),\scrH(g))\).
Since the covariance structure is fixed,
this implies that any weakly convergent subsequence
of \((\mathsf C_{\delta})_{\delta \in (0,\delta_0)}\)
must converge to the same limit \(\mathsf C \in L^2_{\rm loc}(\R^2)\),
showing that we indeed have that \(\mathsf C_{\delta} \to \mathsf C\) weakly as \(\delta\to 0\).
This concludes the proof of~\eqref{eq:main-covariance}.

Note also that \(\mathsf C\)
is rotationally invariant thanks to
Proposition~\ref{prop:invariance},
\textit{i.e.}\ \(\mathsf C(z)\) depends only on \(|z|\),
and we get from~\eqref{eq:uniform-boundedness} that
\begin{equation}
      \label{eq:covariance-bound-weak}
      \mathsf C(z) \leq C\, \bigl[1+\log \bigl( 1+ \log^+ (1/|z|) \bigr)\bigr]\,.
\end{equation}

\subsubsection{Step 2. Spectral measure and continuity of the covariance kernel}
We now show that \(\mathsf C(z)\) is continuous outside zero, relying on a Fourier inversion formula.
Since \(\mathsf C\) can be interpreted as a tempered distribution (thanks e.g.\ to~\eqref{eq:covariance-bound-weak}) and since the covariance distribution is of positive type,
the Bochner--Schwartz theorem
\cite[Theorem~IX.10]{RS75}
provides a nonnegative tempered (i.e. locally finite and with at most polynomial mass growth 
at infinity), spectral measure \(\nu=\nu_{\mathsf C}\)
such that for any \(\phi\in \mathcal{S}(\R^2)\)
\begin{equation}
      \label{eq-spectral-convention}
      \int_{\R^2} \phi(\xi) \,\nu(\dd\xi) = \frac{1}{(2\pi)^2}\int_{\R^2} \widehat{\phi}(z) \mathsf C(z) \,\dd z \,,
\end{equation}
with the convention~\eqref{prop:eq-sobolev-convention} for the Fourier transform.
Because of rotational invariance, we also have that \(\nu\) is radially symmetric.
We now have the following claim that bounds the spectral mass of growing balls (it is proven below).  

\begin{proposition}
      \label{prop:loglog}
      We have that \(\nu (\widebar{B}(0,R)) \sim \log\log R\) as \(R\to \infty\).
\end{proposition}

Because of the radial symmetry, 
letting \(\nu_{\rm rad}\) be the radial projection
of \(\nu\) (\textit{i.e.}\ the push-forward of \(\nu\) by \(\xi\mapsto |\xi|\)),
the Fourier inversion can then be written as follows:
\begin{equation}
\label{prop:eq-Bessel-inversion}
 \mathsf C(z) = C_0 + \int_{[0,\infty)}\bigl\{J_0( \rho |z|) -  \one_{\{\rho\leq 1\}}\bigr\}\, \nu_{\rm rad}( \dd  \rho) \,,
\end{equation}
where \(J_0\) is the Bessel function
\(J_0(u)=\frac1{2\pi}\int_0^{2\pi}\e^{iu\cos\theta}\,\dd\theta\) and \(C_0=\nu_{\rm rad}([0,1])\).

The integrability of the right-hand
side of~\eqref{prop:eq-Bessel-inversion}
can be justified by using the elementary bounds
\(|J_0(u)-1|\le Cu^2\) for \(u\leq 1\)
and \(|J_0(u)|\le C u^{-1/2}\) for \(u\geq 1\),
together with the local finiteness of $\nu$ and 
Proposition~\ref{prop:loglog},
which ensures that
\begin{equation}
	\label{eq:Fubini-nu-rad}
	\begin{split}
	\int_{(1,\infty)} \rho^{-1/2} \nu_{\rm rad}(\dd \rho) =&\frac{1}{2}\int_{(1,\infty)} a^{-3/2} \nu_{\rm rad}\bigl( (1,a]   \bigr) \dd a   \\ \leq & \frac{1}{2}\int_{(1,\infty)} a^{-3/2} \nu\bigl(   \widebar{B}(0,a)  \bigr) \dd  a  <+\infty \,,
	\end{split}
\end{equation}
where the equality follows from writing $\rho^{-1/2}=\frac{1}{2}\int_{[\rho,\infty)}a^{-3/2}\dd a$ and applying Fubini--Tonelli.
Thus the integral \eqref{prop:eq-Bessel-inversion}
converges absolutely, uniformly
for \(z\) in a compact set of \(\R^2\setminus\{0\}\):
since \(z \mapsto J_0(\rho |z|)\) is
a continuous function of \(z\),
this shows that \(z\mapsto \mathsf C(z)\)
is continuous on \(\R^2\setminus\{0\}\).

\begin{proof}[Proof of Proposition~\ref{prop:loglog}]
Consider a nonnegative
\(\varphi\in C_c^{\infty}(\R^2)\) which is radially
symmetric and of integral~\(1\),
so that \(\widehat{\varphi}\) is also radially symmetric and verifies \(\widehat{\varphi}(0)=1\) and \(|\widehat{\varphi}|\leq 1\).
Recall that
\(\varphi_{\gep}(\cdot) = \gep^{-2} \varphi(\cdot/\gep)\)
and notice that
\(\widehat \varphi_{\gep}(\cdot) = \widehat{\varphi} (\gep \cdot)\).

By the spectral representation~\eqref{eq-spectral-convention} we have that 
\begin{equation*}
 \Cov(\scrH(f),\scrH(g)) 
 = \int_{\R^2\times \R^2} f(x) g(y) \mathsf C(x-y) \dd x \dd y
 = \int_{\R^2}\widehat f(\xi) \, \overline{\widehat g(\xi)}\,\nu(\dd\xi) \,, 
\end{equation*}
so that taking \(f=g=\varphi_{1/r}\) we obtain
\begin{equation}\label{prop:eq-spectral-average}
 \int_{\R^2} |\widehat{\varphi} (\xi/r)|^2\, \nu(\dd \xi) = \Var\bigl(\scrH(\varphi_{1/r})\bigr)
 \sim\log\log r \qquad \text{ as } r\to\infty \,,
\end{equation}
where the last asymptotics follows from 
Proposition~\ref{prop:shrinking-variance}.

\smallskip
\textit{Upper bound.}
Let \(\eta\in(0,1)\) be fixed.
Since \(\widehat{\varphi}(0) =1\) and is continuous, we get that \(|\widehat{\varphi}(\xi)|^2\geq 1-\eta\) for all \(|\xi|\leq c_{\eta}\).
We thus obtain that
\[
\int_{\R^2} |\widehat{\varphi} ( c_{\eta} \xi/R)|^2\, \nu(\dd \xi) \geq (1-\eta) \nu( \widebar{B}(0,R)) \,.
\]
With~\eqref{prop:eq-spectral-average} applied with \(r= R/c_{\eta}\), this shows
\begin{equation}
      \label{prop:eq-spectral-growth}
      \nu( \widebar{B}(0,R))  \leq  (1+o(1)) (1-\eta)^{-1} \log\log(R/c_{\eta}) \,.
\end{equation}
Since \(\eta>0\) is arbitrary, this shows that \(\nu( \widebar{B}(0,R)) \leq (1+o(1)) \log \log R\) as \(R\to\infty\),
showing the upper bound in Proposition~\ref{prop:loglog}.

\smallskip
\textit{Lower bound.}
Let us fix some \(\eta \in (0,1)\). Since \(|\widehat{\varphi}|\leq 1\), we get that 
\begin{equation}
      \label{eq:spectral-upper-first}
      \int_{\R^2} \Bigl| \widehat \varphi\Bigl(\frac{\xi}{\eta R} \Bigr)\Bigr|^2 \nu (\dd \xi ) 
      \leq  \nu (\bar B(0,R) ) + \int_{|\xi|>R} \Bigl| \widehat \varphi\Bigl(\frac{\xi}{\eta R} \Bigr)\Bigr|^2\, \nu(\dd \xi) \,.
\end{equation}
Since \(\varphi\) is smooth and compactly supported, \(|\widehat \varphi(\xi)|\) decreases faster than any polynomial at \(+\infty\).
Let us simply use that \(|\widehat \varphi(\xi)|^2\leq C |\xi|^{-1}\) for all \(|\xi|\geq 1\),
so that  
\[
\int_{|\xi|>R} \Bigl| \widehat \varphi\Bigl(\frac{\xi}{\eta R} \Bigr)\Bigr|^2\, \nu(\dd \xi) \leq C \int_{(R,\infty) } (\rho/\eta R)^{-1} \nu_{\rm rad}(\dd \rho) \leq  C \eta R \int_{(R,\infty) } a^{-2} \nu(\widebar{B}(0,a) )\dd a \,,
\]
where for the last inequality we have used that $\rho^{-1}=\int_{[\rho,\infty)}a^{-2}\dd a$ and applied Fubini--Tonelli, as in~\eqref{eq:Fubini-nu-rad}.
Now, using~\eqref{prop:eq-spectral-growth}, we get that \(\nu( \widebar{B}(0,a) ) \leq C (1+ \log \log a)\) for all \(a>R\) with \(R\) large enough, so we end up with 
\[
\int_{|\xi|>R} \Bigl| \widehat \varphi\Bigl(\frac{\xi}{\eta R} \Bigr)\Bigr|^2\, \nu(\dd \xi) \leq  C \eta R \int_{ (R, \infty) } a^{-2} (1+ \log \log a) \dd a \leq C' \eta \log \log R
\]
for all \(R\) sufficiently large.
Going back to~\eqref{eq:spectral-upper-first} and using~\eqref{prop:eq-spectral-average} with \(r = \eta R\), we thus get that as \(R\to\infty\),
\[
\nu( \widebar{B}(0,R)) \geq (1+o(1)) \log\log (\eta R) - C\eta \log \log R \,.
\]
This proves the lower bound in
Proposition~\ref{prop:loglog}, since \(\eta>0\) is arbitrary.
\end{proof}

\subsubsection{Step 3. Divergence of the covariance kernel near the diagonal}

In this section, we prove the sharp asymptotics of \(\mathsf C(z)\) as \(|z|\downarrow 0\), namely~\eqref{eq:main-diagonal}.
It relies on the inversion formula~\eqref{prop:eq-Bessel-inversion}: using the notation \(R=1/|z|\)
to match with the notation
of the previous step, we need to show that 
\[
\int_{[0,\infty)} \bigl\{J_0( \rho/R)- \one_{\{\rho\leq 1\}}\bigr\} \nu_{\rm rad}( \dd  \rho) \sim \log \log R \qquad \text{ as } R \to \infty \,.
\]
Note that
\(\int_{[0,1]}(J_0( \rho/R) -  \one_{\{\rho\leq 1\}})  \nu_{\rm rad} (\dd \rho)\)
remains bounded by a constant, uniformly for \(R> 1\),
so we only need to show that
\begin{equation}
      \label{eq:asymp-integral}
      \int_{(1,\infty)} J_0( \rho/R) \nu_{\rm rad}( \dd  \rho) \sim \log \log R \qquad \text{ as } R \to \infty \,.
\end{equation}

To show this, we use that \(|J_0(u)-1| \leq C u^2\) for \(u\leq 1\) and \(|J_0(u)| \leq C u^{-1/2}\) for \(u\geq 1\), to get that 
\begin{equation}
	\label{eq:radial-asymp-1}
	\begin{split}
		\Big|\int_{(1,\infty)} J_0( \rho/R) \nu_{\rm rad}( \dd  \rho) - & \nu_{\rm rad}((1,R]) \Big|
	\\
	& \leq  C \int_{(0,R]} (\rho/R)^{2} \nu_{\rm rad} (\dd \rho) + C \int_{(R,\infty)} (\rho/R)^{-1/2}  \nu_{\rm rad} (\dd \rho) \,.
	\end{split}
\end{equation}
Since by Proposition~\ref{prop:loglog} we have that 
\[
\nu_{\rm rad}((1,R]) = \nu( \widebar{B}(0,R)) - \nu( \widebar{B}(0,1)) \sim \log \log R \qquad \text{ as } R \to \infty \,,
\]
it only remains to show that the two terms on the right-hand side of~\eqref{eq:radial-asymp-1} are \(o(\log\log R)\).

\smallskip
\noindent
\textit{First term in~\eqref{eq:radial-asymp-1}.}
Using that \(\rho^2= 2 \int_{(0,\rho]} a \dd a\) and Fubini--Tonelli, we get similarly as in \eqref{eq:Fubini-nu-rad} that 
\[
\int_{(0,R]} (\rho/R)^{2} \nu_{\rm rad} (\dd \rho) = \frac{2}{R^2} \int_{(0,R]} a\, \nu_{\rm rad}( (a,R]) \dd a = \nu(\widebar{B}(0,R)) - \frac{2}{R^2} \int_{(0,R]} a\, \nu(\widebar{B}(0,a)) \dd a \,.
\]
Since \(a\mapsto \nu(\widebar{B}(0,a))\) is slowly varying\footnote{A function \(f\) is said to be slowly varying if for any \(\alpha >0\) we have that \(\lim_{x\to\infty} f(\alpha x)/f(x) = 1\).} thanks to Proposition~\ref{prop:loglog}, thanks to \cite[Proposition~1.5.8]{BGT89} we get that \(\int_{(0,R]} a\, \nu(\widebar{B}(0,a)) \dd a \sim \frac12 R^2 \nu(\widebar{B}(0,R) )\) as \(R\to\infty\).
We thus obtain that 
\[
\int_{(0,R]} (\rho/R)^{2} \nu_{\rm rad} (\dd \rho) = \nu(\widebar{B}(0,R)) - (1+o(1)) \nu(\widebar{B}(0,R)) = o(\log \log R) \quad \text{ as } R\to\infty \,.
\]

\smallskip
\noindent
\textit{Second term in~\eqref{eq:radial-asymp-1}.}
We proceed similarly as above. 
Using that $\rho^{-1/2}=\frac{1}{2}\int_{[\rho,\infty)}a^{-3/2}\dd a$
and applying Fubini--Tonelli as in~\eqref{eq:Fubini-nu-rad}, we get
\[
\begin{split}
	\int_{(R,\infty)} (\rho/R)^{-1/2}  \nu_{\rm rad} (\dd \rho) & = \frac{R^{1/2}}{2} \int_{(R,\infty)} a^{-3/2} \nu_{\rm rad}((R,a]) \dd a \\
	& = \frac{R^{1/2}}{2} \int_{(R,\infty)} a^{-3/2} \nu(\widebar{B}(0,a)) \dd a  -  \nu(\widebar{B}(0,R)) \,.
\end{split}
\]
Again, since \(a\mapsto \nu(\widebar{B}(0,a))\) is slowly varying, we get that \(\int_{(R,\infty)} a^{-3/2} \nu(\widebar{B}(0,a)) \dd a \sim 2 R^{-1/2} \nu(\widebar{B}(0,R))\) as \(R\to\infty\), see \cite[Proposition~1.5.10]{BGT89}.
This yields
\[
\int_{(R,\infty)} (\rho/R)^{-1/2}  \nu_{\rm rad} (\dd \rho) = (1+o(1)) \nu(\widebar{B}(0,R)) -\nu(\widebar{B}(0,R))  = o(\log\log R)  \quad \text{as } R\to\infty \,.
\]

\smallskip
\noindent
We have thus shown that both terms on the right-hand side of~\eqref{eq:radial-asymp-1} are \(o(\log \log R)\).  
This concludes the proof of~\eqref{eq:asymp-integral}.
\qed

\section{Non-Gaussianity}\label{ng:section}

In this section, we prove Theorem~\ref{thm:nongaussian}, which implies that
the critical 2D KPZ process~$\scrH_t^\vartheta$ is non-Gaussian. 
Let us fix \(t>0\) and $\vartheta\in\R$.
Fix also a nonzero nonnegative
$\psi\in C_c^\infty(\R^2)$ such that $\int \psi =1$; for the general case, one simply needs to divide $\psi$ by its integral.
The proof will be
by contradiction: in a nutshell, 
we show that if $\scrH_t^\vartheta(\psi)$ were Gaussian,
this would contradict the singularity of the~SHF.

\subsection{Main steps of the proof}

Let us first outline the key steps before giving the details in the rest of this section.

\subsubsection{Step 1. (Gaussian) Rigidity with respect to the initial condition}

By Corollary~\ref{prop:spatial-identification}, we have
\begin{equation}
      \label{ng:flat-spatial-limit}
\widetilde{\scrH}_t^{(r)}(\zero,\psi)\coloneqq  \int\psi(x)\Bigl(\log\scrZ_t^\vartheta(g_r^{(x)},\one)
       -\bbE\bigl[\log\scrZ_t^\vartheta(g_r^{(x)},\one)\bigr]\Bigr)\,\dd x
 \xrightarrow[\; r\downarrow0 \;]{d}
 \scrH_t^\vartheta(\zero,\psi),
\end{equation}
where $g_r^{(x)}$ is the heat kernel at time $r$ centred at $x$
and $\zero$ denotes the flat initial condition.

A first crucial observation is that if we replace $\one$
in \(\scrZ_t^\vartheta(g_r^{(x)},\one)\)
by a bounded Lipschitz function $f$ with
$0<c\le f\le C<\infty$,
we have
\begin{equation}\label{ng:terminal-change}
 \log\scrZ_t^\vartheta(g_r^{(x)},f)
 -\log\scrZ_t^\vartheta(g_r^{(x)},\one)
 =\log\mu_{t,r}^x(f),
\end{equation}
where
\begin{equation}\label{ng:profiles}
 \mu_{t,r}^x(\dd y)
 \coloneqq \frac{\scrZ_t^\vartheta(g_r^{(x)},\dd y)}
          {\scrZ_t^\vartheta(g_r^{(x)},\one)}
\end{equation}
can be interpreted as the endpoint distribution at time \(t\) of a continuum polymer
with partition function \(\scrZ_t^\vartheta(g_r^{(x)},\one)\).
Since $c\le f\le C$, the quantity in \eqref{ng:terminal-change} lies in
$[\log c,\log C]$, uniformly in $r$ and $x$.

Replacing \(\one\) in the left hand side of \eqref{ng:flat-spatial-limit} by \(f =\e^{h}\) for some bounded Lipschitz \(h\), then setting
\begin{equation}\label{ng:F}
 \widebar{\scrH}_t^{(r)}(h,\psi)\coloneqq \int\psi(x)\Bigl(\log\scrZ_t^\vartheta(g_r^{(x)},f)
       -\bbE\bigl[\log\scrZ_t^\vartheta(g_r^{(x)},f)\bigr]\Bigr)\,\dd x \,,
\end{equation}
we have that
\begin{equation}\label{ng:bounded-difference}
 \big| \widebar{\scrH}_t^{(r)}(h,\psi) - \widebar{\scrH}_t^{(r)}(\zero,\psi)  \big| \le \log(C/c).
\end{equation}
Let us stress that contrary to the definitions~\eqref{eq:regularised-H} or~\eqref{prop:eq-direct-cutoff}, the centring in~\eqref{ng:F} depends on the initial condition \(h = \log f\). 
This way, \(\widebar{\scrH}_t^{(r)}(h,\psi)\) is a \textit{centred} random variable; note that it coincides with \(\widetilde{\scrH}_t^{(r)}(h,\psi)\) when \(h=\zero\).

\begin{remark}
      \label{rem:initial-condition}
      More generally, let us consider two bounded Lipschitz initial conditions  \(h_0\) and \(h_0'\).
      Replacing \(f,\one\) by \(\e^{h_0},\e^{h_0'}\) in \eqref{ng:terminal-change} yields that, for any \(x\in \R^2\),
      \[
      \inf \{h_0-h_0'\}\leq \log\scrZ_t^\vartheta(g_r^{(x)},\e^{h_0})
 -\log\scrZ_t^\vartheta(g_r^{(x)},\e^{h_0'}) \le \sup \{h_0-h_0'\}\,.
      \]
      When centring both terms by a common 
      \(c_t^{\vartheta,(r)}\coloneqq  \bbE[ \log \scrZ_{t}^\vartheta ( \one, g_r^{(x)})]\), recalling the definition~\eqref{eq:regularised-H}, this yields that
      for any \(\delta>0\)
      \[
      \begin{aligned}
      \big| \scrH_t^{\vartheta, (\delta)}(h_0,\psi) - \scrH_t^{\vartheta,(\delta)}(h_0',\psi) \big|
      &\leq \|\psi\|_1\, \|h_0-h_0'\|_{\infty}.
      \end{aligned}
      \]
\end{remark}

A crucial observation from Section~\ref{sec:polymer-convexity} is that
both sides of \eqref{ng:flat-spatial-limit} are of Ehrhard type,
see in particular Proposition~\ref{prop:Ehrhard}.
The same is true for~\eqref{ng:F} and its limit as \(r\downarrow 0\),
and the inequality~\eqref{ng:bounded-difference} also remains valid in the limit.
This is recorded in the following lemma, proven in Section~\ref{sec:ng-proof-Ehrhard} below.
For simplicity of notation, we let
\[
X_r \coloneqq \widebar{\scrH}_t^{(r)}(\zero,\psi)
\qquad \text{ and }\qquad
Y_r \coloneqq \widebar{\scrH}_t^{(r)}(h,\psi) \,.
\]

\begin{lemma}\label{ng:positive-sums}
For every bounded Lipschitz function $f$ with $0<c\le f\le C<\infty$,
the laws of $(X_r,Y_r)$ are tight as $r\downarrow0$.
Every joint subsequential limit $(X,Y)$ is centred and satisfies
\[
 X\stackrel{\mathrm d}=\scrH_t^\vartheta(\zero,\psi),\qquad
 |Y-X|\le\log(C/c)\quad\text{almost surely}.
\]
Both $Y$ and $X+Y$ are of Ehrhard type.
\end{lemma}

If we now assume that \(\scrH_t^\vartheta(\zero,\psi)\) is Gaussian,
then Lemma~\ref{ng:positive-sums} allows us
to apply the following Gaussian rigidity result.
A related rigidity result for convex functions of Gaussian vectors is given in \cite[Theorem~3.4]{FMP}.

\begin{lemma}[Gaussian rigidity]\label{ng:rigidity}
Let $X,Y$ be \emph{centred} real random variables with
$X\sim\mathcal N(0,\sigma^2)$, $\sigma\ge0$, and
$|Y-X|\le K<\infty$ almost surely.
If $Y$ and $X+Y$ are of Ehrhard type, then $Y=X$ almost surely.
\end{lemma}

Under the Gaussianity assumption, Lemma~\ref{ng:rigidity} implies that
\(Y=X\) almost surely for every joint subsequential limit \((X, Y)\) in
Lemma~\ref{ng:positive-sums}. 
We next deduce from this equality an unlikely
consequence for the polymer endpoint distributions.

\subsubsection{Step 2. Consequence of Gaussian rigidity on the endpoint distribution}

Applying Lemmas~\ref{ng:positive-sums} and~\ref{ng:rigidity}
to every subsequential weak limit of \((X_r, Y_r)\) gives that
$Y_r-X_r\to0$ in probability.
Recalling \eqref{ng:terminal-change}, this leads to the following consequence (details are given in Section~\ref{sec:ng-polymer-endpoint}).
If $\scrH_t^\vartheta(\zero,\psi)$ is Gaussian or a constant,
then, for every bounded Lipschitz function $f$ with
$0<c\le f\le C<\infty$ and every $1\le p<\infty$,
\begin{equation}\label{ng:vanishing-perturbation}
 \int\psi(x)\pigl(\log\mu_{t,r}^x(f)
             -\bbE\bigl[\log\mu_{t,r}^x(f)\bigr]\pigr)\,\dd x
 \xrightarrow[\;r\downarrow0\;]{}0 \qquad \text{in }L^p(\bbP).
\end{equation}

The freedom to choose $f$ in~\eqref{ng:vanishing-perturbation}
suggests that limits of the averaged polymer endpoint distribution at
time $t$ are deterministic: this is formalized by Lemma~\ref{ng:deterministic-pairings} below.
Define, for $r,t>0$,
\begin{equation}\label{ng:averages}
 \mu_{t,r}^\psi\coloneqq \int\psi(x)\mu_{t,r}^x\,\dd x\in \mathcal M_1(\R^2),
\end{equation}
where $\mathcal M_1(\R^2)$ denotes the space of probability measures on $\R^2$,
equipped with weak convergence.
By choosing $f=\one+\varepsilon\varphi$
and then differentiating w.r.t.\ $\varepsilon$, we can deduce from~\eqref{ng:vanishing-perturbation} the following result.

\begin{lemma}[Deterministic endpoint limits]\label{ng:determinism}
If \eqref{ng:vanishing-perturbation} holds for every bounded positive Lipschitz function $f$ bounded away from zero, then
for any \(\varphi\in C_c(\R^2)\)
\begin{equation}\label{ng:deterministic-pairings}
 \bbE\Bigl[\bigl|\mu_{t,r}^\psi(\varphi)-\bbE\bigl[\mu_{t,r}^\psi(\varphi)\bigr]\bigr|\Bigr]
 \longrightarrow0\qquad \text{ as } r\downarrow0\,.
\end{equation}
Additionally, every subsequential limit in distribution of $\mu_{t,r}^\psi$ (for the vague topology) is a 
deterministic finite positive measure of mass
at most one.
\end{lemma}

The next lemma shows that the deterministic limit points arising from Lemma
\ref{ng:determinism} cannot be the zero measure.

\begin{lemma}[Nonzero deterministic limits]\label{ng:nonzero-limit}
If $\scrH_t^\vartheta(\zero,\psi)$ is Gaussian, the laws of the family of random measures
$(\mu_{t,r}^\psi)_{r\in (0,1)}$ (equipped with the vague topology) are tight as
$r\downarrow0$. Every subsequential weak
limit $\mu_t$ is a deterministic finite positive measure with $0<\mu_t(\R^2)\le1$.
\end{lemma}

The proofs of Lemmas~\ref{ng:vanishing-perturbation} and~\ref{ng:nonzero-limit} are deferred to Section~\ref{sec:ng-polymer-endpoint}.

\subsubsection{Step 3. Contradiction and conclusion of the proof}

Finally we prove two contradicting statements for such a deterministic limiting measure~$\mu$
obtained from Lemma~\ref{ng:nonzero-limit}.

\begin{lemma}[Comparison with the later SHF]\label{ng:later-comparison}
Assume that $\scrH_t^\vartheta(\zero,\psi)$ is Gaussian and let $\mu_t$ be a weak limit of $(\mu_{t,r}^\psi)_{r\in (0,1)}$ as in Lemma~\ref{ng:nonzero-limit}.
With
positive probability, $\mu_t$ and
$\scrZ_{2t/3,t}^\vartheta(\one,\dd y)$ are not mutually singular.
\end{lemma}

\noindent
The key idea behind the proof is to integrate out disorder on the time interval $[t/3, 2t/3]$ for the SHF that defines $\mu_{t,r}^\psi$.

On the other hand, the almost sure singularity of the SHF with Lebesgue initial condition,
see \cite[Theorem~1.1]{CSZ25},
together with $\bbE\bigl[\scrZ_{2t/3,t}^\vartheta(\one,\dd y)\bigr]=\dd y$,
gives the opposite relation.

\begin{lemma}[Singularity with deterministic measures]\label{ng:singularity}
For every deterministic finite positive measure $\mu$ on $\R^2$,
we have that $\mu\perp \scrZ_{2t/3,t}^\vartheta(\one,\dd y)$ a.s., that is, $\mu$ is singular with respect to $\scrZ_{2t/3,t}^\vartheta(\one,\dd y)$.
\end{lemma}

The proofs of these two lemmas are postponed to Section~\ref{sec:ng-singularity}.
Let us now show how these ingredients complete the argument.

\begin{proof}[Proof of Theorem~\ref{thm:nongaussian}]
Suppose that $\scrH_t^\vartheta(\zero,\psi)$ is Gaussian, including a
constant one.
Lemma~\ref{ng:nonzero-limit} gives the existence of a nonzero deterministic
measure $\mu_t$ arising as a weak subsequential limit of $\mu_{t,r}^\psi$ as \(r\downarrow0\).
By Lemma~\ref{ng:later-comparison}, this
measure \(\mu_t\) and $\scrZ_{2t/3,t}^\vartheta(\one,\dd y)$ are not mutually
singular with positive probability.
On the other hand, since \(\mu_t\) is deterministic, Lemma~\ref{ng:singularity}
says that they are mutually singular almost surely. This is a contradiction.
\end{proof}

\subsection{Ehrhard type and Gaussian rigidity: Lemmas~\ref{ng:positive-sums} and~\ref{ng:rigidity}}
\label{sec:ng-proof-Ehrhard}

We now prove the ``Gaussian rigidity'' results that are used in Step~1.

\begin{proof}[Proof of Lemma~\ref{ng:positive-sums}]
The convergence of Corollary~\ref{prop:spatial-identification}, \textit{i.e.}\ \(X_r \stackrel{\mathrm d}{\to} X\), and the bound
\eqref{ng:bounded-difference} give joint tightness of $(X_r,Y_r)$.
Every joint subsequential limit $(X,Y)$ is such that \(X \stackrel{\mathrm d}{=} \scrH_t^\vartheta(\zero,\psi)\)
and satisfies $|Y-X|\le\log(C/c)$.
The limit $X$ is centred by Corollary~\ref{prop:spatial-identification},
and the bounded differences $Y_r-X_r$ have mean zero, which passes
to the limit, so $Y$ is centred as well.

For fixed $r>0$, the variables $X_r$, $Y_r$ and $X_r+Y_r$ are of
Ehrhard type, as explained in Section~\ref{sec:polymer-convexity};
see also \cite[Lemma~3.1 and Section~4]{LX26}.
Indeed, their lattice approximations are nonnegative linear combinations
of centred log partition functions, which are convex in the same Gaussian disorder.
The weak closure in Lemma~\ref{ng:quantile-closure} therefore shows
that $Y$ and $X+Y$ are of Ehrhard type.
\end{proof}

\begin{proof}[Proof of Lemma~\ref{ng:rigidity}]
Suppose first that $\sigma>0$. 
By Lemma~\ref{ng:quantile-closure}, $Y$ has the law of $h(G)$,
where $G$ is standard Gaussian and 
$h$ is a finite, nondecreasing and convex function.
Our goal is to show that $h$ must be affine.

The bound $|Y-X|\le K$ implies that
$Y$ has the same Gaussian tails as $X$, i.e.,
\begin{equation}\label{ng:gaussian-tail-rates}
 \lim_{z\to\infty}\frac{\log\bbP(Y>z)}{z^2}
 =\lim_{z\to\infty}\frac{\log\bbP(Y<-z)}{z^2}
 =-\frac1{2\sigma^2}.
\end{equation}

Let $az+b$ be a supporting line of \(h\), so $h(z)\ge az+b$ for all
$z$, with $a\ge0$ by monotonicity. 
If $a>\sigma$, then
\[
 \bbP(Y>z)=\bbP(h(G)>z)\ge\bbP(aG+b>z),
\]
and hence $\liminf_{z\to\infty} z^{-2} \log \bbP(Y>z) \geq -1/(2a^2)>-1/(2\sigma^2)$,
contradicting~\eqref{ng:gaussian-tail-rates}.  Similarly, if $0\leq a<\sigma$, the bound $\bbP(Y<-z)\le\bbP(aG+b<-z)$ contradicts~\eqref{ng:gaussian-tail-rates}. 
Therefore every supporting line of $h$ has slope \(\sigma\) so this proves that \(h\) is affine, namely \(h(z) =\sigma z +b\).  
The fact that $Y = \sigma G +b$ is 
centred means that
$Y\sim\mathcal N(0,\sigma^2)$.

Apply the same argument to $X+Y$, which is of Ehrhard type and
differs from $2X$ by at most $K$. It follows that
$X+Y\sim\mathcal N(0,4\sigma^2)$, and hence
\[
 \bbE\bigl[(X-Y)^2\bigr]
 =2\bbE[X^2]+2\bbE[Y^2]-\bbE\bigl[(X+Y)^2\bigr]=0.
\]
This concludes the proof. When $\sigma=0$, it is easy to see that $X=Y=0$.
\end{proof}

\subsection{About the polymer endpoint distribution: Lemmas~\ref{ng:determinism} and~\ref{ng:nonzero-limit}}
\label{sec:ng-polymer-endpoint}

We now turn to the proof of the estimates in Step~2,
which highlights the consequence of Gaussian rigidity on the polymer endpoint distribution.

\begin{proof}[Proof of~\eqref{ng:vanishing-perturbation}]
By Lemma~\ref{ng:positive-sums}, every sequence $r_n\downarrow0$ has a
subsequence along which $(X_{r_n},Y_{r_n})$ converges in law to a
centred pair $(X,Y)$ satisfying the hypotheses of
Lemma~\ref{ng:rigidity}, since $X$ is Gaussian by assumption.
Thus $Y=X$ almost surely, so $Y_r-X_r\to0$ in probability.
The bound~\eqref{ng:bounded-difference} upgrades this to convergence in
every finite $L^p$, proving~\eqref{ng:vanishing-perturbation}.   
\end{proof}

\begin{proof}[Proof of Lemma~\ref{ng:determinism}]
Without loss of generality, we may take
$\varphi\in C_c^\infty(\R^2)$ with $\|\varphi\|_\infty\le1$ and
$0<\varepsilon\le1/2$. 
Since $\mu_{t,r}^x$ is a probability measure we have
$\mu_{t,r}^x(\one+\varepsilon\varphi)=1+\varepsilon\mu_{t,r}^x(\varphi)$.
Then, using that $|\varepsilon^{-1}\log(1+\varepsilon a)-a|\le \varepsilon/[2(1-\varepsilon)^2]$ for any $|a|\le1$ we otbain that 
\[
\Bigl|\frac{1}{\varepsilon}\log\mu_{t,r}^x(\one+\varepsilon\varphi) - \mu_{t,r}^x(\varphi) \Bigr| \leq \frac{\varepsilon}{2(1-\varepsilon)^2} 
\]
and hence
\begin{equation}\label{ng:terminal-derivative}
 \left|\frac1\varepsilon\int\psi(x)\log\mu_{t,r}^x(\one+\varepsilon\varphi)\,\dd x
                    -\mu_{t,r}^\psi(\varphi)\right|\le\frac{\varepsilon}{2(1-\varepsilon)^2}.
\end{equation}
Centring both terms by their mean at most doubles this error. 
Thus, applying
\eqref{ng:vanishing-perturbation} with $f=\one+\varepsilon\varphi$ and $p=1$, while~$\varepsilon$ remains fixed, we get that
\[
 \limsup_{r\downarrow0}
 \bbE\Bigl[\bigl|\mu_{t,r}^\psi(\varphi)-\bbE\bigl[\mu_{t,r}^\psi(\varphi)\bigr]\bigr|\Bigr]
 \le\frac{\varepsilon}{(1-\varepsilon)^2}.
\]
Letting $\varepsilon\downarrow0$ yields \eqref{ng:deterministic-pairings}.

We now turn to the second part of the lemma.
Every subsequential limit $\mu_t$ is a random measure (on the space of locally finite
measures equipped with the vague topology) with mass at most
one. 
This boundedness allows us to pass to the limit in \eqref{ng:deterministic-pairings}, giving that 
\[
\mu_t(\varphi)=\bbE\bigl[\mu_t(\varphi)\bigr] \  a.s., \qquad \forall\, \varphi \in C_c^\infty(\R^2) \,.
\]
Applying this to a
countable measure-determining family of $\varphi \in C_c^\infty(\R^2)$ then shows that 
$\mu_t(\cdot)=\bbE[\mu_t(\cdot)]$ a.s., proving that $\mu_t$ is deterministic.
\end{proof}

\begin{proof}[Proof of Lemma~\ref{ng:nonzero-limit}]
Tightness for the vague topology follows from the fact that
$\mu_{t,r}^\psi$ has mass one almost surely.
Let $\mu_t$ be an arbitrary deterministic subsequential limit as \(r\downarrow 0\), and fix any nonzero
$\varphi_0\in C_c^\infty(\R^2)$ with $0\le\varphi_0\le1$.
It suffices to prove that for some $c>0$ we have that 
\begin{equation}\label{911}
 \bbE\bigl[\mu(\varphi_0)\bigr]\ge c \,.
\end{equation}

Recall the definition~\eqref{ng:profiles} of \(\mu_{t,r}^x\).
By Jensen's inequality we have that 
\[
\log\bbE\bigl[\mu_{t,r}^x(\varphi_0)\bigr]
 \ge\bbE\bigl[\log\scrZ_t^\vartheta(g_r^{(x)},\varphi_0)\bigr]
    -\bbE\bigl[\log\scrZ_t^\vartheta(g_r^{(x)},\one)\bigr] \,.
\]
Recalling that $\mu_{t,r}^\psi(\varphi_0)\coloneqq \int\psi(x)\mu_{t,r}^x(\varphi_0)\,\dd x$, 
it suffices to prove the lower bound
\begin{equation}\label{ng:positive-local-mass}
\bbE\bigl[\log\scrZ_t^\vartheta(g_r^{(x)},\varphi_0)\bigr]
    -\bbE\bigl[\log\scrZ_t^\vartheta(g_r^{(x)},\one)\bigr]
 \ge-C >-\infty \,,
\end{equation}
for some $C$ that is uniform in $x\in\operatorname{supp}(\psi)$ and $r$ small. 
Indeed, integrating w.r.t.\ $\psi(x)\,\dd x$
and passing to the subsequential limit $r\downarrow 0$ gives 
\eqref{911} with $c=\e^{-C}$.

The desired bound~\eqref{ng:positive-local-mass}  follows easily from the next two bounds:
\begin{equation}\label{ng:log-mean}
 \Bigl|\bbE\bigl[\log\scrZ_t^\vartheta(g_r^{(x)},\one)\bigr]
             +\frac12\log\log \Big(\frac1r\Big)\Bigr|\le C <\infty,
 \qquad 0<r<r_0,\quad x\in\R^2,
\end{equation}
and
\begin{equation}\label{ng:log-mean2}
 \bbE\bigl[\log\scrZ_t^\vartheta(g_r^{(x)},\varphi_0)\bigr]
 \ge-\frac12\log\log \Big(\frac1r\Big)-C,
\end{equation}
where $C$ is does not depend on $x\in\operatorname{supp}(\psi)$ and $r$ is small.
Indeed, the two terms $\frac12\log\log(1/r)$ cancel in their difference in \eqref{ng:positive-local-mass}. 

It remains to establish~\eqref{ng:log-mean} and~\eqref{ng:log-mean2}. 
The first estimate \eqref{ng:log-mean} is the continuum 
analogue of Lemma \ref{lem:centering} and follows the same proof:
it requires negative-first and second
moment bounds, 
that can be found in \cite[Proposition~5.3 with $p=-1$, and (3.14)]{LX26}.
For the second estimate \eqref{ng:log-mean2}, 
we apply~\eqref{eq:chen-two} to \(\log \bigl(\scrZ_t^\vartheta(g_r^{(x)},\varphi_0)/\bbE[\scrZ_t^\vartheta(g_r^{(x)},\varphi_0)]\bigr)\)
to get that 
\[
\bbE\bigl[\log\scrZ_t^\vartheta(g_r^{(x)},\varphi_0)\bigr]
\ge
-\frac12\log\bbE\bigl[
  \scrZ_t^\vartheta(g_r^{(x)},\varphi_0)^2\bigr]
+2\log\bbE\bigl[
  \scrZ_t^\vartheta(g_r^{(x)},\varphi_0)\bigr].
\]
Since $0\le\varphi_0\le1$, we have
\[
 \bbE\bigl[\scrZ_t^\vartheta(g_r^{(x)},\varphi_0)\bigr]
 =(g_{t+r}*\varphi_0)(x)\ge c>0,
 \qquad
 \bbE\bigl[\scrZ_t^\vartheta(g_r^{(x)},\varphi_0)^2\bigr]\le C\log(1/r),
\]
uniformly for $x\in\operatorname{supp}\psi$ and small $r$, which plugged in the above proves~\eqref{ng:log-mean2}. 
This concludes the proof.
\end{proof}

\subsection{Continuity vs.\ singularity with respect to the SHF: Lemmas~\ref{ng:later-comparison} and~\ref{ng:singularity}}
\label{sec:ng-singularity}

We now prove the two inputs in Step~3.

\begin{proof}[Proof of Lemma~\ref{ng:later-comparison}]
Fix a subsequence $r_n\downarrow0$ along which
$\mu_{t,r_n}^\psi$ converges in distribution to some deterministic measure $\mu_t$. 
Our goal is to show that, with positive probability, \(\mu_t\) has a non-trivial part that is absoluetly continuous w.r.t.\ \(\scrZ_{2t/3,t}^\vartheta(\one,\dd y)\).

\smallskip
\noindent
\textit{Integrating out disorder on \([t/3,2t/3]\). }
Let
$\cG\coloneqq \cF_{[0,t/3]}\vee\cF_{[2t/3,t]}$, where $\cF_{[a, b]}$ denotes the $\sigma$-algebra generated
by the SHF increments $\scrZ_{u, v}^\vartheta$ with $[u, v]\subset [a, b]$. Note that
conditioning on $\cG$ simply integrates out the SHF on $[t/3,2t/3]$.

It follows from Lemma~\ref{ng:determinism}
that for any $\varphi\in C_c(\R^2)$ the convergence 
$\mu_{t,r_n}^\psi(\varphi) \to \mu_t(\varphi)$ also holds in $L^1(\bbP)$, because
the limiting measure $\mu_t$ is deterministic and the measures have bounded mass.
By Jensen's inequality, we also have that
the following conditional expectation converges to the same deterministic limit:
\begin{equation*}
 \bbE\bigl[\mu_{t,r_n}^\psi(\varphi)\mid\cG\bigr]
 \xrightarrow[\;n\to\infty\;]{} \mu_t(\varphi) \qquad \text{in } L^1(\bbP).
\end{equation*}
Along a further subsequence if necessary, this implies that almost surely, the conditional probability distribution \(\bbE[\mu_{t,r_n}^\psi(\cdot)\mid\cG]\) converges vaguely to \(\mu_t\),
that is
\begin{equation}
      \label{ng:conditional-endpoint-limit}
\bbE\bigl[\mu_{t,r_n}^\psi(\cdot)\mid\cG \bigr]
 \xrightarrow[\;n\to\infty\;]{v} \mu_t  \qquad \bbP\text{-a.s.}
\end{equation}

\smallskip
\noindent
\textit{Extracting an absolutely continuous part. }
Recalling the definition~\eqref{ng:profiles} of \(\mu_{t,r}^x(\dd y)\), we can write 
\begin{equation}
\label{ng:conditional-endpoint-ratio}
\bbE\left[\mu_{t,r}^x(\dd y) \mid \cG \right] =
\bbE\biggl[\frac{1}{R_1}\,
   \frac{\scrZ_t^\vartheta(g_r^{(x)},\dd y)}
        {\scrZ_{t/3}^\vartheta(g_r^{(x)},\one)}
   \;\bigg|\;\cG\biggr] \,,
\end{equation}
where \(R_1\) is a correcting factor given by
\[
 R_1= R_1(r,x)\coloneqq \frac{\scrZ_t^\vartheta(g_r^{(x)},\one)}{\scrZ_{t/3}^\vartheta(g_r^{(x)},\one)}.
\]
We prove below that 
\begin{equation}
      \label{ng:middle-conditional-mean}
\bbE\biggl[
   \frac{\scrZ_t^\vartheta(g_r^{(x)},\dd y)}
        {\scrZ_{t/3}^\vartheta(g_r^{(x)},\one)}
   \,\bigg|\,\cG\biggr]
 \le\frac{3}{2\pi t}\scrZ_{2t/3,t}^\vartheta(\one,\dd y) \,.
\end{equation}
To deal with the prefactor \(1/R_1\) inside~\eqref{ng:conditional-endpoint-ratio}, we let \(a>0\) some parameter (chosen sufficiently small below) 
and we introduce 
\begin{equation}
      \label{ng:define-eta}
\eta_{t,r}(\dd y) \coloneqq  \bbE\biggl[
   \int\psi(x)\one_{\{R_1(r,x)\ge a\}}\mu_{t,r}^x(\dd y)\,\dd x
   \;\bigg|\;\cG\biggr]  \,,
\end{equation}
so that almost surely
\begin{equation}
      \label{ng:retained-measure}
 \eta_r(\dd y) \leq \frac{1}{a} 
   \int\psi(x)\bbE\biggl[
   \frac{\scrZ_t^\vartheta(g_r^{(x)},\dd y)}
        {\scrZ_{t/3}^\vartheta(g_r^{(x)},\one)} 
   \;\bigg|\;\cG\biggr]  \dd x  \leq \frac{3}{2\pi ta} \scrZ_{2t/3,t}^\vartheta(\one,\dd y)  \,,
\end{equation}
where the last inequality follows from~\eqref{ng:middle-conditional-mean}.
We also show below that 
that for some nonzero $\varphi_0\in C_c^\infty(\R^2)$ with
$0\le\varphi_0\le1$, the parameter $a$ can be chosen sufficiently small such that, for some $c>0$,
\begin{equation}\label{ng:retained-local-mass}
 \bbE\bigl[\eta_{t,r}(\varphi_0)\bigr]\ge c \quad \mbox{uniformly in $r$ small}.
\end{equation}

The two estimates \eqref{ng:retained-measure}-\eqref{ng:retained-local-mass} allow us to conclude the proof.
Indeed, by~\eqref{ng:retained-measure}, we have that \(\eta_{t,r}\) and any of its weak limit \(\eta_t\) 
are absolutely continuous with respect to \(\scrZ_{2t/3,t}^\vartheta(\one,\dd y)\).
On the other hand, \eqref{ng:retained-local-mass} ensures that any weak limit \(\eta_t\) is non-zero with positive probability.
Additionally, we have \(\eta_{t,r}\leq \bbE[\mu_{t,r}^{\psi}\mid \mathcal{G}]\) as follows directly from~\eqref{ng:define-eta}, and this inequality also remains true by taking weak limits, 
giving \(\eta_t \leq \mu_t\) thanks to~\eqref{ng:conditional-endpoint-limit}.
In other words, we have found some measure \(\eta_t \leq \mu_t\), which is non-zero and absolutely continuous w.r.t.\ \(\scrZ_{2t/3,t}^\vartheta(\one,\dd y)\) with positive probability.

This concludes the argument and it thus only remains to show~\eqref{ng:middle-conditional-mean} and~\eqref{ng:retained-local-mass}.

\smallskip
\noindent
\textit{Proof of~\eqref{ng:middle-conditional-mean}. }
We start by noticing that 
\begin{equation*}
\bbE\biggl[
   \frac{\scrZ_t^\vartheta(g_r^{(x)},\dd y)}
        {\scrZ_{t/3}^\vartheta(g_r^{(x)},\one)}
   \,\bigg|\,\cG\biggr] = \int (\mu_{t/3,r}^x * g_{t/3})(z)\,\scrZ_{2t/3,t}^\vartheta(\dd z,\dd y),
\end{equation*}
where $\mu_{t/3,r}^x(\dd w)= \scrZ_{t/3}^\vartheta(g_r^{(x)},\dd w)/\scrZ_{t/3}^\vartheta(g_r^{(x)},\one)$ as in~\eqref{ng:profiles} and the heat kernel $g_{t/3}$ arises from integrating out the SHF on the time interval $[t/3, 2t/3]$. 
For a proof of such an identity, see e.g.\ \cite[Proposition~2.4]{CM}. 
Since $\mu_{t/3,r}^x$ is a probability measure and \(g_{t/3}(y) \leq 3/(2\pi t)\) uniformly in \(y\),
we have the uniform bound
\begin{equation*}
 (\mu_{t/3,r}^x * g_{t/3})(z)
 =\int g_{t/3}(z-w)\,\mu_{t/3,r}^x(\dd w)
 \le\frac{3}{2\pi t}.
\end{equation*}
This proves~\eqref{ng:middle-conditional-mean}.

\smallskip
\noindent
\textit{Proof of~\eqref{ng:retained-local-mass}. }
By the definition
of $\eta_{t,r}$ in~\eqref{ng:define-eta}, we have
\begin{equation}
      \label{ng:retained-local-mass-start}
 \bbE\bigl[\eta_r(\varphi_0)\bigr]
 =\int\psi(x)\bbE\bigl[\mu_{t,r}^x(\varphi_0)
                 \one_{\{R_1(r,x)\ge a\}}\bigr]\,\dd x.
\end{equation}
We want a positive lower bound on the expectation inside
this integral, 
uniformly for $x\in\operatorname{supp}\psi$
and $r$ small.
valid
uniformly for $r$ small.
For fixed $x$, let us set
\[
 R_2\coloneqq \frac{\scrZ_t^\vartheta(g_r^{(x)},\varphi_0)}
                 {\scrZ_{t/3}^\vartheta(g_r^{(x)},\one)} \,,
\]
so in particular $R_2/R_1=\mu_{t,r}^x(\varphi_0)$.
We thus have
\[
 \bbE\bigl[\mu_{t,r}^x(\varphi_0)\one_{\{R_1\ge a\}}\bigr] = \bbE\Bigr[\frac{R_2}{R_1}\Bigl] - \bbE\Bigr[\frac{R_2}{R_1}\one_{\{R_1< a\}}\Bigl]
 \geq  \bbE\Bigr[\frac{R_2}{R_1}\Bigl] - \bbP(R_1<a) \,,
\]
using also that \(R_2\leq R_1\).
First, notice that by Jensen's inequality we have that
\[
\bbE\Bigl[\frac{R_2}{R_1}\Bigr]
 \ge\exp\bigl(\bbE[\log R_2]-\bbE[\log R_1]\bigr)
 \ge \e^{-C},
\]
where we have used that \(\bbE[\log R_1] \leq \log \bbE[R_1] =0\) since \(\bbE[R_1]=1\) and also that \(\bbE[\log R_2]\ge-C\) uniformly in \(r\), by applying~\eqref{ng:positive-local-mass} at time \(t/3\) and \eqref{ng:log-mean}--\eqref{ng:log-mean2} at time $t$.
On the other hand, by Markov's inequality we have for \(a\in (0,1)\)
\[
\bbP(R_1<a) \leq \frac{\bbE[ (\log R_1)^-]}{\log (1/a)} \leq \frac{\bbE[R_1] - \bbE[\log R_1]}{\log(1/a)} \leq \frac{1+C}{\log(1/a)}
\]
where we have used first that \( (\log b)^-\leq b - \log b\) for all \(b>0\) and then, as above, that \(\bbE[R_1]=1\) and $\bbE[\log R_1]\ge-C$ (by applying \eqref{ng:log-mean} at time $t/3$ and $t$).

All together, we have that uniformly for \(r\) small
\[
\bbE\bigl[\mu_{t,r}^x(\varphi_0)\one_{\{R_1\ge a\}}\bigr]  \geq \e^{-C} - \frac{1+C}{\log (1/a)} \geq \frac12 \e^{-C} \,,
\]
the last inequality being valid for \(a\) small enough.
Plugged in~\eqref{ng:retained-local-mass-start}, this concludes the proof of~\eqref{ng:retained-local-mass}.
\end{proof}

\begin{proof}[Proof of Lemma~\ref{ng:singularity}]
Let \(\mu\) be a non-zero deterministic measure and 
let us decompose~$\mu$ into $\mu=\mu^{\rm a}+\mu^{\rm s}$, where \(\mu^{\rm a}\), \(\mu^{\rm s}\) are the absolutely continuous and
singular parts w.r.t.\ the Lebesgue measure. 
We show that both parts are a.s.\ singular
w.r.t.\ $\scrZ_{2t/3,t}^\vartheta(\one,\dd y)$. 

Indeed, since
$\scrZ_{2t/3,t}^\vartheta(\one,\dd y)$ is a.s.\ singular w.r.t.\ Lebesgue measure by \cite[Theorem~1.1]{CSZ25}, $\mu^{\rm a}$ must be singular
w.r.t.\ $\scrZ_{2t/3,t}^\vartheta(\one,\dd y)$.
On the other hand, let $B$ be a Lebesgue set with \(\mathrm{Leb}(B) =0\) and $\mu^{\rm s}(B^c)=0$.
We have $\bbE[\scrZ_{2t/3,t}^\vartheta(\one,\one_B)]= \mathrm{Leb}(B) =0$ so that $\scrZ_{2t/3,t}^\vartheta(\one,\one_B)=0$ a.s., and thus $\mu^{\rm s}$ is singular w.r.t.\ $\scrZ_{2t/3,t}^\vartheta(\one,\dd y)$.
\end{proof}

\appendix

\section{On the block expansion: proof of Proposition~\ref{prop:block-inputs-iii}}
\label{app:technical-proposition}

The goal of this section is to prove Proposition~\ref{prop:block-inputs-iii}.
The result is similar to that of \cite[Corollary 3.5]{GT26},
but it is not straightforward to adapt their proof to our setting. 
We present here a complete proof for the sake of completeness.

\subsection{Reminder of chaos expansions, second moment and renewal estimates}
We gather in this section a number of auxiliary results on local limit theorems,
second moment estimate and asymptotics.

\medskip
\noindent
{\bf Chaos expansion.}
The estimates in Proposition \ref{prop:block-inputs-iii}, will be done through the following chaos expansion of the
partition functions. More precisely, denote by
\begin{equation*}
      \label{eq:xi}
	\xi_N(n,z) \coloneqq  e^{\beta_N\omega(n,z)-\beta_N^2/2}-1
\end{equation*}
which verifies \(\bbE[\xi_N(n,z)] = 0\) and
\[
\Var[\xi_N(n,z)] =  e^{\beta_N^2}-1 =:\sigma_N^2 \,.
\]
Expanding, essentially, the exponential Gibbs
factor in the partition function, we have the
following expansion for the point-to-point
partition function $Z^{\beta_N}_{m, n}(x, y)$ (we refer to \cite[Section~2]{CSZ26} for details):
\begin{equation}
      \label{eq:expan2}
\begin{split}
Z^{\beta_N}_{m, n}(x, y)
      &= \sum_{k=0}^\infty
      \sum_{\substack{z_1, \ldots, z_{k}\in \Z^2 \\ m <n_1<\cdots <n_{k}\leq n}}
      \Bigl\{\prod_{i=1}^{k} q_{n_i-n_{i-1}}(z_i-z_{i-1}) \xi_N(n_i, z_i) \Bigr\}  q_{n-n_k} (y-z_k)  \\
      & = 
      q_{n-m}(y-x)
      +
      \sum_{k=1}^\infty
      \!\! \sum_{\substack{z_1, \ldots, z_{k}\in \Z^2 \\ m<n_1<\cdots <n_k \le n}} \!\!\!\!\!\!\!\!\!\!
      \begin{tikzpicture}[scale=0.45]
      \draw [fill] (14, 1)  circle [radius=0.1];
      \draw[thick] (0,0) to [out=70,in=160] (2,1);
      \draw [fill] (2, 1)  circle [radius=0.1];
      \draw[thick] (2,1) to [out=-50,in=110] (3,-1);
      \draw [fill] (3, -1)  circle [radius=0.1];
      \draw[thick] (3,-1) to [out=80,in=-180] (4.5,0.5);
      \draw [fill] (4.5, 0.5)  circle [radius=0.1];
      \draw[thick] (4.5,0.5) to [out=60,in=120] (6, 0.5);
      \draw [fill] (6, 0.5)  circle [radius=0.1];
      \draw[thick] (6,0.5) to [out=-80,in=160] (7, -1);
      \draw [fill] (7, -1)  circle [radius=0.1];
      \draw[thick] (7,-1) to [out=60,in=200] (8, 0);
      \node at (9,0) {$\cdots$};\node at (10.2,0) {$\cdots$};
      \draw[thick] (11,1) to [out=10,in=100] (13, 0);
      \draw [fill] (13, 0)  circle [radius=0.1];
      \draw[thick] (13,0) to [out=60,in=180] (14, 1);
      \draw [] (0, 0)  circle [radius=0.1];
      \node at (13,-0.5) {\scalebox{0.6}{$\xi_N(n_k, z_k)$}};
      \node at (2,1.7) {\scalebox{0.6}{$\xi_N(n_1,z_1)$}};
      \node at (0,-0.5)  {\scalebox{0.6}{$(m, x)$}};
      \node at (14.5,0.5)  {\scalebox{0.6}{$(n, y)$}};
      \end{tikzpicture},
\end{split}
\end{equation}
using the convention \((n_0,z_0) =(m,x)\) and \(q_0(x) = \one_{\{x=0\}}\); the term \(k=0\) simply corresponds to \(q_{n-m}(y-x)\).
In the graphical representation, the laces represent the transition kernels $q_{n_i-n_{i-1}}(z_i-z_{i-1})$,
the solid dots
represent the disorder variables $\xi_N(n_i, z_i)$.
(Note that disorder is excluded at time $m$ and included at time $n$.)

\medskip

\noindent
{\bf Second moments.} Using the orthogonality of disorder weights $\xi_N$, we can use the chaos expansion to
obtain formulas for the second moments of the variance of the partition functions.
More precisely, let us introduce the following collision renewal density: set \(U_N(n,x) = \one_{\{x=0\}}\) if \(n=0\) and
\begin{equation}\label{eq:U}
	U_N(n,x) =
	\sum_{k=1}^\infty \sum_{\substack{z_1,\ldots, z_{k-1} \in \Z^2 \\ 0 < n_1 < \ldots < n_{k-1} < n}}
	 \prod_{j=1}^{k} \sigma_N^2\, q_{n_j-n_{j-1}}(z_j-z_{j-1})^2
\end{equation}
if \(n\geq 1\), with again the convention
$(n_0,z_0)=(0,0)$, $(n_{k},z_{k})\coloneqq (n,x)$.
We also denote
\begin{align}\label{U bar}
\overline U_N(n,x) \coloneqq \sigma_N^2 U_N(n,x) \qquad
\overline U_N(n) \coloneqq \sum_{x\in \Z^2} \overline U_N(n,x) \,.
\end{align}
By orthogonality of the chaos expansion, for $n\ge1$ we have
\(\bbE[Z_{0,n}(0,x)^2]
= \frac{1+\sigma_N^2}{\sigma_N^2}
U_N(n,x)\).

Let us also introduce
\begin{equation} \label{dick-green2}
	G_{\vartheta}(t) = \int_0^\infty
	\frac{ e^{(\vartheta - \gamma) s} \,
	s \, t^{s-1}}{\Gamma(s+1)} \, \dd s \,,
	\qquad  \quad t \in (0,1] \,, \
	\vartheta \in \R \,,
\end{equation}
with \(\gamma\) the Euler-Mascheroni constant.
The function $G_\vartheta$ is related to the Dickman subordinator \cite{CSZ19a} for $t\in [0,1]$.
Note also that \(G_{\vartheta}(t) = \e^{\vartheta - \gamma} \nu'( t\, \e^{\vartheta - \gamma})\),
where \(\nu(T) \coloneqq \int_0^\infty \frac{T^s}{\Gamma(s+1)} \dd s\)
is the so-called Volterra function \cite{CM}.

\begin{proposition}
      \label{prop:renewal-UN}
 For any fixed $\delta > 0$ and uniformly in $\delta N \le n \le N$,
as $N\to\infty$, we have
\begin{equation}
      \label{eq:asU1}
	\overline U_N(n) = \frac{\pi}{N} \, G_\vartheta\big(\tfrac{n}{N}\big)
	(1+ o(1)).
\end{equation}
Furthermore, there exists $C>0$ such that uniformly for any $1\le n\le N$, we have the upper bound
\begin{align}\label{eq: U upper bound}
    \overline U_N(n)
    \le 
    \frac{C}{N} G_{\vartheta}(\frac{n}{N}).
\end{align}
Additionally, there exists a constant \(C=C_{\vartheta}>0\) such that for any \(1\leq n \leq N\),
\begin{equation}
\label{eq:uniform-bound-U}
\overline U_N (n,x) \leq C \overline U_N(n) \hat g_{n} (x) \,,
\end{equation}
where \(\hat g_n\) is a \(\Z^2\) version of a heat kernel:
\begin{equation}
      \label{def:hat-g}
\hat g_n(x) \coloneqq \frac{1}{n} \e^{- \cst |x|^2/n}  \one_{\{n+x_1+x_2\in 2\Z\}}\,,
\end{equation}
with a constant \(\cst \leq 1/2\).
By convention, we set \(\hat g_0(x) = \one_{\{x=0\}}\) in the case \(n=0\).
\end{proposition}

The proof of~\eqref{eq:asU1} and \eqref{eq: U upper bound} are both given in \cite[Theorem~2.3]{CSZ19a}.
The bound~\eqref{eq:uniform-bound-U}
does not seem to appear exactly
in this form in the literature (even though it is close to  \cite[Lemma~3.5]{CSZ23}):
we provide a proof in Subsection~\ref{app:technical}.

\smallskip
In the following, we choose the constant \(\cst\)
in the definition of \(\hat g_n\) small enough
such that Lemma~\ref{lem:RW-kernel-gaussian} below holds, and also small enough such that \(q_n(x) \leq C \hat g_n(x)\),
with a constant \(C\) uniform in \(n,x\).
We stress that \(\hat g_n\) may not be of mass one, but has a bounded mass \(\sum_{x \in \Z^2} \hat g_n(x) \leq C\), uniformly in \(n\).

\subsection{Proof of~Proposition~\ref{prop:block-inputs-iii}}

We first give the details
for \eqref{eq:app-centred-string} and then we adapt the
argument to \eqref{eq:app-open-string}.
Recall that
\begin{align*}
D_i(x,y)
=\mathsf{Z}_i(x,y)- \mathsf Z_i(x, \one) q_{N_i}(y)
= {Z}_{N_{i-1},N_i}(x,y)- Z_{N_{i-1}, N_i}(x, \one) q_{N_i}(y).
\end{align*}
Our goal is to compute the second moment of \(D_{u,v}^{\rm closed}=D^{\mathrm{closed}}_{u,v}\coloneqq  q_{N_{u-1}}\lact
  D_u\circ\cdots\circ D_{v-1}
 \circ \mathsf Z_v\ract\one\)
using a chaos expansion for kernels $D_i$ and their compositions.
To better illustrate the scheme, we first present the estimate in the case of two consecutive blocks,
\textit{i.e.}\ we estimate $\|D^{\mathrm{closed}}_{u,u+1}\|_{L^2}$.
We then show how to adapt the idea of the proof to an arbitrary number of blocks,
providing an estimate on $\|D^{\mathrm{closed}}_{u,v}\|_{L^2}$.

\medskip
\noindent
{\bf Proof of~\eqref{eq:app-centred-string} in the case of one $D$-block. }
For a technical reason we first present the case \(u\geq 2\); we explain in Remark~\ref{rem:first-block} how the case \(u=1\) is treated with similar estimates.
By definition of $\|D^{\mathrm{closed}}_{u,u+1}\|_{L^2}$ and independence, we have that
\begin{equation*}
      \bbE \big[ (D^{\mathrm{closed}}_{u,u+1})^2 \big]
       =
      \bbE \big[ (
            q_{N_{u-1}}\lact
            D_u \circ \mathsf Z_{u+1}\ract\one
       )^2 \big]
        =
      q_{N_{u-1}}^{\otimes 2}\lact
      \bbE [D_u^{\otimes 2}] \circ \bbE[\mathsf Z_{u+1}^{\otimes 2}]
      \ract\one^{\otimes 2}.
\end{equation*}
By the chaos expansion \eqref{eq:expan2}, recalling \eqref{U bar},
we can check that
\begin{equation}
      \label{eq:Dclosed-first-step}
      \begin{split}
      \bbE \big[ (D^{\mathrm{closed}}_{u,u+1})^2 \big]
      & =  \,\,\,
      \sum_{\substack{N_{u-1}< a_1 \le b_1 \le N_u  \\
                                    \, \,N_u<a_2\le b_2 \le N_{u+1}}}
      \,  \sum_{x_1,y_1,x_2\in \mathbb Z^2}
      q_{a_1}(x_1)^2
      \,  \overline U_N(b_1-a_1,y_1-x_1) \\
      & \hskip 4cm \times
      \bigl(q_{a_2-b_1}(x_2-y_1) - q_{a_2}(x_2)\bigr)^2
      \, \overline U_N(b_2-a_2) \,.
      \end{split}
\end{equation}
Using the upper bound
in~\eqref{eq:uniform-bound-U} and also \(q_{a_1}(x_1) \leq C \hat g_{a_1}(x_1)\) as noted after Proposition~\ref{prop:renewal-UN},
we can then first focus on controlling the spatial sum.
For \(a_1\leq b_1 < a_2\leq b_2\) fixed
\begin{equation}
\label{eq:spatial-sum-1}
    \sum_{x_1,y_1,x_2 \in \Z^2} \hat g_{a_1}(x_1)^2  \hat g_{b_1-a_1}(y_1-x_1) \bigl(q_{a_2-b_1}(x_2-y_1) - q_{a_2}(x_2)\bigr)^2 \,.
\end{equation}
For this, we may use the following lemma, whose proof is postponed to Subsection~\ref{app:technical}.

\begin{lemma}
\label{lem:RW-kernel-gaussian}
There is a constant \(C>0\) such that, if the constant \(\cst\) is small enough in the definition~\eqref{def:hat-g} of \(\hat g_n\), then for any \(k,\ell\)
\begin{equation}
\label{eq:RW-convolution}
\sum_{x\in \Z^2} \hat g_{k}(x) \hat g_{\ell}( y-x) \leq C  \hat g_{k+\ell}(y)
\end{equation}
and for any \(k <\ell\)
\begin{equation}
\label{eq:RW-kernel-square}
    \sum_{y\in \Z^2} \hat g_{k}(y) \bigl[q_{\ell-k}(w-y) - q_{\ell}(w)\bigr]^2  \leq  C\, \frac{k}{\ell(\ell-k)} \hat g_{\ell}(w) \,.
\end{equation}
\end{lemma}

With this lemma at hand, in~\eqref{eq:spatial-sum-1}
we can first bound \(\hat g_{a_1}(x_1)^2 \leq  \frac{1}{a_1} \hat g_{a_1}(x_1)\) 
then sum over~\(x_1\) (thanks to~\eqref{eq:RW-convolution}),
then sum over \(y_1\) (thanks to~\eqref{eq:RW-kernel-square})
and finally sum over \(x_2\) (with \(\sum_{w} \hat g_s(w) \leq C\)):
we end up with
\begin{equation}
      \label{eq:spatial-summing}
\sum_{x_1,y_1,x_2 \in \Z^2} q_{a_1}(x_1)^2  \hat g_{b_1-a_1}(y_1-x_1) \bigl(q_{a_2-b_1}(x_2-y_1) - q_{a_2}(x_2)\bigr)^2
\leq C \frac{b_1}{a_1 a_2 (a_2-b_1)} \,.
\end{equation}
Going back to~\eqref{eq:Dclosed-first-step},
this gives
\begin{equation*}
      \bbE \big[ (D^{\mathrm{closed}}_{u,u+1})^2 \big]
       \leq C  \,\,\,
      \sum_{\substack{N_{u-1}< a_1 \le b_1 \le N_u  \\
      \, \,N_u<a_2\le b_2 \le N_{u+1}}} \, \frac{b_1}{a_1}
      \, \overline U_N(b_1-a_1) \,  \frac{1}{a_2 (a_2-b_1)}
      \, \overline U_N(b_2-a_2) \,.
\end{equation*}
Then, thanks to Proposition~\ref{prop:renewal-UN}-\eqref{eq: U upper bound} and a Riemann sum approximation,
employing the shorthand notation
$\bar N_u=N_u/N$,
we obtain that
\begin{equation}
      \label{eq:bound-D-integral}
      \bbE \big[ (D^{\mathrm{closed}}_{u,u+1})^2 \big]
      \leq C
      \iint\limits_{\substack{\bar N_{u-1}< s_1 \le t_1<\bar N_u  \\
      \bar N_u<s_2\le t_2<\bar N_{u+1}}}
      \,\,
       \frac{t_1}{s_1} G_\vartheta (t_1-s_1) \, \frac{1}{s_2(s_2-t_1)} \, \,
      G_\vartheta (t_2-s_2) \,\,  \dd \vec s \,  \dd \vec t \,.
\end{equation}
It remains to estimate this time integral.
(Note that the integral is infinite in the case \(u=1\), we refer to Remark~\ref{rem:first-block} for an explanation on how to adapt the estimate in that case.)

First, using that \(\int_0^t G_\vartheta(s)\,\dd s \leq c (\log\frac{1}{t})^{-1}\)
as \(t\downarrow 0\), see \cite[Proposition~1.6]{CSZ19a},
we get that
\(\int_{s_m\leq t_m \leq \bar N_{i_m}}G_\vartheta(t_m-s_m) \dd t_m \leq C/ \log (\bar N_{i_m}^{-1})\).
Thus, we have that
\begin{equation}
      \label{eq:G-last-term}
\iint\limits_{\bar N_{u}<s_2\le t_2\le \bar N_{u+1}}
 	G_\vartheta (t_2-s_2)\frac{t_{1} \, \dd t_2 \,\dd s_2}{s_2(s_2-t_{1})}
\leq	 \frac{C}{\log(\bar N_{u+1}^{-1})}
	 \int_{\bar N_{u}}^{\bar N_{u+1}}
	 \frac{t_{1} \, \dd s_2}{s_2(s_2-t_{1})} \,.
 \end{equation}
For the remaining integral over \(s_2\), we have that
\begin{equation*}
\int_{\bar N_{u}}^{\bar N_{u+1}}
\frac{t_{1} \, \dd s_2}{s_2(s_2-t_{1})}
= \int_{\bar N_{u}}^{\bar N_{u+1}}
	 \Big( \frac{1}{s_2-t_{1}} -\frac{1}{s_2} \Big)
=
\int^{\bar N_{u+1}-t_1}_{\bar N_u-t_1} \frac{\dd s_2}{s_2}
-
\int^{\bar N_{u+1}}_{\bar N_u} \frac{\dd s_2}{s_2}
\leq \int_{\bar N_{u}-t_{1}}^{\bar N_{u}}
	  \frac{\dd s_2}{s_2} \,.
\end{equation*}
We can then distinguish the cases  $t_{1}\le \bar N_{u}/2$ and
$t_{1} > \bar N_{u}/2$.
In the former case, the integral is at most $2t_{1}/ \bar N_{u}$ by a uniform bound on the integrand,
while in the latter we can use the exact value of the integral, which is
$\log (\frac{\bar N_{u}}{\bar N_{u}-t_{1}})$.
It is not difficult to check that the two cases
can be combined to obtain the bound
\begin{equation}
\label{eq:introduce-backward-weight}
\int_{\bar N_{u}}^{\bar N_{u+1}}
\frac{t_{1} \, \dd s_2}{s_2(s_2-t_{1})} 
\leq C\,  \frac{t_1}{\bar N_u} \log \Bigl( \frac{2 \bar N_{u}}{ \bar N_{u} - t_1}\Bigr) \,.
\end{equation}
Going back to~\eqref{eq:bound-D-integral}, and using that \(t_1 \leq \bar N_u\),
it gives that
\begin{equation*}
      \bbE \big[ (D^{\mathrm{closed}}_{u,u+1})^2 \big]
      \leq  \frac{C}{\log(\bar N_{u+1}^{-1})} \iint\limits_{\bar N_{u-1}< s_1 \le t_1<\bar N_u }\frac{1}{s_1}  G_\vartheta (t_1-s_1)  \log \Bigl( \frac{2 \bar N_{u}}{ \bar N_{u} - t_1}\Bigr)  \dd s_1 \dd t_1  \,.
\end{equation*}
Now, one can prove (see Subsection~\ref{app:technical} below)
that for any
\(0<2\mathsf S\le\mathsf T<1\),
\begin{equation}
      \label{eq:backward-weight-initial}
      \iint\limits_{\mathsf S< s \le t< \mathsf T}\frac{1}{s}  G_\vartheta (t-s)  \log \Bigl( \frac{2 \mathsf T}{ \mathsf T - t}\Bigr)  \dd s \dd t \leq \frac{C}{|\log \mathsf T|} \log (\mathsf T/\mathsf S) \,.
\end{equation}
All together, we end up with
\begin{equation*}
     \bbE \big[ (D^{\mathrm{closed}}_{u,u+1})^2 \big]  \leq \frac{C}{\log(\bar N_{u+1}^{-1})} \frac{C \log(\bar N_u/\bar N_{u-1})}{ \log(\bar N_u^{-1})} \leq  \frac{C /\log \rho}{ (K-u) (K-u-1)} \,,
\end{equation*}
where we used the definition
\eqref{eq:adapted-grid} of $N_i$ and $\bar N_i=N_i/N$, which give
\[
 \log\frac{\bar N_u}{\bar N_{u-1}}\le C\log\rho,
 \qquad
 \log\frac1{\bar N_i}\ge c(K-i)\log\rho,
 \qquad 1\le i\le K-1.
\]

This completes the proof of \eqref{eq:app-centred-string}
when there is only one block of~$D_u$ in the expansion.

\begin{remark}
      \label{rem:first-block}
In the case $u=1$, the integral in \eqref{eq:bound-D-integral} diverge at $s_1=0$.
In that case, in the Riemann sum approximation one can in fact keep 
the restriction
$s_1>1/(2N)$ in the integral comparison, so that 
\[
\bbE \big[ (D^{\mathrm{closed}}_{1,2})^2 \big]
      \leq C
      \iint\limits_{\substack{1/(2N)< s_1 \le t_1<\bar N_1  \\
      \bar N_1<s_2\le t_2<\bar N_{2}}}
      \,\,
       \frac{t_1}{s_1} G_\vartheta (t_1-s_1) \, \frac{1}{s_2(s_2-t_1)} \, \,
      G_\vartheta (t_2-s_2) \,\,  \dd \vec s \,  \dd \vec t \,.
\]
Then the proof proceeds identically, and the required bound follows from
\eqref{eq:backward-weight-initial} applied with
$\mathsf S=1/(2N)$ and $\mathsf T=\bar N_1$,
since $ \log (\mathsf T /\mathsf S) =\log(2 N_1) \le C\log\rho$.
\end{remark}

\medskip
\noindent
{\bf Proof of~\eqref{eq:app-centred-string} in the general case of several \(D\)-blocks ($v\ge u+2$). }
Let us denote by $i_1<i_2<\cdots < i_m$ (with \(m\leq v-u+1\)) the indices of the blocks where disorder terms are sampled in the chaos expansion of~\(D_i\) or of $Z_v\ract \one$.
We note that necessarily $i_1=u$ and $i_m=v$.
Indeed, if we denote by $\widetilde D^{\mathrm{closed}}_{u,v}$
the contribution to $D^{\mathrm{closed}}_{u,v}$
by restricting the chaos expansion of
$Z_v \ract \one$ to just the constant $1$, then
\begin{align*}
	\widetilde D^{\mathrm{closed}}_{u,v}
\coloneqq  	q_{N_{u-1}}\lact
 	D_u\circ\cdots\circ D_{v-1}
	 \ract\one
= 0,
\end{align*}
since $D_{v-1} \ract\one = 0$ by \eqref{D-mass}.
Similarly, if there is no disorder sampled in the chaos expansion of $D_u$ and we denote
this restriction by $\widetilde D_u$, then
\begin{align*}
q_{N_{u-1}}\lact \widetilde D_u  (y)
=\sum_{x\in \Z^2} q_{N_{u-1}}(x) \big( q_{N_u-N_{u-1}}(y-x)- q_{N_u}(y) \big)
=0\,.
\end{align*}

Using the chaos expansion,
we can thus write the second moment as
\begin{equation}
      \label{eq:second-moment-D-general}
\begin{split}
      \bbE \big[ (D^{\mathrm{closed}}_{u,v})^2 \big]
      & =
      \sum_{m=2}^{v-u+1}\sum_ {u=i_1<\cdots<i_m=v} \,\,
      \sum_{\substack{  N_{i_j-1}< a_j \le b_j\le N_{i_j}  \\
            x_1,y_1,\ldots,x_m,y_m\in\mathbb Z^2 }}
      q_{a_1}(x_1)^2 \overline U_N(b_1-a_1,y_1-x_1)
      \\
      & \hskip 1.5cm \times \prod_{j=2}^m \bigl[
      q_{a_{j}-b_{j-1}}(x_j-y_{j-1}) - q_{a_j}(x_j)
      \bigr]^2  \overline U_N(b_j-a_j,y_j-x_j) .
\end{split}
\end{equation}

We now follow the same steps as in the simpler case treated above.
We use the bound
\(\overline U_N(b_j-a_j,y_j-x_j) \leq C \overline U_N(b_j-a_j) \hat g_{b_j-a_j}(y_j-x_j)\)
from Proposition~\ref{prop:renewal-UN}
and we then first focus,
for fixed \((a_i,b_i)_{i=1}^m\),
on the sum over the spatial coordinates
\[
\sum_{ x_1,y_1,\ldots,x_m,y_m\in\mathbb Z^2}
      q_{a_1}(x_1)^2
     \prod_{j=1}^{m}  \hat g_{b_{j}-a_{j}}(y_{j}-x_{j})
      \prod_{j=2}^m \bigl[q_{a_{j}-b_{j-1}}(x_{j}-y_{j-1}) - q_{a_{j}}(x_{j})
      \bigr]^2  \,.
\]
Then, bounding \(q_{a_1}(x_1)^2\leq \frac{C}{a_1} \hat g_{a_1}(x_1)\)
and applying Lemma~\ref{lem:RW-kernel-gaussian} to sum
successively over \(x_1,y_1, \ldots, x_m, y_m\)
(\textit{i.e.}\ we sum the space indices in their ``chronological order'' and we use~\eqref{eq:RW-convolution} when summing over \(x_j\) and~\eqref{eq:RW-kernel-square} when summing over \(y_i\)),
we obtain similarly to \eqref{eq:spatial-summing} that the above is bounded by
\[
C^m \frac{1}{a_1} \prod_{j=2}^{m} \frac{b_{j-1}}{a_j (a_j-b_{j-1})} \,.
\]

Plugged inside~\eqref{eq:second-moment-D-general} and
using Proposition~\ref{prop:renewal-UN}-\eqref{eq: U upper bound}
to deal with \(\overline U_N(b_j-a_j)\),
a Riemann sum approximation then yields, analogously to~\eqref{eq:bound-D-integral},
that \(\bbE \big[ (D^{\mathrm{closed}}_{u,v})^2 \big]\) is bounded by
\begin{equation}\label{sec5:2nd-mom}
\sum_{m=2}^{v-u+1} C^m \sum_{u=i_1<\cdots<i_m=v}  \;\;
\iint\limits_{\substack{\bar N_{i_j-1}<s_j\le t_j\le \bar N_{i_j}  \colon\\ 1\le j\le m }}
  \frac{1}{s_1}\prod_{j=1}^{m} G_\vartheta (t_j-s_j)
\prod_{j=2}^{m}\frac{t_{j-1}}{s_j (s_j-t_{j-1})}
\dd \vec s \, \dd \vec t  \,,
\end{equation}
where we recall the shorthand notation \(\bar N_u = N_u/N\).
Again, in the case \(u=1\), the above integral is in fact infinite: one simply needs to replace \(\bar{N}_0\) by \(1/(2N)\) in the Riemann sum approximation and proceed as in Remark~\ref{rem:first-block}.

We now integrate the time variables backwards,
fixing for the moment $m$ and the indices $u=i_1<\cdots<i_m=v$.
First, for the last integral, similarly to~\eqref{eq:G-last-term}--\eqref{eq:introduce-backward-weight},
we get that
\begin{equation}
      \label{sec5:last-block}
      \iint\limits_{\bar N_{i_m-1}<s_m\le t_m\le \bar N_{i_m}}
       	G_\vartheta (t_m-s_m)\frac{t_{m-1} \, \dd t_m \,\dd s_m}{s_m(s_m-t_{m-1})}
      \leq	 \frac{C}{\log(\bar N_{i_m}^{-1})} \times \frac{t_{m-1}}{\bar N_{i_m-1}} \mathsf{W}_{\bar N_{i_{m-1}}}(t_{m-1}) \,,
\end{equation}
where we have defined the following ``backward weight''
\begin{align}\label{eq:backward-weight}
 \mathsf{W}_\mathsf{T}(\tau)\coloneqq \log\frac{2\mathsf{T}}{\mathsf{T}-\tau}, \qquad \text{for $\mathsf{T}>0$},
\end{align}
The reason for expressing the bound~\eqref{sec5:last-block}
in this way is that it then allows us to use
the following estimate iteratively
(its proof is postponed to Subsection~\ref{app:technical}).

\begin{lemma}
      \label{lem:w-iter}
There is a constant \(C>0\) such that,
if $1\geq \mathsf{T} \geq 4 \mathsf{S} $
and $\tau \leq \mathsf{S}$,
\begin{equation}
\label{eq:backward-block-bound}
\iint_{\mathsf{S}<s<t<\mathsf{T}} G_\vartheta(t-s) \frac{\tau}{s(s-\tau)} \, \mathsf{W}_\mathsf{T}(t)  \dd s \,\dd t
\leq	\frac{C}{|\log \mathsf{T}|} \,\frac{\tau}{\mathsf{S}} \,\mathsf{W}_{\mathsf{S}}(\tau)\,,
\end{equation}
\end{lemma}

Before we start the iteration,
let us first make some useful observations.
First of all, since $t_{j-1}\le \bar N_{i_{j-1}}$, we have
\begin{align}\label{sec5:t/N}
 \frac{t_{j-1}}{\bar N_{i_j-1}}
 \le \frac{\bar N_{i_{j-1}}}{\bar N_{i_j-1}}
 \le \Bigl( \frac{C}{\rho} \Bigr)^{i_j-i_{j-1}-1}.
\end{align}
Second, since for any fixed $\tau$
the function $\mathsf{T}\mapsto \mathsf{W}_\mathsf{T}(\tau)$ is non-increasing for $\mathsf{T}\ge\tau$,
we also have for $t_{j-1}\le \bar N_{i_{j-1}}\le \bar N_{i_j-1}$, 
\begin{align}\label{sec5:w<w}
 \mathsf{W}_{\bar N_{i_j-1}}(t_{j-1})\le \mathsf{W}_{\bar N_{i_{j-1}}}(t_{j-1}).
\end{align}

To start the iteration, we first apply \eqref{sec5:t/N} and \eqref{sec5:w<w} to \eqref{sec5:last-block}, to obtain
\begin{align*}
&	  \iint\limits_{\bar N_{i_m-1}<s_m\le t_m\le \bar N_{i_m}}
	  \hskip -5mm
 	  G_\vartheta (t_m-s_m)\frac{t_{m-1} \, \dd t_m \,\dd s_m}{s_m(s_m-t_{m-1})}
	  \notag
\, \leq \,	 \frac{C}{\log\bar N_{i_m}^{-1}}
 	 \Big(\frac{C}{\rho}\Big)^{i_m-i_{m-1}-1}
 	 \mathsf{W}_{\bar N_{i_{m-1}}}(t_{m-1}),
\end{align*}
which when plugged  into \eqref{sec5:2nd-mom} gives
\begin{align*}
  	  \bbE \big[ (D^{\mathrm{closed}}_{u,v})^2 \big]
& \le
       	  \sum_{m=2}^{v-u+1}
	  \sum_{u=i_1<\cdots<i_m=v}
	 \hskip 4mm
	 \iint\limits_{\substack{\bar N_{i_j-1}<s_j\le t_j\le \bar N_{i_j}  \colon\\ 1\le j\le m-1 }}
	 \, \frac{C^m}{s_1}\prod_{j=1}^{m-1}
	  G_\vartheta (t_j-s_j)
	\\
&\hskip 0.6cm \times
 	   \, \prod_{j=2}^{m-1}\frac{t_{j-1}}{s_j(s_j-t_{j-1})}
 	   \frac{C}{\log(\bar N_{i_m}^{-1})}
	   \, \Big(\frac{C}{\rho}\Big)^{i_m-i_{m-1}-1}
	 \mathsf{W}_{\bar N_{i_{m-1}}}(t_{m-1})  \,\dd\vec s \dd \vec t \,.
\end{align*}
We then run the recursion, applying Lemma~\ref{lem:w-iter} for
$j=m-1,\ldots,2$, with
$\mathsf{S}=\bar N_{i_j-1}$, $\mathsf{T}=\bar N_{i_j}$ and $\tau=t_{j-1}$ and using~\eqref{sec5:t/N}--\eqref{sec5:w<w};
note that the assumption $\mathsf{T} \ge 4 \mathsf{S}$ holds whenever $\rho$ is sufficiently large.
We then obtain
\begin{equation*}
      \begin{split}
          \bbE \big[ (D^{\mathrm{closed}}_{u,v})^2 \big]
          & \le
          \sum_{m=2}^{v-u+1}\sum_{u=i_1<\cdots<i_m=v}
          C^m \Big(\frac{C}{\rho}\Big)^{\sum^m_{j=2}(i_j-i_{j-1}-1)}
          \prod_{j=2}^m \frac{1}{\log (\bar N_{i_j}^{-1}) }
           \\
          & \hskip 4cm \times
          \iint_{\bar N_{u-1}<s_1\le t_1\le \bar N_u}
          \frac{ G_\vartheta (t_1-s_1)}{s_1} \, \mathsf{W}_{\bar N_u}(t_1) \,\dd s_1 \, \dd t_1
          \,,
      \end{split}
\end{equation*}
with the last integral bounded by
\(C \log(\bar N_u/\bar N_{u-1})/\log(\bar N_u^{-1})
\le C \log\rho/\log(\bar N_u^{-1})\),
thanks to \eqref{eq:backward-weight-initial}.

Using also that \(\sum_{j=2}^{m}(i_j-i_{j-1}-1)=v-u+1-m\),
we obtain that
\begin{equation*}
  \bbE \big[ (D^{\mathrm{closed}}_{u,v})^2 \big] \leq  \log \rho \sum_{m=2}^{v-u+1} C^m  \Big(\frac{C}{\rho}\Big)^{v-u+1-m} \sum_{u=i_1<\cdots<i_m=v}
    \prod_{j=1}^{m}\frac{C}{\log\bar N_{i_j}^{-1}}
    \, .
\end{equation*}

Now, we have that \(\log(\bar N_{u}^{\,-1}) \ge C (K-u) \log \rho\) and
\(\log(\bar N_v^{-1}) \geq C (K-v) \log \rho\).
For the indices $j=2,\ldots,m-1$, since \(i_j\leq v\leq K-1\) we simply use that
$\log \bar N_{i_j}^{-1} 
\ge c(K-i_j)\log\rho
\ge c\log\rho$.
On the other hand the number of choices for the indices $i_2<\cdots < i_{m-1}$ is
$\binom{v-u-1}{m-2}$.
Putting everything together, we obtain
\begin{align*}
 \bbE\big[(D^{\mathrm{closed}}_{u,v})^2\big]
 &\le\frac{C}{(K-u)(K-v)\log\rho}
 \sum_{m=2}^{v-u+1}\binom{v-u-1}{m-2}
 \bigg(\frac{C}{\log\rho}\bigg)^{m-2}
 \bigg(\frac C\rho\bigg)^{v-u+1-m}\\
 &=\frac{C}{(K-u)(K-v)\log\rho}
 \bigg(\frac{C}{\log\rho}+\frac C\rho\bigg)^{v-u-1} \leq  \frac{C  (C'/\log \rho)^{v-u} }{(K-u)(K-v)}.
\end{align*}
This concludes the proof of \eqref{eq:app-centred-string}
by taking the square root
and possibly enlarging \(\rho\)
to ensure that \(\gep_{\rho}^2\coloneqq C'/\log \rho <1\).

\medskip
\noindent
{\bf Proof of~\eqref{eq:app-open-string}.}
The proof to \eqref{eq:app-open-string} is almost identical to
that of \eqref{eq:app-centred-string}
and follows the same steps.
Let us briefly describe how they adapt here.

The second moment  of \(q_{N_{u-1}}\lact D_u\circ\cdots\circ D_{v-1} \circ Z_{N_{v-1},N}\ract\one\)
can be decomposed similarly as in~\eqref{eq:second-moment-D-general},
except that the last indices \(a_m,b_m\) are in the time-interval \(N_{v-1} < a_m \leq b_m \leq N\).

Again, we can decouple the time and space variable using
the bound from Proposition~\ref{prop:renewal-UN},
then sum over the space variables as above.
We end up with the same bound as in~\eqref{sec5:2nd-mom},
but with the last time indices verifying
\(\bar N_{v-1} < s_m \leq t_m \leq 1\).
The bound~\eqref{sec5:last-block} is then replaced by
\[
\begin{split}
      \iint\limits_{\bar N_{v-1}<s_m\le t_m\le 1}
            G_\vartheta (t_m-s_m)\frac{t_{m-1} \, \dd t_m \,\dd s_m}{s_m(s_m-t_{m-1})}
      & \leq	 C  \frac{t_{m-1}}{\bar N_{i_m-1}} \mathsf{W}_{\bar N_{i_m-1}}(t_{m-1}) \\
      & \leq C \Bigl(\frac{C}{\rho}\Bigr)^{i_{m}- i_{m-1}-1}\mathsf{W}_{\bar N_{i_{m-1}}}(t_{m-1}) \,,
\end{split}
\]
using again~\eqref{sec5:t/N}--\eqref{sec5:w<w} (here by convention \(i_m=v\)).

We can then apply the same iteration as above: thanks to Lemma~\ref{lem:w-iter} (and~\eqref{sec5:t/N}--\eqref{sec5:w<w}),
we get that the second moment of
\(q_{N_{u-1}}\lact D_u\circ\cdots\circ D_{v-1}\circ Z_{N_{v-1},N}\ract\one\) is bounded by
\begin{align*}
& C \log \rho \sum_{m=2}^{v-u+1} C^m  \Big(\frac{C} {\rho}\Big)^{v-u+1-m} \sum_{u=i_1<\cdots<i_m=v}
      \prod_{j=1}^{m-1}\frac{C}{\log\bar N_{i_j}^{-1}}   \\
&\quad \le\frac{C}{K-u}
\sum_{m=2}^{v-u+1}\binom{v-u-1}{m-2}
 \bigg(\frac{C}{\log\rho}\bigg)^{m-2}
 \bigg(\frac C\rho\bigg)^{v-u+1-m} \leq \frac{C}{K-u} \Bigl( \frac{C}{\log \rho} + \frac{C}{\rho}\Bigr)^{v-u-1} \,.
\end{align*}
This shows \eqref{eq:app-open-string}.

\subsection{Some technical estimates}
\label{app:technical}

In this section, we prove the bound~\eqref{eq:uniform-bound-U} in Proposition~\ref{prop:renewal-UN},
and also Lemmas~\ref{lem:RW-kernel-gaussian}
and~\ref{lem:w-iter}.

\begin{proof}[Proof of \eqref{eq:uniform-bound-U} in Proposition~\ref{prop:renewal-UN}]
For any fixed \(k \geq 1\) and \(m_1,\ldots, m_k\geq 1\) with \(m_1+\cdots+m_k=n\),
let \((\zeta_i)_{1\leq i\leq k}\) be independent random variables with distribution \(\rmP(\zeta_j =z) = q_{m_j}(z)^2/q_{2m_j} (0)\).
Then, the proof of~\eqref{eq:uniform-bound-U} reduces to showing that
\begin{equation}
      \label{eq:LLT-difficult}
\rmP(\zeta_1 +\cdots + \zeta_k = x) \leq \frac{C}{n} \, \e^{ -c |x|^2/n}\,,
\end{equation}
for some absolute constants \(c,C>0\) which do not depend on \(k, n_1,\ldots, n_k\) and \(x\).
Indeed, plugged in the definition~\eqref{eq:U} and summing over the space variables (and setting \(m_i=n_{i}-n_{i-1}\)), we can rewrite it as
\[
\begin{split}
U_N(n,x) & = \sum_{k=1}^{\infty} \sum_{0 < n_1 < \ldots < n_{k-1} < n}
      \prod_{j=1}^{k} \sigma_N^2\, q_{2(n_j-n_{j-1})}(0) \rmP(\zeta_1 +\cdots + \zeta_k = x)\\
      & \leq \frac{C}{n} \, \e^{ -c |x|^2/n} \sum_{k=1}^{\infty} \sum_{0 < n_1 < \ldots < n_{k-1} < n}
      \prod_{j=1}^{k} \sigma_N^2\, q_{2(n_j-n_{j-1})}(0)
      = \frac{C}{n} \, \e^{ -c |x|^2/n} U_N(n) \,,
\end{split}
\]
which is the desired inequality.
We are thus reduced to proving~\eqref{eq:LLT-difficult}.

Using that the coordinates $S_n^1+S_n^2$ and
$S_n^1-S_n^2$ are independent one-dimensional simple random walks, we can further reduce~\eqref{eq:LLT-difficult}
to a one-dimensional problem.
In particular \(q_{m}(z) = p_m(z^1+z^2)p_m(z^1-z^2)\), where \(p_m(x)\) is the transition kernel of the simple random walk on \(\Z\).
Letting \(X_i = \zeta_i^1+\zeta_i^2\) and \(Y_i=\zeta_i^1-\zeta_i^2\), we thus get that \(X_i,Y_i\) are independent, with law
\(\rmP(X_i =w)= p_{m_i}(w)^2/p_{2m_i}(0)\).
We are thus reduced to proving the following one-dimensional local limit theorem:
\begin{equation}
      \label{eq:LLT-difficult-1D}
\rmP(X_1 +\cdots X_k =w ) \leq \frac{C}{\sqrt{n}} \, \e^{ -c |w|^2/n}\,.
\end{equation}

First of all, notice that from \cite[Equation~(3.49)]{CSZ23}
we get that for any \(\lambda \geq 0\)
\[
\rmE[\e^{\lambda X_i}] \leq \e^{c \lambda^2 m_i} \,.
\]
(The proof is for \(\zeta_i\) in place of \(X_i\) but can be applied to \(X_i\) by using Cauchy--Schwarz's inequality and adjusting the constant.)
Using Chernoff's bound, this thus gives, say for \(w\geq 0\),
\[
\rmP( X_1 +\cdots X_k \geq w) \leq \inf_{\lambda \geq 0} \e^{-\lambda w +c \lambda^2 n} = 
\e^{-w^2/(4cn)}
\]
recalling first that \(m_1+\cdots+m_k =n\)
and then choosing \(\lambda = w/(2cn)\) for the second inequality.
We can turn this into a local bound as in~\eqref{eq:LLT-difficult-1D}, by observing that the law of \(X_1+\cdots +X_k\) is unimodal
(on \(\{w\in \Z, w=n \mod 2\}\)) as
the convolution of unimodal laws.
From this we deduce that, for \(w\geq 4\),
\[
\rmP( X_1 +\cdots X_k \in  [w/2,w]) \geq c w \rmP(X_1+\cdots + X_k =w)  \,,
\]
since there are at least \(cw\) elements in \([w/2,w]\) with the same parity as \(n\).
Combined with the tail bound above, this yields
\[
\rmP(X_1+\cdots X_k =w) \leq \frac{C}{|w|} \e^{- c |w|^2/n} \,,
\]
which gives the bound~\eqref{eq:LLT-difficult-1D} in the case \(|w| \geq \sqrt{n}\).

To deal with the case \(|w|\leq \sqrt{n}\),
we make the observation that
since \(p_{m_i}(w) = 2^{-m_i} \binom{m_i}{(m_i-w)/2}\) whenever \(w\) has the same parity as \(m_i\),
the random variable \(W_i\coloneqq \frac12 (m_i-X_i)\) has a hypergeometric distribution of parameters \(m_i,2m_i\).
In particular \(W_i\) has a (discrete) log-concave distribution, and so does \(W_1+\cdots +W_k\).
Thus \cite[Theorem~1.1]{BMM20} gives
that 
\[
\sup_{w} \rmP(W_1+\cdots +W_k =w) \leq \Bigl( 1+ 2 \sum_{i=1}^k\Var(W_i)\Bigr)^{-1/2} \leq \bigl(1+ n/4\bigr)^{-1/2} \,,
\]
since \(\Var(W_i)= \frac{m_i^2}{4(2m_i-1)} \geq \frac18 m_i\).
We thus conclude that
\[
\sup_{w} \rmP(X_1+\cdots +X_k =w) \leq \frac{C}{\sqrt{n}} \,,
\]
which shows~\eqref{eq:LLT-difficult-1D}
in the case \(|w|\leq \sqrt{n}\).
\end{proof}

\begin{proof}[Proof of Lemma~\ref{lem:RW-kernel-gaussian}]
The proof of the convolution bound~\eqref{eq:RW-convolution}
is a simple consequence of the following classical inequality for \(a,b\geq 1\)
\begin{equation}
      \label{eq:square-completion}
      \sum_{y}  \e^{-|y|^2/a} \e^{- |x-y|^2/b} \leq \frac{C ab}{a+b} \e^{- |x|^2/(a+b)} \,,
\end{equation}
which is proven by a completion of the square.
We thus focus on the proof of~\eqref{eq:RW-kernel-square}.

First of all, if \(k<\ell \leq 3 k\), we first use that \((a+b)^2\leq 2 a^2 +2 b^2\) and we bound \(q_{\ell-k}(w-y)^2 \leq \frac{C}{\ell-k} q_{\ell-k}(w-y)\) and similarly for \(q_{\ell}(w)^2\) to get that
\[
 \sum_y \hat g_k(y)[q_{\ell-k}(w-y)-q_\ell(w)]^2
 \le\frac{C}{\ell-k}\sum_y \hat g_k(y) q_{\ell-k}(w-y) + \frac{C}{\ell} q_{\ell}(w)\,,
\]
using also that \(\sum_y \hat g_k(y) \leq C\) for the second term.
Now, we may bound
\(q_{\ell-k}(w-y)\le C \hat g_{\ell-k}(w-y)\)
and use~\eqref{eq:RW-convolution} for the first term
and simply \(q_{\ell}(w) \leq C \hat g_{\ell}(w)\)
for the second term:
we obtain that
\[
\sum_y \hat g_k(y)[q_{\ell-k}(w-y)-q_\ell(w)]^2 \leq \frac{C}{\ell-k}  \hat g_{\ell}(w) \leq C' \frac{k}{\ell (\ell-k)} \hat g_{\ell}(w) \,.
\]

In the case where \(\ell >3k\),
using \(q_{\ell}=q_k*q_{\ell-k}\) and Jensen's
inequality, we get that
\[
[q_{\ell-k}(w-y)-q_\ell(w)]^2 \leq \sum_z q_k(z) [q_{\ell-k}(w-y)-q_{\ell-k}(w-z)]^2 \,.
\]
We now may use the following Gaussian gradient bound (for $z,z'$ in the same parity class)
\begin{equation*}
      |q_n(z)-q_n(z')|^2\le \frac{C|z-z'|^2}{n^{3}}
  \Bigl[\e^{- c_1 |z|^2/n} + \e^{-c_1 |z'|^2/n}\Bigr] \,,
\end{equation*}
which is e.g.\ a direct consequence from the gradient bound \cite[Theorem~5.1, Equation~(14)]{HS93} (applied to the two-step random walk on \(\Z^2_{\rm even}\)).
Bounding also \(q_k(z)\leq C \hat g_k(z)\), this gives that
\[
\sum_{y} \hat g_k(y)[q_{\ell-k}(w-y)-q_\ell(w)]^2
\leq \frac{C}{(\ell-k)^3} \sum_{y,z} \hat g_k(y)\hat g_k(z)\, |z-y|^2
\Bigl[\e^{-c_1 \frac{|w-y|^2}{\ell-k}} + \e^{-c_1 \frac{|w-z|^2}{\ell-k}} \Bigr] \,.
\]
By symmetry, we only have to treat one of the two terms in the sum,
say the first one (with \(\e^{-c |w-y|^2/(\ell-k)}\)).
Bounding \(|z-y|^2 \leq 2|z|^2 + 2 |y|^2\) and using that
\(\sum_{z} \hat g_k(z) |z|^2 \leq C k\),
we get that the above is bounded by
\[
\frac{C}{(\ell-k)^3} \sum_{y} \hat g_k(y) (k+|y|^2)
\, \e^{-c_1 \frac{|w-y|^2}{\ell-k}}
\leq  \frac{C'}{(\ell-k)^3} \sum_{y} \e^{- \frac{\cst}{2} \frac{|y|^2}{k}}
\e^{-c_1 \frac{|w-y|^2}{\ell-k}}  \,,
\]
where we have
used that \((k+|y|^2) \hat g_k(y) \leq C \e^{ - \cst |y|^2/2k}\)
(recall the definition~\eqref{def:hat-g} of \(\hat g_k\)).
Applying \eqref{eq:square-completion} with
$a=2k/\cst$ and $b=(\ell-k)/c_1$, we obtain
\[
\sum_{y} \e^{- \frac{\cst}{2} \frac{|y|^2}{k}}
\e^{-c_1 \frac{|w-y|^2}{\ell-k}}  \leq C \frac{k(\ell-k)}{\ell} \e^{- \cst \frac{|w|^2}{\ell}}
\]
where we have chosen \(\cst \leq c_1/2\) so
that \(\frac{2 k}{\cst} + \frac{\ell-k}{c_1} \leq \frac{\ell}{\cst}\),
uniformly for \(\ell>3k\).
All together, this proves that
\[
\sum_{y} \hat g_k(y)[q_{\ell-k}(w-y)-q_\ell(w)]^2  \leq \frac{C k}{(\ell-k)^2} \hat g_{\ell} (w)  \leq \frac{C' k}{ \ell (\ell-k)} \hat g_{\ell} (w) \,,
\]
for \(\ell >3k\).
\end{proof}

\begin{proof}[Proof of Lemma~\ref{lem:w-iter} and~\eqref{eq:backward-weight-initial}]
We use the standard estimates which are easy consequences of \cite[Proposition~1.6]{CSZ19a}:
\begin{align}\label{eq:renewal estimates}
G_\vartheta(r)
\le
\frac{C}{r(1+\log(1/r))^2},
\qquad
\int_0^h G_\vartheta(r)\,\dd r
\le
\frac{C}{1+\log(1/h)},
\qquad
0<h\le1.
\end{align}
Let us first show Lemma~\ref{lem:w-iter} and
then turn to \eqref{eq:backward-weight-initial}.
Recall that \(\mathsf W_{\mathsf T}(t) = \log(\frac{2 \mathsf T}{\mathsf T - t})\) is defined in~\eqref{eq:backward-weight}.

For Lemma~\ref{lem:w-iter}, we first estimate the integral over $t$ in the left-hand side of~\eqref{eq:backward-block-bound}:
\begin{align*}
    \int^{\mathsf T}_s G_{\vartheta}(t-s) \mathsf W_{\mathsf T}(t) \dd t
    =
    \int^{\mathsf T-s}_0 G_{\vartheta}(r)\log \frac{2\mathsf T}{\mathsf T-s-r} \dd r,
\end{align*}
where we have substituted $r\coloneqq t-s$. We split this integral at $r=\frac{\mathsf T-s}{2}$. For $r\in (0,\frac{\mathsf T-s}{2}]$, we have
\begin{align*}
\log\frac{2T}{\mathsf T-s -r}
=
\log\frac{2T}{\mathsf T-s}
+
\log\frac{\mathsf T-s}{\mathsf T-s-r}
\le
\mathsf W_{\mathsf T}(s)+\log 2
\le
2\mathsf W_{\mathsf T}(s).
\end{align*}
Hence together with \eqref{eq:renewal estimates} we have
\begin{align*}
    \int^{\frac{\mathsf T-s}{2}}_0 G_{\vartheta}(r)\log \frac{2\mathsf T}{\mathsf T-s-r} \dd r
    \le
    2 \mathsf W_{\mathsf T}(s)
    \int^{\frac{\mathsf T-s}{2}}_0
    G_{\vartheta}(r)
    \dd r
    \le
    \frac{C}{1+\log \frac{1}{\mathsf T-s}}
    \mathsf W_{\mathsf T}(s).
\end{align*}
For $r\in(\frac{\mathsf T-s}{2},\mathsf T-s)$, we observe that
\begin{equation*}
      \label{eq:int>}
      \begin{split}
          \int^{\mathsf T-s}_{\frac{\mathsf T-s}{2}}
          G_{\vartheta}(r) \, \log & \frac{2\mathsf T}{\mathsf T-s-r}
          \dd r
           \le
          \frac{C}{(\mathsf T-s) \big(1+ \log \frac{1}{\mathsf T-s}\big)^2}
          \int^{\mathsf T-s}_{\frac{\mathsf T-s}{2}}
          \log \frac{2\mathsf T}{\mathsf T-s-r}
          \dd r  \\
          & =
          \frac{C (\mathsf T-s)}{(\mathsf T-s) \big(1+ \log \frac{1}{\mathsf T-s}\big)^2}
          \Big(1+ \log \frac{4\mathsf T}{\mathsf T-s}\Big)
          \le \frac{C}{\big(1+\log \frac{1}{\mathsf T-s}\big)^2}
          \mathsf W_{\mathsf T}(s).
      \end{split}
\end{equation*}
Combining the last two estimates, we thus obtain:
\begin{equation}
      \label{eq:internal-integral}
 \int^{\mathsf T}_s G_{\vartheta}(t-s) \mathsf W_{\mathsf T}(t) \dd t \leq  \frac{C}{1+ \log \frac{1}{\mathsf T-s}} \mathsf W_{\mathsf T} (s)
    \le
    \frac{C}{|\log \mathsf T|}\mathsf W_{\mathsf T}(s).
\end{equation}
Plugged into the left-hand side of~\eqref{eq:backward-block-bound}, this gives
\begin{equation}
      \label{eq:second-step-integral}
      \iint_{\mathsf{S}<s<t<\mathsf{T}}
      G_\vartheta(t-s)
      \frac{\tau\,\mathsf{W}_{\mathsf{T}}(t)}{s(s-\tau)}\,\dd s\,\dd t
       \le
      \frac{C}{ |\log \mathsf T| }
      \int^{\mathsf T}_{\mathsf S}
      \frac{\tau}{s(s-\tau)}
      \mathsf W_{\mathsf T}(s)
      \dd s.
\end{equation}

We now split the above integral at $s = \frac{\mathsf T}{2}$;
note that this is a valid integration interval as we assumed $4\mathsf S\le \mathsf T$.
For $s\in (\mathsf S,\frac{\mathsf T}{2}]$, notice that $\mathsf W_{\mathsf T}(s) = \log \frac{2\mathsf T}{\mathsf T-s} \le \log 4$. Hence
\begin{equation*}
      \begin{split}
          \int^{\frac{\mathsf T}{2}}_{\mathsf S}
          \frac{\tau}{s(s-\tau)}
          \mathsf W_{\mathsf T}(s)
          \dd s
          & \le
          C \int^{\frac{\mathsf T}{2}}_{\mathsf S}
          \frac{\tau}{s(s-\tau)}
          \dd s \\
          & \le
           C\int^\infty_{\mathsf S}
          \Big(\frac{1}{s-\tau} - \frac{1}{s}\Big)
          \dd s
          =
          C \log \frac{1}{1-\tau/\mathsf S} \leq C \frac{\tau}{\mathsf S} \log \frac{2}{1-\tau/\mathsf S}\,,
      \end{split}
\end{equation*}
where the last inequality follows from straightforward calculations
(treating the case \(\tau/\mathsf S \in (0,\frac12]\)
and \(\tau/\mathsf S \in (\frac12,1)\) separately).

It now remains to treat the integral over $s\in (\frac{\mathsf T}{2},\mathsf T)$.
Notice that in this region, $\frac{\tau}{s(s-\tau)}\le \frac{8\tau}{\mathsf T^2}$. Therefore, we have that
\begin{align*}
    \int^{\mathsf T}_{\frac{\mathsf T}{2}}
    \frac{\tau}{s(s-\tau)}
    \mathsf W_{\mathsf T}(s)
    \dd s
    & \le
    \frac{C\tau}{\mathsf T^2}
    \int^{\mathsf T}_{\frac{\mathsf T}{2}}
    \mathsf W_{\mathsf T}(s)
    \dd s
    =
    \frac{C\tau}{\mathsf T^2}
    \int^{\mathsf T}_{\frac{\mathsf T}{2}}
    \log \frac{2\mathsf T}{\mathsf T-s}
    \dd s\notag \\
    & =
    \frac{C\tau }{\mathsf T}
    \int^{\frac{1}{2}}_{0}
    \log \frac{2}{u}
    \dd u
    \le
    \frac{C\tau}{\mathsf S}
    \le
    \frac{C\tau}{\mathsf S}\mathsf W_{\mathsf S}(\tau),
\end{align*}
where we have applied the change of variable $u\coloneqq 1-\frac{s}{\mathsf T}$ above, and in the last inequality we observe that $\mathsf W_{\mathsf S}(\tau)\ge \log 2$ for all $\tau\in (0,\mathsf S)$.
Combined with the previous estimate this shows that
\[
 \int^{\mathsf T}_{\mathsf S}
    \frac{\tau}{s(s-\tau)}
    \mathsf W_{\mathsf T}(s) \leq \frac{C\tau}{\mathsf S}\mathsf W_{\mathsf S}(\tau)
\]
so that plugged into~\eqref{eq:second-step-integral} this
finishes the proof of \eqref{eq:backward-block-bound}.

\smallskip
We now turn to~\eqref{eq:backward-weight-initial}; recall that \(\log(\frac{2\mathsf T}{\mathsf T -t})=\mathsf W_{\mathsf T}(t)\).
We again first estimate the integral over $t$ in the left-hand side thanks to~\eqref{eq:internal-integral}.
Plugged in~\eqref{eq:backward-weight-initial}, this yields
\begin{align*}
    \iint\limits_{\mathsf S< s \le t< \mathsf T}\frac{1}{s}  G_\vartheta (t-s)
      \mathsf W_{\mathsf T}(t)
      \dd s \dd t
      \le
      \frac{C}{|\log \mathsf T|}
      \int^{\mathsf T}_{\mathsf S}
      \frac{1}{s} \mathsf W_{\mathsf T}(s)
      \dd s.
\end{align*}
Similar to the above, we split the integral at $s = \frac{\mathsf T}{2}$:
\begin{align*}
    \int^{\mathsf T/2}_{\mathsf S}
      \frac{1}{s} \mathsf W_{\mathsf T}(s)
      \dd s
      \le
      \log 4 \int^{\mathsf T/2}_{\mathsf S}
      \frac{\dd s}{s}
      \le
      C\log \frac{\mathsf T}{\mathsf S},
\end{align*}
where we have used $\mathsf W_{\mathsf T}(s)\le \log 4$ for all $s\in (\mathsf S, \mathsf T/2)$. On the other hand,
\begin{align*}
    \int^{\mathsf T}_{\mathsf T/2}
      \frac{1}{s} \mathsf W_{\mathsf T}(s)
      \dd s
      \le
      \frac{2}{\mathsf T}
      \int^{\mathsf T}_{\mathsf T/2}
      \mathsf W_{\mathsf T}(s)
      \dd s
      =
      \frac{2}{\mathsf T}
      \int^{\mathsf T/2}_{0}
      \log \frac{2\mathsf T}{u}
      \dd u
      \le C,
\end{align*}
where we have applied the change of variable $u\coloneqq \mathsf T-s$.
Assembling the above inequalities therefore yield the desired bound in~\eqref{eq:backward-weight-initial}.
\end{proof}

\section{Moment estimates: proof of Lemma~\ref{lem:second-moment-hat}}
\label{app:renewal-comparison}

Let us recall that we have set
\begin{equation}
\label{eq:renewal-gap-VD}
 R_s\coloneqq \sum_{j=1}^s q_{2j}(0),
 \qquad \vartheta_N(s)\coloneqq \frac{\pi}{\sigma_N^2}-\pi R_s \,.
\end{equation}
Here $\vartheta_N(s)$ is the remaining distance from the critical
renewal value.
First, we have the following estimate (we use the same notation as in \cite{BCT25,BTZ26}).

\begin{claim}\label{claim:VN}
There exists a constant $C>0$ such that, for every integer
$s\ge0$ with $\vartheta_N(s)>0$,
i.e.\ $\sigma_N^2 R_s< 1$,
\begin{equation}\label{eq:renewal-VN}
 0\le
 \frac{\pi}{\vartheta_N(s)}
 -\sigma_N^2\bbE[Z_s(0)^2]
 \le \frac{C}{\vartheta_N(s)^2}.
\end{equation}
Moreover, given any $t_0>0$, 
there is
$C_{\vartheta,t_0}>0$ such that, for all large $N$
and all integers $1\le s\le t_0N$,
\[
 \sigma_N^2\bbE[Z_s(0)^2]
 \le \frac{C_{\vartheta,t_0}}{\log(1+N/s)}.
\]
\end{claim}

\begin{proof}
The second bound follows from
\cite[Propositions~A.2 and~1.6]{CSZ19a}.
We prove~\eqref{eq:renewal-VN}.

We assume $s \ge 1$, 
since for $s=0$ 
the difference in~\eqref{eq:renewal-VN} vanishes.
The chaos expansion gives
\[
 \sigma_N^2\bbE[Z_s(0)^2]
 =
 \sigma_N^2 \bigg\{ 1 + \sum_{k=1}^{\infty}\sigma_N^{2k}
 \sum_{\substack{n_1,\ldots,n_k\ge1\\
                  n_1+\cdots+n_k\le s}}
 \prod_{j=1}^k q_{2n_j}(0) \bigg\}\,.
\]
On the other hand, since $\sigma_N^2R_s<1$,
recalling the definition of $\vartheta_N$, we can write
\[
 \frac{\pi}{\vartheta_N(s)}
 =
 \sigma_N^2\sum_{k=0}^{\infty}(\sigma_N^2R_s)^k.
\]
Comparing these two expansions, we obtain
\[
 0\le
 \frac{\pi}{\vartheta_N(s)}
 -\sigma_N^2\bbE[Z_s(0)^2]
 =
 \sigma_N^2\sum_{k=1}^{\infty}\sigma_N^{2k}
 \sum_{n_1,\ldots,n_k=1}^{s}
 \one_{\{n_1+\cdots+n_k>s\}}
 \prod_{j=1}^k q_{2n_j}(0)
\]
Bounding $\one_{\{n_1+\cdots+n_k>s\}}\le(n_1+\cdots+n_k)/s$, 
the right hand side at most
\[
\begin{aligned}
 \sigma_N^2
 \sum_{k=1}^{\infty} 
 \sigma_N^{2k}
 \, \frac{k}{s} \,
 \bigg(\sum_{n=1}^s n \, q_{2n}(0)\bigg)
 \bigg(\sum_{n=1}^s q_{2n}(0)\bigg)^{k-1}
 & \le C \sigma_N^4
 \sum_{k=1}^{\infty} k \,
 (\sigma_N^{2} \, R_s)^{k-1 } \\
& \le \frac{C \, \sigma_N^4}{(1-\sigma_N^2R_s)^2}
 \le \frac{C}{\vartheta_N(s)^2} \,,
\end{aligned}
\]
where we used $q_{2n}(0)\le C/n$
and $1-\sigma_N^2R_s=\sigma_N^2\vartheta_N(s)/\pi$.
\end{proof}

We first prove Lemma~\ref{lem:second-moment-hat}
and we then derive Lemma~\ref{lem:moments-blocks}
which gives some important moment estimates.

\begin{proof}[Proof of Lemma~\ref{lem:second-moment-hat}]
The first-collision formula~\eqref{eq:local-replica-formula}
gives
\begin{equation}\label{eq:local-replica-formula-app}
 \bbE[\widehat Z_{s,t}(0)^2]
 =
 1+\sigma_N^2\sum_{u=s+1}^{t}
 q_{2u}(0)\, \bbE[Z_{t-u}(0)^2].
\end{equation}
Since $\vartheta_N(s) = \frac{\pi}{\sigma_N^2}-\pi R_s = \vartheta_N(t) + \pi (R_t-R_s)$, we can write
\[
 \frac{\vartheta_N(s)}{\vartheta_N(t)}
 =
 1+\frac{\pi}{\vartheta_N(t)}
 \sum_{u=s+1}^{t}q_{2u}(0) \,,
\]
therefore
\[
 0\le
 \frac{\vartheta_N(s)}{\vartheta_N(t)}
 -\bbE[\widehat Z_{s,t}(0)^2]
 =
 \sum_{u=s+1}^{t}q_{2u}(0)
 \bigg(
 \frac{\pi}{\vartheta_N(t)}
 -\sigma_N^2\bbE[Z_{t-u}(0)^2]
 \bigg)\,.
\]
It is convenient to replace $\frac{\pi}{\vartheta_N(t)}$ by $\frac{\pi}{\vartheta_N(t-u)}$,
in order to apply Claim~\ref{claim:VN}.
To this purpose, we note that
\begin{equation}\label{eq:identity-vartheta}
 \frac1{\vartheta_N(t)}-\frac1{\vartheta_N(t-u)}
 =
 \frac{\pi(R_t-R_{t-u})}
 {\vartheta_N(t)\vartheta_N(t-u)}
 \le
 \frac{\pi(R_t-R_{t-u})}{\vartheta_N(t)^2} \,,
\end{equation}
because $\vartheta_N(t-u)\ge\vartheta_N(t)>0$,
therefore Claim~\ref{claim:VN} for $s=t-u$ yields
\begin{equation}\label{eq:second-moment-error}
\begin{aligned}
& \sum_{u=s+1}^{t}q_{2u}(0)
 \bigg(
 \frac{\pi}{\vartheta_N(t)}
 -\sigma_N^2\bbE[Z_{t-u}(0)^2]
 \bigg)
 \le
 \sum_{u=s+1}^{t}q_{2u}(0)
 \frac{\pi^2(R_t-R_{t-u})+C}{\vartheta_N(t)^2}\\
 &\qquad\qquad \le
 \frac{C}{\vartheta_N(t)^2}
 \Bigg\{ \sum_{u=s+1}^{t}q_{2u}(0)
 \sum_{v=t-u+1}^t q_{2v}(0)  + R_t - R_s \Bigg\} \,.
\end{aligned}
\end{equation}
Using $q_{2n}(0)\le C/n$, we see that the double sum is uniformly bounded by a multiple of
\[
 \sum_{u=1}^{t}\frac1u
       \sum_{v=t-u+1}^{t}\frac1v
 =
 \sum_{\substack{u,v \in \{1,\ldots, t\} \\ u+v > t}} \frac1{uv}
 \le \sum_{\substack{u,v \in \{1,\ldots, t\} \\ u+v > t}} \frac{1}{t}\, \frac{u+v}{uv}
 = \frac{2}{t} \sum_{u=1}^t \frac{1}{u}
 \sum_{v=t-u+1}^{t} 1
 = 2 \,.
\]
This proves
the first bound of the lemma.
Dividing by $\vartheta_N(s)/\vartheta_N(t)$ gives
\eqref{eq:consequence-lem-SM}.
Finally, if $1\le s\le t/2$, then
$1+R_t-R_s\le C\log(t/s)$, which proves the last assertion.
\end{proof}

We next derive second and fourth moment bounds for the
partition function on each geometric block.

\begin{lemma}\label{lem:moments-blocks}
Recall the definition of $\mathsf{Z}_i$ from \eqref{def:Zi},
so in particular \(\mathsf Z_i(q_{N_{i-1}}) = Z_{N_{i-1},N_i}(q_{N_{i-1}},\one)\).
Then there exists a constant \(C=C_{\rho}\) such that, for any $1\le i\le K$, we have
\begin{equation}\label{eq:app-block-moments}
 \begin{aligned}
 \Var\bigl(\mathsf Z_i(q_{N_{i-1}})\bigr)&\leq  \frac{C}{K+1-i} \,,\\
 \bbE\pigl[\bigl|\mathsf Z_i(q_{N_{i-1}})-1\bigr|^4\pigr]
 &\le\frac{C}{(K+1-i)^2} \, .
 \end{aligned}
\end{equation}
\end{lemma}

Let us stress that Lemma~\ref{lem:moments-blocks}
also treats the initial and final block, which satisfy respectively
\[
\Var\bigl(\mathsf Z_1(0)\bigr)\le\frac{C_\rho}{\log N} \,,
\qquad
\bbE\bigl|\mathsf Z_1(0)-1\bigr|^4
 \le\frac{C_\rho}{(\log N)^2}
\]
and
\[
\Var\bigl(\mathsf Z_K(q_{N_{K-1}})\bigr)\le C_\rho,
\qquad
\bbE\bigl|\mathsf Z_K(q_{N_{K-1}})-1\bigr|^4\le C_\rho.
\]

\begin{proof}
By the first inequality in \eqref{eq:rough-bound},
\[
 \Var\bigl(\mathsf Z_i(q_{N_{i-1}})\bigr)
 =
 \bbE\bigl[\widehat Z_{N_{i-1},N_i}(0)^2\bigr]-1
 \le
 C\frac{R_{N_i}-R_{N_{i-1}}}{\log(1+N/N_i)},
 \qquad 1\le i\le K.
\]
We bound the numerator and denominator separately.
By the grid definition \eqref{eq:adapted-grid},
\[
 N\rho^{i-K}\le N_i\le 3N\rho^{i-K},
 \qquad 1\le i\le K.
\]
Consequently, for $i\ge2$, we have
$N_i/N_{i-1}\le3\rho$ and hence, using $q_{2r}(0)\le C/r$,
\[
 R_{N_i}-R_{N_{i-1}}
 \le C\log\frac{N_i}{N_{i-1}}
 \le C\log(3\rho).
\]
For $i=1$, the numerator is also bounded by some constant $C_\rho$,
since $N_0=0$ and $N_1\le2\rho$.
Bounding the denominator by
\(
 \log(1+N/N_i)
 \ge \log(1+\rho^{K-i}/3)
 \ge c_\rho(K+1-i)
\),
we conclude that
\[
 \Var\bigl(\mathsf Z_i(q_{N_{i-1}})\bigr)
 \le \frac{C_\rho}{K+1-i},
 \qquad 1\le i\le K.
\]

As far as the fourth-moment is concerned, the bound follows from applying \cite[Theorem~1.11]{CSZ25},
which precisely shows that
\[
\bbE\bigl[ |Z_{N_{i-1},N_i}(q_{N_{i-1}}) - 1|^4\bigr]
\leq \mathfrak{C}_4 \Var \bigl(Z_{N_{i-1},N_i}(q_{N_{i-1}})\bigr)^{2} \leq \frac{C}{(K+1-i)^2}\,,
\]
where the last inequality follows from the bound
on \(\Var(\mathsf Z_i(q_{N_{i-1}}))\) obtained above.
\end{proof}

\section{Declaration on the use of AI}
\label{app:AI-declaration}

The results presented here have been developed through extensive
interactions between the authors and AI (ChatGPT~5.5,~5.6 and~6), which
contributed to the tightness argument, the proof of convergence to
a unique limit, and its non-Gaussianity. We describe below some of
the main mathematical contributions and changes of strategy,
without attempting to give a complete account of these interactions.

\begin{enumerate}
\item \textbf{Tightness and the initial input.}
We began by asking whether centring and diffusive rescaling,
without an additional amplitude normalisation, suffice for a
nontrivial limit. We suggested Gu--Tsai's geometric block
factorisation~\cite{GT26} as a starting point and also supplied
Chen's Gaussian convexity theorem~\cite{Chen26}, which had already
been used in~\cite{LX26}. From this
theorem, AI derived a covariance bound and combined it with
negative-moment and second-moment estimates to obtain tightness
of the spatial logarithmic field. Further exchanges, prompted by
our questions and requests for checks, led to revisions and
simplifications.

\item \textbf{Developing the convergence strategy.}
Using tightness as an input, we then sought convergence of the
discrete field and a construction of the limit directly from
the Stochastic Heat Flow (SHF). We suggested revisiting the
local-remainder decomposition in~\cite{CSZ20}, which helped refine
a decomposition at two intermediate times. The first time $m=[\delta N]$ defines a regularisation by
removing the disorder on the initial interval $(0,m]$.
The second separates a local correction from a comparison
of two polymer endpoint distributions evolved through
the same future environment.

\item \textbf{An unsuccessful switch-off construction.}
To pass from the comparison of partition functions to their
logarithms, we still needed suitable moment bounds. We suggested
introducing auxiliary partition functions with disorder removed
on selected geometric time intervals, adapting a construction
from~\cite{CSZ25}, so that conditional Gaussian
convexity~\cite{Chen26} could be applied. AI developed this
suggestion into a sparse switch-off construction and explored
several refinements. This initially appeared to close the
remaining estimate, but a subsequent audit identified a gap
in a contraction argument. We abandoned this auxiliary
construction.

\item \textbf{Completion and refinement of convergence.}
When we asked AI to reconsider the available approaches
and look for simpler arguments, it proposed a simpler use of
conditional Gaussian convexity. Further exchanges developed
this into the logarithmic moment bounds needed to complete
the convergence proof, without the additional switch-off
construction. The regularisation remained the one introduced
above, with disorder removed only on $(0,m]$, and the proof
retained the decomposition at two intermediate times and the
supporting geometric block estimates. Subsequent exchanges
led to simplifications and to the extension to arbitrary
regularisation parameters tending to zero. 
The resulting
argument proves Theorem~\ref{thm:main}: the discrete centred
log-partition functions converge to a unique limit, constructed
as a measurable functional of the SHF by proving that its
regularised logarithmic observables are Cauchy.

\item \textbf{Non-Gaussianity.}
We initiated the investigation of non-Gaussianity of subsequential limits
(before having the proof of convergence to a unique limit) while keeping
open the possibility of a Gaussian limit. After several attempts,
we suggested revisiting earlier unsuccessful approaches.
AI then developed a proof using Gaussian rigidity under
bounded perturbations in the terminal condition;
a related rigidity result for convex functions of Gaussian vectors is given in \cite[Theorem~3.4]{FMP}.
Exchanges led to clarifications and simplifications of the proof.
\end{enumerate}

We supplied structural ideas, assessed the proposed arguments,
and guided the choice and revision of approaches. AI
developed these suggestions and generated further proof ideas that were not contained
in the original prompts. Some proposals were fruitful; others contained
errors or introduced unnecessary complications. Checking and
revising these arguments was an essential part of the work.
We take responsibility for the mathematical claims and the
final manuscript.

\end{document}